\documentclass[reqno,11pt]{amsart}
\pdfoutput1

\usepackage{amssymb,color,hyperref,mathrsfs,stmaryrd, tikz-cd}
\usepackage{amsmath}
\usepackage{mathtools}

\usepackage{xcolor}
\usepackage[curve,matrix,arrow]{xy}

\newtheorem{theorem}{Theorem}[section]

\newtheorem{lemma}[theorem]{Lemma}
\newtheorem{proposition}[theorem]{Proposition}
\newtheorem{corollary}[theorem]{Corollary}
\newtheorem{theoremA}{Theorem}

\newtheorem{corollaryA}[theoremA]{Corollary}

\theoremstyle{definition}
\newtheorem{definition}[theorem]{Definition}

\newtheorem{hypothesis}[theorem]{Hypothesis}

\makeatletter
\let\save@mathaccent\mathaccent
\newcommand*\if@single[3]{%
	\setbox0\hbox{${\mathaccent"0362{#1}}^H$}%
	\setbox2\hbox{${\mathaccent"0362{\kern0pt#1}}^H$}%
	\ifdim\ht0=\ht2 #3\else #2\fi
}
\newcommand*\rel@kern[1]{\kern#1\dimexpr\macc@kerna}
\newcommand*\widebar[1]{\@ifnextchar^{{\wide@bar{#1}{0}}}{\wide@bar{#1}{1}}}
\newcommand*\wide@bar[2]{\if@single{#1}{\wide@bar@{#1}{#2}{1}}{\wide@bar@{#1}{#2}{2}}}
\newcommand*\wide@bar@[3]{%
	\begingroup
	\def\mathaccent##1##2{%
		\let\mathaccent\save@mathaccent
		\if#32 \let\macc@nucleus\first@char \fi
		\setbox\z@\hbox{$\macc@style{\macc@nucleus}_{}$}%
		\setbox\tw@\hbox{$\macc@style{\macc@nucleus}{}_{}$}%
		\dimen@\wd\tw@
		\advance\dimen@-\wd\z@
		\divide\dimen@ 3
		\@tempdima\wd\tw@
		\advance\@tempdima-\scriptspace
		\divide\@tempdima 10
		\advance\dimen@-\@tempdima
		\ifdim\dimen@>\z@ \dimen@0pt\fi
		\rel@kern{0.6}\kern-\dimen@
		\if#31
		\overline{\rel@kern{-0.6}\kern\dimen@\macc@nucleus\rel@kern{0.4}\kern\dimen@}%
		\advance\dimen@0.4\dimexpr\macc@kerna
		\let\final@kern#2%
		\ifdim\dimen@<\z@ \let\final@kern1\fi
		\if\final@kern1 \kern-\dimen@\fi
		\else
		\overline{\rel@kern{-0.6}\kern\dimen@#1}%
		\fi
	}%
	\macc@depth\@ne
	\let\math@bgroup\@empty \let\math@egroup\macc@set@skewchar
	\mathsurround\z@ \frozen@everymath{\mathgroup\macc@group\relax}%
	\macc@set@skewchar\relax
	\let\mathaccentV\macc@nested@a
	\if#31
	\macc@nested@a\relax111{#1}
		\else
		\def\gobble@till@marker##1\endmarker{}%
		\futurelet\first@char\gobble@till@marker#1\endmarker
		\ifcat\noexpand\first@char A\else
		\def\first@char{}%
		\fi
		\macc@nested@a\relax111{\first@char}%
		\fi
		\endgroup
	}
	\makeatother

\author{Julian Kaspczyk}
\title{A characterization of the groups $PSL_n(q)$ and $PSU_n(q)$ by their $2$-fusion systems, $q$ odd}
\address{Institute of Mathematics, University of Aberdeen, Fraser Noble Building, Aberdeen AB24 3UE, UK}
\address{Technische Universität Dresden, Institut für Algebra, 01069 Dresden, Germany}
\email{julian.kaspczyk@gmail.com}

\begin{document}
	\maketitle
	
	\begin{abstract}
		Let $q$ be a nontrivial odd prime power, and let $n \ge 2$ be a natural number with $(n,q) \ne (2,3)$. We characterize the groups $PSL_n(q)$ and $PSU_n(q)$ by their $2$-fusion systems. This contributes to a programme of Aschbacher aiming at a simplified proof of the classification of finite simple groups.   
	\end{abstract}
	
	
	\section{Introduction}
	The classification of finite simple groups (CFSG) is one of the greatest achievements in the history of mathematics. Its proof required around 15,000 pages and spreads out over many hundred articles in various journals. Many mathematicians from all over the world were involved in the proof, whose final steps were published in 2004 by Aschbacher and Smith, after it was prematurely announced as finished already in 1983. Because of its extreme length, a simplified and shortened proof of the CFSG would be very valuable. There are three programmes working towards this goal: the Gorenstein-Lyons-Solomon programme (see \cite{GLS0}), the Meierfrankenfeld-Stellmacher-Stroth programme (see \cite{MSS}) and Aschbacher’s programme. 
	
	The goal of Aschbacher’s programme is to obtain a new proof of the CFSG by using \textit{fusion systems}. The standard examples of fusion systems are the fusion categories of finite groups over $p$-subgroups ($p$ a prime). If $G$ is a finite group and $S$ is a $p$-subgroup of $G$ for some prime $p$, then the \textit{fusion category} of $G$ over $S$ is defined to be the category $\mathcal{F}_S(G)$ given as follows: the objects of $\mathcal{F}_S(G)$ are precisely the subgroups of $S$, the morphisms in $\mathcal{F}_S(G)$ are precisely the group homomorphisms between subgroups of $S$ induced by conjugation in $G$, and the composition of morphisms in $\mathcal{F}_S(G)$ is the usual composition of group homomorphisms. Abstract fusion systems are a generalization of this concept. A fusion system over a finite $p$-group $S$, where $p$ is a prime, is a category whose objects are the subgroups of $S$ and whose morphisms behave as if they are induced by conjugation inside a finite group containing $S$ as a $p$-subgroup. For the precise definition, we refer to \cite[Part I, Definition 2.1]{AKO}. A fusion system is called \textit{saturated} if it satisfies certain axioms motivated by properties of fusion categories of finite groups over Sylow subgroups (see \cite[Part I, Definition 2.2]{AKO}). If $G$ is a finite group and $S_1, S_2 \in \mathrm{Syl}_p(G)$ for some prime $p$, then $\mathcal{F}_{S_1}(G)$ and $\mathcal{F}_{S_2}(G)$ are easily seen to be isomorphic (in the sense of \cite[p. 560]{AO}). Given a finite group $G$, a prime $p$ and a Sylow $p$-subgroup $S$ of $G$, we refer to $\mathcal{F}_S(G)$ as the \textit{$p$-fusion system} of $G$.
	
	Originally considered by the representation theorist Puig, fusion systems have become an object of active research in finite group theory, representation theory and algebraic topology. It has always been a problem of great interest in the theory of fusion systems to translate group-theoretic concepts into suitable concepts for fusion systems. For example, there is a notion of normalizers and centralizers of $p$-subgroups in fusion systems, a notion of the center of a fusion system, a notion of factor systems, a notion of normal subsystems of saturated fusion systems and a notion of simple saturated fusion systems (see \cite[Parts I and II]{AKO}). Roughly speaking, Aschbacher’s programme consists of the following two steps. 
	\begin{enumerate}
		\item[1.] Classify the simple saturated fusion systems on finite $2$-groups. Use the original proof of the CFSG as a “template”. 
		\item[2.] Use the first step to give a new and simplified proof of the CFSG. 
	\end{enumerate} 
	
	There is the hope that several steps of the original proof of the CFSG become easier when working with fusion systems. For example, in the original proof of the CFSG, the study of centralizers of involutions plays an important role. The $2'$-cores of the involution centralizers, i.e. their largest normal odd order subgroups, cause serious difficulties and are obstructions to many arguments. Such difficulties are not present in fusion systems since cores do not exist in fusion systems. This is suggested by the well-known fact that the $2$-fusion system of a finite group $G$ is isomorphic to the $2$-fusion system of $G/O(G)$, where $O(G)$ denotes the $2'$-core of $G$. For an outline of and recent progress on Aschbacher’s programme, we refer to \cite{Aschbacher2019}. 
	
	So far, Aschbacher’s programme has focused mainly on Step~1, while not much has been done on Step~2. An important part of Step~2 is to identify finite simple groups from their $2$-fusion systems. The present paper contributes to Step~2 of Aschbacher’s programme by characterizing the finite simple groups $PSL_n(q)$ and $PSU_n(q)$ in terms of their $2$-fusion systems, where $n \ge 2$ and where $q$ is a nontrivial odd prime power with $(n,q) \ne (2,3)$. 
	
	In order to state our results, we introduce some notation and recall some definitions. Let $G$ be a finite group. A \textit{component} of $G$ is a quasisimple subnormal subgroup of $G$, and a \textit{$2$-component} of $G$ is a perfect subnormal subgroup $L$ of $G$ such that $L/O(L)$ is quasisimple. The natural homomorphism $G \rightarrow G/O(G)$ induces a one-to-one correspondence between the set of $2$-components of $G$ and the set of components of $G/O(G)$ (see \cite[Proposition 4.7]{GLS2}). We use $Z^{*}(G)$ to denote the full preimage of the center $Z(G/O(G))$ in $G$. In Step~2 of Aschbacher's programme, one may assume that a finite group $G$ is a minimal counterexample to the CFSG. Such a group $G$ has the following property.
	
	\begin{align}
		\label{CK}
		\tag{$\mathcal{CK}$}
		&\mbox{Whenever $x\in G$ is an involution and $J$ is a $2$-component of $C_G(x)$},\\ 
		&\mbox{then $J/Z^*(J)$ is a known finite simple group.}\nonumber
	\end{align}
	
	By a known finite simple group, we mean a finite simple group appearing in the statement of the CFSG. 
	
	For each integer $n \ne 0$, we use $n_2$ to denote the $2$-part of $n$, i.e. the largest power of $2$ dividing $n$. Given odd integers $a, b$ with $|a|, |b| > 1$, we write $a \sim b$ provided that $(a-1)_2 = (b-1)_2$ and $(a+1)_2 = (b+1)_2$. If $q$ is a nontrivial prime power and if $n$ is a positive integer, then we write $PSL_n^{+}(q)$ for $PSL_n(q)$ and $PSL_n^{-}(q)$ for $PSU_n(q)$. With this notation, we can now state our main results.  
	
	\begin{theoremA}
		\label{A}
		Let $q$ be a nontrivial odd prime power, and let $n \ge 2$ be a natural number. Let $G$ be a finite simple group. Suppose that $G$ satisfies (\ref{CK}) if $n \ge 6$. Then the $2$-fusion system of $G$ is isomorphic to the $2$-fusion system of $PSL_n(q)$ if and only if one of the following holds: 
		\begin{enumerate}
			\item[(i)] $G \cong PSL_n^{\varepsilon}(q^{*})$ for some nontrivial odd prime power $q^{*}$ and some $\varepsilon \in \lbrace +,- \rbrace$ with $\varepsilon q^{*} \sim q$;
			\item[(ii)] $n = 2$, $\vert PSL_2(q) \vert_2 = 8$, and $G \cong A_7$;
			\item[(iii)] $n = 3$, $(q+1)_2 = 4$, and $G \cong M_{11}$. 
		\end{enumerate} 
	\end{theoremA}
	
	Our second main result is an extension of Theorem \ref{A}. In order to state it, we briefly mention some concepts from the local theory of fusion systems. Let $\mathcal{F}$ be a saturated fusion system on a finite $p$-group $S$ for some prime $p$, and let $\mathcal{E}$ be a normal subsystem of $\mathcal{F}$. In \cite[Chapter 6]{generalizedfittingsubsystem}, Aschbacher introduced a subgroup $C_S(\mathcal{E})$ of $S$, which plays the role of the centralizer of $\mathcal{E}$ in $S$. In \cite[Chapter 9]{generalizedfittingsubsystem}, he defined a normal subsystem $F^{*}(\mathcal{F})$ of $\mathcal{F}$, called the \textit{generalized Fitting subsystem} of $\mathcal{F}$, and proved that $C_S(F^{*}(\mathcal{F})) = Z(F^{*}(\mathcal{F}))$, where the latter denotes the center of $F^{*}(\mathcal{F})$.
	
	\begin{theoremA}
		\label{B}
		Let $q$ be a nontrivial odd prime power, and let $n \ge 2$ be a natural number. If $n = 2$, suppose that $q \equiv 1$ or $7 \mod 8$. Let $G$ be a finite simple group, and let $S$ be a Sylow $2$-subgroup of $G$. Suppose that $\mathcal{F}_S(G)$ has a normal subsystem $\mathcal{E}$ on a subgroup $T$ of $S$ such that $\mathcal{E}$ is isomorphic to the $2$-fusion system of $PSL_n(q)$ and such that $C_S(\mathcal{E}) = 1$. Then $\mathcal{F}_S(G)$ is isomorphic to the $2$-fusion system of $PSL_n(q)$. In particular, if $n \le 5$ or if $G$ satisfies (\ref{CK}), then one of the properties (i)-(iii) from Theorem \ref{A} holds. 
	\end{theoremA}
	
	\begin{corollaryA}
		\label{C} 
		Let $q$ be a nontrivial odd prime power, and let $n \ge 2$ be a natural number. If $n = 2$, suppose that $q \equiv 1$ or $7 \mod 8$. Let $G$ be a finite simple group, and let $S$ be a Sylow $2$-subgroup of $G$. Suppose that $F^{*}(\mathcal{F}_S(G))$ is isomorphic to the $2$-fusion system of $PSL_n(q)$. Then $\mathcal{F}_S(G)$ is isomorphic to the $2$-fusion system of $PSL_n(q)$. In particular, if $n \le 5$ or if $G$ satisfies (\ref{CK}), then one of the properties (i)-(iii) from Theorem \ref{A} holds. 
	\end{corollaryA}
	
	The paper is organized as follows. In Sections \ref{preliminaries_groups_fusion} and \ref{auxiliary_results_linear_unitary}, we collect several results needed for the proofs of our main results. Preliminary results on abstract finite groups and abstract fusion systems are proved in Section \ref{preliminaries_groups_fusion}. Section \ref{auxiliary_results_linear_unitary} presents some results on linear and unitary groups over finite fields, mainly focussing on $2$-local properties and on the automorphisms of these groups. 
	
	In Section \ref{small_cases}, we will verify Theorem \ref{A} for the case $n \le 5$. Our proofs strongly depend on work of Gorenstein and Walter \cite{GorensteinWalter} (for $n = 2$), on work of Alperin, Brauer and Gorenstein \cite{AlperinBrauerGorenstein1}, \cite{AlperinBrauerGorenstein2} (for $n = 3$) and on work of Mason \cite{Mason2}, \cite{Mason1}, \cite{Mason3} (for $n = 4$ and $n = 5$). 
	
	For $n \ge 6$, we will prove Theorem \ref{A} by induction over $n$. In order to do so, we will consider a finite group $G$ realizing the $2$-fusion system of $PSL_n(q)$, where $q$ is a nontrivial odd prime power and where $n \ge 6$ is a natural number such that Theorem \ref{A} is true with $m$ instead of $n$ for any natural number $m$ with $6 \le m < n$. We will also assume that $O(G) = 1$ and that $G$ satisfies (\ref{CK}). To prove that Theorem \ref{A} is satisfied for the natural number $n$, we will prove the existence of a normal subgroup $G_0$ of $G$ such that $G_0$ is isomorphic to a nontrivial quotient of $SL_n^{\varepsilon}(q^{*})$ for some nontrivial odd prime power $q^{*}$ and some $\varepsilon \in \lbrace +, - \rbrace$ with $\varepsilon q^{*} \sim q$. This will happen in Sections \ref{Preliminary discussion and notation}-\ref{G0}.
	
	In Section \ref{Preliminary discussion and notation}, we will introduce some notation and prove some preliminary lemmas. Section \ref{2-components of involution centralizers} describes the $2$-components of the centralizers of involutions of $G$. In Section \ref{components_of_centralizers}, we will use signalizer functor methods to describe the components of the centralizers of certain involutions of $G$. This will be used in Section \ref{G0} to construct the subgroup $G_0$ of $G$. One of the main tools here will be a version of the Curtis-Tits theorem \cite[Chapter 13, Theorem 1.4]{GLS8} and a related theorem of Phan reproved by Bennett and Shpectorov in \cite{BennettShpectorov}. 
	
	Finally, in Section \ref{proofs_main_results}, we will give a full proof of Theorem \ref{A} (basically summarizing Sections \ref{small_cases}-\ref{G0}), and we will prove Theorem \ref{B} and Corollary \ref{C}.
	
	\medskip
	
	\textit{Notation and Terminology.} Our notation and terminology are fairly standard. The reader is referred to \cite{Gorenstein}, \cite{GLS2}, \cite{KurzweilStellmacher} for unfamiliar definitions on groups and to \cite{AKO}, \cite{Craven} for unfamiliar definitions on fusion systems. 
	
	However, we shall now explain some particularly important notation and definitions (before stating our main results, we already introduced some other important definitions). 
	
	Given a map $\alpha: A \rightarrow B$ and an element or a subset $X$ of $A$, we write $X^{\alpha}$ for the image of $X$ under $\alpha$. Also, if $C \subseteq A$ and $D \subseteq B$ such that $C^{\alpha} \subseteq D$, we use $\alpha|_{C,D}$ to denote the map $C \rightarrow D, c \mapsto c^{\alpha}$. Given two maps $\alpha: A \rightarrow B$ and $\beta: B \rightarrow C$, we write $\alpha\beta$ for the map $A \rightarrow C, a \mapsto (a^{\alpha})^{\beta}$. 
	
	Sometimes, we will interprete the symbols $+$ and $-$ as the integers $1$ and $-1$, respectively. For example, if $n$ is an integer and if $\varepsilon$ is assumed to be an element of $\lbrace +,- \rbrace$, then $n \equiv \varepsilon \mod 4$ shall express that $n \equiv 1 \mod 4$ if $\varepsilon = +$ and that $n \equiv -1 \mod 4$ if $\varepsilon = -$. 
	
	Let $G$ be a finite group. We write $G^{\#}$ for the set of non-identity elements of $G$. Given an element $g$ of $G$ and an element or a subset $X$ of $G$, we write $X^g$ for $g^{-1}Xg$. The inner automorphism $G \rightarrow G, x \mapsto x^g$ is denoted by $c_g$. For subgroups $Q$ and $H$ of $G$, we write $\mathrm{Aut}_H(Q)$ for the subgroup of $\mathrm{Aut}(Q)$ consisting of all automorphisms of $Q$ of the form $c_h|_{Q,Q}$, where $h \in N_H(Q)$. 
	
	We write $L(G)$ for the subgroup of $G$ generated by the components of $G$ and $L_{2'}(G)$ for the subgroup of $G$ generated by the $2$-components of $G$. We say that $G$ is \textit{core-free} if $O(G) = 1$. If $G$ is core-free and if $L$ is a subnormal subgroup of $G$, then $L$ is said to be a \textit{solvable $2$-component} of $G$ if $L \cong SL_2(3)$ or $PSL_2(3)$. 
	
	Let $n$ be a natural number. Then we use $E_{2^n}$ to denote an elementary abelian $2$-group of order $2^n$, and we say that $n$ is the \textit{rank} of $E_{2^n}$. The maximal rank of an elementary abelian $2$-subgroup of a finite $2$-group $S$ is said to be the \textit{rank} of $S$. It is denoted by $m(S)$. 
	
	Now let $p$ be a prime, and let $\mathcal{F}$ be a fusion system on a finite $p$-group $S$. Then $S$ is said to be the \textit{Sylow group} of $\mathcal{F}$, and $\mathcal{F}$ is said to be \textit{nilpotent} if $\mathcal{F} = \mathcal{F}_S(S)$. Given a fusion system $\mathcal{F}_1$ on a finite $p$-group $S_1$, we say that $\mathcal{F}$ and $\mathcal{F}_1$ are \textit{isomorphic} if there is a group isomorphism $\varphi: S \rightarrow S_1$ such that 
	\begin{equation*}
		\mathrm{Hom}_{\mathcal{F}_1}(Q^{\varphi},R^{\varphi}) = \lbrace (\varphi^{-1}|_{Q^{\varphi},Q})\psi(\varphi|_{R,R^{\varphi}}) \ \vert \ \psi \in \mathrm{Hom}_{\mathcal{F}}(Q,R) \rbrace
	\end{equation*} 
	for all $Q,R \le S$. In this case, we say that $\varphi$ \textit{induces an isomorphism} from $\mathcal{F}$ to $\mathcal{F}_1$. Let $Q$ be a normal subgroup of $S$. If $P$ and $R$ are subgroups of $S$ containing $Q$ and if $\alpha: P \rightarrow R$ is a morphism in $\mathcal{F}$ such that $Q^{\alpha} = Q$, we write $\alpha/Q$ for the group homomorphism $P/Q \rightarrow R/Q$ induced by $\alpha$. The fusion system $\mathcal{F}/Q$ on $S/Q$ with $\mathrm{Hom}_{\mathcal{F}/Q}(P/Q,R/Q) = \lbrace \alpha/Q \ \vert \ \alpha \in \mathrm{Hom}_{\mathcal{F}}(P,R), Q^{\alpha} = Q \rbrace$ for all $P, R \le S$ containing $Q$ is said to be the \textit{factor system} of $\mathcal{F}$ modulo $Q$. 
	
	Suppose now that $\mathcal{F}$ is saturated. We write $\mathfrak{foc}(\mathcal{F})$ for the focal subgroup of $\mathcal{F}$ and $\mathfrak{hnp}(\mathcal{F})$ for the hyperfocal subgroup of $\mathcal{F}$. We say that $\mathcal{F}$ is \textit{quasisimple} if $\mathcal{F}/Z(\mathcal{F})$ is simple and $\mathfrak{foc}(\mathcal{F}) = S$. A \textit{component} of $\mathcal{F}$ is a subnormal quasisimple subsystem of $\mathcal{F}$. Given a normal subsystem $\mathcal{E}$ of $S$ and a subgroup $R$ of $S$, we write $\mathcal{E}R$ for the product of $\mathcal{E}$ and $R$, as defined in \cite[Chapter 8]{generalizedfittingsubsystem}.

	\section{Preliminaries on finite groups and fusion systems} 
	\label{preliminaries_groups_fusion}
	In this section, we present some general results on finite groups and fusion systems.
	
	\subsection{Preliminaries on finite groups}
	\begin{lemma} (\cite[3.2.8]{KurzweilStellmacher})
		\label{normalizers_p-subgroups} 
		Let $G$ be a finite group, and let $N$ be a normal $p'$-subgroup of $G$ for some prime $p$. Set $\widebar G := G/N$. If $R$ is a $p$-subgroup of $G$, then we have $N_{\overline{G}}(\overline{R}) = \overline{N_G(R)}$ and $C_{\overline{G}}(\overline{R}) = \overline{C_G(R)}$. 
	\end{lemma}
	
	\begin{corollary} 
		\label{centralizers_p-elements} 
		Let $G$ be a finite group, and let $N$ be a normal $p'$-subgroup of $G$ for some prime $p$. Set $\widebar G := G/N$. If $x \in G$ has order $p$, then we have $C_{\widebar G}(\widebar x) = \widebar{C_G(x)}$. 
	\end{corollary}  
	
	\begin{lemma}
		\label{E8 subgroups of central quotients} 
		Let $G$ be a finite group, and let $Z$ be a cyclic central subgroup of $G$. Then each $E_8$-subgroup of $G/Z$ has an involution which is the image of an involution of $G$. 
	\end{lemma}
	
	\begin{proof}
		Let $Z \le E \le G$ such that $E/Z \cong E_8$. Let $R$ be a Sylow $2$-subgroup of $E$. Then $E = RZ$. It suffices to show that $R$ has an involution not lying in $R \cap Z$. Assume that any involution of $R$ is an element of $R \cap Z$. Then $R$ has a unique involution since $Z$ is cyclic. We have $R/(R \cap Z) \cong RZ/Z = E/Z \cong E_8$, and so $R$ is not cyclic. Applying \cite[5.3.7]{KurzweilStellmacher}, we conclude that $R$ is generalized quaternion. In particular, $Z(R)$ has order $2$, and so we have $R \cap Z = Z(R)$. Since $R$ is a generalized quaternion group, $R/Z(R)$ is dihedral. In particular, $E/Z \cong R/(R \cap Z) = R/Z(R) \not\cong E_8$. This contradiction shows that $R$ has an involution not lying in $R \cap Z$, as required.
	\end{proof}
	
	The following proposition is well-known. We include a proof since we could not find a reference in which it appears in the form given here. 
	
	\begin{proposition}
		\label{2-components modulo odd order subgroup}
		Let $G$ be a finite group, and let $N$ be a normal subgroup of $G$ with odd order. If $L$ is a $2$-component of $G$, then $LN/N$ is a $2$-component of $G/N$. The map from the set of $2$-components of $G$ to the set of $2$-components of $G/N$ sending each $2$-component $L$ of $G$ to $LN/N$ is a bijection. Moreover, if $N \le K \le G$ and $K/N$ is a $2$-component of $G/N$, then $O^{2'}(K)$ is the associated $2$-component of $G$. 
	\end{proposition}
	
	\begin{proof}
		Let $L$ be a $2$-component of $G$. Hence, $L$ is a perfect subnormal subgroup of $G$ such that $L/O(L)$ is quasisimple. Clearly, $LN/N$ is perfect and subnormal in $G/N$. Also, we have $(LN/N)/O(LN/N) \cong L/O(L)$, and so $(LN/N)/O(LN/N)$ is quasisimple. It follows that $LN/N$ is a $2$-component of $G/N$. 
		
		Let $N \le K \le G$ such that $K/N$ is a $2$-component of $G/N$. In order to prove the second statement of the proposition, it is enough to show that there is precisely one $2$-component $L$ of $G$ such that $LN/N = K/N$. 
		
		Since $K/N$ is subnormal in $G/N$, we have that $K$ is subnormal in $G$. Therefore, $L := O^{2'}(K)$ is subnormal in $G$. Since $O^{2'}(K/N) = K/N$, we have that $K/N = LN/N$. Clearly, $O^{2'}(L) = L$. We have $L/O(L) \cong (LN/N)/O(LN/N) = (K/N)/O(K/N)$, and so $L/O(L)$ is quasisimple. Applying \cite[Lemma 4.8]{GLS2}, we conclude that $L$ is a $2$-component of $G$. 
		
		Now let $L_0$ be a $2$-component of $G$ such that $K/N = L_0 N /N$. Then $K = L_0 N$. In particular, $L_0$ is a subgroup of $K$ with odd index in $K$. Since $L_0$ is subnormal in $G$, we have that $L_0$ is subnormal in $K$. Applying \cite[Lemma 1.1.11]{Ballester}, we conclude that $L_0 = O^{2'}(L_0) = O^{2'}(K) = L$. The proof of the second statement of the proposition is now complete. The third statement also follows from the above arguments. 
	\end{proof}
	
	\begin{lemma}
		\label{GW 2.18} 
		Let $G$ be a finite group, and let $n$ be a positive integer. Assume that $L_1$, \dots, $L_n$ are the distinct $2$-components of $G$, and assume that $L_i \trianglelefteq G$ for all $1 \le i \le n$. Let $x$ be a $2$-element of $G$, and let $L$ be a $2$-component of $C_G(x)$. Then $L$ is a $2$-component of $C_{L_i}(x)$ for some $1 \le i \le n$.
	\end{lemma}
	
	\begin{proof}
		By \cite[Corollary 3.2]{GW}, we have $L_{2'}(C_G(x)) = L_{2'}(C_{L_{2'}(G)}(x))$, and by \cite[Lemma 2.18 (iii)]{GW}, we have $L_{2'}(C_{L_{2'}(G)}(x)) = \prod_{i=1}^n L_{2'}(C_{L_i}(x))$. Using basic properties of $2$-components, as presented in \cite[Proposition 4.7]{GLS2}, it is not hard to deduce that $L$ is a $2$-component of $C_{L_i}(x)$ for some $1 \le i \le n$.
	\end{proof}
	
	The concepts introduced by the following two definitions will play a crucial role in the proof of Theorem \ref{A} (see \cite{GW} for a detailed study of these concepts). 
	
	\begin{definition}
		\label{def 2-balance} 
		Let $G$ be a finite group, $k$ be a positive integer and $A$ be an elementary abelian $2$-subgroup of $G$. 
		\begin{enumerate}
			\item[(i)] For each nontrivial elementary abelian $2$-subgroup $E$ of $G$, we define 
			\begin{equation*}
				\Delta_G(E) := \bigcap_{a \in E^{\#}} O(C_G(a)). 
			\end{equation*} 
			\item[(ii)] We say that $G$ is \textit{$k$-balanced with respect to $A$} if whenever $E$ is a subgroup of $A$ of rank $k$ and $a$ is a non-trivial element of $A$, we have 
			\begin{equation*}
				\Delta_G(E) \cap C_G(a) \le O(C_G(a)). 
			\end{equation*} 
			\item[(iii)] We say that $G$ is \textit{$k$-balanced} if whenever $E$ is an elementary abelian $2$-subgroup of $G$ of rank $k$ and $a$ is an involution of $G$ centralizing $E$, we have 
			\begin{equation*}
				\Delta_G(E) \cap C_G(a) \le O(C_G(a)).
			\end{equation*}  
			\item[(iv)] By saying that $G$ is \textit{balanced} (respectively, \textit{balanced with respect to $A$}), we mean that $G$ is $1$-balanced (respectively, $1$-balanced with respect to $A$). 
		\end{enumerate}
	\end{definition} 
	
	\begin{definition}
		\label{local k-balance} 
		Let $G$ be a finite quasisimple group, and let $k$ be a positive integer. Then $G$ is said to be \textit{locally $k$-balanced} if whenever $H$ is a subgroup of $\mathrm{Aut}(G)$ containing $\mathrm{Inn}(G)$, we have 
		\begin{equation*} 
			\Delta_H(E) = 1
		\end{equation*}
		for any elementary abelian $2$-subgroup $E$ of $H$ of rank $k$. We say that $G$ is \textit{locally balanced} if $G$ is locally $1$-balanced.   
	\end{definition}  
	
	We need the following proposition for the proof of Theorem \ref{A}. It includes \cite[Theorem 6.10]{GW} and some additional statements, which should be also known. We include a proof for the convenience of the reader. 
	
	\begin{proposition}
		\label{GW 6.10}
		Let $k$ be a positive integer, and let $G$ be a finite group. For each elementary abelian $2$-subgroup $A$ of $G$ of rank at least $k+1$, let
		\begin{equation*}
			W_A := \langle \Delta_G(E) \ \vert \ E \le A, m(E) = k \rangle.
		\end{equation*} 
		Then, for any elementary abelian $2$-subgroup $A$ of $G$ of rank at least $k+1$, the following hold: 
		\begin{enumerate}
			\item[(i)] $(W_A)^g = W_{A^g}$ for all $g \in G$.
			\item[(ii)] Suppose that $A$ has rank at least $k + 2$ and that $G$ is $k$-balanced with respect to $A$. Then $W_A$ has odd order. Moreover, if $A_0$ is a subgroup of $A$ of rank at least $k+1$, then we have $W_A = W_{A_0}$ and $N_G(A_0) \le N_G(W_A)$. 
		\end{enumerate} 
	\end{proposition}
	
	In order to prove Proposition \ref{GW 6.10}, we need the following theorem.
	
	\begin{theorem} (\cite[Theorem 6.9]{GW})
		\label{signalizer functor on k-balanced group}
		Let $k$ be a positive integer, $G$ be a finite group and $A$ be an elementary abelian $2$-subgroup of $G$ of rank at least $k+2$. Suppose that $G$ is $k$-balanced with respect to $A$. Then we obtain an $A$-signalizer functor on $G$ (in the sense of \cite[Definition 4.37]{Gorenstein1983}) by defining 
		\begin{equation*}
			\theta(C_G(a)) := \langle \Delta_G(E) \cap C_G(a) : \ E \le A, m(E) = k \rangle
		\end{equation*} 
		for each $a \in A^{\#}$.  
	\end{theorem}
	
	We also need the following lemma. 
	
	\begin{lemma}
		\label{closure determined by E8}
		Let the notation be as in Theorem \ref{signalizer functor on k-balanced group}. Suppose that $A_0$ is subgroup of $A$ of rank $k+1$. Then we have 
		\begin{equation*}
			\theta(G,A) := \langle \theta(C_G(a)) \ \vert \ a \in A^{\#} \rangle = \langle \Delta_G(E) \ \vert \ E \le A_0, m(E) = k \rangle =: W_{A_0}. 
		\end{equation*}
	\end{lemma} 
	
	\begin{proof}
		To prove this, we follow arguments found on pp. 40-41 of \cite{Mason2}. 
		
		Since $\theta$ is an $A$-signalizer functor on $G$, $\theta(C_G(a))$ is $A$-invariant and in particular $A_0$-invariant for each $a \in A^{\#}$. Consequently, $\theta(G,A)$ is $A_0$-invariant. By the Solvable Signalizer Functor Theorem \cite[11.3.2]{KurzweilStellmacher}, $\theta$ is complete (in the sense of \cite[Definition 4.37]{Gorenstein1983}). In particular, $\theta(G,A)$ has odd order. Applying \cite[Proposition 11.23]{GLS2}, we conclude that
		\begin{equation*}
			\theta(G,A) = \langle C_{\theta(G,A)}(E) \ \vert \ E \le A_0, m(E) = k \rangle.
		\end{equation*}
		Since $\theta$ is complete, we have $C_{\theta(G,A)}(a) = \theta(C_G(a))$ for each $a \in A^{\#}$. By definition of $\theta$ and since $G$ is $k$-balanced with respect to $A$, we have $\theta(C_G(a)) \le O(C_G(a))$ for each $a \in A^{\#}$. So, if $E$ is a subgroup of $A_0$ of rank $k$, then
		\begin{equation*}
			C_{\theta(G,A)}(E) = \bigcap_{a \in E^{\#}}C_{\theta(G,A)}(a) = \bigcap_{a \in E^{\#}} \theta(C_G(a)) \le \bigcap_{a \in E^{\#}} O(C_G(a)) = \Delta_G(E). 
		\end{equation*} 
		It follows that $\theta(G,A) \le W_{A_0}$. 
		
		Let $E \le A_0$ with $m(E) = k$. Clearly, $\Delta_G(E)$ is $A$-invariant. As a consequence of \cite[Proposition 11.23]{GLS2}, we have 
		\begin{equation*}
			\Delta_G(E) = \langle \Delta_G(E) \cap C_G(a) \ \vert \ a \in A^{\#} \rangle. 
		\end{equation*}
		By definition of $\theta$, we have $\Delta_G(E) \cap C_G(a) \le \theta(C_G(a))$ for each $a \in A^{\#}$. It follows that $\Delta_G(E) \le \theta(G,A)$. Consequently, $W_{A_0} \le \theta(G,A)$.  
	\end{proof} 
	
	\begin{proof}[Proof of Proposition \ref{GW 6.10}]
		It is straightforward to verify (i). 
		
		To verify (ii), let $A$ be an elementary abelian $2$-subgroup of $G$ of rank at least $k+2$ such that $G$ is $k$-balanced with respect to $A$. Let $\theta$ be the $A$-signalizer functor on $G$ given by Theorem \ref{signalizer functor on k-balanced group}, and let $\theta(G,A) := \langle \theta(C_G(a)) \ \vert \ a \in A^{\#} \rangle$. As a consequence of Lemma \ref{closure determined by E8}, we have $\theta(G,A) = W_A$. By the proof of Lemma \ref{closure determined by E8}, $W_A = \theta(G,A)$ has odd order. 
		
		Now let $A_0$ be a subgroup of $A$ of rank at least $k+1$. By Lemma \ref{closure determined by E8}, $W_A = \theta(G,A) \le W_{A_0} \le W_A$, and so $W_A = W_{A_0}$. Finally, if $g \in N_G(A_0)$, then $(W_A)^g = (W_{A_0})^g = W_{(A_0)^g} = W_{A_0} = W_A$, and hence $N_G(A_0) \le N_G(W_A)$. 
	\end{proof}
	
	\subsection{Preliminaries on fusion systems}
	\begin{lemma}
		\label{factor_systems_fusion_categories} 
		Let $p$ be a prime, $G$ be a finite group, $N$ be a normal subgroup of $G$ and $S \in \mathrm{Syl}_p(G)$. Then the canonical group isomorphism $S/(S \cap N) \rightarrow SN/N$ induces an isomorphism from $\mathcal{F}_S(G)/(S \cap N)$ to $\mathcal{F}_{SN/N}(G/N)$.
	\end{lemma} 
	
	\begin{proof}
		Let $\varphi$ denote the canonical group isomorphism $S/(S \cap N) \rightarrow SN/N$. Let $P$ and $Q$ be two subgroups of $S$ such that $S \cap N$ is contained in both $P$ and $Q$. Set $\widetilde P := P/(S \cap N)$, $\widetilde Q := Q/(S \cap N)$, $\widebar P := PN/N$ and $\widebar Q := QN/N$. Moreover, define $\widetilde{\mathcal{F}} := \mathcal{F}_S(G)/(S \cap N)$ and $\widebar{\mathcal{F}} := \mathcal{F}_{SN/N}(G/N)$. It is enough to show that 
		\begin{equation*}
			\mathrm{Hom}_{\widebar{\mathcal{F}}}(\widebar P, \widebar Q) = \lbrace (\varphi^{-1}|_{\widebar P,\widetilde P})\alpha(\varphi|_{\widetilde Q, \widebar Q}) \mid \alpha \in \mathrm{Hom}_{\widetilde{\mathcal{F}}}(\widetilde P, \widetilde Q)\rbrace.
		\end{equation*}
		Let $\alpha \in \mathrm{Hom}_{\widetilde{\mathcal{F}}}(\widetilde P, \widetilde Q)$. Then there exists $g \in G$ with $P^g \le Q$ and $\alpha = (c_g|_{P,Q})/(S \cap N)$. By a direct calculation, $(\varphi^{-1}|_{\widebar P, \widetilde P})\alpha (\varphi|_{\widetilde Q, \widebar Q}) = c_{gN}|_{\widebar P, \widebar Q} \in \mathrm{Hom}_{\widebar{\mathcal{F}}}(\widebar P, \widebar Q)$.

		Now let $\widebar \alpha \in \mathrm{Hom}_{\widebar{\mathcal{F}}}(\widebar P, \widebar Q)$. Then there exists $g \in G$ with $\widebar{P}^{gN} \le \widebar{Q}$ and $\widebar \alpha = c_{gN}|_{\widebar P, \widebar Q}$. Clearly, $P^g \le QN$. Since $S \cap N \le Q$, we have that $Q$ is a Sylow $p$-subgroup of $QN$. Since $P^g$ is a $p$-subgroup of $QN$, it follows that there exists an element $n \in N$ with $P^{gn} \le Q$. Set $\alpha := (c_{gn}|_{P,Q})/(S \cap N)$. Then a direct calculation shows that $\widebar{\alpha} = ({\varphi}^{-1}|_{\widebar P, \widetilde P})\alpha(\varphi|_{\widetilde Q,\widebar Q})$. 
	\end{proof} 
	
	\begin{corollary}\textnormal{(\cite[Part II, Exercise 2.1]{AKO})}
		\label{corollary_factor_systems_fusion_categories} 
		Let $p$ be a prime, $G$ be a finite group and $S \in \mathrm{Syl}_p(G)$. Then the canonical group isomorphism $S \rightarrow \widebar S := S O_{p'}(G)/O_{p'}(G)$ induces an isomorphism from $\mathcal{F}_S(G)$ to $\mathcal{F}_{\widebar S}(G/O_{p'}(G))$.
	\end{corollary}
	
	\begin{lemma}
		\label{fusion systems of quasisimple groups}
		Let $K_1$ and $K_2$ be two quasisimple finite groups. If the $2$-fusion systems of $K_1$ and $K_2$ are isomorphic, then the $2$-fusion systems of $K_1/Z(K_1)$ and $K_2/Z(K_2)$ are isomorphic. 
	\end{lemma} 
	
	\begin{proof}
		Suppose that the $2$-fusion systems of $K_1$ and $K_2$ are isomorphic. Let $S_i$ be a Sylow $2$-subgroup of $K_i$ and $\mathcal{F}_i := \mathcal{F}_{S_i}(K_i)$ for $i \in \lbrace 1,2 \rbrace$. As a consequence of \cite[Corollary 1]{Glauberman}, we have $Z(\mathcal{F}_i) = S_i \cap Z^{*}(K_i)$ for $i \in \lbrace 1,2 \rbrace$. Since $K_1$ and $K_2$ are quasisimple, we have $Z^{*}(K_i) = Z(K_i)$ and hence $Z(\mathcal{F}_i) = S_i \cap Z(K_i)$ for $i \in \lbrace 1,2 \rbrace$. Since $\mathcal{F}_1 \cong \mathcal{F}_2$, it follows that 
		\begin{equation*}
			\mathcal{F}_1/(S_1 \cap Z(K_1)) = \mathcal{F}_1/Z(\mathcal{F}_1) \cong \mathcal{F}_2/Z(\mathcal{F}_2) = \mathcal{F}_2/(S_2 \cap Z(K_2)). 
		\end{equation*} 
		Applying Lemma \ref{factor_systems_fusion_categories}, we may conclude that the $2$-fusion system of $K_1/Z(K_1)$ is isomorphic to the $2$-fusion system of $K_2/Z(K_2)$. 
	\end{proof} 
	
	\begin{lemma}
		\label{lemma on strongly closed subgroups}
		Let $S$ be a finite $2$-group, and let $A$ and $B$ be normal subgroups of $S$ such that $S$ is the internal direct product of $A$ and $B$. Suppose that $A \cong Q_8$. Let $\mathcal{F}$ be a (not necessarily saturated) fusion system on $S$. Assume that $A$ and $B$ are strongly $\mathcal{F}$-closed and that there is an automorphism $\alpha \in \mathrm{Aut}_{\mathcal{F}}(S)$ such that $\alpha|_{A,A}$ has order $3$, while $\alpha|_{B,B} = \mathrm{id}_B$. Then each strongly $\mathcal{F}$-closed subgroup of $S$ contains or centralizes $A$. 
	\end{lemma} 
	
	\begin{proof}
		Let $C$ be a strongly $\mathcal{F}$-closed subgroup of $S$ not containing $A$. Our task is to show that $C$ centralizes $A$. 
		
		Since $A$ and $C$ are strongly $\mathcal{F}$-closed, we have that $A \cap C$ is strongly $\mathcal{F}$-closed. In particular, $\alpha$ normalizes $A \cap C$. It is easy to see that an automorphism of $Q_8$ with order $3$ does not normalize any maximal subgroup of $Q_8$. So, as $\alpha|_{A,A}$ has order $3$ and normalizes $A \cap C$, we have that $A \cap C$ has order $1$ or $2$. 
		
		By \cite[8.2.7]{KurzweilStellmacher}, we have
		\begin{equation*}
			[C,\langle \alpha \rangle] = [[C,\langle \alpha \rangle], \langle \alpha \rangle]. 
		\end{equation*} 
		We claim that $[C,\langle \alpha \rangle] \le A \cap C$. Let $c \in C$ and $\beta \in \langle \alpha \rangle$. Let $a \in A$ and $b \in B$ such that $c = ab$. Since $A$ and $B$ commute and since $\beta$ normalizes $A$ and centralizes $B$, we have
		\begin{equation*}
			[c,\beta] = c^{-1}c^{\beta} = b^{-1}a^{-1}a^{\beta}b^{\beta} = a^{-1}a^{\beta} \in A \cap C.
		\end{equation*} 
		Thus $[C,\langle \alpha \rangle] \le A \cap C$, as asserted.
		
		Since $A \cap C$ has order $1$ or $2$, we have $[A \cap C, \langle \alpha \rangle] = 1$. So it follows that 
		\begin{equation*}
			[C, \langle \alpha \rangle] = [[C,\langle \alpha \rangle],\langle \alpha \rangle] \le [A \cap C,\langle \alpha \rangle] =1. 
		\end{equation*} 
		Now we prove that $C$ centralizes $A$. Let $c \in C$ and $a \in A$, $b \in B$ with $c = ab$. We have $c^{-1}c^{\alpha} \in [C,\langle \alpha \rangle] = 1$, whence $c^{\alpha} = c$. Thus $ab = (ab)^{\alpha} = a^{\alpha}b$ and hence $a = a^{\alpha}$. As remarked above, $\alpha$ does not normalize any maximal subgroup of $A$. So $a$ cannot have order $4$. By the structure of $A \cong Q_8$, it follows that $a \in Z(A)$. This implies that $c = ab$ centralizes $A$. 
	\end{proof}

	We need the following definition in order to state the next proposition. 
	
	\begin{definition}
		A nonabelian finite simple group $G$ is said to be a \textit{Goldschmidt group} provided that one of the following holds:
		\begin{enumerate}
			\item[(1)] $G$ has an abelian Sylow $2$-subgroup. 
			\item[(2)] $G$ is isomorphic to a finite simple group of Lie type in characteristic $2$ of Lie rank $1$. 
		\end{enumerate} 
	\end{definition} 
	
	\begin{proposition}
		\label{subsystems induced by 2-components} 
		Let $G$ be a finite group, and let $S$ be a Sylow $2$-subgroup of $G$. Assume that for each $2$-component $L$ of $G$, the factor group $L/Z^{*}(L)$ is a known finite simple group. Let $\mathfrak{L}_{2'}$ denote the set of $2$-components $L$ of $G$ such that $L/Z^{*}(L)$ is not a Goldschmidt group. Then the following hold: 
		\begin{enumerate}
			\item[(i)] Let $L$ be a $2$-component of $G$. Then $\mathcal{F}_{S \cap L}(L)$ is a component of $\mathcal{F}_S(G)$ if and only if $L \in \mathfrak{L}_{2'}$. 
			\item[(ii)] The map from $\mathfrak{L}_{2'}$ to the set of components of $\mathcal{F}_S(G)$ sending each element $L$ of $\mathfrak{L}_{2'}$ to $\mathcal{F}_{S \cap L}(L)$ is a bijection. 
		\end{enumerate} 
	\end{proposition}
	
	\begin{proof}
		Let $L$ be a $2$-component of $G$. Set $\mathcal{G} := \mathcal{F}_{S \cap L}(L)$. Since $L$ is subnormal in $G$, we have that $\mathcal{G}$ is subnormal in $\mathcal{F}_S(G)$ (see \cite[Part I, Proposition 6.2]{AKO}). Therefore, $\mathcal{G}$ is a component of $\mathcal{F}_S(G)$ if and only if $\mathcal{G}$ is quasisimple. We have $\mathfrak{foc}(\mathcal{G}) = S \cap L' = S \cap L$ by the focal subgroup theorem \cite[Chapter 7, Theorem 3.4]{Gorenstein}, and so $\mathcal{G}$ is quasisimple if and only if $\mathcal{G}/Z(\mathcal{G})$ is simple. As a consequence of \cite[Corollary 1]{Glauberman}, we have $Z(\mathcal{G}) = S \cap Z^{*}(L)$. Lemma \ref{factor_systems_fusion_categories} implies that $\mathcal{G}/Z(\mathcal{G})$ is isomorphic to the $2$-fusion system of $L/Z^{*}(L)$. By \cite[Theorem 5.6.18]{Aschbacher2021}, the $2$-fusion system of $L/Z^{*}(L)$ is simple if and only if $L \in \mathfrak{L}_{2'}$. So $\mathcal{G}$ is a component of $\mathcal{F}_S(G)$ if and only if $L \in \mathfrak{L}_{2'}$, and (i) holds.
		
		(ii) follows from \cite[(1.8)]{Aschbacher2020}.
	\end{proof}
	
	\begin{lemma}
		\label{2-nilpotence lemma} 
		Let $G$ be a finite group with $O(G) = 1$, and let $S$ be a Sylow $2$-subgroup of $G$. Let $n \ge 1$ be a natural number, and let $L_1, \dots, L_n$ be pairwise distinct subgroups of $G$ such that $L_i$ is either a component or a solvable $2$-component of $G$ for each $1 \le i \le n$. Set $Q := (S \cap L_1)\cdots(S \cap L_n)$. Assume that $Q \trianglelefteq S$ and that $\mathcal{F}_S(G)/Q$ is nilpotent. Then, if $L_0$ is a component or a solvable $2$-component of $G$, we have $L_0 = L_i$ for some $1 \le i \le n$. 
	\end{lemma} 
	
	\begin{proof}
		Let $L^s(G)$ denote the subgroup of $G$ generated by the components and the solvable $2$-components of $G$. By \cite[6.5.2]{KurzweilStellmacher} and \cite[Proposition 13.5]{GLS2}, $L^s(G)$ is the central product of the subgroups of $G$ which are components or solvable $2$-components. Set $L := L_1 \cdots L_n \trianglelefteq L^s(G)$. 
		
		Let $\mathcal{G} := \mathcal{F}_{S \cap L^s(G)}(L^s(G))$. Clearly, $S \cap L = (S \cap L_1) \cdots (S \cap L_n) = Q$. Lemma \ref{factor_systems_fusion_categories} implies that the $2$-fusion system of $L^s(G)/L$ is isomorphic to $\mathcal{G}/Q$. By hypothesis, $\mathcal{F}_S(G)/Q$ is nilpotent, and so $\mathcal{G}/Q$ is nilpotent. So the $2$-fusion system of $L^s(G)/L$ is nilpotent. Applying \cite[Theorem 1.4]{Linckelmann}, we conclude that $L^s(G)/L$ is $2$-nilpotent. 
		
		Now let $L_0$ be a component or a solvable $2$-component of $G$. If $L_0 \le L$, then we have $L_0 = L_i$ for some $1 \le i \le n$ since otherwise $L_0 \le Z(L)$, which is impossible. So it suffices to show that $L_0 \le L$.
		
		If $L_0$ is a component of $G$, then $L_0/(L_0 \cap L)$ is both perfect and $2$-nilpotent, which implies that $L_0 \le L$, as needed. 
		
		Suppose now that $L_0$ is a solvable $2$-component of $G$. Assume that $L_0 \not\le L$. Then $L_0 \cap L \le Z(L_0)$. Since $L_0$ is a solvable $2$-component of $G$, it follows that $L_0/(L_0 \cap L)$ is isomorphic to $SL_2(3)$ or $PSL_2(3)$. On the other hand, $L_0/(L_0 \cap L)$ is $2$-nilpotent. This contradiction shows that $L_0 \le L$, as required.
		\hfill \qed
	\end{proof}
	
	\begin{corollary}
		\label{corollary_to_show_that_2-components_are_all_2-components} 
		Let $G$ be a finite group, and let $S$ be a Sylow $2$-subgroup of $G$. Let $n \ge 1$ be a natural number, and let $L_1, \dots, L_n$ be pairwise distinct $2$-components of $G$. Assume that $Q := (S \cap L_1)\cdots(S \cap L_n)$ is a normal subgroup of $S$ and that $\mathcal{F}_S(G)/Q$ is nilpotent. Then, if $L_0$ is a $2$-component of $G$, we have $L_0 = L_i$ for some $1 \le i \le n$.  
	\end{corollary}   
	
	\begin{proposition}
		\label{corollary_oliver}
		Let $p$ be a prime, and let $\mathcal{E}$ be a simple saturated fusion system on a finite $p$-group $T$. Suppose that $\mathcal{E}$ is tamely realized (in the sense of \cite[Section 2.2]{Andersen}) by a nonabelian known finite simple group $K$ such that $\mathrm{Out}(K)$ is $p$-nilpotent. Assume moreover that $G$ is a nonabelian finite simple group containing a Sylow $p$-subgroup $S$ with $T \le S$ such that $\mathcal{E} \trianglelefteq \mathcal{F}_S(G)$ and $C_S(\mathcal{E}) = 1$. Then $\mathcal{F}_S(G)$ is tamely realized by a subgroup $L$ of $\mathrm{Aut}(K)$ containing $\mathrm{Inn}(K)$ such that the index of $\mathrm{Inn}(K)$ in $L$ is coprime to $p$. 
	\end{proposition}
	
	\begin{proof}
		Set $\mathcal{F} := \mathcal{F}_S(G)$. By a result of Bob Oliver, namely by \cite[Corollary 2.4]{Oliver2016}, $\mathcal{F}$ is tamely realized by a subgroup $L$ of $\mathrm{Aut}(K)$ containing $\mathrm{Inn}(K)$. We are going to show that the index of $\mathrm{Inn}(K)$ in $L$ is coprime to $p$. 
		
		Let $S_0$ be a Sylow $p$-subgroup of $L$. Then $\mathcal{F} \cong \mathcal{F}_{S_0}(L)$. Clearly, $O^p(G) = G$, and so $\mathfrak{hnp}(\mathcal{F}) = S$ by the hyperfocal subgroup theorem \cite[Theorem 1.33]{Craven}. It follows that $\mathfrak{hnp}(\mathcal{F}_{S_0}(L)) = S_0$.  
		
		By the hyperfocal subgroup theorem \cite[Theorem 1.33]{Craven}, $S_0 = \mathfrak{hnp}(\mathcal{F}_{S_0}(L)) = O^p(L) \cap S_0$. Consequently, $O^p(L)$ has $p'$-index in $L$, whence $O^p(L) = L$. So we have $O^p(L/\mathrm{Inn}(K)) = L/\mathrm{Inn}(K)$. On the other hand, $L/\mathrm{Inn}(K)$ is $p$-nilpotent since $\mathrm{Out}(K)$ is $p$-nilpotent. It follows that $L/\mathrm{Inn}(K)$ is a $p'$-group, as claimed. 
	\end{proof}

	\section{Auxiliary results on linear and unitary groups}
	\label{auxiliary_results_linear_unitary}
	In this section, we collect several results on linear and unitary groups needed for the proofs of our main results. Some of the results stated here are known, while others seem to be new. For the convenience of the reader, we also include proofs of known results when we could not find a reference in which they appear in the form stated here.
	
	\subsection{Basic definitions} 
	We begin with some basic definitions. Let $q$ be a nontrivial prime power, and let $n$ be a positive integer. The \textit{general linear group} $GL_n(q)$ is the group of all invertible $n \times n$ matrices over $\mathbb{F}_q$ under matrix multiplication. The \textit{special linear group} $SL_n(q)$ is the subgroup of $GL_n(q)$ consisting of all $n \times n$ matrices over $\mathbb{F}_q$ with determinant $1$. The center of $GL_n(q)$ consists of all scalar matrices $\lambda I_n$ with $\lambda \in (\mathbb{F}_q)^{*}$. We have $Z(SL_n(q)) = SL_n(q) \cap Z(GL_n(q))$. Set $PGL_n(q) := GL_n(q)/Z(GL_n(q))$ and $PSL_n(q) := SL_n(q)/Z(SL_n(q))$. By \cite[Kapitel II, Satz 6.10]{Huppert} and \cite[Kapitel II, Hauptsatz 6.13]{Huppert}, $SL_n(q)$ is quasisimple if $n \ge 2$ and $(n,q) \ne (2,2),(2,3)$. 
	
	As in \cite[Kapitel II, Bemerkung 10.5 (b)]{Huppert}, we consider the \textit{general unitary group} $GU_n(q)$ as the subgroup of $GL_n(q^2)$ consisting of all $(a_{ij}) \in GL_n(q^2)$ satisfying the condition $((a_{ij})^q)(a_{ij})^t = I_n$. The \textit{special unitary group} $SU_n(q)$ is the subgroup of $GU_n(q)$ consisting of all elements of $GU_n(q)$ with determinant $1$. By \cite[Kapitel II, Hilfssatz 8.8]{Huppert}, we have $SL_2(q) \cong SU_2(q)$. The center of $GU_n(q)$ consists of all scalar matrices $\lambda I_n$, where $\lambda \in (\mathbb{F}_{q^2})^{*}$ and $\lambda^{q+1} = 1$. We have $Z(SU_n(q)) = SU_n(q) \cap Z(GU_n(q))$. Set $PGU_n(q) := GU_n(q)/Z(GU_n(q))$ and $PSU_n(q) := SU_n(q)/Z(SU_n(q))$. By \cite[Theorems 11.22 and 11.26]{Grove}, $SU_n(q)$ is quasisimple if $n \ge 2$ and $(n,q) \ne (2,2),(2,3),(3,2)$.
	
	We write $(P)GL_n^{+}(q)$ and $(P)SL_n^{+}(q)$ for $(P)GL_n(q)$ and $(P)SL_n(q)$, respectively. Also, we write $(P)GL_n^{-}(q)$ for $(P)GU_n(q)$ and $(P)SL_n^{-}(q)$ for $PSU_n(q)$. 
	
	\subsection{Central extensions of $PSL_n(q)$ and $PSU_n(q)$}
	
In the proofs of the following two lemmas, we use the terminology of \cite[Section 33]{FiniteGroupTheory}. 

\begin{lemma} 
	\label{Schur_PSL} 
	Let $n \ge 3$ be a natural number, and let $q$ be a nontrivial odd prime power. Let $H$ be a perfect central extension of $PSL_n(q)$. Then there is a subgroup $Z \le Z(SL_n(q))$ such that $H \cong SL_n(q)/Z$. 
\end{lemma}  

\begin{proof}
	By \cite[pp. 312-313]{GLS3}, the Schur multiplier of $PSL_n(q)$ is isomorphic to $C_{(n,q-1)} \cong Z(SL_n(q))$. From \cite[33.6]{FiniteGroupTheory}, we see that this is just another way to say that $SL_n(q)$ is the universal covering group of $PSL_n(q)$. Applying \cite[33.6]{FiniteGroupTheory} again, we conclude that $H \cong SL_n(q)/Z$ for some $Z \le Z(SL_n(q))$. 
\end{proof}

\begin{lemma}
	\label{Schur_PSU}
	Let $n \ge 3$ be a natural number, and let $q$ be a nontrivial odd prime power. Let $H$ be a perfect central extension of $PSU_n(q)$. Assume that $(n,q) \ne (4,3)$ or that $Z(H)$ is a $2$-group. Then there is a subgroup $Z \le Z(SU_n(q))$ such that $H \cong SU_n(q)/Z$. 
\end{lemma} 

\begin{proof}
	Suppose that $(n,q) \ne (4,3)$. By \cite[pp. 312-313]{GLS3}, the Schur multiplier of $PSU_n(q)$ is isomorphic to $C_{(n,q+1)} \cong Z(SU_n(q))$. As in the proof of Lemma \ref{Schur_PSL}, we conclude that $H \cong SU_n(q)/Z$ for some $Z \le Z(SU_n(q))$. 
	
	Suppose now that $(n,q)=(4,3)$ and that $Z(H)$ is a $2$-group. Let $G := PSU_4(3)$, and let $\widetilde G$ be the universal covering group of $G$. Clearly, the Schur multiplier of $G$ is isomorphic to $Z(\widetilde G)$. By \cite[pp. 312-313]{GLS3}, the Schur multiplier of $G$ is isomorphic to $C_4 \times C_3 \times C_3$. Thus $Z(\widetilde G) \cong C_4 \times C_3 \times C_3$. Clearly, if $Z \le Z(\widetilde G)$, then $Z(\widetilde G / Z) = Z(\widetilde G)/Z$. Let $Q$ be the unique Sylow $3$-subgroup of $Z(\widetilde G)$. By \cite[33.6]{FiniteGroupTheory}, $\widetilde G$ is a central extension of $SU_4(3)$ and of $H$. Since $SU_4(3)$ has a center of order $4$, we have $SU_4(3) \cong \widetilde G / Q$. Let $Z \le Z(\widetilde G)$ with $H \cong \widetilde G / Z$. As $Z(H)$ is a $2$-group, we have $Q \le Z$, whence $H \cong \widetilde G / Z \cong (\widetilde G / Q) / (Z/Q)$ is isomorphic to a quotient of $SU_4(3)$ by a central subgroup. 
\end{proof}
	
	\subsection{Involutions}
	In this subsection, we collect several results on the involutions of the groups $(P)GL_n^{\varepsilon}(q)$ and $(P)SL_n^{\varepsilon}(q)$, where $q$ is a nontrivial odd prime power, $n \ge 2$ and $\varepsilon \in \lbrace +,- \rbrace$. 
	
	\begin{lemma}
		\label{involutions_GL(n,q)} 
		Let $q$ be a nontrivial odd prime power, and let $n \ge 2$. Let $T$ be an element of $GL_n(q)$ such that $T^2 = \lambda I_n$ for some $\lambda \in \mathbb{F}_q^{*}$. Then one of the following holds: 
		\begin{enumerate}
			\item[(i)] There is some $\mu \in \mathbb{F}_q^{*}$ such that $\lambda = \mu^2$, and $T$ is $GL_n(q)$-conjugate to a diagonal matrix with diagonal entries in $\lbrace \mu, -\mu \rbrace$. 
			\item[(ii)] $n$ is even, $\lambda$ is a non-square element of $\mathbb{F}_q^{*}$, and $T$ is $GL_n(q)$-conjugate to the matrix
			\begin{equation*}
				\begin{pmatrix}  & I_{n/2} \ \\ \lambda I_{n/2} & \end{pmatrix}. 
			\end{equation*} 
			Moreover, we have $C_{GL_n(q)}(T) \cong GL_{\frac{n}{2}}(q^2)$.
		\end{enumerate}
	\end{lemma} 
	
	\begin{proof}
		We identify the field $\mathbb{F}_q$ with the subfield of $\mathbb{F}_{q^2}$ consisting of all $x \in \mathbb{F}_{q^2}$ satisfying $x^q = x$. It is easy to note that any element of $\mathbb{F}_q^{*}$ is the square of an element of $\mathbb{F}_{q^2}^{*}$. Let $\mu \in \mathbb{F}_{q^2}^{*}$ with $\lambda = \mu^2$. 
		
		If $\mu \in \mathbb{F}_q$, then basic linear algebra shows that $T$ is diagonalizable over $\mathbb{F}_q$, and it follows that (i) holds.
		
		Assume now that $\mu \not\in \mathbb{F}_q$. Then $\lambda$ is a non-square element of $\mathbb{F}_q^{*}$. 
		Let $V$ be an $n$-dimensional vector space over $\mathbb{F}_q$, and let $B$ be an ordered basis of $V$. Let $\varphi$ be the element of $GL(V)$ such that $\varphi$ is represented by $T$ with respect to $B$. Clearly, $(1,\mu)$ is an $\mathbb{F}_q$-basis of $\mathbb{F}_{q^2}$. Using that $\varphi^2 = \lambda \mathrm{id}_V$, one can check that $V$ becomes a vector space over $\mathbb{F}_{q^2}$ by defining  
		\begin{equation*}
			(x + y \mu) v := xv + y v^{\varphi}
		\end{equation*} 
		for all $x,y \in \mathbb{F}_q$ and $v \in V$. Let $m$ be the dimension of $V$ over $\mathbb{F}_{q^2}$, and let $(v_1,\dots,v_m)$ be an $\mathbb{F}_{q^2}$-basis of $V$. Then $B_0 := (v_1,\dots,v_m,\mu v_1, \dots, \mu v_m)$ is an $\mathbb{F}_q$-basis of $V$. In particular, $n = 2m$ is even. For $1 \le i \le m$, we have $v_i^{\varphi} = \mu v_i$ and $(\mu v_i)^{\varphi} = (v_i)^{\varphi^2} = \lambda v_i$. So, with respect to $B_0$, $\varphi$ is represented by the matrix
		\begin{equation*}
			M := \begin{pmatrix}  & I_{n/2} \ \\ \lambda I_{n/2} & \end{pmatrix}. 
		\end{equation*}
		It follows that $T$ and $M$ are $GL_n(q)$-conjugate.
		
		Let $\psi$ be an automorphism of $V$ as an $\mathbb{F}_q$-vector space centralizing $\varphi$. For $x,y \in \mathbb{F}_q$ and $v \in V$, we have
		\begin{equation*}
			((x+y\mu)v)^{\psi} = (xv+yv^{\varphi})^{\psi} = xv^{\psi} + yv^{\psi\varphi} = (x+y\mu) v^{\psi},
		\end{equation*}
		whence $\psi$ is $\mathbb{F}_{q^2}$-linear. Conversely, if $\psi$ is $\mathbb{F}_{q^2}$-linear, then 
		\begin{equation*}
			v_i^{\psi\varphi} = \mu v_i^{\psi} = (\mu v_i)^{\psi} = v_i^{\varphi\psi} 
		\end{equation*}
		and hence $\psi\varphi = \varphi\psi$. It follows that the centralizer of $\varphi$ in the general linear group of $V$ as an $\mathbb{F}_q$-vector space is equal to the general linear group of $V$ as an $\mathbb{F}_{q^2}$-vector space. Thus $C_{GL_n(q)}(T) \cong GL_{\frac{n}{2}}(q^2)$. So (ii) holds.  
	\end{proof}
	
	\begin{lemma}
		\label{diagonalizable_involutions_GU(n,q)}
		Let $q$ be a nontrivial odd prime power, and let $n \ge 2$ be a natural number. Let $T \in GU_n(q)$. 
		\begin{enumerate}
			\item[(i)] If $T^2 = \lambda I_n$ for some $\lambda \in \mathbb{F}_{q^2}^{*}$, then $\lambda$ is a square in $\mathbb{F}_{q^2}^{*}$.  
			\item[(ii)] If $T^2 = \rho^2 I_n$ for some $\rho \in \mathbb{F}_{q^2}^{*}$ with $\rho^{q+1} = 1$, then $T$ is $GU_n(q)$-conjugate to a diagonal matrix with diagonal entries in $\lbrace \rho,-\rho \rbrace$.
			\item[(iii)] If $T^2 = \rho^2 I_n$ for some $\rho \in \mathbb{F}_{q^2}^{*}$ with $\rho^{q+1} \ne 1$, then $n$ is even, and we have $C_{GU_n(q)}(T) \cong GL_{\frac{n}{2}}(q^2)$. 
		\end{enumerate} 
	\end{lemma} 
	
	\begin{proof}
		Suppose that $T^2 = \lambda I_n$ for some $\lambda \in \mathbb{F}_{q^2}^{*}$. Since $T^2 \in GU_n(q)$, we have that $\lambda^{q+1} = 1$. It is easy to see that any element $x$ of $\mathbb{F}_{q^2}^{*}$ with $x^{q+1} = 1$ is a square in $\mathbb{F}_{q^2}^{*}$. So (i) holds. 
		
		A proof of (ii) and (iii) can be extracted from \cite[pp. 314-315]{Phan1975}. 
	\end{proof}
	
	\begin{proposition}
		\label{involutions of PSL(n,q)}
		Let $q$ be a nontrivial odd prime power, and let $n \ge 2$ be a natural number. Let $\rho$ be an element of $\mathbb{F}_q^{*}$ of order $(n,q-1)$. For each even natural number $i$ with $2 \le i < n$, let
		\begin{equation*}
			\widetilde{t_i} := \begin{pmatrix} I_{n-i} & \\  & -I_i \end{pmatrix} \in SL_n(q)
		\end{equation*}
		and let $t_i$ be the image of $\widetilde{t_i}$ in $PSL_n(q)$.
		\begin{enumerate}
			\item[(i)] Assume that $n$ is odd. Then each involution of $PSL_n(q)$ is $PSL_n(q)$-conjugate to $t_i$ for some even $2 \le i < n$.
			\item[(ii)] Assume that $n$ is even and that there is some $\mu \in \mathbb{F}_q^{*}$ with $\rho = \mu^2$. For each odd natural number $i$ with $1 \le i < n$, the matrix 
			\begin{equation*}
				\widetilde{t_i} := \begin{pmatrix} \mu I_{n-i} & \\  & -\mu I_i \end{pmatrix}
			\end{equation*}
			lies in $SL_n(q)$. Let $t_i$ denote the image of $\widetilde{t_i}$ in $PSL_n(q)$ for each odd $1 \le i < n$. Then each involution of $PSL_n(q)$ is $PSL_n(q)$-conjugate to $t_i$ for some (even or odd) $1 \le i \le \frac{n}{2}$. 
			\item[(iii)] Assume that $n$ is even and that $\rho$ is a non-square element of $\mathbb{F}_q$. Let 
			\begin{equation*}
				\widetilde w := \begin{pmatrix}  & I_{n/2} \ \\ \rho I_{n/2} & \end{pmatrix}. 
			\end{equation*} 
			If $\widetilde w \in SL_n(q)$, then each involution of $PSL_n(q)$ is $PSL_n(q)$-conjugate to to $t_i$ for some even $2 \le i \le \frac{n}{2}$ or to $w := \widetilde{w} Z(SL_n(q)) \in PSL_n(q)$. If $\widetilde w \not\in SL_n(q)$, then each involution of $PSL_n(q)$ is $PSL_n(q)$-conjugate to $t_i$ for some even $2 \le i \le \frac{n}{2}$.  
		\end{enumerate} 
	\end{proposition}
	
	\begin{proof}
		We follow arguments found in the proof of \cite[Lemma 1.1]{Phan1972}.
		
		Assume that $n$ is odd. Then $Z(SL_n(q))$ has odd order, and therefore, any involution of $PSL_n(q)$ is the image of an involution of $SL_n(q)$. As a consequence of Lemma \ref{involutions_GL(n,q)}, each involution of $SL_n(q)$ is $SL_n(q)$-conjugate to $\widetilde{t_i}$ for some even $2 \le i < n$. So (i) follows. 
		
		Assume now that $n$ is even and that $\rho = \mu^2$ for some $\mu \in \mathbb{F}_q^{*}$. Note that $Z(SL_n(q))$ equals $\langle \rho I_n \rangle$. We claim that $\mu^n = -1$. Since $\mu^{2n} = \rho^n = 1$, we have that $\mu^n = 1$ or $-1$. If $\mu^n = 1$, then $\mu \in \langle \rho \rangle$, and so $\rho$ is a square in $\langle \rho \rangle$, which is impossible. So we have $\mu^n = -1$. It follows that $\widetilde{t_i} \in SL_n(q)$ for each odd $1 \le i < n$. Now let $T \in SL_n(q)$ such that $TZ(SL_n(q)) \in PSL_n(q)$ is an involution. Then we have $T^2 = \rho^{\ell}I_n = \mu^{2\ell}I_n$ for some $1 \le \ell \le (n,q-1)$. Using Lemma \ref{involutions_GL(n,q)}, we conclude that $T$ is $SL_n(q)$-conjugate to a diagonal matrix $D \in SL_n(q)$ with diagonal entries in $\lbrace \mu^{\ell}, - \mu^{\ell} \rbrace$. Let $1 \le i < n$ such that $- \mu^{\ell}$ occurs precisely $i$ times as a diagonal entry of $D$. If $i$ is odd, we may assume that $D = \mu^{\ell-1} \widetilde{t_i}$, and if $i$ is even, we may assume that $D = \mu^{\ell} \widetilde{t_i}$. In either case, the image of $D$ in $PSL_n(q)$ is $t_i$. Hence, $TZ(SL_n(q))$ is $PSL_n(q)$-conjugate to $t_i$. Noticing that $t_i$ is $PSL_n(q)$-conjugate to $t_{n-i}$, we conclude that (ii) holds.  
		
		Now assume that $n$ is even and that $\rho$ is a non-square element of $\mathbb{F}_q$. Again let $T$ be an element of $SL_n(q)$ such that $TZ(SL_n(q)) \in PSL_n(q)$ is an involution. We have $T^2 = \rho^{\ell}I_n$ for some $1 \le \ell \le (n,q-1)$. Assume that $\ell$ is even. Then Lemma \ref{involutions_GL(n,q)} implies that $T$ or $-T$ is $SL_n(q)$-conjugate to $\rho^{\frac{\ell}{2}}\widetilde{t_i}$ for some even $2 \le i \le \frac{n}{2}$. It follows that $TZ(SL_n(q))$ is $PSL_n(q)$-conjugate to $t_i$ for some even $2 \le i \le \frac{n}{2}$. Assume now that $\ell$ is odd. As $\rho$ is not a square in $\mathbb{F}_q$, but $\rho^{\ell-1}$ is a square in $\mathbb{F}_q$, $\rho^{\ell}$ cannot be a square in $\mathbb{F}_q$. Using Lemma \ref{involutions_GL(n,q)}, we may conclude that $T$ is $GL_n(q)$-conjugate to the matrix
		\begin{equation*}
			M := \begin{pmatrix} 0 & \rho^{\ell} & & & \\
				1 & 0 & & &\\ 
				& & \ddots & & \\
				& & & 0 & \rho^{\ell} \\
				& & & 1 & 0
			\end{pmatrix} \in SL_n(q).
		\end{equation*}
		It is rather easy to see that $T$ and $M$ are even conjugate in $SL_n(q)$. Let $k := \frac{\ell-1}{2}$. It is not hard to show that the matrices
		\begin{equation*}
			\begin{pmatrix} 0 & \rho^{\ell} \\ 1 & 0 \end{pmatrix} \ \textnormal{and} \ \begin{pmatrix} 0 & \rho^{k+1} \\ \rho^k & 0 \end{pmatrix}
		\end{equation*} 
		are $SL_2(q)$-conjugate. So it follows that $M$ and hence $T$ is $SL_n(q)$-conjugate to $\rho^k M_2$, where
		\begin{equation*}
			M_2 := \begin{pmatrix} 0 & \rho & & & \\
				1 & 0 & & &\\ 
				& & \ddots & & \\
				& & & 0 & \rho \\
				& & & 1 & 0
			\end{pmatrix} \in SL_n(q).
		\end{equation*} 
		Consequently, the images of $T$ and $M_2$ in $PSL_n(q)$ are conjugate. Furthermore, as $\mathrm{det}(M_2) = \mathrm{det}(\widetilde w)$, we see that $\widetilde w \in SL_n(q)$. Also, $\widetilde w$ is $SL_n(q)$-conjugate to $M_2$, and so $TZ(SL_n(q))$ is $PSL_n(q)$-conjugate to $w$. 
	\end{proof}
	
	\begin{lemma}
		\label{centralizer of w in PSL(n,q)} 
		Let $q$ be a nontrivial odd prime power and let $n \ge 4$ be an even natural number. Let $\rho$ be an element of $\mathbb{F}_q^{*}$ of order $(n,q-1)$. Suppose that $\rho$ is a non-square element of $\mathbb{F}_q$ and that 
		\begin{equation*}
			\widetilde w := \begin{pmatrix} & I_{n/2} \\ \rho I_{n/2} & \end{pmatrix}
		\end{equation*} 
		lies in $SL_n(q)$. Denote the image of $\widetilde w$ in $PSL_n(q)$ by $w$. Set $C := C_{PSL_n(q)}(w)$. Let $P$ be a Sylow $2$-subgroup of $C$. Then the following hold: 
		\begin{enumerate}
			\item[(i)] $C$ has a unique $2$-component $J$, and $J$ is isomorphic to a nontrivial quotient of $SL_{\frac{n}{2}}(q^2)$. 
			\item[(ii)] We have $P \cap J \trianglelefteq P$, and the factor system $\mathcal{F}_P(C)/(P \cap J)$ is nilpotent.  
			\item[(iii)] If $n \ge 6$, then $P$ has rank at least $4$. 
		\end{enumerate} 
	\end{lemma} 
	
	\begin{proof}
		Set $C_0 := C_{SL_n(q)}(\widetilde w)/Z(SL_n(q)) \le C$. By a direct argument, $C_0$ has index $2$ in $C$. So the $2$-components of $C$ are precisely the $2$-components of $C_0$. One may deduce from Lemma \ref{involutions_GL(n,q)} that $C_{SL_n(q)}(\widetilde w)$ has a normal subgroup $\widetilde J$ isomorphic to $SL_{\frac{n}{2}}(q^2)$ such that the corresponding factor group is cyclic. Let $J$ be the image of $\widetilde J$ in $PSL_n(q)$. Then $J$ is isomorphic to a nontrivial quotient of $SL_{\frac{n}{2}}(q^2)$. Moreover, $J \trianglelefteq C_0$ and $C_0/J$ is cyclic. Therefore, $J$ is the only $2$-component of $C_0$ and hence the only $2$-component of $C$. Thus (i) holds.
		
		We have $P \cap J \trianglelefteq P$ because $J \trianglelefteq C$. By Lemma \ref{factor_systems_fusion_categories}, the factor system $\mathcal{F}_P(C)/(P \cap J)$ is isomorphic to the $2$-fusion system of $C/J$. Since $C_0$ has index $2$ in $C$ and $C_0/J$ is abelian, we have that $C/J$ is $2$-nilpotent. So $C/J$ has a nilpotent $2$-fusion system, and (ii) follows. 
		
		We now prove (iii). Assume that $n \ge 6$. Let $u$ denote the image of 
		\begin{equation*}
			\begin{pmatrix} 0 & \rho & & & \\
				1 & 0 & & & \\
				& & \ddots & & \\
				& & & 0 & \rho \\
				& & & 1 & 0  \end{pmatrix} \in SL_n(q)
		\end{equation*} 
		in $PSL_n(q)$. It is easy to see that there exist $a, b \in \mathbb{F}_q$ with $a^2 \rho - b^2 \rho^2 = 1$. Let $s$ be the image of 
		\begin{equation*}
			\begin{pmatrix}
				-b \rho & a \rho & & & \\
				-a & b\rho & & & \\
				& & \ddots & & \\
				& & & -b \rho & a \rho \\
				& & & -a & b \rho
			\end{pmatrix} \in SL_n(q)
		\end{equation*}
		in $PSL_n(q)$. By a direct calculation, $s \in C_{PSL_n(q)}(u)$. Another direct calculation shows that $s$ is an involution. Let $z_1$ denote the image of 
		\begin{equation*}
			\begin{pmatrix} -I_2 & \\ & I_{n-2}\end{pmatrix} \in SL_n(q)
		\end{equation*}
		in $PSL_n(q)$, and let $z_2$ denote the image of 
		\begin{equation*}
			\begin{pmatrix} I_2 & & \\ & - I_2 & \\ & & I_{n-4} \end{pmatrix} \in SL_n(q)
		\end{equation*} 
		in $PSL_n(q)$. Then one can easily verify that $\langle s, u, z_1, z_2 \rangle \le C_{PSL_n(q)}(u)$ is isomorphic to $E_{16}$. So a Sylow $2$-subgroup of $C_{PSL_n(q)}(u)$ has rank at least $4$. This is also true for $P$ as $w$ and $u$ are conjugate (see Proposition \ref{involutions of PSL(n,q)}).
	\end{proof}
	
	\begin{lemma}
		\label{involution_centralizers_are_core_free} 
		Let $n \ge 2$ be a natural number and let $\varepsilon \in \lbrace +,- \rbrace$. Also, let $T \in GL_n^{\varepsilon}(3) \setminus Z(GL_n^{\varepsilon}(3))$ such that $T^2 \in Z(GL_n^{\varepsilon}(3))$. Then $C_{GL_n^{\varepsilon}(3)}(T)$ is core-free.  
	\end{lemma} 
	
	\begin{proof}
		By Lemmas \ref{involutions_GL(n,q)} and \ref{diagonalizable_involutions_GU(n,q)}, we either have $C_{GL_n^{\varepsilon}(3)}(T) \cong GL_i^{\varepsilon}(3) \times GL_{n-i}^{\varepsilon}(3)$ for some $1 \le i < n$, or $n$ is even and $C_{GL_n^{\varepsilon}(3)}(T) \cong GL_{n/2}(9)$. So we have that $C_{GL_n^{\varepsilon}(3)}(T)$ is core-free. 
	\end{proof}
	
	It is easy to deduce the following two corollaries from Lemma \ref{involution_centralizers_are_core_free}.
	
	\begin{corollary}
		\label{involution_centralizers_are_core-free_corollary}
		Let $n \ge 2$ be a natural number and let $\varepsilon \in \lbrace +,- \rbrace$. Then any involution centralizer in $SL_n^{\varepsilon}(3)$ is core-free. 
	\end{corollary} 
	
	\begin{corollary}
		\label{involution_centralizers_are-core-free_corollary_2}
		Let $n \ge 2$ be a natural number and let $\varepsilon \in \lbrace +,- \rbrace$. Then any involution centralizer in $PGL_n^{\varepsilon}(3)$ is core-free. 
	\end{corollary}
	
	\subsection{Sylow $2$-subgroups and $2$-fusion systems}
	In this subsection, we consider several properties of Sylow $2$-subgroups and $2$-fusion systems of linear and unitary groups.
	
	\begin{lemma}
		\label{sylow_GL_2(q)}
		(\cite[p. 142]{CarterFong}) Let $q$ be a nontrivial odd prime power. Let $k,s \in \mathbb{N}$ such that $2^k$ is the $2$-part of $q-1$ and that $2^s$ is the $2$-part of $q+1$. Then:   
		\begin{enumerate}
			\item[(i)] Assume that $q \equiv 1 \ \mathrm{mod} \ 4$. Then
			\begin{equation*} 
				\left \lbrace 
				\begin{pmatrix}
					\lambda & \ \ \\
					\ \ & \ \mu \\ 
				\end{pmatrix} 
				\ : \ \lambda, \mu \ \textnormal{are $2$-elements of } \mathbb{F}_{q}^{*} \right \rbrace 
				\cdot
				\left \langle 
				\begin{pmatrix}
					0 \ & \ 1 \\
					1 \ & \ 0 \\ 
				\end{pmatrix} \right \rangle
			\end{equation*} 
			is a Sylow $2$-subgroup of $GL_2(q)$. In particular, the Sylow $2$-subgroups of $GL_2(q)$ are isomorphic to the wreath product $C_{2^k} \wr C_2$.
			\item[(ii)] If $q \equiv 3 \ \mathrm{mod} \ 4$, then the Sylow $2$-subgroups of $GL_2(q)$ are semidihedral of order $2^{s+2}$. 
		\end{enumerate} 
	\end{lemma}
	
	\begin{lemma}
		\label{sylow_GU_2(q)}
		(\cite[p. 143]{CarterFong}) Let $q$ be a nontrivial odd prime power. Let $k,s \in \mathbb{N}$ such that $2^k$ is the $2$-part of $q-1$ and that $2^s$ is the $2$-part of $q+1$. Then:  
		\begin{enumerate}
			\item[(i)] If $q \equiv 1 \ \mathrm{mod} \ 4$, then the Sylow $2$-subgroups of $GU_2(q)$ are semidihedral of order $2^{k+2}$.
			\item[(ii)] If $q \equiv 3 \ \mathrm{mod} \ 4$, then the Sylow $2$-subgroups of $GU_2(q)$ are isomorphic to the wreath product $C_{2^s} \wr C_2$. If $\varepsilon \in \mathbb{F}_{q^2}^{*}$ has order $2^s$, then a Sylow $2$-subgroup of $GU_2(q)$ is concretely given by 
			\begin{equation*} 
				W := \left \lbrace 
				\begin{pmatrix}
					\lambda \ & \ \ \\
					\ \ & \ \mu \\ 
				\end{pmatrix} 
				\ : \ \lambda, \mu \in \langle \varepsilon \rangle \right \rbrace 
				\cdot
				\left \langle 
				\begin{pmatrix}
					0 \ & \ 1 \\
					1 \ & \ 0 \\ 
				\end{pmatrix} \right \rangle.
			\end{equation*} 
		\end{enumerate} 
	\end{lemma}
	
	\begin{lemma}
		\label{sylow_SL_2(q)} 
		(\cite[Kapitel II, Satz 8.10 a)]{Huppert}) If $q$ is a nontrivial odd prime power, then a Sylow $2$-subgroup of $SL_2(q)$ is generalized quaternion of order $(q^2-1)_2$.  
	\end{lemma}
	
	\begin{lemma}
		\label{sylow_PSL_2(q)} 
		(\cite[Kapitel II, Satz 8.10 b)]{Huppert}) If $q$ is a nontrivial odd prime power, then $PSL_2(q)$ has dihedral Sylow $2$-subgroups of order $\frac{1}{2}(q^2-1)_2$. 
	\end{lemma} 
	
	\begin{lemma}
		\label{sylow_power_2} 
		(\cite[Lemma 1]{CarterFong}) Let $q$ be a nontrivial odd prime power and let $\varepsilon \in \lbrace +,- \rbrace$. Let $r$ be a positive integer. Let $W_r$ be a Sylow $2$-subgroup of $GL_{2^r}^{\varepsilon}(q)$. Then $W_r \wr C_2$ is isomorphic to a Sylow $2$-subgroup of $GL_{2^{r+1}}^{\varepsilon}(q)$. A Sylow $2$-subgroup of $GL_{2^{r+1}}^{\varepsilon}(q)$ is concretely given by
		\begin{equation*}
			\left \lbrace 
			\begin{pmatrix}
				A & \ \\
				\ & B \\ 
			\end{pmatrix} 
			\ : \ A, B \in W_r \right \rbrace 
			\cdot
			\left \langle 
			\begin{pmatrix}
				\ & I_{2^r} \\
				I_{2^r} & \ \\ 
			\end{pmatrix} \right \rangle.
		\end{equation*} 
	\end{lemma} 
	
	\begin{lemma}
		\label{sylow_general}
		(\cite[Theorem 1]{CarterFong}) Let $q$ be a nontrivial odd prime power and let $n$ be a positive integer. Let $\varepsilon \in \lbrace +,- \rbrace$. Let $0 \le r_1 < \dots < r_t$ such that $n = 2^{r_1} + \dots + 2^{r_t}$. Let $W_i \in \mathrm{Syl}_2(GL_{2^{r_i}}^{\varepsilon}(q))$ for all $1 \le i \le t$. Then $W_1 \times \dots \times W_t$ is isomorphic to a Sylow $2$-subgroup of $GL_n^{\varepsilon}(q)$. A Sylow $2$-subgroup of $GL_n^{\varepsilon}(q)$ is concretely given by
		\begin{equation*}
			\left \lbrace 
			\begin{pmatrix}
				A_1 & \ & \ \\
				\ & \ddots & \ \\ 
				\ & \ & A_t
			\end{pmatrix} 
			\ : \ A_i \in W_i \right \rbrace.
		\end{equation*} 
	\end{lemma}
	
	\begin{lemma}
		\label{sylow_2} 
		Let $q$ be a prime power with $q \equiv 3 \mod 4$. Let $W$ be a Sylow $2$-subgroup of $GL_2(q)$, and let $m \in \mathbb{N}$ such that $|W| = 2^m$. Then: 
		\begin{enumerate}
			\item[(i)] $W$ is semidihedral. In particular, there are elements $a, b \in W$ with $\mathrm{ord}(a)=2^{m-1}$ and $\mathrm{ord}(b)=2$ such that $a^b = a^{2^{m-2}-1}$.
			\item[(ii)] We have $W \cap SL_2(q) = \langle a^2 \rangle \langle ab \rangle$. 
			\item[(iii)] Let $1 \le \ell \le 2^{m-1}$. If $\ell$ is odd, then $a^{\ell}$ has determinant $-1$, and $a^{\ell}b$ has determinant $1$. If $\ell$ is even, then $a^{\ell}$ has determinant $1$, and $a^{\ell}b$ has determinant $-1$.  
			\item[(iv)] The involutions of $W$ are precisely the elements $a^{2^{m-2}}$ and $a^{\ell}b$, where $2 \le \ell \le 2^{m-1}$ is even.
		\end{enumerate}     
	\end{lemma} 
	
	\begin{proof}
		By Lemma \ref{sylow_GL_2(q)} (ii), we have (i). 
		
		Let $W_0 := W \cap SL_2(q)$. By Lemma \ref{sylow_SL_2(q)}, $W_0$ is generalized quaternion. Also, $W_0$ is a maximal subgroup of $W$ since $SL_2(q)$ has index $q-1$ in $GL_2(q)$ and $q \equiv 3 \mod 4$. By \cite[Chapter 5, Theorem 4.3 (ii) (b)]{Gorenstein}, we have $\Phi(W) = \langle a^2 \rangle$. So the maximal subgroups of $W$ are precisely the groups $M_1 := \langle a \rangle$, $M_2 :=\langle a^2 \rangle \langle b \rangle$ and $M_3 := \langle a^2 \rangle \langle ab \rangle$. One can check that $M_1 \cong C_{2^{n-1}}$, $M_2 \cong D_{2^{n-1}}$ and $M_3 \cong Q_{2^{n-1}}$. Consequently, $W_0 = \langle a^2 \rangle \langle ab \rangle$, and (ii) holds.  
		
		(iii) follows from (ii) since any element of $W \setminus W_0$ has determinant $-1$. 
		
		The proof of (iv) is an easy exercise. 
	\end{proof}
	
	\begin{lemma}
		\label{centralizer_sylow_SL}
		Let $q$ be a nontrivial odd prime power, $n$ a positive integer and $\varepsilon \in \lbrace +,- \rbrace$. Let $0 \le r_1 < \dots < r_t$ such that $n = 2^{r_1} + \dots + 2^{r_t}$. Then there is a Sylow $2$-subgroup $W$ of $G := GL_n^{\varepsilon}(q)$ containing all diagonal matrices in $G$ with $2$-power order such that $C_W(W \cap SL_n^{\varepsilon}(q))$ consists precisely of the matrices 
		\begin{equation*}
			\begin{pmatrix}
				\lambda_1 I_{2^{r_1}} \ & \ & \ \\
				\ & \ddots \ & \ \\
				\ & \ & \lambda_t I_{2^{r_t}}
			\end{pmatrix},
		\end{equation*}
		where $\lambda_1, \dots, \lambda_t$ are $2$-elements of $\mathbb{F}_q^{*}$ if $G = GL_n(q)$ and $2$-elements of $\mathbb{F}_{q^2}^{*}$ with $\lambda_i^{q+1}=1$ (for each $1 \le i \le t$) if $G=GU_n(q)$. 
	\end{lemma}
	
	\begin{proof}
		Using Lemmas \ref{sylow_GL_2(q)} and \ref{sylow_GU_2(q)}, one can check that the centralizer of a Sylow 2-subgroup of $SL_2^{\varepsilon}(q)$ inside a Sylow 2-subgroup of $GL_2^{\varepsilon}(q)$ is the Sylow 2-subgroup of $Z(GL_2^{\varepsilon}(q))$. Applying Lemma \ref{sylow_power_2} and arguing by induction, one can see that a similar statement holds for the centralizer of a Sylow 2-subgroup of $SL_{2^r}^{\varepsilon}(q)$ inside a Sylow 2-subgroup of $GL_{2^r}^{\varepsilon}(q)$ for all $r \ge 0$. Now we may apply Lemma \ref{sylow_general} to obtain a Sylow 2-subgroup of $G$ with the desired properties.   
	\end{proof}
	
	\begin{lemma}
		\label{center_fusion_system_SL(n,q)}
		Let $q$ be a nontrivial odd prime power, $n$ a positive integer and $\varepsilon \in \lbrace +,- \rbrace$. Let $G := SL_n^{\varepsilon}(q)$, and let $S$ be a Sylow $2$-subgroup of $G$. Then we have $Z(\mathcal{F}_S(G)) = S \cap Z(G)$. 
	\end{lemma} 
	
	\begin{proof}
		Let $0 \le r_1 < \dots < r_t$ such that $n = 2^{r_1} + \dots + 2^{r_t}$. By Lemma \ref{centralizer_sylow_SL}, we may assume that $Z(S)$ consists precisely of the matrices
		\begin{equation*}
			\begin{pmatrix}
				\lambda_1 I_{2^{r_1}} \ & \ & \ \\
				\ & \ddots \ & \ \\
				\ & \ & \lambda_t I_{2^{r_t}}
			\end{pmatrix},
		\end{equation*}
		where $\lambda_1, \dots, \lambda_t$ are 2-elements of $\mathbb{F}_q^{*}$ with $\lambda_1^{2^{r_1}}\cdots \lambda_t^{2^{r_t}} = 1$ if $G = SL_n(q)$ and 2-elements of $\mathbb{F}_{q^2}^{*}$ with $\lambda_i^{q+1}=1$ (for each $1 \le i \le t$) and $\lambda_1^{2^{r_1}}\cdots \lambda_t^{2^{r_t}} = 1$ if $G=SU_n(q)$. Moreover, by Lemma \ref{centralizer_sylow_SL}, we may assume that $S$ contains each diagonal matrix in $G$ of $2$-power order.
		
		Let $x$ be an element of $Z(S)$ with diagonal blocks $\lambda_1 I_{2^{r_1}}, \dots, \lambda_t I_{2^{r_t}}$. One can easily see that $x$ is $G$-conjugate to any diagonal matrix in $G$ that is obtained from $x$ by permuting its diagonal entries. It follows that, if $\lambda_i \ne \lambda_j$ for some $1 \le i \ne j \le t$, then $x \not\in Z(\mathcal{F}_S(G))$. This implies $Z(\mathcal{F}_S(G)) = S \cap Z(G)$. 
	\end{proof}
	
	\begin{proposition}
		\label{SL_SU_fusion} 
		Let $n$ be a positive integer. Let $q, q^{*}$ be nontrivial odd prime powers, and let $\varepsilon, \varepsilon^{*} \in \lbrace +,-\rbrace$. If $\varepsilon q \sim \varepsilon^{*}q^{*}$, then the $2$-fusion systems of $SL_n^{\varepsilon}(q)$ and $SL_n^{\varepsilon^{*}}(q^{*})$ are isomorphic.   
	\end{proposition} 
	
	\begin{proof}
		Assume that $\varepsilon \ne \varepsilon^{*}$. From $\varepsilon q \sim \varepsilon^{*}q^{*}$, it is easy to deduce that $\varepsilon q \equiv \varepsilon^{*} q^{*} \mod 8$ and $(q^2-1)_2 = ((q^{*})^2-1)_2$. So, in view of the remarks at the bottom of p. 11 of \cite{BMO2012}, we may apply \cite[Proposition 3.3 (a)]{BMO2012} to conclude that the 2-fusion system of $SL_n^{\varepsilon}(q)$ is isomorphic to the 2-fusion system of $SL_n^{\varepsilon^{*}}(q^{*})$. 
		
		Assume now that $\varepsilon = \varepsilon^{*}$. Using Dirichlet's theorem \cite[Theorem 3.3.1]{FineRosenberger}, one can easily see that there is an odd prime $q_0$ with $\varepsilon q \sim \varepsilon q^{*} \sim -\varepsilon q_0$. By the preceding paragraph, both the 2-fusion system of $SL_n^{\varepsilon}(q)$ and the 2-fusion system of $SL_n^{\varepsilon}(q^{*})$ are isomorphic to the 2-fusion system of $SL_n^{-\varepsilon}(q_0)$. Consequently, the 2-fusion systems of $SL_n^{\varepsilon}(q)$ and $SL_n^{\varepsilon^{*}}(q^{*})$ are isomorphic.  
	\end{proof}
	
	\begin{proposition}
		\label{PSL_PSU_fusion}
		Let $n$ be a positive integer. Let $q, q^{*}$ be nontrivial odd prime powers, and let $\varepsilon, \varepsilon^{*} \in \lbrace +,-\rbrace$. If $\varepsilon q \sim \varepsilon^{*} q^{*}$, then the $2$-fusion systems of $PSL_n^{\varepsilon}(q)$ and $PSL_n^{\varepsilon^{*}}(q^{*})$ are isomorphic. 
	\end{proposition} 
	
	\begin{proof}
		Let $S$ and $S^{*}$ be Sylow $2$-subgroups of $G:= SL_n^{\varepsilon}(q)$ and $G^{*} := SL_n^{\varepsilon^{*}}(q^{*})$, respectively. By Proposition \ref{SL_SU_fusion}, $\mathcal{F} := \mathcal{F}_S(G)$ and $\mathcal{F}^{*} := \mathcal{F}_{S^{*}}(G^{*})$ are isomorphic. Therefore, $\mathcal{F}/Z(\mathcal{F})$ and $\mathcal{F}^{*}/Z(\mathcal{F}^{*})$ are isomorphic. Lemma \ref{center_fusion_system_SL(n,q)} implies that $\mathcal{F}/(S \cap Z(G))$ and $\mathcal{F}^{*}/(S^{*} \cap Z(G^{*}))$ are isomorphic. Now the proposition follows from Lemma \ref{factor_systems_fusion_categories}.  
    	\end{proof}
	
	The following lemma shows together with \cite[Theorem 5.6.18]{Aschbacher2021} that the $2$-fusion system of $PSL_n(q)$ is simple whenever $q$ is odd and $n \ge 3$. 
	
	\begin{lemma}
		\label{PSL as Goldschmidt group} 
		Let $q$ be a nontrivial odd prime power and $n \ge 2$ a natural number such that $(n,q) \ne (2,3)$. Moreover, let $\varepsilon$ be an element of $\lbrace +,- \rbrace$. Then $PSL_n^{\varepsilon}(q)$ is a Goldschmidt group if and only if $n = 2$ and $q \equiv 3$ or $5 \ \mathrm{mod} \ 8$. 
	\end{lemma} 
	
	\begin{proof}
		Set $G := PSL_n^{\varepsilon}(q)$. 
		
		Assume that $n = 2$. Then $G \cong PSL_2(q)$. By Lemma \ref{sylow_PSL_2(q)}, $G$ has dihedral Sylow $2$-subgroups of order $\frac{1}{2}(q-1)_2(q+1)_2$. So, if $q \equiv 3$ or $5 \ \mathrm{mod} \ 8$, then $G$ has abelian Sylow $2$-subgroups and is thus a Goldschmidt group. If $q \equiv 1$ or $7 \ \mathrm{mod} \ 8$, then the Sylow $2$-subgroups of $G$ are dihedral of order at least $8$ and hence nonabelian. Moreover, if $q \equiv 1$ or $7 \ \mathrm{mod} \ 8$, then \cite[Theorem 37]{Steinberg} shows that $G$ is not isomorphic to a finite simple group of Lie type in characteristic $2$ of Lie rank $1$. So $G$ is not a Goldschmidt group if $q \equiv 1$ or $7 \ \mathrm{mod} \ 8$. 
		
		Assume now that $n \ge 3$. Again, we see from \cite[Theorem 37]{Steinberg} that there is no finite simple group of Lie type in characteristic $2$ of Lie rank $1$ which is isomorphic to $G$. Also, $G$ has a subgroup isomorphic to $SL_2^{\varepsilon}(q) \cong SL_2(q)$, and therefore, the Sylow $2$-subgroups of $G$ are nonabelian. Consequently, $G$ is not a Goldschmidt group.
	\end{proof} 
	
	\begin{lemma}
		\label{elementary abelian subgroups of SL(n,q)}
		Let $n$ be a positive integer, $q$ a nontrivial odd prime power and $\varepsilon \in \lbrace +,- \rbrace$. Let $E$ be the subgroup of $SL_n^{\varepsilon}(q)$ consisting of the diagonal matrices in $SL_n^{\varepsilon}(q)$ with diagonal entries in $\lbrace 1,-1 \rbrace$. Then $\vert E \vert = 2^{n-1}$. Moreover, any elementary abelian $2$-subgroup of $SL_n^{\varepsilon}(q)$ is conjugate to a subgroup of $E$.
	\end{lemma} 
	
	\begin{proof}
		It is straightforward to check that $\vert E \vert = 2^{n-1}$. 
		
		Let $E_0$ be an elementary abelian $2$-subgroup of $SL_n^{\varepsilon}(q)$. We show that $E_0$ is conjugate to a subgroup of $E$. Using Dirichlet's theorem \cite[Theorem 3.3.1]{FineRosenberger}, one can see that there is an odd prime number $q^{*}$ with $-q \sim q^{*}$, and Proposition \ref{SL_SU_fusion} shows that the $2$-fusion systems of $SU_n(q)$ and $SL_n(q^{*})$ are isomorphic. Therefore, it is enough to consider the case $\varepsilon = +$.
		
		Since $E_0$ is an elementary abelian $2$-group, any two elements of $E_0$ commute, and any element of $E_0$ is diagonalizable (see Lemma \ref{involutions_GL(n,q)}). It follows that $E_0$ is simultaneously diagonalizable, and this implies that $E_0$ is conjugate to a subgroup of $E$.  
	\end{proof}
	
	\begin{lemma}
		\label{PSL(n,q)-automorphisms of S}
		Let $q$ be a nontrivial odd prime power, $n \ge 3$ a natural number and $S$ a Sylow $2$-subgroup of $PSL_n(q)$. Then $\mathrm{Aut}_{PSL_n(q)}(S) = \mathrm{Inn}(S)$. 
	\end{lemma} 
	
	\begin{proof}
		Let $R \in \mathrm{Syl}_2(SL_n(q))$ such that $S$ is the image of $R$ in $PSL_n(q)$. Let $T$ be a Sylow $2$-subgroup of $GL_n(q)$ with $R \le T$. By \cite[Theorem 1]{Kondratev}, we have $N_{GL_n(q)}(R) = T C_{GL_n(q)}(T)$. So we have that $\mathrm{Aut}_{SL_n(q)}(R)$ is a $2$-group. Since the image of $N_{SL_n(q)}(R)$ in $PSL_n(q)$ equals $N_{PSL_n(q)}(S)$ (see \cite[Kapitel I, Hilfssatz 7.7 c)]{Huppert}), it follows that $\mathrm{Aut}_{PSL_n(q)}(S)$ is a $2$-group, and this implies $\mathrm{Aut}_{PSL_n(q)}(S) = \mathrm{Inn}(S)$. 
	\end{proof}
	
	\subsection{$k$-connectivity}
	In this subsection, we prove some connectivity properties of the Sylow $2$-subgroups of $SL_n(q)$ and $PSL_n(q)$, where $q$ is a nontrivial odd prime power and $n \ge 6$. We will work with the following definition (see \cite[Section 8]{GW}):
	\begin{definition} 
		Let $S$ be a finite $2$-group, and let $k$ be a positive integer. If $A$ and $B$ are elementary abelian subgroups of $S$ of rank at least $k$, then $A$ and $B$ are said to be \textit{$k$-connected} if there is a sequence 
		\begin{equation*}
			A = A_1, A_2, \dots, A_n = B \ \ \ \ (n \ge 1)
		\end{equation*} 
		of elementary abelian subgroups $A_i$, $1 \le i \le n$, of $S$ with rank at least $k$ such that  
		\begin{equation*} 
			A_i \subseteq A_{i+1} \ \textnormal{or} \ A_{i+1} \subseteq A_i
		\end{equation*}
		for all $1 \le i \le n-1$. The group $S$ is said to be \textit{$k$-connected} if any two elementary abelian subgroups of $S$ of rank at least $k$ are $k$-connected.
	\end{definition} 
	
	\begin{lemma} (\cite[Lemma 8.4]{GW})
		\label{GW_lemma_connectivity}
		Let $S$ be a finite $2$-group, and let $k$ be a positive integer. If $S$ has a normal elementary abelian subgroup of rank at least $2^{k-1}+1$, then $S$ is $k$-connected. 
	\end{lemma}
	
	\begin{lemma}
		\label{connectivity_q_1_mod_4}
		Let $q$ be a nontrivial odd prime power with $q \equiv 1 \ \mathrm{mod} \ 4$, and let $n \ge 6$ be a natural number. Then the Sylow $2$-subgroups of $PSL_n(q)$ and those of $SL_n(q)$ are $3$-connected.
	\end{lemma} 
	
	\begin{proof}
		Let $W_0$ be the unique Sylow $2$-subgroup of $GL_1(q)$, and let $W_1$ be the Sylow $2$-subgroup of $GL_2(q)$ given in Lemma \ref{sylow_GL_2(q)} (i). For each $r \ge 2$, let $W_r$ be the Sylow $2$-subgroup of $GL_{2^r}(q)$ obtained from $W_{r-1}$ by the construction given in the last statement of Lemma \ref{sylow_power_2}. Let $0 \le r_1 < \dots < r_t$ such that $n = 2^{r_1} + \dots + 2^{r_t}$, and let $W$ be the Sylow $2$-subgroup of $GL_n(q)$ obtained from $W_{r_1}, \dots, W_{r_t}$ by using the last statement of Lemma \ref{sylow_general}. 
		
		Let $R$ denote the subgroup of $GL_n(q)$ consisting of all diagonal matrices $D \in GL_n(q)$, where $D^2 \in Z(GL_n(q))$ and any diagonal element of $D$ is a $2$-element of $\mathbb{F}_q^{*}$. It is easy to note that $R \trianglelefteq W$. 
		
		Set $R_0 := R \cap SL_n(q)$. Then $\Omega_1(R_0)$, the subgroup of $R_0$ generated by all involutions of $R_0$, is elementary abelian of order $2^{n-1} \ge 2^5$, and $\Omega_1(R_0) \trianglelefteq W \cap SL_n(q)$. Also, $R_0 Z(SL_n(q))/Z(SL_n(q))$ is a normal elementary abelian subgroup of $(W \cap SL_n(q))Z(SL_n(q))/Z(SL_n(q))$, and one can easily check that the order of $R_0 Z(SL_n(q))/Z(SL_n(q))$ is at least $2^5$. Lemma \ref{GW_lemma_connectivity} implies that $W \cap SL_n(q)$ and its image in $PSL_n(q)$ are $3$-connected. 
	\end{proof}
	
	Lemma \ref{GW_lemma_connectivity} and the proof of Lemma \ref{connectivity_q_1_mod_4} show that we also have the following: 
	
	\begin{lemma}
		Let $q$ be a nontrivial odd prime power with $q \equiv 1 \ \mathrm{mod} \ 4$, and let $n \ge 6$ be a natural number. Then the Sylow $2$-subgroups of $PSL_n(q)$ and those of $SL_n(q)$ are $2$-connected.
	\end{lemma} 
	
	We now study the case $q \equiv 3 \mod 4$. 
	
	\begin{lemma}
		\label{connectivity_q_congruent_3_mod_4}
		Let $q$ be a nontrivial odd prime power with $q \equiv 3 \ \mathrm{mod} \ 4$, and let $n \ge 6$ be a natural number. Then the Sylow $2$-subgroups of $PSL_n(q)$ and those of $SL_n(q)$ are $2$-connected. If $n \ge 10$, then we even have that the Sylow $2$-subgroups of $PSL_n(q)$ and those of $SL_n(q)$ are $3$-connected. 
	\end{lemma} 
	
	\begin{proof}
		Let $W_0$ denote the unique Sylow $2$-subgroup of $GL_1(q)$, and let $W_1$ be a Sylow $2$-subgroup of $GL_2(q)$. By Lemma \ref{sylow_GL_2(q)} (ii), $W_1$ is semidihedral. Let $m \in \mathbb{N}$ with $|W_1| = 2^m$. Also, let $h,a \in W_1$ such that $\mathrm{ord}(h) = 2^{m-1}$, $\mathrm{ord}(a) = 2$ and $h^a = h^{2^{m-2}-1}$. Set $z := -I_2 = h^{2^{m-2}}$. For each $r \ge 2$, let $W_r$ be the Sylow $2$-subgroup of $GL_{2^r}(q)$ obtained from $W_{r-1}$ by the construction given in the last statement of Lemma \ref{sylow_power_2}. Let $0 \le r_1 < \dots < r_t$ such that $n = 2^{r_1} + \dots + 2^{r_t}$, and let $W$ be the Sylow $2$-subgroup of $GL_n(q)$ obtained from $W_{r_1}, \dots, W_{r_t}$ by using the last statement of Lemma \ref{sylow_general}.
		
		Given a natural number $\ell \ge 1$ and elements $x_1,\dots,x_{\ell} \in GL_2(q)$, we write $\mathrm{diag}(x_1,\dots,x_{\ell})$ for the block diagonal matrix 
		\begin{equation*}
			\begin{pmatrix} x_1 & & \\ & \ddots & \\ & & x_{\ell} \end{pmatrix}.
		\end{equation*}
		For each natural number $r \ge 1$, let $A_r$ denote the subgroup of $GL_{2^r}(q)$ consisting of the matrices $\mathrm{diag}(x_1, \dots, x_{2^{r-1}})$, where either $x_i \in \langle z \rangle$ for all $1 \le i \le 2^{r-1}$ or $x_i$ is an element of $\langle h \rangle$ with order $4$ for all $1 \le i \le 2^{r-1}$. By induction over $r$, one can see that $A_r \trianglelefteq W_r$ for all $r \ge 1$. Also, let $\widetilde{A_r} := \Omega_1(A_r)$ for all $r \ge 1$. Clearly, $\widetilde{A_r} \trianglelefteq W_r$ for all $r \ge 1$.   
		
		We now consider two cases. 
		
		\medskip
		
		\textit{Case 1: $n$ is even.}
		
		Let $E$ be the subgroup of $GL_n(q)$ consisting of the matrices $\mathrm{diag}(x_1,\dots,x_{\frac{n}{2}})$, where either $x_i \in \langle z \rangle$ for all $1 \le i \le \frac{n}{2}$ or $x_i$ is an element of $\langle h \rangle$ with order $4$ for all $1 \le i \le \frac{n}{2}$. Let $\widetilde{E} := \Omega_1(E)$. Since $A_{r_i} \trianglelefteq W_{r_i}$ for all $1 \le i \le t$, we have that $E$ and $\widetilde E$ are normal subgroups of $W$. Lemma \ref{sylow_2} (iii) shows that $E \le W \cap SL_n(q)$.
		
		As $\widetilde E$ is elementary abelian of order $2^{\frac{n}{2}}$, Lemma \ref{GW_lemma_connectivity} implies that $W \cap SL_n(q)$ is $2$-connected, and even $3$-connected if $n \ge 10$. Since $EZ(SL_n(q))/Z(SL_n(q))$ is a normal elementary abelian subgroup of $(W \cap SL_n(q))Z(SL_n(q))/Z(SL_n(q))$ with order $2^{\frac{n}{2}}$, Lemma \ref{GW_lemma_connectivity}  also shows that a Sylow $2$-subgroup is $2$-connected, and even $3$-connected if $n \ge 10$.  
		
		\medskip
		
		\textit{Case 2: $n$ is odd.}
		
		Now let $E$ denote the subgroup of $GL_n(q)$ consisting of the matrices 
		\begin{equation*}
			\left( \begin{array}{c|cc} 1 & \begin{matrix} & &  \end{matrix} \\ \hline \begin{matrix} & \\ & \\ & \end{matrix} & \begin{matrix} x_1 & & \\ & \ddots & \\ & & x_{\frac{n-1}{2}} \end{matrix} \end{array} \right),
		\end{equation*} 
		where $x_i \in \langle z \rangle$ for all $1 \le i \le \frac{n-1}{2}$. Since $\widetilde{A_{r_i}} \trianglelefteq W_{r_i}$ for all $2 \le i \le t$, we have that $E$ is a normal subgroup of $W \cap SL_n(q)$. Moreover, $E$ is elementary abelian of order $2^{\frac{n-1}{2}}$. Lemma \ref{GW_lemma_connectivity} implies that $W \cap SL_n(q)$ is $2$-connected, and even $3$-connected if $n \ge 11$. There is nothing else to show since the Sylow $2$-subgroups of $PSL_n(q)$ are isomorphic to those of $SL_n(q)$ (as $n$ is odd). 
	\end{proof}
	
	We show next that the groups $SL_n(q)$, where $6 \le n \le 9$ and $q \equiv 3 \mod 4$, and the groups $PSL_n(q)$, where $7 \le n \le 9$ and $q \equiv 3 \mod 4$, also have $3$-connected Sylow $2$-subgroups. 
	
	\begin{lemma}
		\label{3-connectivity_7} 
		Let $q$ be a nontrivial odd prime power with $q \equiv 3 \mod 4$. Then the Sylow $2$-subgroups of $SL_6(q)$ and those of $SL_7(q)$ are $3$-connected. 
	\end{lemma} 
	
	\begin{proof}
		Let $W_1$ be a Sylow $2$-subgroup of $GL_2(q)$, let $W_2$ be the Sylow $2$-subgroup of $GL_4(q)$ obtained from $W_1$ by the construction given in the last statement of Lemma \ref{sylow_power_2}, and let $W$ be the Sylow $2$-subgroup of $GL_6(q)$ obtained from $W_1$ and $W_2$ by using the last statement of Lemma \ref{sylow_general}.  
		
		From Lemma \ref{sylow_general}, we see that the Sylow $2$-subgroups of $SL_7(q)$ are isomorphic to those of $GL_6(q)$. So it is enough to show that $W$ and $W \cap SL_6(q)$ are $3$-connected. Given elements $x_1,x_2,x_3 \in GL_2(q)$, we write $\mathrm{diag}(x_1,x_2,x_3)$ for the block diagonal matrix
		\begin{equation*}
			\begin{pmatrix} x_1 & & \\ & x_2 & \\ & & x_3 \end{pmatrix}. 
		\end{equation*} 
		Let $A$ be the subgroup of $W \cap SL_6(q)$ consisting of the matrices $\mathrm{diag}(x_1,x_2,x_3)$, where $x_i \in \langle -I_2 \rangle$ for $1 \le i \le 3$. Clearly, $A \cong E_8$. We prove the following: 
		\begin{enumerate} 
			\item[(1)] If $E$ is an elementary abelian subgroup of $W$ of rank at least $3$, then $E$ is $3$-connected to an elementary abelian subgroup of $W \cap SL_6(q)$ of rank at least $3$. 
			\item[(2)] If $E$ is an elementary abelian subgroup of $W \cap SL_6(q)$ of rank at least $3$, then $E$ is $3$-connected to $A$ in $W \cap SL_6(q)$. 
		\end{enumerate}
		By (1) and (2), any elementary abelian subgroup of $W$ of rank at least $3$ is $3$-connected to $A$, and so $W$ is $3$-connected. Similarly, (2) implies that $W \cap SL_6(q)$ is $3$-connected.  
		
		Let $Z := \langle \mathrm{diag}(-I_2, I_2, I_2), \mathrm{diag}(I_2,-I_2,-I_2) \rangle$. Since $Z \le Z(W)$, we have that any elementary abelian subgroup of $W$ of rank at least $3$ is $3$-connected to an $E_8$-subgroup of $W$ containing $Z$. Also, any elementary abelian subgroup of $W \cap SL_6(q)$ of rank at least $3$ is $3$-connected (in $W \cap SL_6(q)$) to an $E_8$-subgroup of $W \cap SL_6(q)$ containing $Z$. Therefore, we only need to consider $E_8$-subgroups containing $Z$ in order to prove (1) and (2). 
		
		So let $E$ be an $E_8$-subgroup of $W$ with $Z \le E$, and let $s \in E \setminus Z$. Suppose that $s = \mathrm{diag}(s_1,s_2,s_3)$, where $s_1,s_2,s_3 \in W_1$. Then $[E,A] = 1$, and it is easy to deduce that $E$ is $3$-connected to $A$, so that $E$ satisfies (1). Also, if $E \le W \cap SL_6(q)$, it is easy to deduce that $E$ satisfies (2).  
		
		Suppose now that 
		\begin{equation*}
			s = \begin{pmatrix}
				s_1 & & \\ 
				& & s_2 \\
				& s_3 &
			\end{pmatrix} 
		\end{equation*}
		for some $s_1, s_2, s_3 \in W_1$. Since $s^2 = I_6$, we have $s_2 = s_3^{-1}$. Let $a$ be an involution of $W_1$ with $a \ne -I_2$. Set $s^{*} := \mathrm{diag}(I_2,a,a^{s_2})$ and $E^{*} := \langle Z, s^{*} \rangle \cong E_8$. Clearly, $E^{*} \le W \cap SL_6(q)$. It is easy to check that $[E,E^{*}] = 1$, which implies that $E$ is $3$-connected to $E^{*}$. So $E$ satisfies (1). If $E \le W \cap SL_6(q)$, then $E$ is $3$-connected to $E^{*}$ in $W \cap SL_6(q)$, and $E^{*}$ is $3$-connected to $A$ in $W \cap SL_6(q)$ since $[E^{*},A] = 1$. Therefore, $E$ satisfies (2) when $E \le W \cap SL_6(q)$.
	\end{proof} 
	
	Let $q$ be a nontrivial odd prime power with $q \equiv 3 \mod 4$. A Sylow $2$-subgroup of $PSL_7(q)$ is isomorphic to a Sylow $2$-subgroup of $SL_7(q)$. So, by Lemma \ref{3-connectivity_7}, the Sylow $2$-subgroups of $PSL_7(q)$ are $3$-connected.
	
	We need the following lemma in order to prove that the Sylow $2$-subgroups of $SL_n(q)$ and $PSL_n(q)$ are $3$-connected when $n \in \lbrace 8,9 \rbrace$. 
	
	\begin{lemma}
		\label{involutions in W2} 
		Let $q$ be a nontrivial odd prime power with $q \equiv 3 \ \mathrm{mod} \ 4$, and let $V$ be a Sylow $2$-subgroup of $GL_4(q)$. Let $u \in V$ with $u^2 = I_4$ or $u^2 = -I_4$. Then there is an involution $v \in V \setminus \langle u, -I_4 \rangle$ which commutes with $u$.
	\end{lemma} 
	
	\begin{proof}
		Fix a Sylow $2$-subgroup $W_1$ of $GL_2(q)$, and let $W_2$ be the Sylow $2$-subgroup of $GL_4(q)$ obtained from $W_1$ by the construction given in the last statement of Lemma \ref{sylow_power_2}. By Sylow’s Theorem, we may assume that $V = W_2$. Let $a$ be an involution of $W_1$ with $a \ne -I_2$. 
		
		First, we consider the case that
		\begin{equation*}
			u = \begin{pmatrix} x & \\ & y \end{pmatrix}
		\end{equation*}
		with elements $x,y \in W_1$. If $x \not\in \langle -I_2 \rangle$ or $y \not\in \langle -I_2 \rangle$, then 
		\begin{equation*}
			\begin{pmatrix} -I_2 & \\ & I_2 \end{pmatrix} \in W_2
		\end{equation*}
		is an involution commuting with $u$ and not lying in $\langle u, -I_4 \rangle$. If $x,y \in \langle -I_2 \rangle$, then we may choose 
		\begin{equation*}
			v := \begin{pmatrix} a & \\ & a \end{pmatrix}.
		\end{equation*}
		
		Assume now that 
		\begin{equation*}
			u = \begin{pmatrix} & x \\ y & \end{pmatrix} 
		\end{equation*}
		with elements $x,y \in W_1$. Let
		\begin{equation*}
			v := \begin{pmatrix}a & \\ & a^x \end{pmatrix}.
		\end{equation*} 
		As $a$ is an involution of $W_1$, we have that $v$ is an involution of $W_2$. By a direct calculation (using that $xy \in \langle -I_2 \rangle$), $v$ has the desired properties. 
	\end{proof} 
	
	\begin{lemma}
		\label{connectivity_9}
		Let $q$ be a nontrivial odd prime power with $q \equiv 3 \ \mathrm{mod} \ 4$. Then the Sylow $2$-subgroups of $SL_8(q)$ and those of $SL_9(q)$ are $3$-connected. 
	\end{lemma} 
	
	\begin{proof}
		Fix a Sylow $2$-subgroup $W_1$ of $GL_2(q)$, let $W_2$ be the Sylow $2$-subgroup of $GL_4(q)$ obtained from $W_1$ by the construction given in the last statement of Lemma \ref{sylow_power_2}, and let $W$ be the Sylow $2$-subgroup of $GL_8(q)$ obtained from $W_2$ by the construction given in the last statement of Lemma \ref{sylow_power_2}. Set $S := W \cap SL_8(q)$.   
		
		From Lemma \ref{sylow_general}, we see that the Sylow $2$-subgroups of $SL_9(q)$ are isomorphic to those of $GL_8(q)$. So it is enough to show that $W$ and $S$ are $3$-connected.
		
		Given a natural number $\ell \ge 1$ and $x_1, \dots, x_{\ell}$ of $GL_2(q) \cup GL_4(q)$, we write $\mathrm{diag}(x_1,\dots,x_{\ell})$ for the block diagonal matrix
		
		\begin{equation*}
			\begin{pmatrix}
				x_1 & & \\
				& \ddots & \\
				& & x_{\ell}
			\end{pmatrix}. 
		\end{equation*}
		
		Set
		\begin{equation*}
			A := \left \lbrace \mathrm{diag}(x_1,x_2,x_3,x_4) \ \vert \ x_i \in \langle -I_2 \rangle \ \forall \ 1 \le i \le 4 \right \rbrace \le S
		\end{equation*}
		and
		\begin{equation*}
			Z := \langle - I_8 \rangle \le S. 
		\end{equation*} 
		Clearly, $A \cong E_{16}$. Since $Z \le Z(W)$, we have that any elementary abelian subgroup of $W$ of rank at least $3$ is $3$-connected to an $E_8$-subgroup of $W$ containing $Z$. Similarly, any elementary abelian subgroup of $S$ of rank at least $3$ is $3$-connected to an $E_8$-subgroup of $S$ containing $Z$. So it suffices to prove that any $E_8$-subgroup $E$ of $W$ with $Z \le E$ is $3$-connected to $A$, where $E$ is even $3$-connected in $S$ to $A$ if $E \le S$. Thus let $E$ be an $E_8$-subgroup of $W$ containing $Z$, and let $x,y \in E$ with $E = \langle Z, x, y \rangle$. 
		
		We consider a number of cases. Below, $a$ will always denote an involution of $W_1$ with $a \ne -I_2$. 
		
		\medskip
		
		\textit{Case 1: $x = \mathrm{diag}(-I_4,I_4)$ and $y = \mathrm{diag}(b_1,b_2)$ for some $b_1,b_2 \in W_2$}.
		
		We determine an involution $y_1 \in C_W(E) \setminus \langle Z, x \rangle$ such that $\langle Z, x, y_1 \rangle \cong E_8$ is $3$-connected to $A$. In the case that $E \le S$, we determine $y_1$ such that $y_1 \in S$ and such that $\langle Z, x, y_1 \rangle$ is $3$-connected to $A$ in $S$. The existence of such an involution $y_1$ easily implies that $E$ is $3$-connected to $A$, and even $3$-connected to $A$ in $S$ if $E \le S$. The involution $y_1$ is given by the following table in dependence of $y$. In each row, $r_1, r_2, r_3, r_4$ are assumed to be elements of $W_1$ such that $y$ is equal to the matrix given in the column “$y$” and such that the conditions in the column “Conditions” (if any) are satisfied. The column “$y_1$” gives the involution $y_1$ with the desired properties. For each row, one can verify the stated properties of $y_1$ by a direct calculation or by using the previous rows. 
		
		\medskip
		
		\begin{tabular}{c | c| c| c}
			\textbf{Case} & \textbf{$y$} & \textbf{Conditions} & {$y_1$} \\
			\hline
			
			1.1 & $\begin{pmatrix} r_1 & & & \\ & r_2 & & \\ & & r_3 & \\ & & & r_4 \end{pmatrix}$ &  & $y$ \\
			
			1.2 & $\begin{pmatrix} r_1 & & & \\ & r_2 & & \\ & & & r_3 \\ & & r_4 & \end{pmatrix}$ & $\langle r_1, r_2 \rangle \not\le \langle -I_2 \rangle$ & $\begin{pmatrix} r_1 & &  \\ & r_2 & \\ & & I_4 \end{pmatrix}$ \\
			
			1.3 & $\begin{pmatrix} r_1 & & & \\ & r_2 & & \\ & & & r_3\\ & & r_4 &  \end{pmatrix}$ & $r_1, r_2 \le \langle -I_2 \rangle$ & $\begin{pmatrix} a & & \\ & a & \\ & & I_4 \end{pmatrix}$ \\
			
			1.4 & $\begin{pmatrix} & r_1 & & \\ r_2 & & & \\ & & r_3 &\\ & & & r_4 \end{pmatrix}$ & $\langle r_3, r_4 \rangle \not\le \langle -I_2 \rangle$ & $\begin{pmatrix} I_4 & & \\ & r_3 & \\ & & r_4 \end{pmatrix}$ \\
			
			1.5 & $\begin{pmatrix} & r_1 & & \\ r_2 & & & \\ & & r_3 &\\ & & & r_4 \end{pmatrix}$ & $r_3, r_4 \le \langle -I_2 \rangle$ & $\begin{pmatrix} I_4 & & \\ & a & \\ & & a \end{pmatrix}$ \\
			
			1.6 & $\begin{pmatrix} & r_1 & & \\r_2 & & & \\ & & & r_3\\ & & r_4 & \end{pmatrix}$ &  & $\begin{pmatrix} & r_1 & \\ r_2 & & \\ & & I_4 \end{pmatrix}$
		\end{tabular}
		
		\medskip
		
		\textit{Case 2: $x = \mathrm{diag}(a_1,a_2)$ and $y = \mathrm{diag}(b_1,b_2)$ for some $a_1, a_2, b_1, b_2 \in W_2$.}
		
		Set $x_1 := \mathrm{diag}(-I_4,I_4)$. Since $E = \langle Z, x, y \rangle \cong E_8$, the elements $x$ and $y$ cannot be both contained in $\langle Z, x_1 \rangle$. Without loss of generality, we may assume that $y \not\in \langle Z, x_1 \rangle$. Then $E_1 := \langle Z, x_1, y \rangle \cong E_8$. The group $E_1$ is $3$-connected to $A$ by Case 1, and it is $3$-connected to $E$ since $E$ and $E_1$ commute. Hence, $E$ is $3$-connected to $A$. Clearly, if $E \le S$, then $E$ is even $3$-connected in $S$ to $A$. 
		
		\medskip
		
		\textit{Case 3: There are $a_1, a_2, b_1, b_2 \in W_2$ with
			\begin{equation*}
				\lbrace x, y \rbrace = \left \lbrace \begin{pmatrix} a_1 & \\ & a_2 \end{pmatrix}, \begin{pmatrix} & b_1 \\ b_2 & \end{pmatrix} \right \rbrace.
		\end{equation*}}
		
		Without loss of generality, we assume that 
		\begin{equation*}
			x = \begin{pmatrix} a_1 & \\ & a_2 \end{pmatrix} \ \textnormal{and} \ y = \begin{pmatrix} & b_1 \\ b_2 & \end{pmatrix}.
		\end{equation*}
		Since $x$ and $y$ are commuting involutions, we have $b_1 = b_2^{-1}$ and $a_2 = {a_1}^{b_1}$. By Lemma \ref{involutions in W2}, there is an involution $\widetilde{a_1} \in W_2 \setminus \langle a_1, -I_4 \rangle$ which commutes with $a_1$. Set 
		\begin{equation*}
			y_1 := \begin{pmatrix} \widetilde{a_1} & \\ & {\widetilde{a_1}}^{b_1} \end{pmatrix}.
		\end{equation*}
		It is easy to see that $y_1 \in S$, and $y_1$ is an involution since $\widetilde{a_1}$ is an involution of $W_2$. We have $[x,y_1] = 1$ since $\widetilde{a_1}$ commutes with $a_1$ and $\widetilde{a_1}^{b_1}$ commutes with ${a_1}^{b_1} = a_2$. A direct calculation using that $b_1 = b_2^{-1}$ shows that we also have $[y,y_1] = 1$. Thus $E = \langle Z,x,y \rangle$ commutes with $E_1 := \langle Z, x, y_1 \rangle$. Since $\widetilde{a_1} \not\in \langle a_1, -I_4 \rangle$, we have $y_1 \not\in \langle Z,x \rangle$ and hence $E_1 \cong E_8$. Applying Case 2, it follows that $E$ is $3$-connected to $A$ (and even $3$-connected in $S$ to $A$ when $E \le S$). 
		
		\medskip
		
		\textit{Case 4: There are $a_1, a_2, b_1, b_2 \in W_2$ with 
			\begin{equation*}
				x = \begin{pmatrix} & a_1 \\ a_2 & \end{pmatrix} \ \textnormal{\textit{and}} \ y = \begin{pmatrix} & b_1 \\ b_2 & \end{pmatrix}.
			\end{equation*}
		}
		
		This case can be reduced to Case 3 since $E = \langle Z, x, y \rangle = \langle Z, x, xy \rangle$.
	\end{proof}
	
	Let $q$ be a nontrivial odd prime power with $q \equiv 3 \mod 4$. A Sylow $2$-subgroup of $PSL_9(q)$ is isomorphic to a Sylow $2$-subgroup of $SL_9(q)$. So, by Lemma \ref{connectivity_9}, the Sylow $2$-subgroups of $PSL_9(q)$ are $3$-connected. 
	
	\begin{lemma}
		Let $q$ be a nontrivial odd prime power with $q \equiv 3 \ \mathrm{mod} \ 4$. Then the Sylow $2$-subgroups of $PSL_8(q)$ are $3$-connected. 
	\end{lemma}
	
	\begin{proof}
		Let $W_1$ be a Sylow $2$-subgroup of $GL_2(q)$. Let $W_2$ be the Sylow $2$-subgroup of $GL_4(q)$ obtained from $W_1$ by the construction given in the last statement of Lemma \ref{sylow_power_2}, and let $W_3$ be the Sylow $2$-subgroup of $GL_8(q)$ obtained from $W_2$ by the construction given in the last statement of Lemma \ref{sylow_power_2}. Set $S := W_3 \cap SL_8(q)$. For each subgroup or element $X$ of $SL_8(q)$, let $\widebar X$ denote the image of $X$ in $PSL_8(q)$. We prove that $\widebar S$ is $3$-connected. 
		
		Given a natural number $\ell \ge 1$ and $x_1, \dots, x_{\ell}$ of $GL_2(q) \cup GL_4(q)$, we write $\mathrm{diag}(x_1,\dots,x_{\ell})$ for the block diagonal matrix
		
		\begin{equation*}
			\begin{pmatrix}
				x_1 & & \\
				& \ddots & \\
				& & x_{\ell}
			\end{pmatrix}. 
		\end{equation*}
		
		Set
		\begin{equation*}
			A := \left\lbrace \mathrm{diag}(x_1,x_2,x_3,x_4) \ \vert \ x_i \in \langle -I_2 \rangle \ \forall \ 1 \le i \le 4 \right\rbrace \le S.
		\end{equation*} 
		We have $\widebar A \cong E_8$.
		
		Set 
		\begin{equation*}
			Z := \left \langle \mathrm{diag}(-I_4,I_4) \right \rangle.
		\end{equation*} 
		We have $\widebar{Z} \le Z(\widebar S)$. Using this, it is easy to note that any elementary abelian subgroup of $\widebar S$ of rank at least $3$ is $3$-connected to an $E_8$-subgroup of $\widebar S$ containing $\widebar Z$. Hence, it suffices to prove that any $E_8$-subgroup of $\widebar S$ containing $\widebar Z$ is $3$-connected to $\widebar A$. 
		
		Let $x, y \in S$ and $B := \langle \widebar Z, \widebar x, \widebar y \rangle$. Suppose that $B \cong E_8$. Considering a number of cases, we will prove that $B$ is $3$-connected to $\widebar A$. Below, $a$ will always denote an involution of $W_1$ with $a \ne -I_2$. 
		
		\medskip
		
		\textit{Case 1: $x = \mathrm{diag}(r_1,r_2,r_3,r_4) \ \textnormal{\textit{and}} \ y = \mathrm{diag}(m_1,m_2)$ for some $r_1,r_2,r_3,r_4 \in W_1$ and $m_1,m_2 \in W_2$.}
		
		We consider a number of subcases. These subcases are given by the rows of the table below. In each row, we assume that $s_1, s_2, s_3, s_4$ are elements of $W_1$ such that $y$ is equal to the matrix given in the column “$y$”. We also assume that the conditions in the column “Conditions” (if any) are satisfied. The column “$y_1$” gives an element of $S$ such that $\widebar{y_1}$ is an involution in $C_{\widebar S}(\widebar E) \setminus \langle \widebar Z, \widebar x \rangle$ and such that $\langle \widebar Z, \widebar x, \widebar{y_1} \rangle$ is $3$-connected to $\widebar A$. The existence of such an element $y_1$ easily implies that $B$ is $3$-connected to $\widebar A$. 
		
		\medskip 
		
		\begin{tabular}{c | c| c| c}
			\textbf{Case} & \textbf{$y$} & \textbf{Conditions} & {$y_1$}\\
			\hline
			
			1.1 & $\begin{pmatrix} s_1 & & & \\ & s_2 & & \\ & & s_3 & \\ & & & s_4 \end{pmatrix}$ &  & $y$\\
			
			1.2 & $\begin{pmatrix} & s_1& & \\ s_2& & & \\ & & s_3 & \\ & & & s_4\end{pmatrix}$ & $x \not\in A$ & $\begin{pmatrix} I_4 & &  \\ & -I_2 & \\ & & I_2 \end{pmatrix}$ \\
			
			1.3 & $\begin{pmatrix} & s_1& & \\ s_2& & & \\ & & s_3 & \\ & & & s_4\end{pmatrix}$ & $x \in A$ & $\begin{pmatrix} a & &  \\ & a^{s_2^{-1}} & \\ & & I_4 \end{pmatrix}$ \\
			
			1.4 & $\begin{pmatrix} & s_1 & & \\s_2 & & & \\ & & & s_3\\ & & s_4 & \end{pmatrix}$ & $x \not\in A$ & $\begin{pmatrix} I_2 & & & \\ & -I_2 & & \\ & & I_2 & \\ & & & -I_2\end{pmatrix}$ \\
			
			1.5 & $\begin{pmatrix} & s_1 & & \\s_2 & & & \\ & & & s_3\\ & & s_4 & \end{pmatrix}$ & $x \in A$ & $\begin{pmatrix} a & & & \\ & a^{s_2^{-1}} & & \\ & & a & \\ & & & a^{s_4^{-1}} \end{pmatrix}$
		\end{tabular}
		
		\medskip 
		
		The subcase that $y$ has the form
		\begin{equation*}
			\begin{pmatrix} s_1 &  & & \\ & s_2 & & \\ & & & s_3\\ & & s_4 & \end{pmatrix}
		\end{equation*}
		can be easily reduced to Cases 1.2 and 1.3.
		
		\medskip 
		
		\textit{Case 2: There are $r_1,r_2,r_3,r_4 \in W_1$ and $m_1, m_2 \in W_2$ with
			\begin{equation*}
				x = \begin{pmatrix} & r_1 & & \\
					r_2 & & & \\
					& & r_3 & \\
					& & & r_4
				\end{pmatrix}  \ \textnormal{\textit{and}} \ y = \begin{pmatrix} m_1 & \\ & m_2 \end{pmatrix}.
		\end{equation*}}
		
		\medskip 
		
		\textit{Case 2.1: There are $s_1,s_2,s_3,s_4 \in W_1$ with 
			\begin{equation*}
				y = \begin{pmatrix} s_1 & & & \\
					& s_2 & & \\
					& & s_3 & \\
					& & & s_4
				\end{pmatrix} \ \textnormal{\textit{or}} \
				y = \begin{pmatrix} & s_1 & & \\
					s_2 & & & \\
					& & s_3 & \\
					& & & s_4
				\end{pmatrix}.
		\end{equation*}}
		Noticing that $\langle \widebar Z, \widebar x, \widebar y \rangle = \langle \widebar Z, \widebar x, \widebar x \widebar y \rangle$, this case can be reduced to Case 1.   
		
		\medskip
		
		\textit{Case 2.2: There are $s_1,s_2,s_3,s_4 \in W_1$ with 
			\begin{equation*}
				y = \begin{pmatrix} s_1 & & & \\
					& s_2 & & \\
					& &  & s_3 \\
					& & s_4 & 
				\end{pmatrix}.
		\end{equation*}}
		Since $B \cong E_8$, we have $\varepsilon x^y = x$, where $\varepsilon \in \lbrace +,- \rbrace$. By a direct calculation, we have 
		\begin{equation*} 
			x^y = \begin{pmatrix} & s_1^{-1}r_1 s_2 & & \\
				s_2^{-1} r_2 s_1 & & & \\
				& & r_4^{s_4} & \\
				& & & r_3^{s_3} \end{pmatrix}.
		\end{equation*} 
		As $x = \varepsilon x^y$, we have $r_1 = \varepsilon s_1^{-1} r_1 s_2$, $r_2 = \varepsilon s_2^{-1}r_2s_1$, $r_3 = \varepsilon r_4^{s_4}$ and $r_4 = \varepsilon r_3^{s_3}$. Note that $\varepsilon s_1^{r_1} = s_2$ and $\varepsilon s_2^{r_2}=s_1$. 
		
		We now consider a number of subsubcases. These subsubcases are given by the rows of the table below. The columns “Condition 1” and “Condition 2” describe the subsubcase under consideration. The column “$y_1$” gives an element $y_1 \in S$ such that $\widebar{y_1}$ is an involution in $C_{\widebar S}(\widebar E) \setminus \langle \widebar Z, \widebar x \rangle$ and such that $\langle \widebar Z, \widebar x, \widebar{y_1} \rangle$ is $3$-connected to $\widebar A$. In each subsubcase, one can see from the above calculations and from the previous cases that $y_1$ indeed has the stated properties. The existence of such an element $y_1$ easily implies that $B$ is $3$-connected to $\widebar A$ in all subsubcases. 
		
		\medskip
		
		\begin{tabular}{c | c| c| c}
			\textbf{Case} & \textbf{Condition 1} & \textbf{Condition 2} & {$y_1$}\\
			\hline
			
			2.2.1 & $x^2 = I_8 = y^2$ & $\langle r_3, r_4 \rangle \not\le \langle -I_2 \rangle$ & $\begin{pmatrix} \varepsilon s_1 & & & \\ & s_2 & & \\ & & \varepsilon r_3 & \\ & & & r_4 \end{pmatrix}$\\
			
			2.2.2 & $x^2 = I_8 = y^2$ & $\langle r_3, r_4 \rangle \le \langle -I_2 \rangle$ & $\begin{pmatrix} & r_1 & & \\ r_2 & & & \\ & & \varepsilon a & \\ & & & a^{s_3} \end{pmatrix}$\\
			
			2.2.3 & $x^2 = -I_8 = y^2$ & & $\begin{pmatrix} \varepsilon s_1 & & & \\ & s_2 & & \\ & & \varepsilon r_3 & \\ & & & r_4 \end{pmatrix}$\\
			
			2.2.4 & $x^2 = I_8, y^2 = -I_8$ & $\langle r_3, r_4 \rangle \not\le \langle -I_2 \rangle$ & $\begin{pmatrix} I_4 & & \\ & \varepsilon r_3 & \\ & & r_4\end{pmatrix}$\\
			
			2.2.5 & $x^2 = I_8, y^2 = -I_8$ & $\langle r_3, r_4 \rangle \le \langle -I_2 \rangle$ & $\begin{pmatrix} I_4 & & \\ & \varepsilon a & \\ & & \varepsilon a^{s_3}\end{pmatrix}$
		\end{tabular}
		
		\medskip
		
		The case that $x^2 = -I_8$ and $y^2 = I_8$ can be easily reduced to Cases 2.2.4 and 2.2.5.   
		
		\medskip
		
		\textit{Case 2.3: There are $s_1,s_2,s_3,s_4 \in W_1$ with 
			\begin{equation*}
				y = \begin{pmatrix}  & s_1 & & \\
					s_2 & & & \\
					& &  & s_3 \\
					& & s_4 & 
				\end{pmatrix}.
		\end{equation*}}
		Since $\langle \widebar Z, \widebar x, \widebar y \rangle = \langle \widebar Z, \widebar x, \widebar x \widebar y \rangle$, this case can be reduced to Case 2.2.  
		
		\medskip
		
		\textit{Case 3: There are $r_1,r_2,r_3,r_4 \in W_1$ and $m_1,m_2 \in W_2$ with 
			\begin{equation*}
				x = \begin{pmatrix} r_1 & & & \\
					& r_2 & & \\
					& &  & r_3 \\
					& & r_4 & 
				\end{pmatrix} \ \textnormal{\textit{and}} \ 
				y = \begin{pmatrix} m_1 & \\ & m_2 \end{pmatrix}. 
		\end{equation*}}
		This case can be reduced to Case 2. 
		
		\medskip
		
		\textit{Case 4: There are $r_1,r_2,r_3,r_4 \in W_1$ and $m_1,m_2 \in W_2$ with 
			\begin{equation*}
				x = \begin{pmatrix} & r_1 & & \\
					r_2 & & & \\
					& &  & r_3 \\
					& & r_4 & 
				\end{pmatrix} \ \textnormal{\textit{and}} \ 
				y = \begin{pmatrix} m_1 & \\ & m_2 \end{pmatrix}.
		\end{equation*}}
		
		In view of Cases 1-3, we may assume that  
		\begin{equation*}
			y = \begin{pmatrix} & s_1 & & \\
				s_2 & & & \\
				& &  & s_3 \\
				& & s_4 &  \end{pmatrix}
		\end{equation*}
		for some $s_1,s_2,s_3,s_4 \in W_1$. Since $\langle \widebar Z, \widebar x, \widebar y \rangle = \langle \widebar Z, \widebar x, \widebar x \widebar y \rangle$, we can now reduce the given case to Case 1. 
		
		\medskip 
		
		\textit{Case 5: There are $a_1,a_2,b_1,b_2 \in W_2$ with 
			\begin{equation*}
				\lbrace x,y \rbrace = 
				\left \lbrace
				\begin{pmatrix} a_1 & \\ & a_2 \end{pmatrix},
				\begin{pmatrix} & b_1 \\ b_2 & \end{pmatrix} 
				\right\rbrace. 
		\end{equation*}}
		
		Without loss of generality, we assume that 
		\begin{equation*}
			x = \begin{pmatrix} a_1 & \\ & a_2 \end{pmatrix} \ \textnormal{and} \ y = \begin{pmatrix} & b_1 \\ b_2 & \end{pmatrix}. 
		\end{equation*} 
		We have $x^2 \in \langle -I_8 \rangle$ since $B = \langle \widebar Z, \widebar x, \widebar y \rangle \cong E_8$, and hence ${a_1}^2 \in \langle -I_4 \rangle$. So, by Lemma \ref{involutions in W2}, there is an involution $\widetilde{a_1} \in W_2 \setminus \langle a_1, -I_4 \rangle$ which commutes with $a_1$. Set 
		\begin{equation*}
			y_1 := \begin{pmatrix} \widetilde{a_1} & \\ & \widetilde{a_1}^{b_1} \end{pmatrix}. 
		\end{equation*}  
		Clearly, $\widebar{y_1}$ is an involution of $\widebar S$. As $[x,y] \in \langle -I_8 \rangle$, we have ${a_1}^{b_1} \in \lbrace a_2, -a_2 \rbrace$. Since $a_1$ and $\widetilde{a_1}$ commute, it follows that $\widetilde{a_1}^{b_1}$ and $a_2$ commute. So we have $[x,y_1] = 1$ and hence $[\widebar{x},\widebar{y_1}] = 1$. Using that $y^2 \in \langle -I_8 \rangle$, one can easily verify that $[y,y_1] = 1$ and hence $[\widebar y, \widebar{y_1}] = 1$. As $\widetilde{a_1} \not\in \langle a_1, -I_4 \rangle$, we have $\widebar{y_1} \not\in \langle \widebar Z, \widebar x \rangle$.
		
		Now $\langle \widebar Z, \widebar x, \widebar{y_1} \rangle$ is an $E_8$-subgroup of $\widebar S$ which commutes with $B$ and which is $3$-connected to $\widebar A$ by Cases 1-4. Thus $B$ is $3$-connected to $\widebar A$. 
		
		\medskip 
		
		\textit{Case 6: There are $a_1,a_2,b_1,b_2 \in W_2$ with
			\begin{equation*}
				x = \begin{pmatrix} & a_1 \\ a_2 & \end{pmatrix} \ \textnormal{\textit{and}} \
				y = \begin{pmatrix} & b_1 \\ b_2 & \end{pmatrix}.
		\end{equation*}}
		Noticing that $\langle \widebar Z, \widebar x, \widebar y \rangle = \langle \widebar Z, \widebar x, \widebar x \widebar y \rangle$, we can reduce this case to Case 5.
	\end{proof}
	
	We summarize the above lemmas in the following corollary. 
	
	\begin{corollary}
		\label{conclusion connectivity}
		Let $q$ be a nontrivial odd prime power and $n \ge 6$. Then the following hold: 
		\begin{enumerate}
			\item[(i)] The Sylow $2$-subgroups of $SL_n(q)$ and those of $PSL_n(q)$ are $2$-connected. 
			\item[(ii)] The Sylow $2$-subgroups of $SL_n(q)$ are $3$-connected. 
			\item[(iii)] If $q \equiv 1 \ \mathrm{mod} \ 4$ or $n \ge 7$, then the Sylow $2$-subgroups of $PSL_n(q)$ are $3$-connected. 
		\end{enumerate} 
	\end{corollary}
	
	Unfortunately, the Sylow $2$-subgroups of $PSL_6(q)$ are not $3$-connected when $q \equiv 3 \mod 4$ (this is not terribly difficult to observe). 
	
	\begin{corollary}
		\label{conclusion_connectivity_2} 
		Let $q$ be a nontrivial odd prime power and $n \ge 6$. Let $G = SL_n(q)$, or $G = PSL_n(q)$ and $n \ge 7$ if $q \equiv 3 \mod 4$. For any Sylow $2$-subgroup $S$ of $G$ and any elementary abelian subgroup $A$ of $S$ with $m(A) \le 3$, there is some elementary abelian subgroup $B$ of $S$ with $A < B$ and $m(B) = 4$. 
	\end{corollary} 
	
	\begin{proof}
		By Corollary \ref{conclusion connectivity}, $S$ is $2$-connected and $3$-connected. Applying \cite[Lemma 8.7]{GW}, the claim follows.
	\end{proof} 
	
	\subsection{Generation} 
	Next we discuss some generational properties of $(P)SL_n(q)$ and $(P)SU_n(q)$, where $n \ge 3$ and $q$ is a nontrivial odd prime power. We need the following definition (see \cite[Section 8]{GW}). 
	
	\begin{definition}
		\label{def_k-generation} 
		Let $G$ be a finite group, let $S$ be a Sylow $2$-subgroup of $G$, and let $k$ be a positive integer. We say that $G$ is \textit{$k$-generated} if 
		\begin{equation*}
			G = \Gamma_{S,k}(G) := \langle N_G(T) \ | \ T \le S, m(T) \ge k \rangle.
		\end{equation*}  
	\end{definition}
	
	The following two lemmas will later prove to be useful. 
	
	\begin{lemma}
		\label{generation1}
		(see \cite{Aschbacher1974}) Let $q$ be a nontrivial odd prime power. Then the groups $SL_3(q)$, $PSL_3(q)$, $SU_3(q)$ and $PSU_3(q)$ are $2$-generated. 
	\end{lemma} 
	
	\begin{lemma}
		\label{generation2} 
		Let $q$ be a nontrivial odd prime power, and let $n \ge 4$ be a natural number. Moreover, let $\varepsilon \in \lbrace +,- \rbrace$ and $Z \le Z(SL_n^{\varepsilon}(q))$. Assume that one of the following holds: 
		\begin{enumerate}
			\item[(i)] $n \ge 5$,
			\item[(ii)] $q \equiv \varepsilon \ \mathrm{mod} \ 8$, 
			\item[(iii)] $Z = 1$. 
		\end{enumerate} 
		Then $SL_n^{\varepsilon}(q)/Z$ is $3$-generated. 
	\end{lemma} 
	
	We need the following lemma in order to prove Lemma \ref{generation2}.
	
	\begin{lemma}
		\label{generation0} 
		(see \cite{Phan1970}, \cite{BennettShpectorov}) Let $q > 2$ be a prime power, and let $n \ge 3$ be a natural number. Let $\varepsilon \in \lbrace +,- \rbrace$. Define 
		\begin{equation*}
			U_1 := \left\lbrace \begin{pmatrix} A & \\ & I_{n-2} \end{pmatrix} \ : \ A \in SL_2^{\varepsilon}(q) \right\rbrace
		\end{equation*} and 
		\begin{equation*}
			U_{n-1} := \left\lbrace \begin{pmatrix} I_{n-2} & \\ & A \end{pmatrix} \ : \ A \in SL_2^{\varepsilon}(q) \right\rbrace. 
		\end{equation*}
		Moreover, for each $2 \le i \le n-2$, let 
		\begin{equation*}
			U_i := \left\lbrace \begin{pmatrix} I_{i-1} & & \\ & A & \\ & & I_{n-i-1} \end{pmatrix} \ : \ A \in SL_2^{\varepsilon}(q) \right\rbrace.
		\end{equation*}
		Then the following hold: 
		\begin{enumerate}
			\item[(i)] We have $SL_n^{\varepsilon}(q) = \langle U_i \ : \ 1 \le i \le n-1 \rangle$.
			\item[(ii)] For each $1 \le i \le n-2$, there is a monomial matrix $m_i$ in $SL_n^{\varepsilon}(q)$ with $U_i^{m_i} = U_{i+1}$. 
		\end{enumerate} 
	\end{lemma}
	
	\begin{proof}[Proof of Lemma \ref{generation2}]
		Let $q$ be a nontrivial odd prime power, and let $n \ge 4$ be a natural number. Moreover, let $\varepsilon \in \lbrace +,- \rbrace$ and $Z \le Z(SL_n^{\varepsilon}(q))$. Suppose that one of the conditions $n \ge 5$, $q \equiv \ \varepsilon \ \mathrm{mod} \ 8$ or $Z = 1$ is satisfied. We have to show that $SL_n^{\varepsilon}(q)/Z$ is $3$-generated.
		
		Let $U_1, \dots, U_{n-1}$ denote the $SL_2^{\varepsilon}(q)$-subgroups of $SL_n^{\varepsilon}(q)$ corresponding to the $2 \times 2$ blocks along the main diagonal (as in Lemma \ref{generation0}). Let $E$ be the subgroup of $SL_n^{\varepsilon}(q)$ consisting of the diagonal matrices in $SL_n^{\varepsilon}(q)$ with diagonal entries in $\lbrace -1,1 \rbrace$.
		
		Assume that $n \ge 5$. Then one can easily see that, for each $i \in \lbrace 1,\dots,n-1\rbrace$, there is an $E_8$-subgroup $E_i$ of $E$ with $E_i \cap Z(SL_n^{\varepsilon}(q)) = 1$ and $[E_i,U_i] = 1$. Hence, $U_iZ/Z$ centralizes $E_iZ/Z \cong E_8$ for each $i \in \lbrace 1,\dots,n-1 \rbrace$. Now, if $S$ is a Sylow $2$-subgroup of $SL_n^{\varepsilon}(q)/Z$ containing $EZ/Z$, we have $U_iZ/Z \le \Gamma_{S,3}(SL_n^{\varepsilon}(q)/Z)$ for each $i \in \lbrace 1,\dots, n-1 \rbrace$, and Lemma \ref{generation0} (i) implies that $SL_n^{\varepsilon}(q)/Z$ is $3$-generated.
		
		We now consider the case $n = 4$. By hypothesis, $Z = 1$ or $q \equiv \varepsilon \ \mathrm{mod} \ 8$. Let
		\begin{equation*}
			U := \left \lbrace \left( \begin{array}{c|c} A & \begin{matrix} 0 \\ 0 \end{matrix} \\ \hline \begin{matrix} 0 & 0 \end{matrix} & 1 \end{array} \right) \ : \ A \in SL_3^{\varepsilon}(q) \right\rbrace.
		\end{equation*}
		If $Z = 1$, set $y := -I_4$. If $q \equiv \varepsilon \ \mathrm{mod} \ 8$, let $\lambda$ be an element of $\mathbb{F}_{q^2}^{*}$ of order $8$ such that $\lambda^{q-\varepsilon} = 1$. Note that $\lambda \in \mathbb{F}_q^{*}$ if $\varepsilon = +$. Also, if $q \equiv \varepsilon \ \mathrm{mod} \ 8$ and $\vert Z \vert = 2$, let $y := \lambda^2 I_4 \in SL_4^{\varepsilon}(q)$, and if $q \equiv \varepsilon \ \mathrm{mod} \ 8$ and $\vert Z \vert = 4$, let $y := \mathrm{diag}(\lambda, \lambda, \lambda, - \lambda) \in SL_4^{\varepsilon}(q)$. 
		
		Let $S_0$ be a Sylow $2$-subgroup of $U$ containing $E \cap U$. Let $\widetilde S$ be a Sylow $2$-subgroup of $SL_4^{\varepsilon}(q)$ containing $S_0$ and $y$. Denote the image of $\widetilde S$ in $SL_4^{\varepsilon}(q)/Z$ by $S$. We have $S \cap UZ/Z = S_0Z/Z \in \mathrm{Syl}_2(UZ/Z)$. By Lemma \ref{generation1}, $UZ/Z \cong U \cong SL_3^{\varepsilon}(q)$ is $2$-generated. So we have
		\begin{equation*}
			UZ/Z = \Gamma_{S_0 Z/Z,2}(UZ/Z) = \langle N_{UZ/Z}(T) \ \vert \ T \le S_0 Z/Z, m(T) \ge 2 \rangle.
		\end{equation*} 
		
		Let $T \le S_0 Z/Z$ with $m(T) \ge 2$ and $\widehat T := \langle T, yZ \rangle$. Clearly, $yZ$ is an involution of $S$ not contained in $UZ/Z$ and centralizing $UZ/Z$. Therefore, we have that $m(\widehat T) \ge 3$ and $N_{UZ/Z}(T) \le N_{SL_n^{\varepsilon}(q)/Z}(\widehat T)$. It follows that $UZ/Z \le \Gamma_{S,3}(SL_n^{\varepsilon}(q)/Z)$. In particular, $U_iZ/Z \le \Gamma_{S,3}(SL_n^{\varepsilon}(q)/Z)$ for $i \in \lbrace 1,2 \rbrace$. 
		
		From Lemma \ref{generation0} (ii), we see that there is some $m \in SL_4^{\varepsilon}(q)$ such that ${U_2}^m = U_3$ and such that $m$ normalizes $\langle E, y \rangle$. So $mZ$ normalizes $\langle EZ/Z, yZ \rangle$. It is easy to note that $\langle EZ/Z, yZ \rangle \cong E_8$, and so we have $mZ \in \Gamma_{S,3}(SL_n^{\varepsilon}(q)/Z)$. It follows that $U_3 Z/Z = (U_2 Z/Z)^{mZ} \le \Gamma_{S,3}(SL_n^{\varepsilon}(q)/Z)$. 
		
		So we have $U_iZ/Z \le \Gamma_{S,3}(SL_n^{\varepsilon}(q)/Z)$ for $i \in \lbrace 1,2,3 \rbrace$, and Lemma \ref{generation0} (i) implies that $SL_n^{\varepsilon}(q)/Z$ is $3$-generated.  
	\end{proof}
	
	\subsection{Automorphisms of $(P)SL_n(q)$} 
	Fix a prime number $p$, a positive integer $f$ and a natural number $n \ge 2$. Set $q := p^f$ and $T := SL_n(q)$. We now briefly describe the structure of $\mathrm{Aut}(T/Z)$, where $Z \le Z(T)$, referring to \cite{Dieudonne} and \cite[Section 2.1]{BurnessGiudici} for further details. 
	
	Let $\mathrm{Inndiag}(T) := \mathrm{Aut}_{GL_n(q)}(T)$. Note that 
	\begin{equation*}
		\mathrm{Inndiag}(T)/\mathrm{Inn}(T) \cong C_{(n,q-1)}. 
	\end{equation*} 
	The map 
	\begin{equation*}
		\phi: T \rightarrow T, (a_{ij}) \mapsto ({a_{ij}}^p)
	\end{equation*}  
	is an automorphism of $T$ with order $f$. One can check that $\phi$ normalizes $\mathrm{Inndiag}(T)$. Set 
	\begin{equation*}
		P\Gamma L_n(q) := \mathrm{Inndiag}(T) \langle \phi \rangle. 
	\end{equation*}  
	It is easy to note that $\langle \phi \rangle \cap \mathrm{Inndiag}(T) = 1$, so that $P \Gamma L_n(q)$ is the inner semidirect product of $\mathrm{Inndiag}(T)$ and $\langle \phi \rangle$. 
	
	The map 
	\begin{equation*}
		\iota: T \rightarrow T, a \mapsto (a^t)^{-1}
	\end{equation*} 
	is an automorphism of $T$ with order $2$. If $n = 2$, then $\iota$ turns out to be an inner automorphism of $T$, while we have $\iota \not\in P\Gamma L_n(q)$ when $n \ge 3$. By a direct calculation, $\iota$ normalizes $\mathrm{Inndiag}(T)$ and commutes with $\phi$. In particular, $A := P \Gamma L_n(q) \langle \iota \rangle$ is a subgroup of $\mathrm{Aut}(T)$, and we have 
	\begin{equation*}
		A/\mathrm{Inndiag}(T) \cong C_f \times C_a,
	\end{equation*} 
	where $a = 1$ if $n = 2$ and $a = 2$ if $n \ge 3$. 
	
	Now let $Z$ be a central subgroup of $T$. It can be easily checked that the natural homomorphism $\mathrm{Aut}(T) \rightarrow \mathrm{Aut}(T/Z)$ is injective. The image of $\mathrm{Inndiag}(T)$ under this homomorphism will be denoted by $\mathrm{Inndiag}(T/Z)$. By abuse of notation, we denote the image of $P\Gamma L_n(q)$ in $\mathrm{Aut}(T/Z)$ again by $P \Gamma L_n(q)$ and the images of $\iota$ and $\phi$ again by $\iota$ and $\phi$, respectively. 
	
	With this notation, we have 
	\begin{equation*}
		\mathrm{Aut}(T/Z) = P \Gamma L_n(q) \langle \iota \rangle. 
	\end{equation*}
	
	Note that the natural homomorphism $\mathrm{Aut}(T) \rightarrow \mathrm{Aut}(T/Z)$ is an isomorphism and that it induces an isomorphism $\mathrm{Out}(T) \rightarrow \mathrm{Out}(T/Z)$. 
	
	The elements of $\mathrm{Inndiag}(T/Z) \setminus \mathrm{Inn}(T/Z)$ are said to be the (non-trivial) \textit{diagonal automorphisms} of $T/Z$. An automorphism of $T/Z$ is called a \textit{field automorphism} if it is conjugate to $\phi^i$ for some $1 \le i < f$. The automorphisms of the form $\alpha \iota$, where $\alpha \in \mathrm{Inndiag}(T/Z)$, are said to be the \textit{graph automorphisms} of $T/Z$. An automorphism of $T/Z$ is said to be a \textit{graph-field automorphism} if it is conjugate to an automorphism of the form $\phi^i \iota$ for some $1 \le i < f$. We remark that these definitions are particular cases of more general definitions, see \cite[Chapter 10]{Steinberg}.
	
	\begin{proposition}
		\label{prop_2-nilpotence_out1}
		Let $q$ be a nontrivial prime power, and let $n \ge 2$. Then $\mathrm{Out}(PSL_n(q))$ is $2$-nilpotent. 
	\end{proposition}
	
	\begin{proof}
		From the above remarks, it is easy to see that $\mathrm{Out}(PSL_n(q))$ is supersolvable. By \cite[Lemma 2.4 (4)]{LiZhangYi}, any supersolvable finite group is $2$-nilpotent, and so the proposition follows.
	\end{proof}
	
	The following proposition also follows from the above remarks. 
	
	\begin{proposition}
		\label{Out_SL_n(3)} 
		Let $n \ge 2$ be a natural number. Then $\mathrm{Out}(SL_n(3))$ is a $2$-group. 
	\end{proposition}
	
	\subsection{Automorphisms of $(P)SU_n(q)$}
	Let $p$ be a prime number, $f$ be a positive integer and $n \ge 3$ be a natural number. Set $q := p^f$ and $T := SU_n(q)$. We now briefly describe the structure of $\mathrm{Aut}(T/Z)$, where $Z \le Z(T)$, referring to \cite{Dieudonne} and \cite[Section 2.3]{BurnessGiudici} for further details.
	
	Let $\mathrm{Inndiag}(T) := \mathrm{Aut}_{GU_n(q)}(SU_n(q))$. It is rather easy to note that 
	\begin{equation*}
		\mathrm{Inndiag}(T)/\mathrm{Inn}(T) \cong C_{(n,q+1)}.
	\end{equation*} 
	The map 
	\begin{equation*}
		\phi: T \rightarrow T, (a_{ij}) \mapsto ({a_{ij}}^p) 
	\end{equation*} 
	is an automorphism of $T$ with order $2f$. One can check that $\phi$ normalizes $\mathrm{Inndiag}(T)$. Set 
	\begin{equation*}
		P \Gamma U_n(q) := \mathrm{Inndiag}(T) \langle \phi \rangle. 
	\end{equation*} 
	It is rather easy to note that $\langle \phi \rangle \cap \mathrm{Inndiag}(T) = 1$, so that $P\Gamma U_n(q)$ is the inner semidirect product of $\mathrm{Inndiag}(T)$ and $\langle \phi \rangle$. Note that 
	\begin{equation*}
		P \Gamma U_n(q) / \mathrm{Inndiag}(T) \cong C_{2f}.
	\end{equation*}
	Now let $Z$ be a central subgroup of $T$. It can be easily checked that the natural homomorphism $\mathrm{Aut}(T) \rightarrow \mathrm{Aut}(T/Z)$ is injective. The image of $\mathrm{Inndiag}(T)$ under this homomorphism will be denoted by $\mathrm{Inndiag}(T/Z)$. By abuse of notation, we denote the image of $P \Gamma U_n(q)$ in $\mathrm{Aut}(T/Z)$ again by $P \Gamma U_n(q)$ and the image of $\phi$ again by $\phi$. 
	
	With this notation, we have 
	\begin{equation*}
		\mathrm{Aut}(T/Z) = P \Gamma U_n(q). 
	\end{equation*} 
	Note that the natural homomorphism $\mathrm{Aut}(T) \rightarrow \mathrm{Aut}(T/Z)$ is an isomorphism and that it induces an isomorphism $\mathrm{Out}(T) \rightarrow \mathrm{Out}(T/Z)$. 
	
	The elements of $\mathrm{Inndiag}(T/Z) \setminus \mathrm{Inn}(T/Z)$ are said to be the (non-trivial) \textit{diagonal automorphisms} of $T/Z$. An automorphism of $T/Z$ is called a \textit{field automorphism} if it is conjugate to $\phi^i$ for some $1 \le i < 2f$ such that $\phi^i$ has odd order. The automorphisms of the form $\alpha \phi^i$, where $\phi^i$ has even order and $\alpha \in \mathrm{Inndiag}(T/Z)$, are said to be the \textit{graph automorphisms} of $T/Z$. There are no graph-field automorphisms of $T/Z$. 
	
	\begin{proposition}
		\label{prop_2-nilpotence_out2} 
		Let $q$ be a nontrivial prime power, and let $n \ge 3$. Then $\mathrm{Out}(PSU_n(q))$ is $2$-nilpotent. 
	\end{proposition} 
	
	\begin{proof}
		We see from the above remarks that $\mathrm{Out}(PSU_n(q))$ is supersolvable. So $\mathrm{Out}(PSU_n(q))$ is $2$-nilpotent by \cite[Lemma 2.4 (4)]{LiZhangYi}.
	\end{proof}
	
	The following proposition also follows from the above remarks. 
	
	\begin{proposition}
		\label{Out_SU_n(3)} 
		Let $n \ge 3$ be a natural number. Then $\mathrm{Out}(SU_n(3))$ is a $2$-group.
	\end{proposition}
	
	\subsection{Some lemmas}
	We now prove several results on the automorphism groups of $(P)SL_n(q)$ and $(P)SU_n(q)$, where $n \ge 2$ and $q$ is a nontrivial odd prime power. 
	
	\begin{lemma}
		\label{involutions of SL_2(q) fixing S2 subgroup}
		Let $q$ be a nontrivial odd prime power. Also, let $T := SL_2(q)$ and $S \in \mathrm{Syl}_2(T)$. Suppose that $\alpha$ and $\beta$ are $2$-elements of $\mathrm{Aut}(T)$ such that $S^{\alpha} = S = S^{\beta}$ and $\alpha|_{S,S} = \beta|_{S,S}$. Then $\alpha = \beta$. 
	\end{lemma}
	
	\begin{proof}
		Let $\gamma := \alpha \beta^{-1} \in C_{\mathrm{Aut}(T)}(S)$. We have $C_{\mathrm{Inndiag}(T)}(S) = 1$ by \cite[Lemma 4.10.10]{GLS3}. Therefore, it suffices to show that $\gamma \in \mathrm{Inndiag}(T)$. Clearly, the images of $\alpha$ and $\beta^{-1}$ in $\mathrm{Aut}(T)/\mathrm{Inndiag}(T)$ are $2$-elements of $\mathrm{Aut}(T)/\mathrm{Inndiag}(T)$. Since $\mathrm{Aut}(T)/\mathrm{Inndiag}(T)$ is abelian, 
		\begin{equation*}
			\gamma \cdot \mathrm{Inndiag}(T) = (\alpha \cdot \mathrm{Inndiag}(T)) \cdot (\beta^{-1} \cdot \mathrm{Inndiag}(T))
		\end{equation*} 
		is still a $2$-element of $\mathrm{Aut}(T)/\mathrm{Inndiag}(T)$. By \cite[Lemma 4.10.10]{GLS3}, $C_{\mathrm{Aut}(T)}(S)$ is a $2'$-group, and so $\gamma$ has odd order. Therefore, $\gamma \cdot \mathrm{Inndiag}(T)$ has odd order. It follows that $\gamma \in \mathrm{Inndiag}(T)$, as required. 
	\end{proof}
	
	\begin{lemma} 
		\label{comp_auto_SL_2}
		Let $q = p^f$, where $p$ is an odd prime and $f$ is a positive integer. Let $T := PSL_2(q)$, and let $\alpha$ be an involution of $\mathrm{Aut}(T)$. Suppose that $C_T(\alpha)$ has a $2$-component $K$. Then we have $2 \mid f$, $(f,p) \ne (2,3)$ and $K \cong PSL_2(p^{\frac{f}{2}})$. In particular, $K$ is a component of $C_T(\alpha)$. 
	\end{lemma} 
	
	\begin{proof}
		Note that $C_T(\alpha) \cong C_{\mathrm{Inn}(T)}(\alpha)$. 
		
		Assume that $\alpha \in \mathrm{Inndiag}(T)$. Noticing that $\mathrm{Inndiag}(T) \cong PGL_2(q)$, we see from Lemma \ref{involutions_GL(n,q)} that $C_{\mathrm{Inndiag}(T)}(\alpha)$ is solvable. Thus $C_T(\alpha) \cong C_{\mathrm{Inn}(T)}(\alpha)$ is solvable, and $C_T(\alpha)$ has no $2$-components, a contradiction to the choice of $\alpha$. 
		
		So we have $\alpha \not\in \mathrm{Inndiag}(T)$. By the structure of $\mathrm{Aut}(PSL_2(q))$ and since $\alpha$ has order $2$, we can write $\alpha$ as a product of an inner-diagonal automorphism and a field automorphism of order $2$. In particular, $f$ must be even. Consulting \cite[Proposition 4.9.1 (d)]{GLS3}, we see that $\alpha$ itself is a field automorphism. So we can apply \cite[Proposition 4.9.1 (b)]{GLS3} to conclude that $C_{\mathrm{Inndiag}(T)}(\alpha) \cong \mathrm{Inndiag}(PSL_2(p^{\frac{f}{2}})) \cong PGL_2(p^{\frac{f}{2}})$. Consequently, $K$ is isomorphic to a $2$-component of $PGL_2(p^{\frac{f}{2}})$. It follows that $(f,p) \ne (2,3)$ and $K \cong PSL_2(p^{\frac{f}{2}})$. 
	\end{proof} 
	
	Before we state the next lemma, we introduce some notational conventions for adjoint Chevalley groups. Given a nontrivial prime power $q$, we denote $A_1(q)$ also by $B_1(q)$ and by $C_1(q)$. Moreover, $B_2(q)$ will be also denoted by $C_2(q)$, and $A_3(q)$ will be also denoted by $D_3(q)$. We also set $D_2(q) := A_1(q) \times A_1(q)$ and $^2D_2(q) := A_1(q^2)$. 
	
	\begin{lemma}
		\label{components of involutory automorphisms} 
		Let $q = p^f$, where $p$ is an odd prime and $f$ is a positive integer. Let $n \ge 3$ be a natural number and $\varepsilon \in \lbrace +,- \rbrace$. Let $T := PSL_n^{\varepsilon}(q)$, and let $\alpha$ be an involution of $\mathrm{Aut}(T)$. Suppose that $C_T(\alpha)$ has a $2$-component $K$. Then $K$ is in fact a component, and one of the following holds: 
		\begin{enumerate}
			\item[(i)] $K \cong SL_i^{\varepsilon}(q)$ for some $2 \le i < n$, where $i > 2$ if $q = 3$;
			\item[(ii)] $n$ is even, and $K$ is isomorphic to a nontrivial quotient of $SL_{\frac{n}{2}}(q^2)$;
			\item[(iii)] $\varepsilon = +$, $f$ is even, $K \cong PSL_n(p^{\frac{f}{2}})$ or $K \cong PSU_n(p^{\frac{f}{2}})$;
			\item[(iv)] $q \ne 3$, $n = 3$ or $4$, and $K \cong PSL_2(q)$;
			\item[(v)] $n$ is odd, $n \ge 5$ and $K \cong B_{\frac{n-1}{2}}(q)$;
			\item[(vi)] $n$ is even and $K \cong C_{\frac{n}{2}}(q)$;
			\item[(vii)] $n$ is even, $n \ge 6$ and $K \cong D_{\frac{n}{2}}(q)$; 
			\item[(viii)] $n$ is even, $n \ge 6$ and $K \cong$ $^2D_{\frac{n}{2}}(q)$.
		\end{enumerate} 
		Here, the (twisted) Chevalley groups appearing in (v)-(viii) are adjoint. 
	\end{lemma}
	
	\begin{proof}
		It can be shown that any involution of $\mathrm{Aut}(T)$ is an inner-diagonal automorphism, a field automorphism, a graph automorphism, or a graph-field automorphism (see \cite[Section 3.1.3]{BurnessGiudici} or \cite[Section 4.9]{GLS3}). 
		
		\medskip
		
		\textit{Case 1: $\alpha \in \mathrm{Inndiag}(T)$, or $\alpha$ is a graph automorphism.}
		
		Set $C^{*} := C_{\mathrm{Inndiag}(T)}(\alpha)$ and $L^{*} := O^{p'}(C^{*})$. One can see from \cite[Theorem 4.2.2 and Table 4.5.1]{GLS3} that $C^{*}/L^{*}$ is solvable and that one of the following holds: 
		\begin{enumerate} 
			\item[(1)] $L^{*}$ is the central product of two subgroups isomorphic to $SL_i^{\varepsilon}(q)$ and $SL_{n-i}^{\varepsilon}(q)$ for some natural number $i$ with $1 \le i \le \frac{n}{2}$, 
			\item[(2)] $n$ is even and $L^{*}$ is isomorphic to a nontrivial quotient of $SL_{\frac{n}{2}}(q^2)$, 
			\item[(3)] $n$ is odd and $L^{*} \cong B_{\frac{n-1}{2}}(q)$,
			\item[(4)] $n$ is even and $L^{*} \cong C_{\frac{n}{2}}(q)$,  
			\item[(5)] $n$ is even and $L^{*} \cong D_{\frac{n}{2}}(q)$, 
			\item[(6)] $n$ is even and $L^{*} \cong$ $^2D_{\frac{n}{2}}(q)$,  
		\end{enumerate} 
		where the (twisted) Chevalley groups appearing in the last four cases are adjoint. Since $C_T(\alpha)$ is isomorphic to $C_{\mathrm{Inn}(T)}(\alpha) \trianglelefteq C^{*}$, we have that $K$ is isomorphic to a $2$-component of $C^{*}$ and thus isomorphic to a $2$-component of $L^{*}$. Therefore, one of the conditions (i)-(viii) is satisfied. 
		
		\medskip

		\textit{Case 2: $\alpha$ is a field automorphism or a graph-field automorphism.}
		
		Again, let $C^{*} := C_{\mathrm{Inndiag}(T)}(\alpha)$. Since the field automorphisms of $PSU_n(q)$ have odd order and $PSU_n(q)$ has no graph-field automorphisms, we have $\varepsilon = +$. Also, $f$ is even since $\alpha$ is a field automorphism or a graph-field automorphism of order $2$. From \cite[Proposition 4.9.1 (a), (b)]{GLS3}, we see that $C^{*} \cong PGL_n(p^{\frac{f}{2}})$ if $\alpha$ is a field automorphism and that $C^{*} \cong PGU_n(p^{\frac{f}{2}})$ if $\alpha$ is a graph-field automorphism. Since $K$ is isomorphic to a $2$-component of $C^{*}$, it follows that (iii) is satisfied.  
	\end{proof}
	
	\begin{corollary}
		\label{q q star corollary} 
		Let $q = p^f$, where $p$ is an odd prime and $f$ is a positive integer. Let $n \ge 2$ be a natural number and $\varepsilon \in \lbrace +,- \rbrace$. Let $Z$ be a central subgroup of $SL_n^{\varepsilon}(q)$ and let $T := SL_n^{\varepsilon}(q)/Z$. Let $\alpha$ be an involution of $\mathrm{Aut}(T)$, and let $K$ be a $2$-component of $C_T(\alpha)$. Then the following hold: 
		\begin{enumerate} 
			\item[(i)] $K$ is a component of $C_T(\alpha)$, and $K/Z(K)$ is a known finite simple group.  
			\item[(ii)] $K/Z(K) \not\cong M_{11}$.
			\item[(iii)] Assume that $K/Z(K) \cong PSL_k^{\varepsilon^{*}}(q^{*})$ for some positive integer $2 \le k \le n$, some nontrivial odd prime power $q^{*}$ and some $\varepsilon^{*} \in \lbrace +,- \rbrace$. Then one of the following holds: 
			\begin{enumerate} 
				\item[(a)] $q^{*} = q$;   
				\item[(b)] $q^{*} = q^2$, $n \ge 4$ is even, $k = \frac{n}{2}$, and $\varepsilon^{*} = +$ if $n \ge 6$;
				\item[(c)] $f$ is even, $k = n$, $q^{*} = p^{\frac{f}{2}}$.
			\end{enumerate} 
		\end{enumerate} 
	\end{corollary} 
	
	\begin{proof}
		Set $\widebar T := T/Z(T) \cong PSL_n^{\varepsilon}(q)$. Let $\overline \alpha$ be the automorphism of $\overline T$ induced by $\alpha$.
		
		Clearly, $\widebar K$ is a $2$-component of $\widebar{C_T(\alpha)}$. It is easy to note that $\widebar{C_T(\alpha)}$ is a normal subgroup of $C_{\widebar T}(\widebar{\alpha})$. So $\widebar K$ is a $2$-component of $C_{\widebar T}(\widebar{\alpha})$. Lemmas \ref{comp_auto_SL_2} and \ref{components of involutory automorphisms} imply that $\widebar K$ is a component of $C_{\widebar T}(\widebar \alpha)$ and that $\widebar{K}/Z(\widebar K)$ is a known finite simple group. Applying \cite[6.5.1]{KurzweilStellmacher}, we conclude that $K'$ is a component of $C_T(\alpha)$. We have $K = K'$ since $K$ is a $2$-component of $C_T(\alpha)$, and so it follows that $K$ is a component of $C_T(\alpha)$. Also, $K/Z(K) \cong \widebar{K}/Z(\widebar K)$, so that $K/Z(K)$ is a known finite simple group. Hence (i) holds. 
		
		If $K/Z(K) \cong M_{11}$, then $\widebar K/Z(\widebar K) \cong M_{11}$, which is not possible by Lemmas \ref{comp_auto_SL_2} and \ref{components of involutory automorphisms}. So (ii) holds. 
		
		Suppose that $K/Z(K) \cong PSL_k^{\varepsilon^{*}}(q^{*})$ for some positive integer $2 \le k \le n$, some nontrivial odd prime power $q^{*}$ and some $\varepsilon^{*} \in \lbrace +,- \rbrace$. By Lemmas \ref{comp_auto_SL_2} and \ref{components of involutory automorphisms}, one of the following holds: 
		\begin{enumerate}
			\item[(1)] $\widebar K/Z(\widebar K) \cong PSL_i^{\varepsilon}(q)$ for some $2 \le i < n$; 
			\item[(2)] $n$ is even and $\widebar K/Z(\widebar K)$ is isomorphic to $PSL_{\frac{n}{2}}(q^2)$; 
			\item[(3)] $f$ is even, $\widebar K \cong PSL_n(p^{\frac{f}{2}})$ or $PSU_n(p^{\frac{f}{2}})$;
			\item[(4)] $q \ne 3$, $n = 3$ or $4$, $\widebar K \cong PSL_2(q)$; 
			\item[(5)] $n$ is odd, $n \ge 5$, $\widebar K \cong B_{\frac{n-1}{2}}(q)$;
			\item[(6)] $n$ is even, $n \ge 4$, $\widebar K \cong C_{\frac{n}{2}}(q)$;
			\item[(7)] $n$ is even, $n \ge 6$, $\widebar K \cong  D_{\frac{n}{2}}(q)$;
			\item[(8)] $n$ is even, $n \ge 6$, $\widebar K \cong$ $^2D_{\frac{n}{2}}(q)$. 
		\end{enumerate} 
		Here, the (twisted) Chevalley groups appearing in (5)-(8) are adjoint. On the other hand, we have $\widebar K/Z(\widebar K) \cong PSL_k^{\varepsilon^{*}}(q^{*})$. Now, if (1) holds, then $PSL_k^{\varepsilon^{*}}(q^{*}) \cong PSL_i^{\varepsilon}(q)$ for some $2 \le i < n$, and \cite[Theorem 37]{Steinberg} shows that this is only possible when $q^{*} = q$, so that (a) holds. Similarly, if (2) holds, then we have (b). Moreover, (3) implies (c) and (4) implies (a). As Theorem \cite[Theorem 37]{Steinberg} shows, the cases (5) and (6) cannot occur, while (7) and (8) can only occur when $n = 6$. As above, one can see that if $n = 6$ and (7) or (8) holds, then we have (a).
	\end{proof}
	
	\begin{lemma}
		\label{local balance q=3}
		Let $n \ge 3$ and $\varepsilon \in \lbrace +,- \rbrace$. Then $SL_n^{\varepsilon}(3)$ is locally balanced (in the sense of Definition \ref{local k-balance}).  
	\end{lemma} 
	
	\begin{proof}
		Set $T := SL_n^{\varepsilon}(3)$. Let $H$ be a subgroup of $\mathrm{Aut}(T)$ containing $\mathrm{Inn}(T)$, and let $x$ be an involution of $H$. It is enough to show that $O(C_H(x)) = 1$. 
		
		Assume that $O(C_H(x)) \ne 1$. Then $x \in \mathrm{Inndiag}(T)$ by \cite[Theorem 7.7.1]{GLS3}. By Propositions \ref{Out_SL_n(3)} and \ref{Out_SU_n(3)}, $\mathrm{Out}(T)$ is a $2$-group. This implies $O(C_H(x)) = O(C_{\mathrm{Inn}(T)}(x)) = O(C_{\mathrm{Inndiag}(T)}(x))$. Since $x$ is an involution of $\mathrm{Inndiag}(T) \cong PGL_n^{\varepsilon}(3)$, we have $O(C_{\mathrm{Inndiag}(T)}(x)) = 1$ by Corollary \ref{involution_centralizers_are-core-free_corollary_2}. Thus $O(C_H(x)) = 1$. This contradiction completes the proof. 
	\end{proof}
	
	\begin{lemma}
		\label{local 2-balance of SL and SU}
		Let $n \ge 3$ be a natural number, let $q$ be a nontrivial odd power, and let $\varepsilon \in \lbrace +,- \rbrace$. Then any non-trivial quotient of $SL_n^{\varepsilon}(q)$ is locally $2$-balanced (in the sense of Definition \ref{local k-balance}). 
	\end{lemma} 
	
	\begin{proof}
		By \cite[Theorem 4.61]{Gorenstein1983} or \cite[Theorem 7.7.4]{GLS3}, $PSL_n^{\varepsilon}(q)$ is locally $2$-balanced. Let $K$ be a non-trivial quotient of $SL_n^{\varepsilon}(q)$. As we have seen, there is an isomorphism $\mathrm{Aut}(K) \rightarrow \mathrm{Aut}(PSL_n^{\varepsilon}(q))$ mapping $\mathrm{Inn}(K)$ to $\mathrm{Inn}(PSL_n^{\varepsilon}(q))$. So the local $2$-balance of $K$ follows from the local $2$-balance of $PSL_n^{\varepsilon}(q)$. 
	\end{proof}
	
	\begin{lemma}
		\label{auto_lemma_SL} 
		Let $q$ be a nontrivial odd prime power and $n \ge 4$ be a natural number. Let $Z \le Z(SL_n(q))$ and $T := SL_n(q)/Z$. Let $K_1$ be the image of 
		\begin{equation*}
			\left\lbrace \begin{pmatrix} A & \\ & I_{n-2} \end{pmatrix} \ : \ A \in SL_2(q) \right\rbrace 
		\end{equation*} 
		in $T$, and let $K_2$ be the image of 
		\begin{equation*}
			\left\lbrace \begin{pmatrix} I_2 & \\ & B \end{pmatrix} \ : \ B \in SL_{n-2}(q) \right\rbrace 
		\end{equation*} 
		in $T$. Let $\alpha$ be an automorphism of $T$ with odd order such that $\alpha$ normalizes $K_1$ and centralizes $K_2$. Then $\alpha|_{K_1,K_1}$ is an inner automorphism. 
	\end{lemma} 
	
	\begin{proof}
		By hypothesis, $q = p^f$ for some odd prime number $p$ and some positive integer $f$. We have $\alpha \in P \Gamma L_n(q)$ since $\alpha$ has odd order and $|\mathrm{Aut}(T)/P \Gamma L_n(q)| = 2$. So there are some $m \in GL_{n}(q)$ and some $1 \le r \le f$ such that, for each element $(a_{ij})$ of $SL_n(q)$, $\alpha$ maps $(a_{ij})Z$ to $((a_{ij})^{p^r})^mZ$. 
		
		Let $x$ be the image of $\mathrm{diag}(-1,-1,1,\dots,1) \in SL_n(q)$ in $T$. Then $x$ is the unique involution of $K_1$, and so we have $x^{\alpha} = x$. This easily implies that 
		\begin{equation*}
			m = \begin{pmatrix} m_1 & \\ & m_2 \end{pmatrix} 
		\end{equation*} 
		for some $m_1 \in GL_2(q)$ and some $m_2 \in GL_{n-2}(q)$. 
		
		Since $\alpha$ centralizes $K_2$, we have $((a_{ij})^{p^r})^{m_2} = (a_{ij})$ for all $(a_{ij}) \in SL_{n-2}(q)$. Therefore, the automorphism $SL_{n-2}(q) \rightarrow SL_{n-2}(q), (a_{ij}) \mapsto (a_{ij})^{p^r}$ is an element of $\mathrm{Inndiag}(SL_{n-2}(q))$. This implies $r = f$. 
		
		Thus, under the isomorphism $\mathrm{Aut}(SL_2(q)) \rightarrow \mathrm{Aut}(K_1)$ induced by the canonical isomorphism $SL_2(q) \rightarrow K_1$, the automorphism $\alpha|_{K_1,K_1}$ of $K_1$ corresponds to the inner-diagonal automorphism $\widetilde \alpha: SL_2(q) \rightarrow SL_2(q), a \mapsto a^{m_1}$, and this automorphism has odd order since $\alpha$ has odd order. The index of $\mathrm{Inn}(SL_2(q))$ in $\mathrm{Inndiag}(SL_2(q))$ is $2$, and so it follows that $\widetilde \alpha \in \mathrm{Inn}(SL_2(q))$. Consequently, $\alpha|_{K_1,K_1} \in \mathrm{Inn}(K_1)$. 
	\end{proof}
	
	By using similar arguments as in the proof of Lemma \ref{auto_lemma_SL}, one can prove the following lemma. 
	
	\begin{lemma}
		\label{auto_lemma_SU}
		Let $q$ be a nontrivial odd prime power and $n \ge 4$ be a natural number. Let $Z \le Z(SU_n(q))$ and $T := SU_n(q)/Z$. Let $K_1$ be the image of 
		\begin{equation*}
			\left\lbrace \begin{pmatrix} A & \\ & I_{n-2} \end{pmatrix} \ : \ A \in SU_2(q) \right\rbrace 
		\end{equation*} 
		in $T$, and let $K_2$ be the image of 
		\begin{equation*}
			\left\lbrace \begin{pmatrix} I_2 & \\ & B \end{pmatrix} \ : \ B \in SU_{n-2}(q) \right\rbrace 
		\end{equation*} 
		in $T$. Let $\alpha$ be an automorphism of $T$ with odd order such that $\alpha$ normalizes $K_1$ and centralizes $K_2$. Then $\alpha|_{K_1,K_1}$ is an inner automorphism. 
	\end{lemma} 
	
	Our next goal is to prove the following lemma. 
	
	\begin{lemma}
		\label{final lemma for 2-balance} 
		Let $q$ and $q^{*}$ be nontrivial odd prime powers. Let $L$ be a group isomorphic to $SL_2(q^{*})$. Let $Q$ be a Sylow $2$-subgroup of $L$. Moreover, let $V$ be a Sylow $2$-subgroup of $GL_2(q)$ and $V_0 := V \cap SL_2(q)$. Suppose that there is a group isomorphism $\psi: V_0 \rightarrow Q$. Let $v_1, v_2, v_3$ be elements of $V$ such that $v_3 = v_1v_2$ and such that the square of any element of $\lbrace v_1,v_2,v_3 \rbrace$ lies in $Z(GL_2(q))$. For each $i \in \lbrace 1,2,3 \rbrace$, let $\alpha_i$ be a $2$-element of $\mathrm{Aut}(L)$ normalizing $Q$ such that 
		\begin{equation*}
			\alpha_i|_{Q,Q} = \psi^{-1} (c_{v_i}|_{V_0,V_0}) \psi. 
		\end{equation*} 
		Then we have 
		\begin{equation*}
			\bigcap_{i=1}^3 O(C_L(\alpha_i)) = 1. 
		\end{equation*} 
	\end{lemma} 
	
	To prove Lemma \ref{final lemma for 2-balance}, we need to prove some other lemmas.
	
	\begin{lemma}
		\label{lemma for 2-balance 1}
		Let $q$ be a nontrivial odd prime power with $q \equiv 1 \ \mathrm{mod} \ 4$, and let $k \in \mathbb{N}$ with $(q-1)_2 = 2^k$. Let $G$ be a group isomorphic to $SL_2(q)$ and $Q \in \mathrm{Syl}_2(G)$. Then the following hold: 
		\begin{enumerate}
			\item[(i)] There are elements $a, b \in Q$ such that $\mathrm{ord}(a) = 2^k$, $\mathrm{ord}(b) = 4$, $a^b = a^{-1}$ and $b^2 = a^{2^{k-1}}$.
			\item[(ii)] Let $a$ and $b$ be as in (i). Then there is a group isomorphism $\varphi: G \rightarrow SL_2(q)$ such that
			\begin{equation*}
				a^{\varphi} = \begin{pmatrix} \lambda & 0 \\ 0 & \lambda^{-1} \end{pmatrix}
			\end{equation*} 
			for some $\lambda \in \mathbb{F}_q^{*}$ with order $2^k$ and
			\begin{equation*}
				b^{\varphi} = \begin{pmatrix} 0 & 1 \\ -1 & 0 \end{pmatrix}.
			\end{equation*}
		\end{enumerate} 
	\end{lemma} 
	
	\begin{proof}
		(i) follows from Lemma \ref{sylow_SL_2(q)}.
		
		We now prove (ii). Assume that $k \ge 3$. By Lemma \ref{sylow_GL_2(q)} (i), 
		\begin{equation*}
			\left\lbrace \begin{pmatrix} \mu & 0 \\ 0 & \mu^{-1} \end{pmatrix} \ : \ \textnormal{$\mu$ is a $2$-element of $\mathbb{F}_q^{*}$} \right\rbrace \left\langle \begin{pmatrix} 0 & 1 \\ -1 & 0 \end{pmatrix} \right\rangle
		\end{equation*}
		is a Sylow $2$-subgroup of $SL_2(q)$. Choose a group isomorphism $\psi: G \rightarrow SL_2(q)$ such that $Q^{\psi} = R$. Clearly, since $k \ge 3$, $Q$ has only one cyclic subgroup of order $2^k$. This implies that  
		\begin{equation*}
			a^{\psi} = \begin{pmatrix} \lambda & 0 \\ 0 & \lambda^{-1} \end{pmatrix}
		\end{equation*}
		for some $\lambda \in \mathbb{F}_q^{*}$ with order $2^k$. Since $b \not\in \langle a \rangle$, we have 
		\begin{equation*}
			b^{\psi} = \begin{pmatrix} 0 & \mu \\ -\mu^{-1} & 0 \end{pmatrix} 
		\end{equation*} 
		for some $2$-element $\mu$ of $\mathbb{F}_q^{*}$. Composing $\psi$ with the automorphism
		\begin{equation*}
			SL_2(q) \rightarrow SL_2(q), \ A \mapsto \begin{pmatrix} \mu^{-1} & 0 \\ 0 & 1 \end{pmatrix}A\begin{pmatrix} \mu & 0 \\ 0 & 1 \end{pmatrix}
		\end{equation*}
		we get a group isomorphism $\varphi: G \rightarrow SL_2(q)$ with the desired properties. This completes the proof of (ii) for the case $k \ge 3$.
		
		Assume now that $k = 2$. Let $\psi: G \rightarrow SL_2(q)$ be a group isomorphism. We have $(a^\psi)^2 = -I_2$ since $-I_2$ is the only involution of $SL_2(q)$ and $\mathrm{ord}(a^2) = 2$. So, by Lemma \ref{involutions_GL(n,q)}, we may assume that 
		\begin{equation*}
			a^{\psi} = \begin{pmatrix}\lambda & 0 \\ 0 & \lambda^{-1} \end{pmatrix}
		\end{equation*}
		for some $\lambda \in \mathbb{F}_q^{*}$ with order $4$. Since $a^b = a^{-1}$, we have
		\begin{equation*}
			\begin{pmatrix}\lambda & 0 \\ 0 & \lambda^{-1} \end{pmatrix}^{b^{\psi}} = \begin{pmatrix}\lambda^{-1} & 0 \\ 0 & \lambda \end{pmatrix}.
		\end{equation*} 
		This implies that 
		\begin{equation*}
			b^{\psi} = \begin{pmatrix} 0 & \mu \\ -\mu^{-1} & 0 \end{pmatrix} 
		\end{equation*}  
		for some $\mu \in \mathbb{F}_q^{*}$. Again we may compose $\psi$ with a suitable diagonal automorphism of $SL_2(q)$ to obtain a group isomorphism $\varphi: G \rightarrow SL_2(q)$ with the desired properties. 
	\end{proof}
	
	By using similar arguments as in the proof of Lemma \ref{lemma for 2-balance 1}, one can prove the following lemma. 
	
	\begin{lemma}
		\label{lemma for 2-balance 2} 
		Let $q$ be a nontrivial odd prime power with $q \equiv 3 \mod 4$, and let $s \in \mathbb N$ with $(q+1)_2 = 2^s$. Let $G$ be a group isomorphic to $SU_2(q)$ and $Q \in \mathrm{Syl}_2(G)$. Then the following hold: 
		\begin{enumerate}
			\item[(i)] There are elements $a, b \in Q$ such that $\mathrm{ord}(a) = 2^s$, $\mathrm{ord}(b) = 4$, $a^b = a^{-1}$ and $b^2 = a^{2^{s-1}}$.
			\item[(ii)] Let $a$ and $b$ be as in (i). Then there is a group isomorphism $\varphi: G \rightarrow SU_2(q)$ such that
			\begin{equation*}
				a^{\varphi} = \begin{pmatrix} \lambda & 0 \\ 0 & \lambda^{-1} \end{pmatrix}
			\end{equation*} 
			for some $\lambda \in \mathbb{F}_{q^2}^{*}$ with order $2^s$ and 
			\begin{equation*}
				b^{\varphi} = \begin{pmatrix} 0 & 1 \\ -1 & 0 \end{pmatrix}.
			\end{equation*}
		\end{enumerate} 
	\end{lemma} 
	
	\begin{lemma}
		\label{lemma for 2-balance 3} 
		Let $q$ be a nontrivial odd prime power with $q \equiv 1 \ \mathrm{mod} \ 4$. Let $\rho$ be a generating element of the Sylow $2$-subgroup of $\mathbb{F}_q^{*}$, and let
		\begin{equation*}
			a := \begin{pmatrix} \rho & \\ & \rho^{-1} \end{pmatrix}, \ \ \ \ b := \begin{pmatrix} 0 & 1 \\ -1 & 0 \end{pmatrix}.
		\end{equation*} 
		Let $V$ be the Sylow $2$-subgroup of $GL_2(q)$ given by Lemma \ref{sylow_GL_2(q)} (i), and let $v,w \in V$ such that $v^2,w^2,(vw)^2 \in Z(GL_2(q))$. Then one of the following holds: 
		\begin{enumerate}
			\item[(i)] $\lbrace v,w,vw \rbrace \cap Z(GL_2(q)) \ne \emptyset$. 
			\item[(ii)] There exist $r,s \in \lbrace v,w,vw \rbrace$ with $a^{r} = a$, $b^{r} = b^3$ and $a^{s} = a^{-1}$. 
		\end{enumerate} 
	\end{lemma} 
	
	\begin{proof}
		It is easy to note that (i) holds if $v$ and $w$ are diagonal matrices. 
		
		Suppose now that $v$ or $w$ is not a diagonal matrix. If neither $v$ nor $w$ is a diagonal matrix, then $vw$ is a diagonal matrix. So there exist $r, s \in \lbrace v,w,vw \rbrace$ such that
		\begin{equation*}
			r = \begin{pmatrix} \lambda_1 & \\ & \lambda_2 \end{pmatrix}, \ \ \ \ s = \begin{pmatrix} & \mu_1 \\ \mu_2 \end{pmatrix}, 
		\end{equation*}  
		where $\lambda_1$, $\lambda_2$, $\mu_1$ and $\mu_2$ are $2$-elements of $\mathbb{F}_q^{*}$.
		
		If $\lambda_1 = \lambda_2$, then (i) holds. Assume now that $\lambda_1 \ne \lambda_2$. Then $\lambda_2 = - \lambda_1$ since $r^2 \in Z(GL_2(q))$, and a direct calculation shows that $a^{r} = a$, $b^{r} = b^3$ and $a^{s} = a^{-1}$. 
	\end{proof} 
	
	\begin{lemma}
		\label{lemma for 2-balance 4} 
		Let $q$ be a nontrivial odd prime power with $q \equiv 3 \ \mathrm{mod} \ 4$, and let $k \in \mathbb{N}$ with $(q+1)_2 = 2^k$. Let $V$ be a Sylow $2$-subgroup of $GL_2(q)$.
		\begin{enumerate}
			\item[(i)] There exist $x, y \in V$ with $\mathrm{ord}(x) = 2^{k+1}$, $\mathrm{ord}(y) = 2$ and $x^y = x^{-1+2^k}$. We have $V \cap SL_2(q) = \langle x^2 \rangle \langle xy \rangle$.
			\item[(ii)] Let $x$ and $y$ be as above, and let $a := x^2$ and $b := xy$. Let $v, w \in V$ with $v^2,w^2,(vw)^2 \in Z(GL_2(q))$. Then one of the following holds: 
			\begin{enumerate}
				\item[(a)] $\lbrace v,w,vw \rbrace \cap Z(GL_2(q)) \ne \emptyset$.
				\item[(b)] There exist $r,s \in \lbrace v,w,vw \rbrace$ such that $a^r = a$, $b^r = b^3$ and $a^s = a^{-1}$. 
			\end{enumerate}
		\end{enumerate} 
	\end{lemma}
	
	\begin{proof}
		(i) follows from Lemma \ref{sylow_2} (i), (ii). 
		
		We now prove (ii). We have $Z(V) = \langle x^{2^k} \rangle$ by Lemma \cite[Chapter 5, Theorem 4.3]{Gorenstein}. Thus $Z(GL_2(q)) \cap V = \langle x^{2^k} \rangle$. Clearly, $\lbrace v,w,vw \rbrace \cap \langle x \rangle \subseteq \langle x^{2^{k-1}} \rangle$.
		
		If $v,w \in \langle x \rangle$, then $v,w \in \langle x^{2^{k-1}} \rangle$, and it easily follows that (a) holds.
		
		Assume now that $v \not\in \langle x \rangle$ or $w \not\in \langle x \rangle$. If neither $v$ nor $w$ lies in $\langle x \rangle$, then $vw \in \langle x \rangle$. Consequently, $\lbrace v,w,vw \rbrace$ has an element $r$ of the form $x^{\ell 2^{k-1}}$ for some $1 \le \ell \le 4$ and an element $s$ of the form $x^i y$ for some $1 \le i \le 2^{k+1}$. If $\ell = 2$ or $4$, then (a) holds. Assume now that $\ell = 1$ or $3$. It is clear that $a^r = a$. Furthermore, we have
		\begin{align*}
			b^r &= (xy)^{x^{\ell 2^{k-1}}} \\
			&= x y^{x^{\ell 2^{k-1}}} \\
			&= x x^{-\ell 2^{k-1}} y x^{\ell 2^{k-1}} y^2 \\
			&= x^{1-\ell 2^{k-1}} (x^y)^{\ell 2^{k-1}} y \\ 
			&= x^{1-\ell 2^{k-1}}(x^{-1+2^k})^{\ell 2^{k-1}} y \\
			&= x^{1-\ell 2^k+\ell 2^{2k-1}} y \\
			&= x^{1-\ell 2^k} y \\
			&\stackrel{\textnormal{$\ell$ odd}}{=} x^{1+2^k}y.
		\end{align*} 
		On the other hand, we have
		\begin{equation*}
			b^3 = (xy)^3 = x^{2^k}xy = x^{1+2^k}y.
		\end{equation*}   
		Consequently, $b^r = b^3$. Finally, we also have
		\begin{equation*}
			a^s = (x^2)^{x^i y} = (x^2)^y = (x^y)^2 = (x^{-1+2^k})^2 = x^{-2} = a^{-1}.
		\end{equation*} 
		Thus (b) holds. 
	\end{proof}
	
	\begin{proof}[Proof of Lemma \ref{final lemma for 2-balance}]
		If $\alpha_j|_{Q,Q} = \mathrm{id}_Q$ for some $j \in \lbrace 1,2,3 \rbrace$, then $\alpha_j = \mathrm{id}_L$ by Lemma \ref{involutions of SL_2(q) fixing S2 subgroup}, which implies that
		\begin{equation*}
			\bigcap_{i=1}^3 O(C_L(\alpha_i)) \le O(C_L(\alpha_j)) = O(L) = 1. 
		\end{equation*} 
		
		Suppose now that $\alpha_i$ acts nontrivially on $Q$ for all $i \in \lbrace 1,2,3 \rbrace$. Let $m \in \mathbb{N}$ with $|Q| = 2^m$. Using Lemma \ref{lemma for 2-balance 3} (together with Sylow’s theorem) and Lemma \ref{lemma for 2-balance 4}, we see that there exist $a,b \in Q$ and $i,j \in \lbrace 1,2,3 \rbrace$ such that the following hold: 
		\begin{enumerate}
			\item[(i)] $\mathrm{ord}(a) = 2^{m-1}$, $\mathrm{ord}(b) = 4$, $a^b = a^{-1}$, $b^2 = a^{2^{m-2}}$;
			\item[(ii)] $a^{\alpha_i} = a$, $b^{\alpha_i} = b^3$, $a^{\alpha_j} = a^{-1}$. 
		\end{enumerate} 
		Clearly, $b^{\alpha_j} = a^{\ell}b$ for some $1 \le \ell \le 2^{m-1}$. 
		
		Assume that $q^{*} \equiv 1 \ \mathrm{mod} \ 4$. By Lemma \ref{lemma for 2-balance 1}, there is group isomorphism $\varphi: L \rightarrow SL_2(q^{*})$ with 
		\begin{equation*} 
			a^{\varphi} = \begin{pmatrix} \lambda & 0 \\ 0 & \lambda^{-1} \end{pmatrix}
		\end{equation*}
		for some generator $\lambda$ of the Sylow $2$-subgroup of $(\mathbb{F}_{q^{*}})^{*}$ and 
		\begin{equation*}
			b^{\varphi} = \begin{pmatrix} 0 & 1 \\ -1 & 0 \end{pmatrix}.
		\end{equation*}
		Set $\beta_k := \varphi^{-1}\alpha_k \varphi$ for $k \in \lbrace 1,2,3 \rbrace$. Let $i$ and $j$ be as in (ii). Also, let 
		\begin{equation*}
			m_i := \begin{pmatrix} 1 & \\ & -1 \end{pmatrix}.
		\end{equation*} 
		Then $\beta_i$ and $c_{m_i}$ normalize $Q^{\varphi}$, and we have $\beta_i|_{Q^{\varphi},Q^{\varphi}} = c_{m_i}|_{Q^{\varphi},Q^{\varphi}}$. Applying Lemma \ref{involutions of SL_2(q) fixing S2 subgroup}, we conclude that $\beta_i = c_{m_i}$. 
		
		Clearly, 
		\begin{equation*}
			\begin{pmatrix} 0 & 1 \\ -1 & 0 \end{pmatrix}^{\beta_j} = \begin{pmatrix} 0 & \mu \\ -\mu^{-1} & 0 \end{pmatrix}
		\end{equation*} 
		for some $2$-element $\mu$ of $(\mathbb{F}_{q^{*}})^{*}$. Set 
		\begin{equation*}
			m_j := \begin{pmatrix} 0 & \mu \\ -1 & 0 \end{pmatrix}.  
		\end{equation*} 
		Then $\beta_j$ and $c_{m_j}$ normalize $Q^{\varphi}$, and we have $\beta_j|_{Q^{\varphi},Q^{\varphi}} = c_{m_j}|_{Q^{\varphi},Q^{\varphi}}$. Applying Lemma \ref{involutions of SL_2(q) fixing S2 subgroup}, we conclude that $\beta_j = c_{m_j}$.
		
		It follows that $C_{SL_2(q^{*})}(\beta_i) \cap C_{SL_2(q^{*})}(\beta_j) = Z(SL_2(q^{*}))$. So we have $C_L(\alpha_i) \cap C_L(\alpha_j) = Z(L)$, and this implies that  
		\begin{equation*}
			\bigcap_{k=1}^3 O(C_L(\alpha_k)) = 1
		\end{equation*}
		since $|Z(L)| = 2$. 
		
		If $q^{*} \equiv 3 \mod 4$, then a very similar argumentation shows that the same conclusion holds. Here, one has to use Lemma \ref{lemma for 2-balance 2} instead of Lemma \ref{lemma for 2-balance 1}, together with the fact that $SL_2(q^{*}) \cong SU_2(q^{*})$. 
	\end{proof}
	
	We bring this section to a close with a proof of the following lemma, which will play an important role in the proof of Theorem \ref{B}. 
	
	\begin{lemma}
		\label{lemma_fusion_out_PSL_PSU} 
		Let $q$ be a nontrivial odd prime power, $\varepsilon \in \lbrace +,- \rbrace$ and $n \ge 2$ a natural number. Set $T := \mathrm{Inn}(PSL_n^{\varepsilon}(q))$. Let $A$ be a subgroup of $\mathrm{Aut}(PSL_n^{\varepsilon}(q))$ such that $T \le A$ and such that the index of $T$ in $A$ is odd. Let $S$ be a Sylow $2$-subgroup of $T$. Then we have $\mathcal{F}_S(T) = \mathcal{F}_S(A)$. 
	\end{lemma} 
	
	To prove Lemma \ref{lemma_fusion_out_PSL_PSU}, we need to prove some other lemmas first. 
	
	\begin{lemma}
		\label{2-autos_W}
		Let $q$ be a nontrivial odd prime power, $\varepsilon \in \lbrace +,- \rbrace$, and let $r$ be positive integer. Also, let $W$ be a Sylow $2$-subgroup of $GL_{2^r}^{\varepsilon}(q)$. Then $\mathrm{Aut}(W)$ is a $2$-group. 
	\end{lemma} 
	
	\begin{proof}
		We proceed by induction over $r$.
		
		Suppose that $r = 1$. If $q \equiv -\varepsilon \mod 4$, then $W$ is semidihedral by Lemmas \ref{sylow_GL_2(q)} and \ref{sylow_GU_2(q)}, and so $\mathrm{Aut}(W)$ is a $2$-group by \cite[Proposition 4.53]{Craven}. If $q \equiv \varepsilon \mod 4$, then $W \cong C_{2^k} \wr C_2$ for some positive integer $k$ by Lemmas \ref{sylow_GL_2(q)} and \ref{sylow_GU_2(q)}, and so $\mathrm{Aut}(W)$ is a $2$-group as a consequence of \cite[Theorem 2]{Fumagalli}. 
		
		Assume now that $r > 1$ and that the lemma is true with $r-1$ instead of $r$. Let $W_0$ be a Sylow $2$-subgroup of $GL_{2^{r-1}}^{\varepsilon}(q)$. Hence, $\mathrm{Aut}(W_0)$ is a $2$-group. By Lemma \ref{sylow_power_2}, we have $W \cong W_0 \wr C_2$. Applying \cite[Theorem 2]{Fumagalli}, we conclude that $\mathrm{Aut}(W)$ is a $2$-group. 
	\end{proof}
	
	\begin{lemma}
		\label{2-autos_W_2}
		Let $q$ be a nontrivial odd prime power, $\varepsilon \in \lbrace +,- \rbrace$, and let $n \ge 3$ be a natural number. Let $T := SL_n^{\varepsilon}(q)$, and let $S$ be a Sylow $2$-subgroup of $\mathrm{Inndiag}(T)$. Then $\mathrm{Aut}_{P\Gamma L_n^{\varepsilon}(q)}(S)$ is a $2$-group. 
	\end{lemma} 
	
	\begin{proof}
		Let $\alpha \in N_{P\Gamma L_n^{\varepsilon}(q)}(S)$. It suffices to show that $c_{\alpha}|_{S,S}$ is a $2$-automorphism of $S$. 
		
		Let $0 \le r_1 < \dots < r_t$ such that $n = 2^{r_1} + \dots + 2^{r_t}$. Let $W_i \in \mathrm{Syl}_2(GL_{2^{r_i}}^{\varepsilon}(q))$ for all $1 \le i \le t$. By Lemma \ref{sylow_general},
		\begin{equation*}
			\left \lbrace 
			\begin{pmatrix}
				A_1 & \ & \ \\
				\ & \ddots & \ \\ 
				\ & \ & A_t
			\end{pmatrix} 
			\ : \ A_i \in W_i \right \rbrace
		\end{equation*} 
		is a Sylow $2$-subgroup of $GL_n^{\varepsilon}(q)$. 
		
		Clearly, $\lbrace c_w|_{T,T} \ \vert \ w \in W \rbrace$ is a Sylow $2$-subgroup of $\mathrm{Inndiag}(T)$. Without loss of generality, we assume that $S = \lbrace c_w|_{T,T} \ \vert \ w \in W \rbrace$.
		
		Let $p$ be the odd prime number and $f$ be the positive integer with $q = p^f$. Since $\alpha \in P \Gamma L_n^{\varepsilon}(q)$, there exist some $m \in GL_n^{\varepsilon}(q)$ and some natural number $\ell$, where $1 \le \ell \le f$ if $\varepsilon = +$ and $1 \le \ell \le 2f$ if $\varepsilon = -$, such that 
		\begin{equation*}
			(a_{ij})^{\alpha} = ((a_{ij})^{p^{\ell}})^m
		\end{equation*}   
		for all $(a_{ij}) \in T$. 
		
		Let 
		\begin{equation*}
			\widebar \alpha: GL_n^{\varepsilon}(q) \rightarrow GL_n^{\varepsilon}(q), (a_{ij}) \mapsto ((a_{ij})^{p^{\ell}})^m.
		\end{equation*} 
		It is easy to note that $\widebar \alpha$ is an automorphism of $GL_n^{\varepsilon}(q)$. Using this, one can see that $\alpha^{-1}(c_w|_{T,T})\alpha = c_{w^{\widebar \alpha}}|_{T,T}$ for all $w \in W$. 
		
		Let $w \in W$. Since $\alpha$ normalizes $S$, there is some $\widetilde w \in W$ with $c_{w^{\widebar \alpha}}|_{T,T} = \alpha^{-1}(c_w|_{T,T})\alpha = c_{\widetilde w}|_{T,T}$. It follows that $w^{\widebar \alpha} \in \widetilde w Z(GL_n^{\varepsilon}(q)) \subseteq W Z(GL_n^{\varepsilon}(q))$. This implies $w^{\widebar \alpha} \in W$ since $W$ is the unique Sylow $2$-subgroup of $W Z(GL_n^{\varepsilon}(q))$. In particular, $\widebar \alpha$ induces an automorphism of $W$. 
		
		Let
		\begin{equation*}
			d_i := \begin{pmatrix} I_{2^{r_1}} & & & & \\ & \ddots & & & \\ & & -I_{2^{r_i}} & & \\ & & & \ddots & \\ & & & & I_{2^{r_t}} \end{pmatrix} 
		\end{equation*} 
		for each $1 \le i \le t$. Then $d_i$ is a central involution of $W$ for each $1 \le i \le t$. So we have that $(d_i)^{\widebar \alpha} = (d_i)^m$ is a central involution of $W$ for each $1 \le i \le t$. As we see from Lemma \ref{centralizer_sylow_SL}, this already implies that $(d_i)^m = d_i$ for each $1 \le i \le t$. So there is some $m_i \in GL_{2^{r_i}}^{\varepsilon}(q)$ for each $1 \le i \le t$ such that
		\begin{equation*}
			m = \begin{pmatrix}
				m_1 & & \\ & \ddots & \\ & & m_t
			\end{pmatrix}.
		\end{equation*} 
		Now 
		\begin{equation*}
			W_r \rightarrow W_r, (a_{ij}) \mapsto ((a_{ij})^{p^{\ell}})^{m_i}
		\end{equation*} 
		is an automorphism of $W_r$ for each $1 \le r \le t$. Applying Lemma \ref{2-autos_W}, we conclude that $\widebar{\alpha}|_{W,W}$ is a $2$-automorphism of $W$. Since $\alpha^{-1}(c_w|_{T,T})\alpha = c_{w^{\widebar \alpha}}|_{T,T}$ for all $w \in W$, it follows that $c_{\alpha}|_{S,S}$ is a $2$-automorphism of $S$, as required. 
	\end{proof}
	
	\begin{corollary}
		\label{2-autos_W_3}
		Let $q$ be a nontrivial odd prime power, $\varepsilon \in \lbrace +,- \rbrace$, and let $n \ge 3$ be a natural number. Let $T := PSL_n^{\varepsilon}(q)$, and let $S$ be a Sylow $2$-subgroup of $\mathrm{Inndiag}(T)$. Then $\mathrm{Aut}_{P \Gamma L_n^{\varepsilon}(q)}(S)$ is a $2$-group. 
	\end{corollary} 
	
	\begin{lemma}
		Let $q$ be a nontrivial odd prime power, $\varepsilon \in \lbrace +,- \rbrace$, and $n \ge 3$ be a natural number. Let $S$ be a Sylow $2$-subgroup of $SL_n^{\varepsilon}(q)Z(GL_n^{\varepsilon}(q))/Z(GL_n^{\varepsilon}(q))$, and let $S_1$ be a Sylow $2$-subgroup of $PGL_n^{\varepsilon}(q)$ containing $S$. Then $N_{PGL_n^{\varepsilon}(q)}(S) = N_{PGL_n^{\varepsilon}(q)}(S_1)$.
	\end{lemma} 
	
	\begin{proof}
		Let $T_1$ be a Sylow $2$-subgroup of $GL_n^{\varepsilon}(q))$ such that $S_1 = T_1 Z(GL_n^{\varepsilon}(q))/Z(GL_n^{\varepsilon}(q))$. Let $T := T_1 \cap SL_n^{\varepsilon}(q)$. Clearly, we have $S = T Z(GL_n^{\varepsilon}(q))/Z(GL_n^{\varepsilon}(q))$. It is rather easy to show $N_{PGL_n^{\varepsilon}(q)}(S) = N_{GL_n^{\varepsilon}(q)}(T)Z(GL_n^{\varepsilon}(q))/Z(GL_n^{\varepsilon}(q))$. By \cite[Theorem 1]{Kondratev}, $N_{GL_n^{\varepsilon}(q)}(T) = T_1 C_{GL_n^{\varepsilon}(q)}(T_1) \le N_{GL_n^{\varepsilon}(q)}(T_1)$. It follows that $N_{PGL_n^{\varepsilon}(q)}(S) \le N_{PGL_n^{\varepsilon}(q)}(S_1)$. It is clear that we also have $N_{PGL_n^{\varepsilon}(q)}(S_1) \le N_{PGL_n^{\varepsilon}(q)}(S)$.
	\end{proof}
	
	\begin{corollary}
		\label{2-autos_W_4} 
		Let $q$ be a nontrivial odd prime power, $\varepsilon \in \lbrace +,- \rbrace$, and let $n \ge 3$ be a natural number. Let $T := PSL_n^{\varepsilon}(q)$, let $S$ be a Sylow $2$-subgroup of $\mathrm{Inn}(T)$, and let $S_1$ be a Sylow $2$-subgroup of $\mathrm{Inndiag}(T)$ containing $S$. Then $N_{\mathrm{Inndiag}(T)}(S) = N_{\mathrm{Inndiag}(T)}(S_1)$.  
	\end{corollary} 
	
	We are now ready to prove Lemma \ref{lemma_fusion_out_PSL_PSU}. 
	
	\begin{proof}[Proof of Lemma \ref{lemma_fusion_out_PSL_PSU}]
		Assume that $n = 2$ and $q \equiv 3$ or $5 \mod 8$. Then $S \cong C_2 \times C_2$ by Lemma \ref{sylow_PSL_2(q)}. There is only one non-nilpotent fusion system on $S$. Since $T$ and $A$ are not $2$-nilpotent, we have that $\mathcal{F}_S(T)$ and $\mathcal{F}_S(A)$ are not nilpotent (see \cite[Theorem 1.4]{Linckelmann}). It follows that $\mathcal{F}_S(T) = \mathcal{F}_S(A)$. 
		
		From now on, we assume that either $n \ge 3$, or $n = 2$ and $q \equiv 1$ or $7 \mod 8$. Let $P, Q \le S$ and $a \in A$ such that $P^a \le Q$. We are going to show that $c_a|_{P,Q}$ is a morphism in $\mathcal{F}_S(T)$. By the Frattini argument, we have $a = wu$ for some $w \in N_A(S)$ and some $u \in T$. We prove that $c_w|_{S,S} \in \mathrm{Inn}(S)$. This clearly implies that $c_a|_{P,Q}$ is a morphism in $\mathcal{F}_S(T)$.
		
		Suppose that $n = 2$. Then $S$ is dihedral of order at least $8$ by Lemma \ref{sylow_PSL_2(q)}, and so $\mathrm{Aut}(S)$ is a $2$-group by \cite[Proposition 4.53]{Craven}. This implies that $\mathrm{Aut}_A(S) = \mathrm{Inn}(S)$, whence $c_w|_{S,S} \in \mathrm{Inn}(S)$.
		
		Suppose now that $n \ge 3$. Let $S_1$ be a Sylow $2$-subgroup of $\mathrm{Inndiag}(PSL_n^{\varepsilon}(q))$ containing $S$. Since $T$ has odd index in $A$, we have that $A \le P \Gamma L_n^{\varepsilon}(q)$. By the Frattini argument, $w = w_1w_2$ for some $w_1 \in N_{P \Gamma L_n^{\varepsilon} (q)}(S_1)$ and some $w_2 \in \mathrm{Inndiag}(PSL_n^{\varepsilon}(q))$. Since $w_1$ normalizes both $S_1$ and $T$, we have that $w_1$ normalizes $S$. And since $w = w_1w_2$ normalizes $S$, we also have that $w_2$ normalizes $S$. So $w_2$ normalizes $S_1$ by Corollary \ref{2-autos_W_4}. Consequently, $w = w_1 w_2 \in N_{P \Gamma L_n^{\varepsilon} (q)}(S_1)$. By Corollary \ref{2-autos_W_3}, $c_w|_{S_1,S_1}$ is a $2$-automorphism of $S_1$. So $c_w|_{S,S}$ is a $2$-automorphism of $S$. Since $S \in \mathrm{Syl}_2(A)$ and $w \in A$, this implies that $c_w|_{S,S} \in \mathrm{Inn}(S)$, as required. 
	\end{proof}
	
	\section{The case $n \le 5$} 
	\label{small_cases} 
	In this section, we verify Theorem \ref{A} for the case $n \le 5$.
	
	\begin{proposition} 
		Let $q$ be a nontrivial odd prime power, and let $G$ be a finite simple group. Then the following are equivalent: 
		\begin{enumerate} 
			\item[(i)] the $2$-fusion system of $G$ is isomorphic to the $2$-fusion system of $PSL_2(q)$;
			\item[(ii)] the Sylow $2$-subgroups of $G$ are isomorphic to those of $PSL_2(q)$; 
			\item[(iii)] $G \cong PSL_2^{\varepsilon}(q^{*})$ for some $\varepsilon \in \lbrace +,- \rbrace$ and some odd prime power $q^{*} \ge 5$ with $\varepsilon q^{*} \sim q$, or $\vert PSL_2(q) \vert_2 = 8$ and $G \cong A_7$. 
		\end{enumerate} 
		In particular, Theorem \ref{A} holds for $n = 2$. 
	\end{proposition} 
	
	\begin{proof}
		The implication (i) $\Rightarrow$ (ii) is clear. 
		
		(ii) $\Rightarrow$ (iii): Assume that the Sylow $2$-subgroups of $G$ are isomorphic to those of $PSL_2(q)$. Hence, $G$ has dihedral Sylow $2$-subgroups of order $\frac{1}{2}(q-1)_2(q+1)_2$. Applying a result of Gorenstein and Walter \cite[Theorem 1]{GorensteinWalter}, we may conclude that $G \cong PSL_2(q^{*})$ for some odd prime power $q^{*} \ge 5$, or $G \cong A_7$. Suppose that the former holds. Then $(q^{*}-1)_2(q^{*}+1)_2 = 2 \vert G \vert_2 = (q-1)_2(q+1)_2$, whence either $q^{*} \sim q$ or $-q^{*} \sim q$. Since $PSL_2(q^{*}) \cong PSU_2(q^{*})$, this implies that the first statement in (iii) is satisfied. If $G \cong A_7$, then $\vert PSL_2(q) \vert_2 = \vert G \vert_2 = 8$, so that the second statement in (iii) is satisfied.  
		
		(iii) $\Rightarrow$ (i): Assume that (iii) holds. Set $G_1 := G$ and $G_2 := PSL_2(q)$. For $i \in \lbrace 1,2 \rbrace$, let $S_i \in \mathrm{Syl}_2(G_i)$ and $\mathcal{F}_i := \mathcal{F}_{S_i}(G_i)$. Clearly, $S_1$ and $S_2$ are dihedral groups of the same order. Let $i \in \lbrace 1,2 \rbrace$. By \cite[Chapter 5, Theorem 4.3]{Gorenstein}, any subgroup of $S_i$ is cyclic or dihedral. By \cite[Proposition 4.53]{Craven}, a dihedral subgroup of $S_i$ with order greater than $4$ cannot be $\mathcal{F}_i$-essential. Since the automorphism group of a finite cyclic $2$-group is itself a $2$-group, a cyclic subgroup of $S_i$ cannot be $\mathcal{F}_i$-essential either. So we have that any $\mathcal{F}_i$-essential subgroup of $S_i$ is a Klein four group. Alperin’s fusion theorem \cite[Part I, Theorem 3.5]{AKO} implies that 
		\begin{equation*}
			\mathcal{F}_i = \langle \mathrm{Aut}_{\mathcal{F}_i}(P) \ \vert \ P \le S_i, P \cong C_2 \times C_2 \ \textnormal{or} \ P = S_i \rangle_{S_i}. 
		\end{equation*}  
		If $\vert S_i \vert = 4$, then $\mathrm{Aut}_{\mathcal{F}_i}(S_i)$ is the unique subgroup of $\mathrm{Aut}(S_i)$ with order $3$, because otherwise $\mathrm{Aut}_{\mathcal{F}_i}(S_i) = \mathrm{Inn}(S_i)$, so that \cite[Theorem 1.4]{Linckelmann} would imply that $G_i$ is $2$-nilpotent. If $\vert S_i \vert \ge 8$, then $\mathrm{Aut}_{\mathcal{F}_i}(S_i) = \mathrm{Inn}(S_i)$ since $\mathrm{Aut}(S_i)$ is a $2$-group by \cite[Proposition 4.53]{Craven}, and for any Klein four subgroup $P$ of $S_i$, we have $\mathrm{Aut}_{\mathcal{F}_i}(P) = \mathrm{Aut}(P)$ by \cite[Chapter 7, Theorem 7.3]{Gorenstein}. As $S_1 \cong S_2$ and as the preceding observations do not depend on whether $i$ is $1$ or $2$, we may conclude that $\mathcal{F}_1 \cong \mathcal{F}_2$, as required. 
	\end{proof} 
	
	\begin{proposition} 
		\label{k=3} 
		Let $q$ be a nontrivial odd prime power, and let $G$ be a finite simple group. Then the following are equivalent: 
		\begin{enumerate} 
			\item[(i)] the $2$-fusion system of $G$ is isomorphic to the $2$-fusion system of $PSL_3(q)$;
			\item[(ii)] the Sylow $2$-subgroups of $G$ are isomorphic to those of $PSL_3(q)$; 
			\item[(iii)] $G \cong PSL_3^{\varepsilon}(q^{*})$ for some $\varepsilon \in \lbrace +,- \rbrace$ and some nontrivial odd prime power $q^{*}$ with $\varepsilon q^{*} \sim q$, or $(q+1)_2 = 4$ and $G \cong M_{11}$. 
		\end{enumerate} 
		In particular, Theorem \ref{A} holds for $n = 3$. 
	\end{proposition} 
	
	\begin{proof}
		The implication (i) $\Rightarrow$ (ii) is clear. 
		
		(ii) $\Rightarrow$ (iii): Assume that the Sylow $2$-subgroups of $G$ are isomorphic to those of $PSL_3(q)$. Hence, a Sylow $2$-subgroup of $G$ is wreathed (i.e. isomorphic to $C_{2^k} \wr C_2$ for some positive integer $k$) if $q \equiv 1 \mod 4$, and semidihedral if $q \equiv 3 \mod 4$. Applying work of Alperin, Brauer and Gorenstein, namely \cite[Third Main Theorem]{AlperinBrauerGorenstein1} and \cite[First Main Theorem]{AlperinBrauerGorenstein2}, we may conclude that either $G \cong PSL_3^{\varepsilon}(q^{*})$ for some $\varepsilon \in \lbrace +,- \rbrace$ and some nontrivial odd prime power $q^{*}$ with $\varepsilon q^{*} \equiv q \mod 4$, or $q \equiv 3 \mod 4$ and $G \cong M_{11}$. If the former holds, then $((q^{*}-\varepsilon)_2)^2(q^{*}+\varepsilon)_2 = \vert G \vert_2 = ((q-1)_2)^2(q+1)_2$, and it easily follows that $\varepsilon q^{*} \sim q$. If $G \cong M_{11}$, then $16 = \vert G \vert_2 = ((q-1)_2)^2(q+1)_2$ and hence $(q+1)_2 = 4$.
		
		(iii) $\Rightarrow$ (i): Assume that (iii) holds. If $q \equiv 1 \mod 4$, then Proposition \ref{PSL_PSU_fusion} implies that the $2$-fusion system of $G$ is isomorphic to the $2$-fusion system of $PSL_3(q)$. Alternatively, this can be seen from \cite[Proposition 5.87]{Craven}. Now suppose that $q \equiv 3 \mod 4$. If $(q+1)_2 \ne 4$, then we could apply Proposition \ref{PSL_PSU_fusion} again, but we are going to argue in a more elementary way. Let $G_1 := G$ and $G_2 := PSL_3(q)$. For $i \in \lbrace 1,2 \rbrace$, let $S_i \in \mathrm{Syl}_2(G_i)$ and $\mathcal{F}_i := \mathcal{F}_{S_i}(G_i)$. Clearly, $S_1$ and $S_2$ are semidihedral groups of the same order. Let $i \in \lbrace 1,2 \rbrace$. By \cite[Chapter 5, Theorem 4.3]{Gorenstein}, any proper subgroup of $S_i$ is cyclic, dihedral or generalized quaternion. By \cite[Proposition 4.53]{Craven}, dihedral subgroups of $S_i$ with order greater than $4$ and generalized quaternion subgroups of $S_i$ with order greater than $8$ cannot be $\mathcal{F}_i$-essential. Since the automorphism group of a finite cyclic $2$-group is itself a $2$-group, a cyclic subgroup of $S_i$ cannot be $\mathcal{F}_i$-essential either. Alperin’s fusion theorem \cite[Part I, Theorem 3.5]{AKO} implies that 
		\begin{equation*}
			\mathcal{F}_i = \langle \mathrm{Aut}_{\mathcal{F}_i}(P) \ \vert \ P \cong C_2 \times C_2, P \cong Q_8, \ \textnormal{or} \ P = S_i \rangle_{S_i}. 
		\end{equation*}  
		Since $\mathrm{Aut}(S_i)$ is a $2$-group by \cite[Proposition 4.53]{Craven}, we have $\mathrm{Aut}_{\mathcal{F}_i}(S_i) = \mathrm{Inn}(S_i)$. From \cite[pp. 10-11, Proposition 1]{AlperinBrauerGorenstein1}, one can see that $\mathrm{Aut}_{\mathcal{F}_i}(P) = \mathrm{Aut}(P)$ for any subgroup $P$ of $S_i$ isomorphic to $C_2 \times C_2$ or $Q_8$. As $S_1 \cong S_2$ and as the preceding observations do not depend on whether $i$ is $1$ or $2$, we may conclude that $\mathcal{F}_1 \cong \mathcal{F}_2$, as required. 
	\end{proof} 
	
	The next two lemmas are required to verify Theorem \ref{A} for the case $n = 4$. 
	
	\begin{lemma}
		\label{fusion_alternating_PSL} 
		Let $q$ be an odd prime power with $q \equiv 3 \mod 8$. Assume that $G=A_{10}$ or $A_{11}$. Then the $2$-fusion system of $G$ is not isomorphic to the $2$-fusion system of $PSL_4(q)$.
	\end{lemma} 
	
	\begin{proof}
		Set $x := (1 \ 2)(3 \ 4) \in G$ and $y := (1 \ 2)(3 \ 4)(5 \ 6)(7 \ 8) \in G$. Let $g \in G$ be an involution. Then the cycle type of $g$ is either that of $x$ or that of $y$. So, by \cite[4.3.1]{KurzweilStellmacher}, $g$ is conjugate to $x$ or $y$ in the ambient symmetric group, which easily implies that $g$ is also $G$-conjugate to $x$ or $y$. The involutions $x$ and $y$ are not $G$-conjugate as they have different cycle types. It follows that $G$ has precisely two conjugacy classes of involutions with representatives $x$ and $y$.   
		
		By a direct calculation, 
		\begin{equation*}
			S := \langle (1 \ 2 \ 3 \ 4)(9 \ 10),(1 \ 2)(3 \ 4),(5 \ 6 \ 7 \ 8)(9 \ 10),(5 \ 6)(7 \ 8),(1 \ 5)(2 \ 6)(3 \ 7)(4 \ 8) \rangle
		\end{equation*} 
		is a Sylow 2-subgroup of $G$. Another calculation confirms that $S$ has precisely 14 involutions whose cycle type is that of $x$ and precisely 29 involutions whose cycle type is that of $y$. So there are precisely two $\mathcal{F}_S(G)$-conjugacy classes of involutions, one of which has 14 elements, while the other one has 29 elements. In order to prove that $\mathcal{F}_S(G)$ is not isomorphic to the 2-fusion system of $PSL_4(q)$, we show that the $2$-fusion system of $PSL_4(q)$ has a conjugacy class of involutions with precisely $17$ elements. 
		
		Let $W_1$ be a Sylow 2-subgroup of $GL_2(q)$, and let $W_2$ be the Sylow 2-subgroup of $GL_4(q)$ obtained from $W_1$ by the construction given in the last statement of Lemma \ref{sylow_power_2}. Let $W := W_2 \cap SL_4(q) \in \mathrm{Syl}_2(SL_4(q))$, and let $R$ be the image of $W$ in $PSL_4(q)$. The involutions of $W_2$ are precisely the elements 
		\begin{equation*}
			\begin{pmatrix}
				a \ & \ \\
				\ & \ b 
			\end{pmatrix} \ \ \textnormal{and} \ \  
			\begin{pmatrix}
				\ & \ c \\
				c^{-1} \ & \ 
			\end{pmatrix},
		\end{equation*}
		where $a, b, c \in W_1$ and $\mathrm{max}\lbrace \mathrm{ord}(a), \mathrm{ord}(b) \rbrace = 2$. Bearing in mind that $W_1$ is semidihedral of order $16$, which holds because of $q \equiv 3 \mod 8$, we may see from Lemma \ref{sylow_2} that $W$ has precisely 35 involutions. As one of them is $-I_4$, and as the product of $-I_4$ with an involution of $W$ different from $-I_4$ is again an involution, we may conclude that $R$ has precisely 17 involutions that are images of involutions of $W$. Since any noncentral involution of $SL_4(q)$ is $SL_4(q)$-conjugate to a diagonal matrix having diagonal entries $1$,$1$,$-1$,$-1$, we have that all the noncentral involutions of $SL_4(q)$ are $SL_4(q)$-conjugate. Thus the 17 involutions of $R$ induced by involutions of $W$ are $PSL_4(q)$-conjugate. As an element of $PSL_4(q)$ induced by an involution cannot be conjugate to an element of $PSL_4(q)$ not induced by an involution, it follows that there is an $\mathcal{F}_R(PSL_4(q))$-conjugacy class of involutions with precisely 17 elements.
	\end{proof}
	
	\begin{lemma}
		\label{fusion_sporadic_PSL}
		Let $q$ be an odd prime power with $q \equiv 5 \mod 8$. Assume that $G=M_{22}$, $M_{23}$ or $McL$. Then the $2$-fusion system of $G$ is not isomorphic to the $2$-fusion system of $PSL_4(q)$.
	\end{lemma} 
	
	\begin{proof} 
		Let $S \in \mathrm{Syl}_2(G)$ and $\mathcal{F} := \mathcal{F}_S(G)$. Let $x$ be an element of $S$ with order $4$ such that $\langle x \rangle$ is fully $\mathcal{F}$-centralized. In other words, we have $C_S(x) \in \mathrm{Syl}_2(C_G(x))$. If $G=M_{22}$ or $M_{23}$, then by \cite{Atlas}, $C_G(x)$ is a $2$-group, whence $C_{\mathcal{F}}(\langle x \rangle) = \mathcal{F}_{C_S(x)}(C_G(x)) = \mathcal{F}_{C_S(x)}(C_S(x))$. If $G=McL$, then by \cite{Atlas}, $G$ has precisely one conjugacy class of elements of order $4$, so that all elements of $S$ with order $4$ are $\mathcal{F}$-conjugate. 
		
		Consequently, we either have that $C_{\mathcal{F}}(\langle x \rangle)$ is nilpotent for all elements $x \in S$ with order 4 such that $\langle x \rangle$ is fully $\mathcal{F}$-centralized, or all elements of $S$ with order 4 are $\mathcal{F}$-conjugate. We are going to show that the $2$-fusion system of $PSL_4(q)$ has neither of these properties. 
		
		Let $\lambda$ be an element of $\mathbb{F}_{q}^{*}$ of order $4$ and let $y$ be the image of $\mathrm{diag}(1,1,\lambda,\lambda^{-1})$ in $PSL_4(q)$. Clearly, $y$ has order $4$. Let $R$ be a Sylow 2-subgroup of $PSL_4(q)$ containing a Sylow $2$-subgroup of $C := C_{PSL_4(q)}(y)$. Clearly, $y \in R$. Let us denote $\mathcal{F}_R(PSL_4(q))$ by $\mathcal{G}$. Then $\langle y \rangle$ is fully $\mathcal{G}$-centralized. The centralizer $C$ is not $2$-nilpotent since it has a subgroup isomorphic to $SL_2(q)$. So, by \cite[Theorem 1.4]{Linckelmann}, $C_{\mathcal{G}}(\langle y \rangle) = \mathcal{F}_{C_R(y)}(C)$ is not nilpotent. 
		
		Let $m$ denote the matrix  
		\begin{equation*}
			\begin{pmatrix}
				0 & \lambda & \ & \ \\
				1 & 0 & \ & \ \\
				\ & \ & 0 & -\lambda \\
				\ & \ & 1 & 0
			\end{pmatrix} \in SL(4,q).
		\end{equation*}
		A direct calculation, using $q \equiv 5 \ \mathrm{mod} \ 8$, shows that $m$ has no eigenvalues, whence $m$ is in particular not diagonalizable. The image of $m$ in $PSL_4(q)$ has order 4, but it is not $PSL_4(q)$-conjugate to $y$. Therefore, $PSL_4(q)$ has more than one conjugacy class of elements with order $4$. Thus there is more than one $\mathcal{G}$-conjugacy class of elements with order $4$.
	\end{proof} 
	
	\begin{proposition}
		\label{n=4} 
		Let $q$ be a nontrivial odd prime power and let $G$ be a finite simple group. Then the following are equivalent:
		\begin{enumerate}
			\item[(i)] the $2$-fusion system of $G$ is isomorphic to the $2$-fusion system of $PSL_4(q)$;
			\item[(ii)] $G \cong PSL_4^{\varepsilon}(q^{*})$ for some $\varepsilon \in \lbrace +,- \rbrace$ and some nontrivial odd prime power $q^{*}$ with $\varepsilon q^{*} \sim q$. 
		\end{enumerate}
		In particular, Theorem \ref{A} holds for $n = 4$. 
	\end{proposition} 
	
	\begin{proof}
		The implication (ii) $\Rightarrow$ (i) is given by Proposition \ref{PSL_PSU_fusion}. 
		
		(i) $\Rightarrow$ (ii): Assume that the $2$-fusion system of $G$ is isomorphic to the $2$-fusion system of $PSL_4(q)$. Then the Sylow $2$-subgroups of $G$ are isomorphic to those of $PSL_4(q)$. Applying Mason’s results \cite[Theorem 1.1 and Corollary 1.3]{Mason1} and \cite[Theorems 1.1 and 3.15]{Mason2}, the latter together with \cite[Theorem 4.10.5 (f)]{GLS3}, we see that one of the following holds:
		\begin{enumerate}
			\item[(1)] $G \cong PSL_4^{\varepsilon}(q^{*})$ for some nontrivial odd prime power $q^{*}$ and some $\varepsilon \in \lbrace +,- \rbrace$ with $\varepsilon q^{*} \equiv q \mod 4$;
			\item[(2)] $G \cong A_{10}$ or $A_{11}$, and $q \equiv 3 \ \mathrm{mod} \ 4$;
			\item[(3)] $G \cong M_{22}$, $M_{23}$ or $McL$, and $q \equiv 5 \ \mathrm{mod} \ 8$.
		\end{enumerate} 
		Let $q_0$ be a nontrivial odd prime power, $\varepsilon_0 \in \lbrace +,- \rbrace$, and $k_0, s_0 \in \mathbb{N}$ such that $2^{k_0} = (q_0 - \varepsilon_0)_2$ and $2^{s_0} = (q_0 + \varepsilon_0)_2$. Then we have 
		\begin{equation*}
			\vert PSL_4^{\varepsilon_0}(q_0) \vert_2 = \frac{\vert GL_4^{\varepsilon_0}(q_0) \vert_2}{2^{k_0}(4,2^{k_0})} = \frac{2 (\vert GL_2^{\varepsilon_0}(q_0) \vert_2)^2}{2^{k_0}(4,2^{k_0})} = \frac{2^{3k_0 + 2s_0 + 1}}{(4,2^{k_0})}.
		\end{equation*} 
		Let $k, s \in \mathbb{N}$ such that $2^k = (q-1)_2$ and $2^s = (q+1)_2$. 
		
		Suppose that (1) holds, and let $k^{*}, s^{*} \in \mathbb{N}$ such that $2^{k^{*}} = (q^{*}-\varepsilon)_2$ and $2^{s^{*}} = (q^{*} + \varepsilon)_2$. Then we have 
		\begin{equation*}
			\frac{2^{3k^{*} + 2s^{*} + 1}}{(4,2^{k^{*}})} = \vert G \vert_2 = \frac{2^{3k + 2s + 1}}{(4,2^{k})}. 
		\end{equation*}  
		Since $\varepsilon q^{*} \equiv q \mod 4$, this easily implies $\varepsilon q^{*} \sim q$.  
		
		Suppose that (2) holds. Then $2^7 = \vert G \vert_2 = 2^{3 + 2s}$, whence $s = 2$ and thus $q \equiv 3 \mod 8$. This is a contradiction to Lemma \ref{fusion_alternating_PSL}. So (2) does not hold. 
		
		Also, (3) cannot hold because of Lemma \ref{fusion_sporadic_PSL}. 
	\end{proof} 
	
	\begin{proposition}
		\label{n=5} 
		Let $q$ be a nontrivial odd prime power, and let $G$ be a finite simple group. Then the following are equivalent:
		\begin{enumerate}
			\item[(i)] the $2$-fusion system of $G$ is isomorphic to the $2$-fusion system of $PSL_5(q)$;
			\item[(ii)] the Sylow $2$-subgroups of $G$ are isomorphic to those of $PSL_5(q)$; 
			\item[(iii)] $G \cong PSL_5^{\varepsilon}(q^{*})$ for some nontrivial odd prime power $q^{*}$ and some $\varepsilon \in \lbrace +,- \rbrace$ with $\varepsilon q^{*} \sim q$. 
		\end{enumerate} 
		In particular, Theorem \ref{A} holds for $n = 5$. 
	\end{proposition}
	
	\begin{proof}
		The implication (i) $\Rightarrow$ (ii) is clear, and the implication (iii) $\Rightarrow$ (i) is given by Proposition \ref{PSL_PSU_fusion}.
		
		(ii) $\Rightarrow$ (iii): Assume that the Sylow $2$-subgroups of $G$ are isomorphic to those of $PSL_5(q)$. Applying work of Mason \cite[Theorem 1.1]{Mason3}, it follows that $G \cong PSL_5^{\varepsilon}(q^{*})$ for some $\varepsilon \in \lbrace +,- \rbrace$ and some nontrivial odd prime power $q^{*}$. In view of Lemma \ref{sylow_general}, it is easy to see that a Sylow $2$-subgroup of $G$ is isomorphic to a Sylow $2$-subgroup of $GL_4^{\varepsilon}(q^{*})$, while a Sylow $2$-subgroup of $PSL_5(q)$ is isomorphic to a Sylow $2$-subgroup of $GL_4(q)$. Now it is easy to deduce from Lemmas \ref{sylow_GL_2(q)}, \ref{sylow_GU_2(q)} and \ref{sylow_power_2} that a Sylow $2$-subgroup of $G$ has a center of order $(q^{*}-\varepsilon)_2$, while a Sylow $2$-subgroup of $PSL_5(q)$ has a center of order $(q-1)_2$. It follows that $(q^{*}-\varepsilon)_2 = (q-1)_2$. Let $k,s,k^{*},s^{*} \in \mathbb{N}$ with $2^k = (q-1)_2, 2^s = (q+1)_2, 2^{k^{*}} = (q^{*} - \varepsilon)_2$ and $2^{s^{*}} = (q^{*} + \varepsilon)_2$. Then
		\begin{equation*}
			2^{4k^{*} + 2s^{*} + 1} = \vert GL_4^{\varepsilon}(q^{*}) \vert_2 = \vert G \vert_2 = \vert GL_4(q) \vert_2 = 2^{4k + 2s + 1}. 
		\end{equation*} 
		Since $2^{k^{*}} = 2^k$, we thus have $k = k^{*}$ and $s = s^{*}$. This implies $\varepsilon q^{*} \sim q$. 
	\end{proof}
	
	\section{The case $n \ge 6$: Preliminary discussion and notation}
	\label{Preliminary discussion and notation}
	Given a natural number $k \ge 6$, we say that $P(k)$ is satisfied if whenever $q_0$ is a nontrivial odd prime power and $H$ is a finite simple group satisfying (\ref{CK}) and realizing the $2$-fusion system of $PSL_k(q_0)$, we have $H \cong PSL_k^{\varepsilon}(q^{*})$ for some nontrivial odd prime power $q^{*}$ and some $\varepsilon \in \lbrace +,- \rbrace$ with $\varepsilon q^{*} \sim q_0$.
	
	In order to establish Theorem \ref{A} for $n \ge 6$, we are going to prove by induction that $P(k)$ is satisfied for all $k \ge 6$. From now on until the end of Section \ref{G0}, we will assume the following hypothesis. 
	
	\begin{hypothesis}
		\label{hypothesis} 
		Let $n \ge 6$ be a natural number such that $P(k)$ is satisfied for all natural numbers $k$ with $6 \le k < n$, and let $q$ be a nontrivial odd prime power. Moreover, let $G$ be a finite group satisfying the following properties: 
		\begin{enumerate}
			\item[(i)] $G$ realizes the $2$-fusion system of $PSL_n(q)$;
			\item[(ii)] $O(G) = 1$;
			\item[(iii)] $G$ satisfies (\ref{CK}).   
		\end{enumerate} 
	\end{hypothesis}

	We will prove the following theorem. 
	
	\begin{theorem}
		\label{theorem1}
		There is a normal subgroup $G_0$ of $G$ isomorphic to a nontrivial quotient of $SL_n^{\varepsilon}(q^{*})$ for some nontrivial odd prime power $q^{*}$ and some $\varepsilon \in \lbrace +, - \rbrace$ with $\varepsilon q^{*} \sim q$. In particular, $P(n)$ is satisfied. 
	\end{theorem} 
	
	The proof of Theorem \ref{theorem1} will occupy Sections \ref{Preliminary discussion and notation}-\ref{G0}. In this section, we introduce some notation and prove some preliminary results needed for the proof.	
	
	For each $A \subseteq \lbrace 1, \dots, n \rbrace$ of even order, let $t_A$ be the image of the diagonal matrix $\mathrm{diag}(d_1,\dots,d_n)$ in $PSL_n(q)$, where 
	\begin{equation*}
		d_i = \begin{cases}
			-1 & \, \text{if $i \in A$} \\
			1 & \, \text{if $i \not\in A$}
		\end{cases}
	\end{equation*}
	for each $1 \le i \le n$. If $i$ is an even natural number with $2 \le i < n$ and $A = \lbrace n-i+1,\dots,n \rbrace$, then we write $t_i$ for $t_A$. We denote $t_2$ by $t$, and we write $u$ for $t_{\lbrace 1,2 \rbrace}$.
	
	We assume $\rho$ to be an element of $\mathbb{F}_q^{*}$ of order $(n,q-1)$. If $\rho$ is a square in $\mathbb{F}_q$, then we assume $\mu$ to be a fixed element of $\mathbb{F}_q$ with $\rho = \mu^2$.
	
	If $n$ is even, $\rho$ is a square in $\mathbb{F}_q$, and $i$ is an odd natural number with $1 \le i < n$, then
	\begin{equation*}
		\begin{pmatrix} \mu I_{n-i} & \\ & -\mu I_i \end{pmatrix}
	\end{equation*}
	is an element of $SL_n(q)$ by Proposition \ref{involutions of PSL(n,q)}, and we will denote its image in $PSL_n(q)$ by $t_i$. 
	
	If $n$ is even and $\rho$ is a non-square element of $\mathbb{F}_q$, then we denote the matrix
	\begin{equation*}
		\begin{pmatrix} & I_{n/2} \\ \rho I_{n/2} & \end{pmatrix}
	\end{equation*} 
	by $\widetilde w$, and if $\widetilde w \in SL_n(q)$, then we use $w$ to denote its image in $PSL_n(q)$. 
	
	Note that, by Proposition \ref{involutions of PSL(n,q)}, any involution of $PSL_n(q)$ is conjugate to $t_i$ for some $1 \le i < n$ such that $t_i$ is defined, or to $w$ (if defined).
	
	Next, we construct a Sylow $2$-subgroup of $C_{PSL_n(q)}(t)$ containing some ``nice'' elements of $PSL_n(q)$. Take a Sylow $2$-subgroup $V$ of $GL_2(q)$ containing each diagonal matrix in $GL_2(q)$ with $2$-elements of $\mathbb{F}_q^{*}$ along the main diagonal. Similarly, we assume $V_2$ to be a Sylow $2$-subgroup of $GL_{n-4}(q)$ containing each diagonal matrix in $GL_{n-4}(q)$ with $2$-elements of $\mathbb{F}_q^{*}$ along the main diagonal. Now let $W$ be a Sylow $2$-subgroup of $GL_{n-2}(q)$ containing 
	\begin{equation*}
		\left \lbrace \begin{pmatrix} A & \\ & B \end{pmatrix} \ : \ A \in V, B \in V_2 \right \rbrace.
	\end{equation*} 
	If $n = 6$, then we assume that $V = V_2$ and that $W$ is the Sylow $2$-subgroup 
	\begin{equation*}
		\left\lbrace\begin{pmatrix}
			A & \\
			& B \\ 
		\end{pmatrix} 
		\ : \ A, B \in V \right \rbrace 
		\cdot
		\left \langle 
		\begin{pmatrix}
			& I_{2} \\
			I_{2} &  \\ 
		\end{pmatrix} \right \rangle
	\end{equation*} 
	of $GL_4(q)$.
	
	Let $\widetilde t := \mathrm{diag}(1,\dots,1,-1,-1) \in SL_n(q)$. Then we have 
	\begin{equation*}
		C_{SL_n(q)}(\widetilde t) = \left \lbrace \begin{pmatrix} A & \\ & B \end{pmatrix} \ : \ A \in GL_{n-2}(q), B \in GL_2(q), \mathrm{det}(A)\mathrm{det}(B) = 1  \right \rbrace.
	\end{equation*} 
	It is easy to note that 
	\begin{equation*}
		\widetilde T := \left \lbrace \begin{pmatrix} A & \\ & B \end{pmatrix} \ : \ A \in W, B \in V, \mathrm{det}(A)\mathrm{det}(B) = 1  \right \rbrace
	\end{equation*} 
	is a Sylow $2$-subgroup of $C_{SL_n(q)}(\widetilde t)$. Let $T$ denote the image of $\widetilde T$ in $PSL_n(q)$. As the centralizer of $t$ in $PSL_n(q)$ is the image of $C_{SL_n(q)}(\widetilde t)$ in $PSL_n(q)$, we have that $T$ is a Sylow $2$-subgroup of $C_{PSL_n(q)}(t)$. We assume $S$ to be a Sylow $2$-subgroup of $PSL_n(q)$ containing $T$. Since $C_S(t) = T \in \mathrm{Syl_2}(C_{PSL_n(q)}(t))$, we have that $\langle t \rangle$ is fully $\mathcal{F}_S(PSL_n(q))$-centralized.
	
	Let $K_1$ be the image of  
	\begin{equation*}
		\left\lbrace \begin{pmatrix} A & \\ & I_2\end{pmatrix} \ : \ A \in SL_{n-2}(q) \right\rbrace
	\end{equation*}
	in $PSL_n(q)$, and let $K_2$ be the image of 
	\begin{equation*}
		\left\lbrace \begin{pmatrix} I_{n-2} & \\ & B\end{pmatrix} \ : \ B \in SL_2(q) \right\rbrace
	\end{equation*}
	in $PSL_n(q)$. Clearly, $K_1$ and $K_2$ are normal subgroups of $C_{PSL_n(q)}(t)$ isomorphic to $SL_{n-2}(q)$ and $SL_2(q)$, respectively. Define $X_1$ to be the image of 
	\begin{equation*}
		\left\lbrace \begin{pmatrix} A & \\ & I_2\end{pmatrix} \ : \ A \in W \cap SL_{n-2}(q) \right\rbrace
	\end{equation*}
	in $PSL_n(q)$, and define $X_2$ to be the image of 
	\begin{equation*}
		\left\lbrace \begin{pmatrix} I_{n-2} & \\ & B \end{pmatrix} \ : \ B \in V \cap SL_2(q) \right\rbrace
	\end{equation*}
	in $PSL_n(q)$. 
	
	Note that $X_1 = T \cap K_1 \in \mathrm{Syl}_2(K_1)$ and $X_2 = T \cap K_2 \in \mathrm{Syl}_2(K_2)$. Define 
	\begin{equation*}
		\mathcal{C}_i := \mathcal{F}_{X_i}(K_i)
	\end{equation*} 
	for $i \in \lbrace 1,2 \rbrace$. By \cite[Part I, Proposition 6.2]{AKO}, $\mathcal{C}_1$ and $\mathcal{C}_2$ are normal subsystems of $\mathcal{F}_T(C_{PSL_n(q)}(t))$.
	
	\begin{lemma}
		\label{components_C1_C2}
		Let $\mathcal{F} := \mathcal{F}_S(PSL_n(q))$. 
		If $q \equiv 1$ or $7 \ \mathrm{mod} \ 8$, then the components of $C_{\mathcal{F}}(\langle t \rangle)$ are precisely the subsystems $\mathcal{C}_1$ and $\mathcal{C}_2$. If $q \equiv 3$ or $5 \ \mathrm{mod} \ 8$, then $\mathcal{C}_1$ is the only component of $C_{\mathcal{F}}(\langle t \rangle)$.
	\end{lemma} 
	
	\begin{proof}
		Set $C := C_{PSL_n(q)}(t)$. It is easy to note that the $2$-components of $C$ are precisely the quasisimple elements of $\lbrace K_1, K_2 \rbrace$. As $n \ge 6$ and as $K_1 \cong SL_{n-2}(q)$ and $K_2 \cong SL_2(q)$, it follows that the $2$-components of $C$ are $K_1$ and $K_2$ if $q \ne 3$, and that $K_1$ is the only $2$-component of $C$ if $q = 3$. 
		
		By Lemma \ref{PSL as Goldschmidt group}, $K_1/Z(K_1)$ is not a Goldschmidt group. If $q \ne 3$, then the lemma just cited also shows that $K_2/Z(K_2)$ is a Goldschmidt group if and only if $q \equiv 3$ or $5 \ \mathrm{mod} \ 8$.
		
		Now we apply Proposition \ref{subsystems induced by 2-components} to conclude that $\mathcal{F}_{T \cap K_1}(K_1)$ and $\mathcal{F}_{T \cap K_2}(K_2)$ are precisely the components of $\mathcal{F}_T(C)$ if $q \equiv 1$ or $7 \ \mathrm{mod} \ 8$, and that $\mathcal{F}_{T \cap K_1}(K_1)$ is the only component of $\mathcal{F}_T(C)$ if $q \equiv 3$ or $5 \ \mathrm{mod} \ 8$. This completes the proof because $C_{\mathcal{F}}(\langle t \rangle) = \mathcal{F}_T(C)$, $\mathcal{C}_1 = \mathcal{F}_{T \cap K_1}(K_1)$ and $\mathcal{C}_2 = \mathcal{F}_{T \cap K_2}(K_2)$. 
	\end{proof}  
	
	\begin{lemma}
		\label{factor_system_C_G(t)_nilpotent} 
		Let $\mathcal{F} := \mathcal{F}_S(PSL_n(q))$.
		Then the factor system $C_{\mathcal{F}}(\langle t \rangle)/X_1X_2$ is nilpotent. 
	\end{lemma} 
	
	\begin{proof}
		Set $C := C_{PSL_n(q)}(t)$. It is easy to note that $X_1 X_2 = K_1 K_2 \cap T$. By Lemma \ref{factor_systems_fusion_categories}, $C_{\mathcal{F}}(\langle t \rangle)/X_1X_2$ is isomorphic to the $2$-fusion system of $C/K_1K_2$. The factor group $C/K_1K_2$ is abelian. This easily implies that $C/K_1K_2$ has a nilpotent $2$-fusion system. Hence $C_{\mathcal{F}}(\langle t \rangle)/X_1X_2$ is nilpotent.
	\end{proof} 
	
	\begin{lemma}
		\label{products in centralizer} 
		Let $A \in W$ and $B \in V$ such that $\mathrm{det}(A)\mathrm{det}(B) = 1$. Let
		\begin{equation*}
			m := \begin{pmatrix}A & \\ & B \end{pmatrix}Z(SL_n(q)) \in T.
		\end{equation*} 
		Then we have $m \in Z(\mathcal{C}_1 \langle m \rangle)$ if and only if $A \in Z(GL_{n-2}(q))$.
	\end{lemma}
	
	\begin{proof}
		By \cite[Proposition 1]{Henke}, we have $\mathcal{C}_1 \langle m \rangle = \mathcal{F}_{X_1 \langle m \rangle}(K_1 \langle m \rangle)$. 
		
		If $A \in Z(GL_{n-2}(q))$, then $m$ is central in $K_1\langle m \rangle$, which implies that $m$ lies in the center of $\mathcal{C}_1 \langle m \rangle$. 
		
		We show now that if $A \not\in Z(GL_{n-2}(q))$, then $m \not\in Z(\mathcal{C}_1 \langle m \rangle)$. Assume to the contrary that $A \not\in Z(GL_{n-2}(q))$, but $m \in Z(\mathcal{C}_1 \langle m \rangle)$. Clearly, $m \in Z(X_1 \langle m \rangle)$. So $m$ centralizes $X_1$. It easily follows that $A$ centralizes $W \cap SL_{n-2}(q)$. Using Sylow’s theorem, we may see from Lemma \ref{centralizer_sylow_SL} that any element $A_0$ of $W$ which centralizes $W \cap SL_{n-2}(q)$ without being central in $GL_{n-2}(q)$ is $SL_{n-2}(q)$-conjugate to an element of $W$ different from $A_0$. As $A$ centralizes $W \cap SL_{n-2}(q)$, but $A \not\in Z(GL_{n-2}(q))$, it follows that $A$ is $SL_{n-2}(q)$-conjugate to an element $A' \in W$ with $A \ne A'$. As $\mathrm{det}(A) = \mathrm{det}(A')$, we have $A' = A''A$ for some $A'' \in W \cap SL_{n-2}(q)$. Now, it follows that $m$ is $K_1$-conjugate to 
		\begin{equation*}
			\begin{pmatrix} A' & \\ & B \end{pmatrix} Z(SL_n(q)) = \begin{pmatrix} A'' & \\ & I_2 \end{pmatrix} \begin{pmatrix}A & \\ & B \end{pmatrix} Z(SL_n(q)) \in X_1 \langle m \rangle. 
		\end{equation*}
		So $m$ is $K_1$-conjugate to an element of $X_1 \langle m \rangle$ which is different from $m$. Therefore, $m \not\in Z(\mathcal{C}_1 \langle m \rangle)$, a contradiction.
	\end{proof} 
	
	\begin{lemma} 
		\label{lemma_hyperfocal_subgroup}
		Set $\mathcal{F} := \mathcal{F}_S(PSL_n(q))$ and $\mathcal{G} := C_{\mathcal{F}}(\langle t \rangle)$. Then $\mathfrak{hnp}(C_{\mathcal{G}}(X_1)) = X_2$. 
	\end{lemma} 
	
	\begin{proof}
		Set $C := C_{PSL_n(q)}(t)$. Note that $C' = K_1 K_2$. 
		
		By \cite[Chapter 7, Theorem 3.4]{Gorenstein}, we have $\mathfrak{foc}(C_{\mathcal{G}}(X_1)) = C_T(X_1) \cap C_C(X_1)' \le C_T(X_1) \cap C' = C_T(X_1) \cap X_1X_2 = Z(X_1)X_2$. As $\mathfrak{hnp}(C_{\mathcal{G}}(X_1)) \le \mathfrak{foc}(C_{\mathcal{G}}(X_1))$, it follows that $\mathfrak{hnp}(C_{\mathcal{G}}(X_1)) \le Z(X_1)X_2$. 
		
		Let $P$ be a subgroup of $C_T(X_1)$ and let $\varphi$ be a $2'$-element of $\mathrm{Aut}_{C_C(X_1)}(P)$. By \cite[8.2.7]{KurzweilStellmacher}, we have 
		\begin{equation*}
			[P,\langle \varphi \rangle] = [P, \langle \varphi \rangle, \langle \varphi \rangle] \le [\mathfrak{hnp}(C_{\mathcal{G}}(X_1)) \cap P,\langle \varphi \rangle] \le [Z(X_1)X_2 \cap P,\langle \varphi \rangle]. 
		\end{equation*} 
		Since $\varphi \in \mathrm{Aut}_{C_C(X_1)}(P)$, $K_2 \trianglelefteq C$ and $X_2 = T \cap K_2$, it follows $[P,\langle \varphi \rangle] \le X_2$. Consequently, $\mathfrak{hnp}(C_{\mathcal{G}}(X_1)) \le X_2$. 
		
		On the other hand, since $K_2 \le O^2(C_C(X_1))$, we have $X_2 \le \mathfrak{hnp}(C_{\mathcal{G}}(X_1))$ by \cite[Theorem 1.33]{Craven}. 
	\end{proof}
	
	\begin{lemma}
		\label{automorphisms of X1} 
		Set $C := C_{PSL_n(q)}(t)$. Then $\mathrm{Aut}_{C}(X_1)$ is a $2$-group.
	\end{lemma} 
	
	\begin{proof}
		Let $m \in N_C(X_1)$. We have 
		\begin{equation*}
			m = \begin{pmatrix} M_1 & \\ & M_2 \end{pmatrix} Z(SL_n(q))
		\end{equation*}
		for some $M_1 \in GL_{n-2}(q)$ and some $M_2 \in GL_2(q)$ with $\mathrm{det}(M_1)\mathrm{det}(M_2) = 1$. Let $A \in W \cap SL_{n-2}(q)$ and 
		\begin{equation*}
			x := \begin{pmatrix} A & \\ & I_2 \end{pmatrix} Z(SL_n(q)) \in X_1. 
		\end{equation*} 
		As $m$ normalizes $X_1$, we have 
		\begin{equation*}
			\begin{pmatrix} A^{M_1} & \\ & I_2 \end{pmatrix} Z(SL_n(q)) = x^m \in X_1.
		\end{equation*} 
		This easily implies that $A^{M_1} \in W \cap SL_{n-2}(q)$. It follows that $M_1$ normalizes $W \cap SL_{n-2}(q)$. By \cite[Theorem 1]{Kondratev}, we have $N_{GL_{n-2}(q)}(W \cap SL_{n-2}(q)) = W C_{GL_{n-2}(q)}(W)$. It follows that $c_m|_{X_1,X_1}$ is a $2$-automorphism.
	\end{proof}
	
	Define $T_1$ to be the image of 
	\begin{equation*}
		\left\lbrace \begin{pmatrix} A & \\ & I_{n-2} \end{pmatrix} \ : \ A \in V \cap SL_2(q) \right\rbrace
	\end{equation*}
	in $PSL_n(q)$ and $T_2$ to be the image of
	\begin{equation*}
		\left\lbrace \begin{pmatrix} I_2 & & \\ & B & \\ & & I_2 \end{pmatrix} \ : \ B \in V_2 \cap SL_{n-4}(q) \right\rbrace
	\end{equation*}
	in $PSL_n(q)$. Clearly, $T_1$ and $T_2$ are subgroups of $X_1$. Recall that we use $u$ to denote $t_{\lbrace 1,2 \rbrace} \in X_1$. The following lemma sheds light on some properties of the centralizer fusion system $C_{\mathcal{C}_1}(\langle u \rangle)$. 
	
	\begin{lemma}
		\label{Centralizer of u} 
		The following hold. 
		\begin{enumerate}
			\item[(i)] We have $C_{X_1}(u) \in \mathrm{Syl}_2(C_{K_1}(u))$. In particular, $\langle u \rangle$ is fully $\mathcal{C}_1$-centralized. 
			\item[(ii)] $\mathfrak{foc}(C_{\mathcal{C}_1}(\langle u \rangle)) = T_1 T_2$.
			\item[(iii)] If $n = 6$ and $q \equiv 3$ or $5 \ \mathrm{mod} \ 8$, then $T_1$ and $T_2$ are the only subgroups of $\mathfrak{foc}(C_{\mathcal{C}_1}(\langle u \rangle))$ which are isomorphic to $Q_8$ and strongly closed in $C_{\mathcal{C}_1}(\langle u \rangle)$.
			\item[(iv)] If $n \ge 7$ and $q \equiv 3$ or $5 \ \mathrm{mod} \ 8$, then $T_1$ is the only subgroup of the intersection $\mathfrak{foc}(C_{\mathcal{C}_1}(\langle u \rangle)) \cap C_{X_1}(T_2)$ which is isomorphic to $Q_8$ and strongly closed in $C_{\mathcal{C}_1}(\langle u \rangle)$.
			\item[(v)] Let $C_1$ be the image of 
			\begin{equation*}
				\left\lbrace \begin{pmatrix} A & \\ & I_{n-2} \end{pmatrix} \ : \ A \in SL_2(q) \right\rbrace
			\end{equation*} 
			in $PSL_n(q)$ and $C_2$ be the image of 
			\begin{equation*}
				\left\lbrace \begin{pmatrix} I_2 & &\\ & B & \\ & & I_2 \end{pmatrix} \ : \ A \in SL_{n-4}(q) \right\rbrace
			\end{equation*} 
			in $PSL_n(q)$. Then any component of $C_{\mathcal{C}_1}(\langle u \rangle)$ lies in $\lbrace \mathcal{F}_{T_1}(C_1), \mathcal{F}_{T_2}(C_2) \rbrace$. Moreover, $\mathcal{F}_{T_1}(C_1)$ is a component if and only if $q \equiv 1$ or $7 \ \mathrm{mod} \ 8$, and $\mathcal{F}_{T_2}(C_2)$ is a component if and only if $n \ge 7$ or $q \equiv 1$ or $7 \ \mathrm{mod} \ 8$. 
		\end{enumerate} 
	\end{lemma} 
	
	\begin{proof}
		Clearly, $C_{K_1}(u)$ is the image of 
		\begin{equation*} 
			\left\lbrace \begin{pmatrix} A & & \\ & B & \\ & & I_2 \end{pmatrix} \ : \ A \in GL_2(q), B \in GL_{n-4}(q), \mathrm{det}(A)\mathrm{det}(B) = 1 \right\rbrace
		\end{equation*} 
		in $PSL_n(q)$. Let $\widetilde W$ be the image of 
		\begin{equation*}
			\left\lbrace \begin{pmatrix} A & & \\ & B & \\ & & I_2 \end{pmatrix} \ : \ A \in V, B \in V_2, \mathrm{det}(A) \mathrm{det}(B) = 1 \right\rbrace
		\end{equation*} 
		in $PSL_n(q)$. Clearly, we have $\widetilde W \le C_{X_1}(u)$. It is easy to note that $\widetilde W$ is a Sylow $2$-subgroup of $C_{K_1}(u)$. Thus $C_{X_1}(u) = \widetilde W \in \mathrm{Syl}_2(C_{K_1}(u))$. Hence (i) holds.  
		
		We have $C_{\mathcal{C}_1}(\langle u \rangle) = \mathcal{F}_{C_{X_1}(u)}(C_{K_1}(u)) = \mathcal{F}_{\widetilde W}(C_{K_1}(u))$. The focal subgroup theorem \cite[Chapter 7, Theorem 3.4]{Gorenstein} implies that $\mathfrak{foc}(C_{\mathcal{C}_1}(\langle u \rangle)) = \widetilde W \cap (C_{K_1}(u))'$. It is easy to see that $(C_{K_1}(u))' = C_1C_2$, where $C_1$ and $C_2$ are as in (v). We thus have $\mathfrak{foc}(C_{\mathcal{C}_1}(\langle u \rangle)) = T_1T_2$. Hence (ii) holds. 
		
		Now we turn to the proofs of (iii) and (iv). Assume that $q \equiv 3$ or $5 \ \mathrm{mod} \ 8$. Clearly, $C_1$ and $C_2$ are normal subgroups of $C_{K_1}(u)$ and we have $T_1 = C_1 \cap \widetilde W$, $T_2 = C_2 \cap \widetilde W$. This implies that $T_1$ and $T_2$ are strongly closed in $C_{\mathcal{C}_1}(\langle u \rangle)$. By Lemma \ref{sylow_SL_2(q)}, we have $T_1 \cong Q_8$ and, if $n = 6$, we also have $T_2 \cong Q_8$. Clearly, any strongly $C_{\mathcal{C}_1}(\langle u \rangle)$-closed subgroup of $\mathfrak{foc}(C_{\mathcal{C}_1}(\langle u \rangle)) = T_1T_2$ is strongly closed in $\mathcal{F}_{T_1T_2}(C_1C_2)$. Hence, in order to prove (iii), it suffices to show that if $n = 6$, then $T_1$ and $T_2$ are the only strongly $\mathcal{F}_{T_1T_2}(C_1C_2)$-closed subgroups of $T_1T_2$ which are isomorphic to $Q_8$. Similarly, in order to prove (iv), it suffices to show that if $n \ge 7$, then $T_1$ is the only subgroup of $T_1T_2$ which centralizes $T_2$, which is isomorphic to $Q_8$, and which is strongly closed in $\mathcal{F}_{T_1T_2}(C_1C_2)$. 
		
		Continue to assume that $q \equiv 3$ or $5 \ \mathrm{mod} \ 8$. In order to prove the two statements just mentioned, we need some observations. As $C_1 \cong SL_2(q)$, we have that $C_1$ is not $2$-nilpotent. So $\mathcal{F}_{T_1}(C_1)$ is not nilpotent by \cite[Theorem 1.4]{Linckelmann}. Again by \cite[Theorem 1.4]{Linckelmann}, it follows that $\mathrm{Aut}_{C_1}(T_1)$ is not a $2$-group. So $\mathrm{Aut}_{C_1}(T_1)$ has an element of order $3$. Similarly, if $n = 6$, then $\mathrm{Aut}_{C_2}(T_2)$ has an element of order $3$. It follows that there is an element $\alpha \in \mathrm{Aut}_{C_1C_2}(T_1T_2)$ such that $\alpha|_{T_1,T_1}$ has order $3$, while $\alpha|_{T_2,T_2} = \mathrm{id}_{T_2}$. Moreover, if $n = 6$, then there is an element $\beta \in \mathrm{Aut}_{C_1C_2}(T_1T_2)$ such that $\beta|_{T_1,T_1} = \mathrm{id}_{T_1}$, while $\beta|_{T_2,T_2}$ has order $3$.
		
		Continue to assume that $q \equiv 3$ or $5 \ \mathrm{mod} \ 8$. If $n = 6$, then the observations in the preceding two paragraphs show together with Lemma \ref{lemma on strongly closed subgroups} that $T_1$ and $T_2$ are the only strongly $\mathcal{F}_{T_1T_2}(C_1C_2)$-closed subgroups of $T_1T_2$ which are isomorphic to $Q_8$. As observed above, this is enough to conclude that (iii) holds. If $n \ge 7$, then we may apply the observations in the preceding two paragraphs together with Lemma \ref{lemma on strongly closed subgroups} to conclude that if $T_0$ is a strongly $\mathcal{F}_{T_1T_2}(C_1C_2)$-closed subgroup of $T_1T_2$ such that $T_0 \cong Q_8$ and such that $T_0$ centralizes $T_2$, then $T_0 = T_1$. As observed above, this is enough to conclude that (iv) holds.
		
		It remains to prove (v). It is easy to note that the $2$-components of $C_{K_1}(u)$ are precisely the quasisimple elements of $\lbrace C_1,C_2 \rbrace$. So (v) can be obtained from Proposition \ref{subsystems induced by 2-components} and Lemma \ref{PSL as Goldschmidt group}.
	\end{proof} 
	
	Let $G$ be as in Hypothesis \ref{hypothesis}. The group $G$ realizes the $2$-fusion system of $PSL_n(q)$. So, if $R$ is a Sylow $2$-subgroup of $G$, then $\mathcal{F}_S(PSL_n(q)) \cong \mathcal{F}_R(G)$. For the sake of simplicity, we will identify $S$ with a Sylow $2$-subgroup $R$ of $G$ and $\mathcal{F}_S(PSL_n(q))$ with $\mathcal{F}_R(G)$. Hence we will work under the following hypothesis.  
	
	\begin{hypothesis}
		\label{hypothesis2} 
		We will treat $G$ as a group with $S \in \mathrm{Syl}_2(G)$ and $\mathcal{F}_S(G) = \mathcal{F}_S(PSL_n(q))$. 
	\end{hypothesis}  
	
	The following lemma will play a key role in the proof of Theorem \ref{theorem1}. 
	
	\begin{lemma}
		\label{key lemma} 
		Let $x$ be an involution of $S$ such that $C_S(x) \in \mathrm{Syl}_2(C_G(x))$. Let $\mathcal{C}$ be a component of $\mathcal{F}_{C_S(x)}(C_G(x))$, and let $k$ be a natural number with $3 \le k < n$. Then the following hold. 
		\begin{enumerate}
			\item[(i)] There is a unique $2$-component $Y$ of $C_G(x)$ such that $\mathcal{C} = \mathcal{F}_{C_S(x) \cap Y}(Y)$.
			\item[(ii)] If $\mathcal{C}$ is isomorphic to the $2$-fusion system of $SL_k(q)$, then we either have that $Y/O(Y) \cong SL_k^{\varepsilon}(q^{*})/O(SL_k^{\varepsilon}(q^{*}))$ for some nontrivial odd prime power $q^{*}$ and some $\varepsilon \in \lbrace +,- \rbrace$ with $q \sim \varepsilon q^{*}$; or $k = 3$, $(q+1)_2 = 4$, and $Y/Z^{*}(Y) \cong M_{11}$.
			\item[(iii)] If $\mathcal{C}$ is isomorphic to the $2$-fusion system of a nontrivial quotient of $SL_k(q^2)$, then $Y/O(Y)$ is isomorphic to a nontrivial quotient of $SL_k^{\varepsilon}(q^{*})$ for some nontrivial odd prime power $q^{*}$ and some $\varepsilon \in \lbrace +,-\rbrace$ with $q^2 \sim \varepsilon q^{*}$. 
		\end{enumerate} 
	\end{lemma} 
	In order to prove Lemma \ref{key lemma}, we need the following observation. 
	
	\begin{lemma}
		\label{lemma_key_lemma} 
		Let $k \ge 6$ be a natural number satisfying $P(k)$. If $q_0$ is a nontrivial odd prime power and $H$ is a known finite simple group realizing the $2$-fusion system of $PSL_k(q_0)$, then $H \cong PSL_k^{\varepsilon}(q^{*})$ for some $\varepsilon \in \lbrace +,- \rbrace$ and some nontrivial odd prime power $q^{*}$ with $\varepsilon q^{*} \sim q_0$.
	\end{lemma} 
	
	\begin{proof} 
		It suffices to show that any known finite simple group $H$ satisfies (\ref{CK}). Without using the CFSG, this is a priori not clear. It can be deduced from \cite[Proposition 5.2.9]{GLS3} if $H$ is alternating, from \cite[Table 4.5.1]{GLS3} if $H$ is a finite simple group of Lie type in odd characteristic, and from \cite[Table 5.3]{GLS3} if $H$ is sporadic. If $H$ is a finite simple group of Lie type in characteristic $2$, then $H$ satisfies (\ref{CK}) since, in this case, no involution centralizer in $H$ has a $2$-component (see \cite[47.8 (3)]{FiniteGroupTheory}).
	\end{proof} 
	
	\begin{proof}[Proof of Lemma \ref{key lemma}] Since $G$ satisfies (\ref{CK}), we have that $Y/Z^{*}(Y)$ is a known finite simple group for each $2$-component $Y$ of $C_G(x)$. Proposition \ref{subsystems induced by 2-components} implies that there is a unique $2$-component $Y$ of $C_G(x)$ with $\mathcal{C} = \mathcal{F}_{C_S(x) \cap Y}(Y)$. Thus (i) holds. 
	
	Suppose that $\mathcal{C}$ is isomorphic to the $2$-fusion system of $SL_k(q_0)/Z$, where either $q_0 = q$ and $Z = 1$, or $q_0 = q^2$ and $Z \le Z(SL_k(q^2))$. In order to prove (ii) and (iii), we need the following three claims. 
	
	\medskip 
	
	(1) \textit{The $2$-fusion systems of $Y/Z^{*}(Y)$ and $PSL_k(q_0)$ are isomorphic.}
	
	As $\mathcal{C} = \mathcal{F}_{C_S(x) \cap Y}(Y)$, we have that the $2$-fusion system of $Y$ is isomorphic to the $2$-fusion system of $SL_k(q_0)/Z$. So, by Corollary \ref{corollary_factor_systems_fusion_categories}, the $2$-fusion system of $Y/O(Y)$ is isomorphic to the $2$-fusion system of $SL_k(q_0)/Z$. Lemma \ref{fusion systems of quasisimple groups} implies that the $2$-fusion systems of $Y/Z^{*}(Y)$ and $PSL_k(q_0)$ are isomorphic.
	
	\medskip
	
	(2) \textit{We have $Y/Z^{*}(Y) \cong PSL_k^{\varepsilon}(q^{*})$ for some nontrivial odd prime power $q^{*}$ and some $\varepsilon \in \lbrace +,- \rbrace$ with $q_0 \sim \varepsilon q^{*}$; or $k = 3$, $(q_0+1)_2 = 4$ and $Y/Z^{*}(Y) \cong M_{11}$.}
	
	If $k = 3$, then it follows from (1) and Proposition \ref{k=3}. If $k \in \lbrace 4, 5 \rbrace$, then it follows from (1) together with Propositions \ref{n=4} and \ref{n=5}. Assume now that $k \ge 6$. By Hypothesis \ref{hypothesis} and since $k < n$, we have that $k$ satisfies $P(k)$. Since $Y/Z^{*}(Y)$ is a known finite simple group, the claim follows from (1) and Lemma \ref{lemma_key_lemma}.
	
	\medskip
	
	(3) \textit{Suppose that $Y/Z^{*}(Y) \cong PSL_k^{\varepsilon}(q^{*})$, where $q^{*}$ and $\varepsilon$ are as in (2). Then we have $Y/O(Y) \cong SL_k^{\varepsilon}(q^{*})/U$, where $U \le Z(SL_k^{\varepsilon}(q^{*}))$ and the index of $U$ in $Z(SL_k^{\varepsilon}(q^{*}))$ is equal to the $2$-part of $|Z(SL_k(q_0))/Z|$.}
	
	The group $Y/O(Y)$ is a perfect central extension of $PSL_k^{\varepsilon}(q^{*})$. Since $Y/O(Y)$ is core-free, the center of $Y/O(Y)$ is a $2$-group. So, by Lemmas \ref{Schur_PSL} and \ref{Schur_PSU}, there is a central subgroup $U$ of $SL_k^{\varepsilon}(q^{*})$ with $Y/O(Y) \cong SL_k^{\varepsilon}(q^{*})/U$. The claim now follows from 
	\begin{align*}
		|PSL_k(q_0)|_2|Z(SL_k(q_0))/Z|_2 &= |SL_k(q_0)/Z|_2 \\
		&= |Y|_2 \\
		&= |Y/Z^{*}(Y)|_2 |Z(Y/O(Y))| \\
		&= |PSL_k(q_0)|_2 |Z(SL_k^{\varepsilon}(q^{*}))/U|.
	\end{align*} 
	Here, the second equality follows from the fact that $Y$ realizes $\mathcal{C}$, the third one holds since $|Z^{*}(Y)|_2 = |Z^{*}(Y)/O(Y)|_2 = |Z(Y/O(Y))|_2 = |Z(Y/O(Y))|$, and the fourth one follows from (1). 
	
	\medskip 
	
	Assume that $q_0 = q$ and $Z = 1$. By (2) and (3), one of the following hold: either $k=3$, $(q+1)_2 = 4$ and $Y/Z^{*}(Y) \cong M_{11}$; or $Y/O(Y) \cong SL_k^{\varepsilon}(q^{*})/U$, where $q^{*}$ is a nontrivial odd prime power, $\varepsilon \in \lbrace +,- \rbrace$, $q \sim \varepsilon q^{*}$, $U \le Z(SL_k^{\varepsilon}(q^{*}))$ and the index of $U$ in $Z(SL_k^{\varepsilon}(q^{*}))$ is equal to the $2$-part of $|Z(SL_k(q))|$. Assume that the latter holds. As $q \sim \varepsilon q^{*}$, we have $(q-1)_2 = (\varepsilon q^{*} - 1)_2 = (q^{*}-\varepsilon)_2$. Since $|Z(SL_k(q))| = (k,q-1)$ and $|Z(SL_k^{\varepsilon}(q^{*}))| = (k,q^{*}-\varepsilon)$, it follows that the $2$-part of $|Z(SL_k(q))|$ is equal to the $2$-part of $|Z(SL_k^{\varepsilon}(q^{*}))|$. It follows that $U=O(Z(SL_k^{\varepsilon}(q^{*}))) = O(SL_k^{\varepsilon}(q^{*}))$. This completes the proof of (ii). 
	
	Assume now that $q_0 = q^2$. Then, since $q^2 \equiv 1 \ \mathrm{mod} \ 4$ , (2) und (3) imply that $Y/O(Y)$ is isomorphic to a nontrivial quotient of $SL_k^{\varepsilon}(q^{*})$ for some nontrivial odd prime power $q^{*}$ and some $\varepsilon \in \lbrace +,- \rbrace$ with $q^2 \sim \varepsilon q^{*}$. Thus (iii) holds. 
	\end{proof}
	
	\section{$2$-components of involution centralizers}
	\label{2-components of involution centralizers}
	In this section, we continue to assume Hypotheses \ref{hypothesis} and \ref{hypothesis2}. We will use the notation introduced in the last section without further explanation. 
	
	The main goal of this section is to describe the $2$-components and the solvable $2$-components of the centralizers of involutions of $G$.
	
	\subsection{The subgroups $K$ and $L$ of $C_G(t)$}
	We start by considering $C_G(t)$. Let $\mathcal{F} := \mathcal{F}_S(G) = \mathcal{F}_S(PSL_n(q))$. Since $\langle t \rangle$ is fully $\mathcal{F}$-centralized, we have that $T = C_S(t) \in \mathrm{Syl}_2(C_G(t))$. Also, note that $\mathcal{F}_T(C_G(t)) = C_{\mathcal{F}}(\langle t \rangle) = \mathcal{F}_T(C_{PSL_n(q)}(t))$.  
	
	\begin{proposition}
		\label{2-component_K} 
		There is a unique $2$-component $K$ of $C_G(t)$ such that $\mathcal{C}_1 = \mathcal{F}_{T \cap K}(K)$. We have $K/O(K) \cong SL_{n-2}^{\varepsilon}(q^{*})/O(SL_{n-2}^{\varepsilon}(q^{*}))$ for some nontrivial odd prime power $q^{*}$ and some $\varepsilon \in \lbrace +,- \rbrace$ with $q \sim \varepsilon q^{*}$. Moreover, $K$ is a normal subgroup of $C_G(t)$. 
	\end{proposition} 
	
	\begin{proof}
		Set $\mathcal{F} := \mathcal{F}_S(G)$. By Lemma \ref{components_C1_C2}, $\mathcal{C}_1$ is a component of $C_{\mathcal{F}}(\langle t \rangle)$. Lemma \ref{key lemma} (i) implies that there is a unique $2$-component $K$ of $C_G(t)$ such that $\mathcal{C}_1 = \mathcal{F}_{T \cap K}(K)$. By definition, the component $\mathcal{C}_1$ is isomorphic to the $2$-fusion system of $SL_{n-2}(q)$. Lemma \ref{key lemma} (ii) implies that $K/O(K) \cong SL_{n-2}^{\varepsilon}(q^{*})/O(SL_{n-2}^{\varepsilon}(q^{*}))$ for some nontrivial odd prime power $q^{*}$ and some $\varepsilon \in \lbrace +,- \rbrace$ with $q \sim \varepsilon q^{*}$.
		
		It remains to show that $K$ is a normal subgroup of $C_G(t)$. Suppose that $\widetilde{K}$ is a $2$-component of $C_G(t)$ such that $K \cong \widetilde{K}$. Set $\widetilde{\mathcal{C}} := \mathcal{F}_{T \cap \widetilde{K}}(\widetilde{K})$. Since $\widetilde K$ is subnormal in $C_G(t)$, it easily follows from \cite[Part I, Proposition 6.2]{AKO} that $\widetilde{\mathcal{C}}$ is subnormal in $C_{\mathcal{F}}(\langle t \rangle)$. Moreover, $\widetilde{\mathcal{C}} \cong \mathcal{C}_1$ as $\widetilde{K} \cong K$. Hence $\widetilde{\mathcal{C}}$ is a component of $C_{\mathcal{F}}(\langle t \rangle)$. But as a consequence of Lemma \ref{components_C1_C2}, there is no component of $C_{\mathcal{F}}(\langle t \rangle)$ which is isomorphic to $\mathcal{C}_1$ and different from $\mathcal{C}_1$. So we have $\mathcal{C}_1 = \widetilde{\mathcal{C}}$. The uniqueness in the first statement of the proposition implies that $K = \widetilde{K}$. Consequently, $C_G(t)$ has no $2$-component which is different from $K$ and isomorphic to $K$. So $K$ is characteristic and hence normal in $C_G(t)$.
	\end{proof}
	
	From now on, $K$, $q^{*}$ and $\varepsilon$ will always have the meanings given to them by Proposition \ref{2-component_K}.
	
	Our next goal is to prove the existence and uniqueness of a normal subgroup $\widebar L$ of $\widebar{C_G(t)} := C_G(t)/O(C_G(t))$ such that $\widebar L \cong SL_2(q^{*})$, and to show that the image $\widebar K$ of $K$ in $\widebar{C_G(t)}$ and $\widebar L$ are the only subgroups of $\widebar{C_G(t)}$ which are components or solvable $2$-components of $\widebar{C_G(t)}$. First, we need to prove some lemmas. 
	
	\begin{lemma}
		\label{centralizer_K} 
		Let $A \in W$ and $B \in V$ such that $\mathrm{det}(A)\mathrm{det}(B) = 1$. Let 
		\begin{equation*}
			m := \begin{pmatrix}A & \\ & B \end{pmatrix} Z(SL_n(q)) \in T. 
		\end{equation*} 
		Set $\widebar{C_G(t)} := C_G(t)/O(C_G(t))$. Then $\widebar m$ centralizes $\widebar K$ if and only if $A \in Z(GL_{n-2}(q))$. 
	\end{lemma} 
	
	\begin{proof}
		By Lemma \ref{products in centralizer}, we have $m \in Z(\mathcal{C}_1 \langle m \rangle)$ if and only if $A \in Z(GL_{n-2}(q))$. Let $\overline{\mathcal{C}_1}$ be the subsystem of $\mathcal{F}_{\widebar{T}}(\widebar{C_G(t)})$ corresponding to $\mathcal{C}_1$ under the isomorphism from $\mathcal{F}_{T}(C_G(t))$ to $\mathcal{F}_{\widebar{T}}(\widebar{C_G(t)})$ given by Corollary \ref{corollary_factor_systems_fusion_categories}. Then we have $\widebar m \in Z(\overline{\mathcal{C}_1}\langle \widebar m \rangle)$ if and only if $A \in Z(GL_{n-2}(q))$. So it is enough to show that $\widebar m \in Z(\overline{\mathcal{C}_1}\langle \widebar m \rangle)$ if and only if $\widebar m$ centralizes $\widebar K$. It is easy to note that $\overline{\mathcal{C}_1}
		= \mathcal{F}_{\widebar{X_1}}(\widebar K)$. As a consequence of Proposition \ref{2-component_K}, we have $\widebar K \trianglelefteq \widebar{C_G(t)}$. By \cite[Proposition 1]{Henke}, we have 
		\begin{equation*}
			\overline{\mathcal{C}_1}\langle \widebar m \rangle = \mathcal{F}_{\widebar{X_1}\langle \widebar m \rangle}(\widebar K \langle \widebar m \rangle).
		\end{equation*} 
		Since $\widebar m$ is a $2$-element of $\widebar{C_G(t)}$, we have $O(\widebar K \langle \widebar m \rangle) = O(\widebar K) = 1$. Applying \cite[Corollary 1]{Glauberman}, it follows that the center of the product $\overline{\mathcal{C}_1}\langle \widebar m \rangle$ is equal to the center of $\widebar K \langle \widebar m \rangle$. It follows that that $\widebar m \in Z(\overline{\mathcal{C}_1}\langle \widebar m \rangle)$ if and only if $\widebar m$ centralizes $\widebar K$, as required. 
	\end{proof}   
	
	\begin{lemma}
		\label{centralizer_q_is_3} 
		Suppose that $q^{*} = 3$. Let $C := C_G(t)$ and $\widebar C := C/O(C)$. Then: 
		\begin{enumerate} 
			\item[(i)] The factor group $\widebar C/\widebar{K}C_{\widebar C}(\widebar K)$ is a $2$-group. 
			\item[(ii)] The centralizer $C_{\widebar{C}}(\widebar u)$ is core-free.
			\item[(iii)] The factor group $C_{\widebar C}(\widebar u)/C_{\widebar C}(\widebar K)$ is core-free. 
		\end{enumerate}
	\end{lemma}
	
	\begin{proof}
		Clearly, $\widebar{C}/\widebar{K}C_{\widebar C}(\widebar K)$ is isomorphic to a subgroup of $\mathrm{Out}(\widebar K)$. Since $q^{*} = 3$, we have that $\widebar{K} \cong SL_{n-2}^{\varepsilon}(3)$. By Propositions \ref{Out_SL_n(3)} and \ref{Out_SU_n(3)}, $\mathrm{Out}(\widebar K)$ is a $2$-group. So (i) holds. 
		
		Set $\widebar{C}_0 := \widebar K C_{\widebar C}(\widebar K)$. As a consequence of (i), $C_{\widebar C}(\widebar u)/C_{\widebar{C}_0}(\widebar u)$ is a $2$-group. Hence, in order to prove (ii), it suffices to show that $C_{\widebar{C}_0}(\widebar u)$ is core-free. As $\widebar u \in \widebar K$, we have $C_{\widebar{C}_0}(\widebar u) = C_{\widebar K}(\widebar u) C_{\widebar C}(\widebar K)$. It follows that $C_{\widebar{C}_0}(\widebar u)/C_{\widebar C}(\widebar K) \cong C_{\widebar K}(\widebar u)/(C_{\widebar K}(\widebar u) \cap C_{\widebar C}(\widebar K)) = C_{\widebar K}(\widebar u)/Z(\widebar K)$. By Corollary \ref{involution_centralizers_are_core-free_corollary}, $C_{\widebar K}(\widebar u)$ is core-free. This easily implies that $C_{\widebar K}(\widebar u)/Z(\widebar K)$ is core-free. It follows that $C_{\widebar{C}_0}(\widebar u)/C_{\widebar C}(\widebar K)$ is core-free. Consequently, $O(C_{\widebar{C}_0}(\widebar u)) = O(C_{\widebar C}(\widebar K)) = 1$. So (ii) follows. 
		
		Finally, (iii) is true since $C_{\widebar C}(\widebar u)/C_{\widebar{C}_0}(\widebar u)$ is a $2$-group and $C_{\widebar{C}_0}(\widebar u)/C_{\widebar C}(\widebar K)$ is core-free.
	\end{proof} 
	
	\begin{lemma} 
		\label{very long lemma} 
		Let $\widebar{C_G(t)} := C_G(t)/O(C_G(t))$. Then there is a unique pair $({A_1}^{+},{A_2}^{+})$ of normal subgroups ${A_1}^{+}$, ${A_2}^{+}$ of $C_{\widebar K}(\widebar u)'$ such that $C_{\widebar K}(\widebar u)' = {A_1}^{+}  \times {A_2}^{+}$, ${A_1}^{+} \cong SL_2^{\varepsilon}(q^{*})$, ${A_2}^{+} \cong SL_{n-4}^{\varepsilon}(q^{*})$ and $\widebar u \in {A_1}^{+}$. Moreover, the following hold. 
		\begin{enumerate}
			\item[(i)] ${A_1}^{+} \cap \widebar{X_1} = \widebar{T_1}$.
			\item[(ii)] ${A_2}^{+} \cap \widebar{X_1} = \widebar{T_2}$.
			\item[(iii)] There is a group isomorphism $\varphi: \widebar{K} \rightarrow SL_{n-2}^{\varepsilon}(q^{*})/O(SL_{n-2}^{\varepsilon}(q^{*}))$ under which ${A_1}^{+}$ corresponds to the image of 
			\begin{equation*}
				\left\lbrace \begin{pmatrix} A & \\ & I_{n-4} \end{pmatrix} \ : \ A \in SL_2^{\varepsilon}(q^{*}) \right\rbrace
			\end{equation*} 
			in $SL_{n-2}^{\varepsilon}(q^{*})/O(SL_{n-2}^{\varepsilon}(q^{*}))$ and under which ${A_2}^{+}$ corresponds to the image of
			\begin{equation*}
				\left\lbrace \begin{pmatrix} I_2 & \\ & B \end{pmatrix} \ : \ B \in SL_{n-4}^{\varepsilon}(q^{*}) \right\rbrace
			\end{equation*} 
			in $SL_{n-2}^{\varepsilon}(q^{*})/O(SL_{n-2}^{\varepsilon}(q^{*}))$.  
		\end{enumerate}
	\end{lemma} 
	
	\begin{proof}
		For each subsystem $\mathcal{G}$ of $\mathcal{F}_T(C_G(t))$, we use $\widebar{\mathcal{G}}$ to denote the subsystem of $\mathcal{F}_{\widebar T}(\widebar{C_G(t)})$ corresponding to $\mathcal{G}$ under the isomorphism from $\mathcal{F}_T(C_G(t))$ to $\mathcal{F}_{\widebar T}(\widebar{C_G(t)})$ given by Corollary \ref{corollary_factor_systems_fusion_categories}. Note that $\overline{\mathcal{C}_1} = \mathcal{F}_{\widebar{X_1}}(\widebar K)$. 
		
		Set $H := SL_{n-2}^{\varepsilon}(q^{*})/O(SL_{n-2}^{\varepsilon}(q^{*}))$. For each even natural number $i$ with $2 \le i \le n-2$, let $h_i$ be the image of $\widetilde{h_i} := \mathrm{diag}(-1,\dots,-1,1,\dots,1) \in SL_{n-2}^{\varepsilon}(q^{*})$ in $H$, where $-1$ occurs precisely $i$ times as a diagonal entry.
		
		We claim that there is a group isomorphism $\varphi: \widebar K \rightarrow H$ such that $\widebar{u}^{\varphi} = h_i$ for some even $2 \le i < n-2$. By Proposition \ref{2-component_K}, we have $\widebar{K} \cong K/O(K) \cong H$. As a consequence of Lemmas \ref{involutions_GL(n,q)} and \ref{diagonalizable_involutions_GU(n,q)}, any involution of $SL_{n-2}^{\varepsilon}(q^{*})$ is conjugate to $\widetilde{h_i}$ for some even $2 \le i \le n-2$. Since any involution of $H$ is induced by an involution of $SL_{n-2}^{\varepsilon}(q^{*})$, it follows that any involution of $H$ is conjugate to $h_i$ for some even $2 \le i \le n-2$. As $\widebar u$ is an involution of $\widebar K$, it follows that there is an isomorphism $\varphi: \widebar K \rightarrow H$ mapping $\widebar u$ to $h_i$ for some even $2 \le i \le n-2$. Assume that $i = n-2$. Then $\widebar u$ is central in $\widebar K$. Thus $\widebar u \in Z(\overline{\mathcal{C}_1})$ and hence $u \in Z(\mathcal{C}_1)$. This is a contradiction to Lemma \ref{center_fusion_system_SL(n,q)} and the definition of $\mathcal{C}_1$. So we have $i < n-2$. 
		
		Set $h := \widebar u^{\varphi} = h_i$. Also, let $H_1$ be the image of 
		\begin{equation*}
			\left\lbrace \begin{pmatrix} A & \\ & I_{n-2-i} \end{pmatrix} \ : \ A \in SL_i^{\varepsilon}(q^{*}) \right\rbrace
		\end{equation*}
		in $H$, and let $H_2$ be the image of 
		\begin{equation*}
			\left\lbrace \begin{pmatrix} I_i & \\ & B \end{pmatrix} \ : \ B \in SL_{n-2-i}^{\varepsilon}(q^{*}) \right\rbrace
		\end{equation*}
		in $H$. For $j \in \lbrace 1,2 \rbrace$, let ${A_j}^{+}$ be the subgroup of $\widebar K$ corresponding to $H_j$ under $\varphi$.
		
		We now proceed in a number of steps in order to complete the proof. 
		
		\medskip
		
		(1) \textit{We have $C_{\widebar K}(\widebar u)' = {A_1}^{+}{A_2}^{+}$, $[{A_1}^{+},{A_2}^{+}] = 1$, ${A_1}^{+}, {A_2}^{+} \trianglelefteq C_{\widebar K}(\widebar u)$, $\widebar u \in {A_1}^{+}$ and $\widebar u \not\in {A_2}^{+}$.}
		
		It is easy to note that $C_H(h)'$ is the central product of $H_1$ and $H_2$ and that $H_1$ and $H_2$ are normal in $C_H(h)$. Therefore, $C_{\widebar K}(\widebar u)'$ is the central product of ${A_1}^{+}$ and ${A_2}^{+}$, and ${A_1}^{+}, {A_2}^{+}$ are normal in $C_{\widebar K}(\widebar u)$. By definition of $H_1$ and $H_2$, we have $h \in H_1$ and $h \not\in H_2$. Thus $\widebar u \in {A_1}^{+}$ and $\widebar u \not\in {A_2}^{+}$.
		
		\medskip
		
		(2) \textit{We have $C_{\widebar{X_1}}(\widebar u) \in \mathrm{Syl}_2(C_{\widebar K}(\widebar u))$, and $\lbrace \mathcal{F}_{\widebar{X_1} \cap {A_1}^{+}}({A_1}^{+}), \mathcal{F}_{\widebar{X_1} \cap {A_2}^{+}}({A_2}^{+}) \rbrace$ contains every component of $C_{\overline{\mathcal{C}_1}}(\langle \widebar u \rangle)$.}
		
		By Lemma \ref{Centralizer of u} (i), we have that $\langle \widebar u \rangle$ is fully $\overline{\mathcal{C}_1}$-centralized. So we have $C_{\widebar{X_1}}(\widebar u) \in \mathrm{Syl}_2(C_{\widebar K}(\widebar u))$. 
		
		Set $P := C_{\widebar{X_1}}(\widebar u)^{\varphi} \in \mathrm{Syl}_2(C_H(h))$. It is easy to note that the $2$-components of $C_H(h)$ are precisely the quasisimple elements of $\lbrace H_1, H_2 \rbrace$. Proposition \ref{subsystems induced by 2-components} implies that the components of $\mathcal{F}_P(C_H(h))$ are precisely the quasisimple elements of $\lbrace \mathcal{F}_{P \cap H_1}(H_1), \mathcal{F}_{P \cap H_2}(H_2) \rbrace$. 
		
		Thus the components of $C_{\overline{\mathcal{C}_1}}(\langle \widebar u \rangle) = \mathcal{F}_{C_{\widebar{X_1}}(\widebar u)}(C_{\widebar K}(\widebar u))$ are precisely the quasisimple elements of $\lbrace \mathcal{F}_{\widebar{X_1} \cap {A_1}^{+}}({A_1}^{+}), \mathcal{F}_{\widebar{X_1} \cap {A_2}^{+}}({A_2}^{+}) \rbrace$.
		
		\medskip 
		
		(3) \textit{$\widebar{X_1} \cap {A_1}^{+}$ and $\widebar{X_1} \cap {A_2}^{+}$ are subgroups of $\mathfrak{foc}(C_{\overline{\mathcal{C}_1}}(\langle \widebar u \rangle))$ and are strongly closed in $C_{\overline{\mathcal{C}_1}}(\langle \widebar u \rangle)$.}
		
		We have $\mathfrak{foc}(C_{\overline{\mathcal{C}_1}}(\langle \widebar u \rangle)) = C_{\widebar{X_1}}(\widebar u) \cap C_{\widebar K}(\widebar u)'$ by the focal subgroup theorem \cite[Chapter 7, Theorem 3.4]{Gorenstein}. So the claim follows from (1). 
		
		\medskip
		
		(4) \textit{Suppose that $n = 6$ and $q \equiv 3$ or $5 \ \mathrm{mod} \ 8$. Then we have $i = 2$ and hence ${A_1}^{+} \cong SL_2^{\varepsilon}(q^{*}) \cong {A_2}^{+}$. Moreover, $\widebar{X_1} \cap {A_1}^{+} = \widebar{T_1}$ and $\widebar{X_1} \cap {A_2}^{+} = \widebar{T_2}$.}
		
		Since $n = 6$ and $2 \le i < n-2 = 4$, we have $i = 2$. Thus ${A_1}^{+} \cong H_1 \cong SL_2^{\varepsilon}(q^{*}) \cong H_2 \cong {A_2}^{+}$. By Proposition \ref{2-component_K}, we have $q \sim \varepsilon q^{*}$, whence $q^{*} \equiv 3$ or $5 \ \mathrm{mod} \ 8$. Clearly, $\widebar{X_1} \cap {A_1}^{+} \in \mathrm{Syl}_2({A_1}^{+})$ and $\widebar{X_1} \cap {A_2}^{+} \in \mathrm{Syl}_2({A_2}^{+})$. Lemma \ref{sylow_SL_2(q)} implies that $\widebar{X_1} \cap {A_1}^{+} \cong Q_8 \cong \widebar{X_1} \cap {A_2}^{+}$. By Lemma \ref{Centralizer of u} (iii), $\widebar{T_1}$ and $\widebar{T_2}$ are the only subgroups of $\mathfrak{foc}(C_{\overline{\mathcal{C}_1}}(\langle \widebar u \rangle))$ which are isomorphic to $Q_8$ and strongly closed in $C_{\overline{\mathcal{C}_1}}(\langle \widebar u \rangle)$. So, by (3), $\lbrace \widebar{X_1} \cap {A_1}^{+}, \widebar{X_1} \cap {A_2}^{+} \rbrace = \lbrace \widebar{T_1}, \widebar{T_2} \rbrace$. We have $\widebar u \in \widebar{T_1}$, and $\widebar u \not\in {A_2}^{+}$ by (1). It follows that $\widebar{X_1} \cap {A_1}^{+} = \widebar{T_1}$ and $\widebar{X_1} \cap {A_2}^{+} = \widebar{T_2}$. 
		
		\medskip
		
		(5) \textit{Suppose that $n = 6$ and $q \equiv 1$ or $7 \ \mathrm{mod} \ 8$, or that $n \ge 7$. Then we have $i = 2$, and hence ${A_1}^{+} \cong SL_2^{\varepsilon}(q^{*})$ and ${A_2}^{+} \cong SL_{n-4}^{\varepsilon}(q^{*})$. Moreover, $\widebar{X_1} \cap {A_1}^{+} = \widebar{T_1}$ and $\widebar{X_1} \cap {A_2}^{+} = \widebar{T_2}$.}
		
		We begin by proving that $\widebar{X_1} \cap {A_2}^{+} = \widebar{T_2}$. As a consequence of Lemma \ref{Centralizer of u} (v), $C_{\overline{\mathcal{C}_1}}(\langle \widebar u \rangle)$ has a component with Sylow group $\widebar{T_2}$. Applying (2), we may conclude that $\widebar{T_2} = \widebar{X_1} \cap {A_1}^{+}$ or $\widebar{X_1} \cap {A_2}^{+}$. Since $\widebar u \in A_1^{+}$ by (1), but $\widebar u \not\in \widebar{T_2}$, we have $\widebar{X_1} \cap {A_2}^{+} = \widebar{T_2}$.
		
		We show next that $i = 2$. Using Proposition \ref{SL_SU_fusion}, or using the order formulas for $\vert SL_{n-4}(q^{*}) \vert$ and $\vert SU_{n-4}(q^{*}) \vert$ given by \cite[Proposition 1.1 and Corollary 11.29]{Grove}, we see that
		\begin{equation*}
			|SL_{n-4}^{\varepsilon}(q^{*})|_2 = |SL_{n-4}(q)|_2 = |T_2| = |{A_2}^{+}|_2 = |H_2|_2 = |SL_{n-2-i}^{\varepsilon}(q^{*})|_2.
		\end{equation*}
		Using the order formulas cited above, we may conclude that $n-2-i = n-4$, whence $i = 2$. In particular, ${A_1}^{+} \cong SL_2^{\varepsilon}(q^{*})$ and ${A_2}^{+} \cong SL_{n-4}^{\varepsilon}(q^{*})$. 
		
		It remains to prove $\widebar{X_1} \cap {A_1}^{+} = \widebar{T_1}$. If $q \equiv 1$ or $7 \ \mathrm{mod} \ 8$, then Lemma \ref{Centralizer of u} (v) shows that $C_{\overline{\mathcal{C}_1}}(\langle \widebar u \rangle)$ has a component with Sylow group $\widebar{T_1}$. Since $\widebar u \in \widebar{T_1}$, but $\widebar u \not\in {A_2}^{+}$, we have $\widebar{X_1} \cap {A_1}^{+} = \widebar{T_1}$ by (2). 
		
		Now suppose that $q \equiv 3$ or $5 \ \mathrm{mod} \ 8$. Then we have $q^{*} \equiv 3$ or $5 \ \mathrm{mod} \ 8$ since $q \sim \varepsilon q^{*}$. So, by Lemma \ref{sylow_SL_2(q)}, a Sylow $2$-subgroup of ${A_1}^{+}$ is isomorphic to $Q_8$. In particular, $\widebar{X_1} \cap {A_1}^{+} \cong Q_8$. By (3), $\widebar{X_1} \cap {A_1}^{+}$ is a subgroup of $\mathfrak{foc}(C_{\overline{\mathcal{C}_1}}(\langle \widebar u \rangle))$ and is strongly closed in $C_{\overline{\mathcal{C}_1}}(\langle \widebar u \rangle)$. Moreover, by (1), $\widebar{X_1} \cap {A_1}^{+}$ centralizes $\widebar{X_1} \cap {A_2}^{+}  = \widebar{T_2}$. Lemma \ref{Centralizer of u} (iv) now implies that $\widebar{T_1} = \widebar{X_1} \cap {A_1}^{+}$. 
		
		\medskip
		
		(6) \textit{$C_{\widebar K}(\widebar u)' = {A_1}^{+} \times {A_2}^{+}$.}
		
		We have ${A_1}^{+} \cong SL_2^{\varepsilon}(q^{*})$ by (4) and (5), and $\widebar u \in Z({A_1}^{+})$ by (1). It follows that $Z({A_1}^{+}) = \langle \widebar u \rangle$. By (1), ${A_1}^{+} \cap {A_2}^{+} \le Z({A_1}^{+})$ and $\widebar u \not\in {A_1}^{+} \cap {A_2}^{+}$. It follows that ${A_1}^{+} \cap {A_2}^{+} = 1$. So (1) implies that $C_{\widebar K}(\widebar u)' = {A_1}^{+} \times {A_2}^{+}$. 
		
		\medskip
		
		(7) \textit{Assume that ${A_1}^{\circ}$ and ${A_2}^{\circ}$ are normal subgroups of $C_{\widebar K}(\widebar u)'$ such that $C_{\widebar K}(\widebar u)' = {A_1}^{\circ} \times {A_2}^{\circ}$, ${A_1}^{\circ} \cong SL_2^{\varepsilon}(q^{*})$, ${A_2}^{\circ} \cong SL_{n-4}^{\varepsilon}(q^{*})$ and $\widebar u \in {A_1}^{\circ}$. Then ${A_1}^{\circ} = {A_1}^{+}$ and ${A_2}^{\circ} = {A_2}^{+}$.}
		
		Let $j \in \lbrace 1,2 \rbrace$. As a consequence of (4) and (5), ${A_j}^{+}$ is either quasisimple or isomorphic to $SL_2(3)$. In either case, it is easy to see that ${A_j}^{+}$ is indecomposable, i.e. ${A_j}^{+}$ cannot be written as an internal direct product of two proper normal subgroups. Moreover, $\vert {A_1}^{+}/({A_1}^{+})' \vert$ and $\vert Z({A_{2}}^{+}) \vert$ as well as $\vert {A_2}^{+}/({A_2}^{+})' \vert$ and $\vert Z({A_{1}}^{+}) \vert$ are coprime. A consequence of the Krull-Remak-Schmidt theorem, namely \cite[Kapitel I, Satz 12.6]{Huppert}, implies that $\lbrace {A_1}^{+}, {A_2}^{+} \rbrace = \lbrace {A_1}^{\circ}, {A_2}^{\circ} \rbrace$. Since $\widebar u \in {A_1}^{+}$ and $\widebar u \not\in {A_2}^{\circ}$, we have ${A_1}^{+} = {A_1}^{\circ}$ and ${A_2}^{+} = {A_2}^{\circ}$.
		
		\medskip
		
		(8) \textit{The isomorphism $\varphi: \widebar{K} \rightarrow H$ maps ${A_1}^{+}$ to the image of
			\begin{equation*}
				\left\lbrace \begin{pmatrix} A & \\ & I_{n-4} \end{pmatrix} \ : \ A \in SL_2^{\varepsilon}(q^{*}) \right\rbrace
			\end{equation*} 
			in $H$ and ${A_2}^{+}$ to the image of
			\begin{equation*}
				\left\lbrace \begin{pmatrix} I_2 & \\ & B \end{pmatrix} \ : \ B \in SL_{n-4}^{\varepsilon}(q^{*}) \right\rbrace
			\end{equation*} 
			in $H$.}
		
		By (4) and (5), we have $i = 2$. So the claim follows from the definitions of ${A_1}^{+}$ and ${A_2}^{+}$.
	\end{proof}
	
	From now on, ${A_1}^{+}$ and ${A_2}^{+}$ will always have the meanings given to them by Lemma \ref{very long lemma}.
	
	\begin{lemma}
		\label{A_1^+_and_A_2^+_are_normal} 
		Let $C := C_G(t)$ and $\widebar C := C/O(C)$. Then ${A_1}^{+}$ and ${A_2}^{+}$ are normal subgroups of $C_{\widebar C}(\widebar u)$.  
	\end{lemma} 
	
	\begin{proof}
		We have $C_{\widebar K}(\widebar u) \trianglelefteq C_{\widebar C}(\widebar u)$ as $\widebar K \trianglelefteq \widebar C$. Thus $C_{\widebar K}(\widebar u)' \trianglelefteq C_{\widebar C}(\widebar u)$. Having this observed, the lemma is immediate from Lemma \ref{very long lemma}. 
	\end{proof} 
	
	Let $C := C_G(t)$ and $\widebar C := C/O(C)$. Next we introduce certain preimages of ${A_1}^{+}$ and ${A_2}^{+}$ in $C_{C}(u)$. By Corollary \ref{centralizers_p-elements}, we have $C_{\widebar{C}}(\widebar u) = \widebar{C_{C}(u)}$. We may see from Proposition \ref{2-components modulo odd order subgroup} that there is a bijection from the set of $2$-components of $C_C(u)$ to the set of $2$-components of $C_{\widebar C}(\widebar u)$ sending each $2$-component $A$ of $C_C(u)$ to $\widebar A$.
	
	Suppose that $q^{*} \ne 3$. Then ${A_1}^{+}$ is a component and hence a $2$-component of $C_{\widebar C}(\widebar u)$. We use $A_1$ to denote the $2$-component of $C_{C}(u)$ corresponding to ${A_1}^{+}$ under the bijection described above.
	
	Suppose that $q^{*} \ne 3$ or $n \ge 7$. Then ${A_2}^{+}$ is a component and hence a $2$-component of $C_{\widebar C}(\widebar u)$. We use $A_2$ to denote the $2$-component of $C_{C}(u)$ corresponding to ${A_2}^{+}$ under the bijection described above. 
	
	Suppose that $q^{*} = 3$. By Lemma \ref{centralizer_q_is_3} (ii), $O(C_{\widebar{C}}(\widebar u)) = 1$. So the factor group $C_{C}(u)/(C_{C}(u) \cap O(C))$ is core-free, whence $O(C_{C}(u)) = C_{C}(u) \cap O(C)$. Let $O(C_{C}(u)) \le A_1 \le C_{C}(u)$ such that $A_1/O(C_{C}(u))$ corresponds to ${A_1}^{+}$ under the natural group isomorphism $C_{C}(u)/O(C_{C}(u)) \rightarrow C_{\widebar{C}}(\widebar u)$. Furthermore, if $n = 6$, let $O(C_{C}(u)) \le A_2 \le C_{C}(u)$ such that $A_2/O(C_{C}(u))$ corresponds to ${A_2}^{+}$ under the natural group isomorphism $C_{C}(u)/O(C_{C}(u)) \rightarrow C_{\widebar{C}}(\widebar u)$. 
	
	\begin{lemma}
		\label{T1 lies in A1} 
		We have $T_1 \le A_1$ and $T_2 \le A_2$. 
	\end{lemma} 
	
	\begin{proof}
		Let $i \in \lbrace 1,2 \rbrace$. Set $C := C_G(t)$ and $\widebar C := C/O(C)$. 
		
		Let $C_C(u) \cap O(C) \le \widetilde{A_i} \le C_C(u)$ such that $\widetilde{A_i}/(C_C(u) \cap O(C))$ corresponds to ${A_i}^{+}$ under the natural group isomorphism $C_C(u)/(C_C(u) \cap O(C)) \rightarrow C_{\widebar C}(\widebar u)$. We have $T_i \le C_C(u)$ and, by Lemma \ref{very long lemma}, $\widebar{T_i} \le {A_i}^{+}$. Thus $T_i \le \widetilde{A_i}$. If ${A_i}^{+} \cong SL_2(3)$, then we have $A_i = \widetilde{A_i}$ and thus $T_i \le A_i$. Assume now that ${A_i}^{+}$ is a component of $C_{\widebar C}(\widebar u)$. Then $A_i$ is the $2$-component of $C_C(u)$ associated to the $2$-component $\widetilde{A_i}/(C_C(u) \cap O(C))$ of $C_C(u)/(C_C(u) \cap O(C))$. So, by Proposition \ref{2-components modulo odd order subgroup}, $A_i = O^{2'}(\widetilde{A_i})$, and hence $T_i \le A_i$.
	\end{proof} 
	
	\begin{lemma}
		\label{very easy lemma} 
		There is an element $g \in G$ such that ${T_1}^g = X_2$ and ${X_2}^g = T_1$. For each such $g \in G$, we have $u^g = t$ and $t^g = u$.  
	\end{lemma} 
	
	\begin{proof}
		The first statement easily follows from $\mathcal{F}_S(G) = \mathcal{F}_S(PSL_n(q))$. By Lemma \ref{sylow_SL_2(q)}, the groups $T_1$ and $X_2$ are generalized quaternion. So $u$ is the only involution of $T_1$ and $t$ is the only involution of $X_2$. Thus $u^g = t$ and $t^g = u$ for any $g \in G$ with ${T_1}^g = X_2$ and ${X_2}^g = T_1$.
	\end{proof} 
	
	With the above lemmas at hand, we can now prove the following proposition. 
	
	\begin{proposition}
		\label{construction of L} 
		Take an element $g \in G$ such that ${T_1}^g = X_2$ and ${X_2}^g = T_1$. Set $C := C_G(t)$ and $\widebar C := C/O(C)$. Let $L := {A_1}^g$. Then the following hold.   
		\begin{enumerate}
			\item[(i)] $L \le C_C(u)$.
			\item[(ii)] $\widebar L$ is subnormal in $\widebar C$ and $\widebar L \cong SL_2(q^{*})$.  
			\item[(iii)] The subgroups $\widebar K$ and $\widebar L$ are the only subgroups of $\widebar C$ which are components or solvable $2$-components of $\widebar C$. In particular, $\widebar K$ and $\widebar L$ are normal subgroups of $\widebar C$. 
		\end{enumerate} 
	\end{proposition} 
	
	\begin{proof}
		By Lemma \ref{very easy lemma}, we have $t^g = u$ and $u^g = t$. Hence $C_C(u)^g = C_C(u)$. As $A_1$ is a subgroup of $C_C(u)$, we thus have $L = {A_1}^g \le C_C(u)$. So (i) holds. 
		
		Before proving (ii), we show that $C_{\widebar L}(\widebar K)$ is a normal subgroup of $\widebar L$ containing $\widebar{X_2}$. Since $C_{\widebar C}(\widebar K) \trianglelefteq \widebar{C}$, we have $C_{\widebar L}(\widebar K) = \widebar{L} \cap C_{\widebar C}(\widebar K) \trianglelefteq \widebar{L}$. Because of Lemma \ref{T1 lies in A1}, we have $X_2 = {T_1}^g \le {A_1}^g = L$. Thus $\widebar{X_2} \le \widebar{L}$. By the definition of $X_2$ and by Lemma \ref{centralizer_K}, we have $\widebar{X_2} \le C_{\widebar C}(\widebar K)$. Thus $\widebar{X_2} \le C_{\widebar L}(\widebar K)$. 
		
		Note that $\widebar{X_2}$ is generalized quaternion by Lemma \ref{sylow_SL_2(q)} and in particular nonabelian. 
		
		We now prove (ii) for the case $q^{*} \ne 3$. Then $A_1$ is a $2$-component of $C_C(u)$. As $g$ normalizes $C_C(u)$ and $L = {A_1}^g$, it follows that $L$ is a $2$-component of $C_C(u)$. So $\widebar{L}$ is a $2$-component of $C_{\widebar C}(\widebar u)$. Moreover, we have $A_1/O(A_1) \cong SL_2(q^{*})$ since $A_1/(A_1 \cap O(C)) \cong \widebar{A_1} = {A_1}^{+} \cong SL_2(q^{*})$. Hence $L/O(L)$ is isomorphic to $SL_2(q^{*})$. The group $C_{\widebar L}(\widebar K) O(\widebar L)/O(\widebar L)$ is normal in $\widebar{L}/O(\widebar L)$, and it is nonabelian since $\widebar{X_2} \le C_{\widebar L}(\widebar K)$. As $\widebar L/O(\widebar L)$ is quasisimple, it follows that $C_{\widebar L}(\widebar K)O(\widebar L) = \widebar L$. So $C_{\widebar L}(\widebar K)$ has odd index in $\widebar L$. Since $\widebar L$ is a $2$-component of $C_{\widebar C}(\widebar u)$, we have $O^{2'}(\widebar L) = \widebar L$. It follows that $\widebar{L} = C_{\widebar L}(\widebar K) \le C_{\widebar C}(\widebar K)$. Since $\widebar L$ is subnormal in $C_{\widebar C}(\widebar u)$ and $C_{\widebar C}(\widebar K) \le C_{\widebar C}(\widebar u)$, we have that $\widebar L$ is subnormal in $C_{\widebar C}(\widebar K)$. Hence $\widebar L$ is subnormal in $\widebar C$. As $\widebar C$ is core-free, we have $O(\widebar{L}) = 1$. It follows that $O(L) = L \cap O(C)$ and hence $\widebar L \cong L/O(L) \cong SL_2(q^{*})$. So we have proved (ii) for the case $q^{*} \ne 3$.
		
		Assume now that $q^{*} = 3$. Then $O(C_C(u)) = C_C(u) \cap O(C)$, $O(C_C(u)) \le A_1 \le C_C(u)$, and $A_1/O(C_C(u))$ corresponds to ${A_1}^{+} \cong SL_2(3)$ under the natural isomorphism $C_C(u)/O(C_C(u)) \rightarrow C_{\widebar C}(\widebar u)$. By Lemma \ref{A_1^+_and_A_2^+_are_normal}, ${A_1}^{+}$ is normal in $C_{\widebar C}(\widebar u)$. Hence, $A_1/O(C_C(u))$ is a normal subgroup of $C_C(u)/O(C_C(u))$ isomorphic to $SL_2(3)$. Since $g$ normalizes $C_C(u)$ and $L = {A_1}^g$, it follows that $O(C_C(u)) \le L$ and that $L/O(C_C(u))$ is a normal subgroup of $C_C(u)/O(C_C(u))$ isomorphic to $SL_2(3)$. Since $L/O(C_C(u))$ corresponds to $\widebar L$ under the natural isomorphism $C_C(u)/O(C_C(u)) \rightarrow C_{\widebar C}(\widebar u)$, it follows that $\widebar L$ is a normal subgroup of $C_{\widebar C}(\widebar u)$ isomorphic to $SL_2(3)$. Recall that $\widebar{X_2} \le C_{\widebar L}(\widebar K) \trianglelefteq \widebar{L}$. As $\widebar{L}$ has order $24$ and $\widebar{X_2}$ has order $8$, $C_{\widebar L}(\widebar K)$ either equals $\widebar L$ or has index $3$ in $\widebar L$. However, if the latter holds, then $\widebar L C_{\widebar C}(\widebar K)/C_{\widebar C}(\widebar K)$ is a normal subgroup of $C_{\widebar C}(\widebar u)/C_{\widebar C}(\widebar K)$ of order $3$, which is a contradiction to Lemma \ref{centralizer_q_is_3} (iii). Thus $\widebar L = C_{\widebar L}(\widebar K) \le C_{\widebar C}(\widebar K)$. As $\widebar L \trianglelefteq C_{\widebar C}(\widebar u)$ and $C_{\widebar C}(\widebar K) \le C_{\widebar C}(\widebar u)$, it follows that $\widebar L$ is normal in $C_{\widebar C}(\widebar K)$ and hence subnormal in $\widebar C$. So we have proved (ii) for the case $q^{*} = 3$.
		
		We now prove (iii). Clearly, $\widebar T \cap \widebar K = \widebar{X_1}$. Also $\widebar T \cap \widebar L = \widebar{X_2}$ since $\vert \widebar{X_2} \vert = \vert SL_2(q) \vert_2 = \vert SL_2(q^{*}) \vert_2 = \vert \widebar L \vert_2$ and $\widebar{X_2} \le \widebar{L}$. As a consequence of Lemma \ref{factor_system_C_G(t)_nilpotent}, the fusion system $\mathcal{F}_{\widebar{T}}(\widebar C)/(\widebar{X_1}\widebar{X_2})$ is nilpotent. Applying Lemma \ref{2-nilpotence lemma}, we may conclude that $\widebar K$ and $\widebar L$ are the only subgroups of $\widebar C$ which are components or solvable $2$-components of $\widebar C$. As $\widebar K$ and $\widebar L$ are not isomorphic, both are characteristic and hence normal in $\widebar C$.
	\end{proof} 
	
	It is not difficult to observe that the definition of $L$ in Proposition \ref{construction of L} is independent of the choice of $g$. From now on, $L$ will always have the meaning given to it by the above proposition.
	
	\subsection{$2$-components of centralizers of involutions conjugate to $t_i$, $i \ne 2$}
	Having described the components and the solvable $2$-components of the group $C_G(t)/O(C_G(t))$, we now turn our attention to centralizers of involutions of $G$ not conjugate to $t$. 
	
	First we recall some notation from Section \ref{Preliminary discussion and notation}. Let $1 \le i < n$. If $i$ is even, then $t_i$ denotes the image of 
	\begin{equation*}
		\begin{pmatrix} I_{n-i} & \\ & -I_i \end{pmatrix}
	\end{equation*} 
	in $PSL_n(q)$. We use $\rho$ to denote an element of $\mathbb{F}_q^{*}$ with order $(n,q-1)$, and if $\rho$ is a square in $\mathbb{F}_q$, then $\mu$ denotes an element of $\mathbb{F}_q^{*}$ with $\mu^2 = \rho$. If $n$ is even, $\rho$ is a square in $\mathbb{F}_q$ and $i$ is odd, then $t_i$ is defined to be the image of 
	\begin{equation*}
		\begin{pmatrix} \mu I_{n-i} & \\ & -\mu I_i \end{pmatrix} \in SL_n(q)
	\end{equation*} 
	in $PSL_n(q)$. It is easy to note that $t_i$ lies in $T$ and hence in $S$ whenever $t_i$ is defined. 
	
	Let $\mathcal{S}$ denote the set of all subgroups $E$ of $PSL_n(q)$ such that there is some elementary abelian $2$-subgroup $\widetilde E \le SL_n(q)$ with $E = \widetilde{E} Z(SL_n(q))/Z(SL_n(q))$. For each $3 \le i \le n$, we define $\mathcal{S}_i$ to be the set of all elements $E$ of $\mathcal{S}$ such that $E$ contains a $PSL_n(q)$-conjugate of $t_j$ for some even $2 \le j < i$.
	
	\begin{lemma}
		\label{components of centralizer fusion systems} 
		Let $1 \le i < n$ such that $t_i$ is defined. Assume that $i \ne 2$, and that $i \le \frac{n}{2}$ if $n$ is even. Let $P$ be a Sylow $2$-subgroup of $C_{PSL_n(q)}(t_i)$ and $\mathcal{F} := \mathcal{F}_P(C_{PSL_n(q)}(t_i))$. Then the following hold.
		\begin{enumerate} 
			\item[(i)] Assume that $i \not\in \lbrace 1,n-1 \rbrace$. Then $\mathcal{F}$ has precisely two components. Denoting them in a suitable way by $\mathcal{E}_1$ and $\mathcal{E}_2$, the following hold.
			\begin{enumerate}
				\item[(a)] $\mathcal{E}_1$ is isomorphic to the $2$-fusion system of $SL_{n-i}(q)$. 
				\item[(b)] $\mathcal{E}_2$ is isomorphic to the $2$-fusion system of $SL_i(q)$. 
				\item[(c)] Let $Y_1$ be the Sylow group of $\mathcal{E}_1$ and let $Y_2$ be the Sylow group of $\mathcal{E}_2$. Then $Y_1Y_2$ is normal in $P$ and $\mathcal{F}/Y_1Y_2$ is nilpotent. The group $Y_i$, where $i \in \lbrace 1,2 \rbrace$, contains a $PSL_n(q)$-conjugate of $t$. Moreover, any elementary abelian subgroup of $Y_1$ of rank at least $2$ is contained in $\mathcal{S}_{n-i}$, and any elementary abelian subgroup of $Y_2$ of rank at least $2$ is contained in $\mathcal{S}_i$. 
			\end{enumerate} 
			\item[(ii)] Assume that $i = 1$ or $i = n-1$. Then $\mathcal{F}$ has a unique component. This component is isomorphic to the $2$-fusion system of $SL_{n-1}(q)$. If $Y$ is its Sylow group, then $Y \trianglelefteq P$ and $\mathcal{F}/Y$ is nilpotent. Moreover, any elementary abelian subgroup of $Y$ of rank at least $2$ is contained in $\mathcal{S}_{n-1}$.
		\end{enumerate} 
	\end{lemma} 
	
	\begin{proof}
		Assume that $i \not\in \lbrace 1,n-1 \rbrace$. By hypothesis, we have $i \ne 2$, and $i \le \frac{n}{2}$ if $n$ is even. It follows that $i \ge 3$ and $n-i \ge 3$. Let $J_1$ be the image of 
		\begin{equation*}
			\left\lbrace \begin{pmatrix} A & \\ & I_i \end{pmatrix} \ : \ A \in SL_{n-i}(q) \right\rbrace
		\end{equation*} 
		in $PSL_n(q)$, and let $J_2$ be the image of 
		\begin{equation*}
			\left\lbrace \begin{pmatrix} I_{n-i} & \\ & A \end{pmatrix} \ : \ A \in SL_i(q) \right\rbrace
		\end{equation*} 
		in $PSL_n(q)$. It is easy to note that $J_1$ and $J_2$ are the only $2$-components of $C_{PSL_n(q)}(t_i)$. Applying Proposition \ref{subsystems induced by 2-components} and Lemma \ref{PSL as Goldschmidt group}, we may conclude that $\mathcal{E}_1 := \mathcal{F}_{P \cap J_1}(J_1)$ and $\mathcal{E}_2 := \mathcal{F}_{P \cap J_2}(J_2)$ are the only components of $\mathcal{F} = \mathcal{F}_P(C_{PSL_n(q)}(t_i))$. Clearly, $\mathcal{E}_1$ is isomorphic to the $2$-fusion system of $SL_{n-i}(q)$, while $\mathcal{E}_2$ is isomorphic to the $2$-fusion system of $SL_i(q)$. Set $Y_1 := P \cap J_1$ and $Y_2 := P \cap J_2$. It is easy to note that $Y_1Y_2 = P \cap J_1J_2$. As $J_1J_2 \trianglelefteq C_{PSL_n(q)}(t_i)$, it follows that $Y_1Y_2 \trianglelefteq P$. By Lemma \ref{factor_systems_fusion_categories}, $\mathcal{F}/Y_1Y_2$ is isomorphic to the $2$-fusion system of $C_{PSL_n(q)}(t_i)/J_1J_2$, and it is easy to note that $C_{PSL_n(q)}(t_i)/J_1J_2$ is $2$-nilpotent. So $\mathcal{F}/Y_1Y_2$ is nilpotent by \cite[Theorem 1.4]{Linckelmann}.  It is clear from the definitions of $J_1$ and $J_2$ that both $J_1$ and $J_2$ contain a $PSL_n(q)$-conjugate of $t$. Hence $Y_k$ has an element which is $PSL_n(q)$-conjugate to $t$ for $k \in \lbrace 1,2 \rbrace$. Clearly, any elementary abelian $2$-subgroup of $J_k$, $k \in \lbrace 1,2 \rbrace$, lies in $\mathcal{S}$. Moreover, any noncentral involution of $J_1$ is $PSL_n(q)$-conjugate to $t_j$ for some even $2 \le j < n-i$, and any noncentral involution of $J_2$ is $PSL_n(q)$-conjugate to $t_j$ for some even $2 \le j < i$. This implies that any elementary abelian subgroup of $Y_1$ of rank at least $2$ is contained in $\mathcal{S}_{n-i}$, and that any elementary abelian subgroup of $Y_2$ of rank at least $2$ is contained in $\mathcal{S}_i$. This completes the proof of (i).
		
		We omit the proof of (ii) since it is very similar to the one of (i). 
	\end{proof}
	
	\begin{proposition}
		\label{2-components_ti_1}
		Let $1 \le i < n$ such that $t_i$ is defined. Assume that $i \not\in \lbrace 1, 2, n-1 \rbrace$, and that $i \le \frac{n}{2}$ if $n$ is even. Let $x$ be an involution of $S$ which is $G$-conjugate to $t_i$. Then $C_G(x)$ has precisely two $2$-components. Denoting them in a suitable way by $J_1$ and $J_2$, the following hold. 
		\begin{enumerate}
			\item[(i)] $J_1/O(J_1)$ is isomorphic to $SL_{n-i}^{\varepsilon}(q^{*})/O(SL_{n-i}^{\varepsilon}(q^{*}))$, where $\varepsilon$ and $q^{*}$ are as in Proposition \ref{2-component_K}. 
			\item[(ii)] $J_2/O(J_2) \cong SL_i^{\varepsilon}(q^{*})/O(SL_i^{\varepsilon}(q^{*}))$, where $\varepsilon$ and $q^{*}$ are as in Proposition \ref{2-component_K}.
			\item[(iii)] Any elementary abelian $2$-subgroup of $J_1$ of rank at least $2$ is $G$-conjugate to a subgroup of $S$ lying in $\mathcal{S}_{n-i}$, and any elementary abelian $2$-subgroup of $J_2$ of rank at least $2$ is $G$-conjugate to a subgroup of $S$ lying in $\mathcal{S}_i$. 
		\end{enumerate} 
	\end{proposition}
	
	\begin{proof}
		Let $\mathcal{F} := \mathcal{F}_S(G) = \mathcal{F}_S(PSL_n(q))$. It suffices to prove the proposition under the assumption that $\langle x \rangle$ is fully $\mathcal{F}$-centralized, and we will assume that this is the case. So we have $C_S(x) \in \mathrm{Syl}_2(C_G(x))$ and $C_S(x) \in \mathrm{Syl}_2(C_{PSL_n(q)}(x))$. Also, $\mathcal{F}_{C_S(x)}(C_G(x)) = C_{\mathcal{F}}(\langle x \rangle) = \mathcal{F}_{C_S(x)}(C_{PSL_n(q)}(x))$.
		
		Clearly, $x$ is $PSL_n(q)$-conjugate to $t_i$. So Lemma \ref{components of centralizer fusion systems} (i) shows together with Lemma \ref{key lemma} (i) that there exist two distinct $2$-components $J_1$ and $J_2$ of $C_G(x)$ satisfying the following conditions, where $Y_1 := C_S(x) \cap J_1$ and $Y_2 := C_S(x) \cap J_2$. 
		\begin{enumerate}
			\item[(1)] $\mathcal{F}_{Y_1}(J_1)$ is isomorphic to the $2$-fusion system of $SL_{n-i}(q)$.
			\item[(2)] $\mathcal{F}_{Y_2}(J_2)$ is isomorphic to the $2$-fusion system of $SL_i(q)$.
			\item[(3)] $Y_1Y_2$ is normal in $C_S(x)$, and $C_{\mathcal{F}}(\langle x \rangle)/Y_1Y_2$ is nilpotent.
			\item[(4)] For $k \in \lbrace 1,2 \rbrace$, $Y_k$ contains a $G$-conjugate of $t$. 
			\item[(5)] Any elementary abelian abelian subgroup of $Y_1$ of rank at least $2$ lies in $\mathcal{S}_{n-i}$, and any elementary abelian subgroup of $Y_2$ of rank at least $2$ lies in $\mathcal{S}_i$.
		\end{enumerate} 
		By (3) and Corollary \ref{corollary_to_show_that_2-components_are_all_2-components}, $J_1$ and $J_2$ are the only $2$-components of $C_G(x)$. It remains to show that $J_1$ and $J_2$ satisfy (i)-(iii). As $Y_k \in \mathrm{Syl}_2(J_k)$ for $k \in \lbrace 1,2 \rbrace$, (5) implies (iii).
		
		We now prove (ii). The proof of (i) will be omitted since it is very similar to the proof of (ii). 
		
		Let $s$ be an element of $J_1$ which is $G$-conjugate to $t$. Set $C := C_G(s)$, $\widehat C := C/O(C)$ and $\widebar{C_G(x)} := C_G(x)/O(C_G(x))$.
		
		Since $\widebar{J_1}$ and $\widebar{J_2}$ are distinct components of $\widebar{C_G(x)}$, we have $[\widebar{J_1},\widebar{J_2}] = 1$ by \cite[6.5.3]{KurzweilStellmacher}. As $\widebar s \in \widebar{J_1}$, it follows that $\widebar{J_2}$ is a component of $C_{\widebar{C_G(x)}}(\widebar s)$. As a consequence of Corollary \ref{centralizers_p-elements} and Proposition \ref{2-components modulo odd order subgroup}, $C_G(x) \cap C$ has a $2$-component $H$ with $\widebar{H} = \widebar{J_2}$. 
		
		By assumption, $s$ is $G$-conjugate to $t$. So, by Proposition \ref{construction of L}, $\widehat{C}$ has a unique normal subgroup $K^{+}$ isomorphic to $SL_{n-2}^{\varepsilon}(q^{*})/O(SL_{n-2}^{\varepsilon}(q^{*}))$ and a unique normal subgroup $L^{+}$ isomorphic to $SL_2(q^{*})$. Moreover, $K^{+}$ and $L^{+}$ are the only subgroups of $\widehat C$ which are components or solvable $2$-components of $\widehat C$.
		
		Clearly, $\widehat{H}$ is a $2$-component of $C_{\widehat{C}}(\widehat x)$. Lemma \ref{GW 2.18} implies that $\widehat H$ is a $2$-component of $C_{K^{+}}(\widehat x)$ or of $C_{L^{+}}(\widehat x)$. By Corollary \ref{q q star corollary} (i), we even have that $\widehat H$ is a component of $C_{K^{+}}(\widehat x)$ or $C_{L^{+}}(\widehat x)$. It is easy to note that $\widehat{H}/Z(\widehat{H}) \cong H/Z^{*}(H) \cong \widebar{J_2}/Z(\widebar{J_2})$. By Corollary \ref{q q star corollary} (ii), we have $\widehat{H}/Z(\widehat H) \not\cong M_{11}$, and so $\widebar{J_2}/Z(\widebar{J_2}) \not\cong M_{11}$. Now (2) and Lemma \ref{key lemma} (ii) imply that $\widebar{J_2} \cong SL_i^{\varepsilon_0}(q_0)/O(SL_i^{\varepsilon_0}(q_0))$ for some nontrivial odd prime power $q_0$ and some $\varepsilon_0 \in \lbrace +,- \rbrace$ with $q \sim \varepsilon_0 q_0$. Hence $\widehat{H}/Z(\widehat H) \cong \widebar{J_2}/Z(\widebar{J_2}) \cong PSL_{i}^{\varepsilon_0}(q_0)$. Note that $\varepsilon q^{*} \sim q \sim \varepsilon_0 q_0$ and in particular $({q^{*}}^2-1)_2 = ({q_0}^2-1)_2$. Applying Corollary \ref{q q star corollary} (iii), we may conclude that $q_0 = q^{*}$ and $\varepsilon_0 = \varepsilon$. Consequently, we have $J_2/O(J_2) \cong SL_i^{\varepsilon}(q^{*})/O(SL_i^{\varepsilon}(q^{*}))$. So we have proved (ii).  
	\end{proof} 
	
	The proof of the following proposition runs along the same lines as that of the previous result.
	
	\begin{proposition}
		\label{2-components_ti_2}
		Suppose that $n$ is odd and $i = n-1$, or that $n$ is even, $i = 1$ and $t_1$ is defined. Let $x$ be an involution of $S$ which is $G$-conjugate to $t_i$. Then $C_G(x)$ has precisely one $2$-component $J$. We have $J/O(J) \cong SL_{n-1}^{\varepsilon}(q^{*})/O(SL_{n-1}^{\varepsilon}(q^{*}))$, where $\varepsilon$ and $q^{*}$ are as in Proposition \ref{2-component_K}. Moreover, any elementary abelian $2$-subgroup of $J$ of rank at least $2$ is $G$-conjugate to a subgroup of $S$ lying in $\mathcal{S}_{n-1}$.
	\end{proposition}
	
	\begin{proof}
		Let $\mathcal{F} := \mathcal{F}_S(G) = \mathcal{F}_S(PSL_n(q))$. It suffices to prove the proposition under the assumption that $\langle x \rangle$ is fully $\mathcal{F}$-centralized, and we will assume that this is the case. So we have $C_S(x) \in \mathrm{Syl}_2(C_G(x))$ and $C_S(x) \in \mathrm{Syl}_2(C_{PSL_n(q)}(x))$. Also, $\mathcal{F}_{C_S(x)}(C_G(x)) = C_{\mathcal{F}}(\langle x \rangle) = \mathcal{F}_{C_S(x)}(C_{PSL_n(q)}(x))$.
		
		Clearly, $x$ is $PSL_n(q)$-conjugate to $t_i$. Lemma \ref{components of centralizer fusion systems} (ii) implies that $C_{\mathcal{F}}(\langle x \rangle)$ has a unique component $\mathcal{E}$, and that $\mathcal{E}$ is isomorphic to the $2$-fusion system of $SL_{n-1}(q)$. Applying Lemma \ref{key lemma} (i), we may conclude that $C_G(x)$ has a unique $2$-component $J$ with $\mathcal{E} = \mathcal{F}_{C_S(x) \cap J}(J)$. By Lemma \ref{key lemma} (ii), $J/O(J) \cong SL_{n-1}^{\varepsilon_0}(q_0)/O(SL_{n-1}^{\varepsilon_0}(q_0))$ for some nontrivial odd prime power $q_0$ and some $\varepsilon_0 \in \lbrace +,- \rbrace$ with $\varepsilon_0 q_0 \sim q$. 
		
		Set $Y := C_S(x) \cap J$. By Lemma \ref{components of centralizer fusion systems} (ii), $Y \trianglelefteq C_S(x)$ and $C_{\mathcal{F}}(\langle x \rangle)/Y$ is nilpotent. Applying Corollary \ref{corollary_to_show_that_2-components_are_all_2-components}, we may conclude that $J$ is the only $2$-component of $C_G(x)$. Using Lemma \ref{components of centralizer fusion systems} (ii), we see that any elementary abelian subgroup of $Y$ of rank at least $2$ lies in $\mathcal{S}_{n-1}$. As $Y \in \mathrm{Syl}_2(J)$, it follows that any elementary abelian $2$-subgroup of $J$ of rank at least $2$ is $G$-conjugate to a subgroup of $S$ lying in $\mathcal{S}_{n-1}$. 
		
		It remains to show that $\varepsilon_0 = \varepsilon$ and $q_0 = q^{*}$. Define $s := t_i$ if $i = 1$ and $s := t_A$, where $A := \lbrace 1, \dots, n-1 \rbrace$, if $i = n-1$. Then we have $s \in C_G(t)$, and $s$ is $G$-conjugate to $x$. Set $\widebar{C_G(t)} := C_G(t)/O(C_G(t))$. Lemma \ref{centralizer_K} shows that $\widebar s$ centralizes $\widebar K$. Hence, $\widebar K$ is a component of $C_{\widebar{C_G(t)}}(\widebar s)$. As a consequence of Corollary \ref{centralizers_p-elements} and Proposition \ref{2-components modulo odd order subgroup}, $C_G(t) \cap C_G(s)$ has a $2$-component $H$ with $\widebar{H} = \widebar K$. Set $C := C_G(s)$ and $\widehat C := C/O(C)$. Then $\widehat H$ is a $2$-component of $C_{\widehat C}(\widehat t)$. Since $s$ is $G$-conjugate to $x$, $\widehat C$ has precisely one component $J^{+}$, and $J^{+}$ is isomorphic to $SL_{n-1}^{\varepsilon_0}(q_0)/O(SL_{n-1}^{\varepsilon_0}(q_0))$. By Lemma \ref{GW 2.18}, $\widehat H$ is a $2$-component of $C_{J^{+}}(\widehat t)$. We see from Corollary \ref{q q star corollary} (i) that $\widehat H$ is in fact a component of $C_{J^{+}}(\widehat t)$. It is easy to see that $\widehat H/Z(\widehat H) \cong H/Z^{*}(H) \cong \widebar{K}/Z(\widebar K) \cong PSL_{n-2}^{\varepsilon}(q^{*})$. Note that $\varepsilon_0 q_0 \sim q \sim \varepsilon q^{*}$ and in particular $({q_0}^2-1)_2 = ({q^{*}}^2-1)_2$. Using this, we may deduce from Corollary \ref{q q star corollary} (iii) that $q_0 = q^{*}$ and $\varepsilon_0 = \varepsilon$.
	\end{proof}
	
	\subsection{$2$-components of centralizers of involutions conjugate to $w$}
	Recall that we assume $\rho$ to be an element of $\mathbb{F}_q^{*}$ with order $(n,q-1)$. Recall moreover that if $n$ is even and $\rho$ is a non-square element of $\mathbb{F}_q$, then $\widetilde w$ denotes the matrix
	\begin{equation*}
		\begin{pmatrix} & I_{n/2} \\ \rho I_{n/2} & \end{pmatrix}
	\end{equation*} 
	and, if $\widetilde w \in SL_n(q)$, then $w$ denotes its image in $PSL_n(q)$. 
	
	\begin{lemma}
		\label{components of centralizer fusion system of w} 
		Suppose that $w$ is defined. Let $P$ be a Sylow $2$-subgroup of $C_{PSL_n(q)}(w)$, and let $\mathcal{F}$ denote the fusion system $\mathcal{F}_P(C_{PSL_n(q)})(w))$. Then $\mathcal{F}$ has precisely one component. This component is isomorphic to the $2$-fusion system of a nontrivial quotient of $SL_{\frac{n}{2}}(q^2)$. If $Y$ is its Sylow subgroup, then we have $Y \trianglelefteq P$, and $\mathcal{F}/Y$ is nilpotent. 
	\end{lemma}
	
	\begin{proof}
		By Lemma \ref{centralizer of w in PSL(n,q)} (i), $C_{PSL_n(q)}(w)$ has precisely one $2$-component $J$, and $J$ is isomorphic to a nontrivial quotient of $SL_{\frac{n}{2}}(q^2)$. Applying Proposition \ref{subsystems induced by 2-components} and Lemma \ref{PSL as Goldschmidt group}, we may conclude that $\mathcal{F}_{P \cap J}(J)$ is the only component of $\mathcal{F}$. The last statement of the lemma is given by Lemma \ref{centralizer of w in PSL(n,q)} (ii).
	\end{proof}
	
	\begin{proposition}
		\label{2-components w}
		Suppose that $w$ is defined. Let $x$ be an involution of $S$ which is $PSL_n(q)$-conjugate to $w$. Then $C_G(x)$ has precisely one $2$-component, say $J$. The group $J/O(J)$ is isomorphic to a nontrivial quotient of $SL_{\frac{n}{2}}^{\varepsilon_0}(q_0)$ for some nontrivial odd prime power $q_0$ and some $\varepsilon_0 \in \lbrace +,-\rbrace$ with $q^2 \sim \varepsilon_0 q_0$.
	\end{proposition} 
	
	\begin{proof}
		Let $\mathcal{F} := \mathcal{F}_S(G) = \mathcal{F}_S(PSL_n(q))$. It suffices to prove the proposition under the assumption that $\langle x \rangle$ is fully $\mathcal{F}$-centralized, and we will assume that this is the case. So we have $C_S(x) \in \mathrm{Syl}_2(C_G(x))$ and $C_S(x) \in \mathrm{Syl}_2(C_{PSL_n(q)}(x))$. Also, $\mathcal{F}_{C_S(x)}(C_G(x)) = C_{\mathcal{F}}(\langle x \rangle) = \mathcal{F}_{C_S(x)}(C_{PSL_n(q)}(x))$.

		As $x$ is $PSL_n(q)$-conjugate to $w$, Lemma \ref{components of centralizer fusion system of w} implies that $C_{\mathcal{F}}(\langle x \rangle)$ has precisely one component, say $\mathcal{E}$, and that $\mathcal{E}$ is isomorphic to the $2$-fusion system of a nontrivial quotient of $SL_{\frac{n}{2}}(q^2)$. By Lemma \ref{key lemma} (i), $C_G(x)$ has a unique $2$-component $J$ such that $\mathcal{E} = \mathcal{F}_{C_S(x) \cap J}(J)$. Set $Y := C_S(x) \cap J$. As a consequence of Lemma \ref{components of centralizer fusion system of w}, we have $Y \trianglelefteq C_S(x)$, and the factor system $C_{\mathcal{F}}(\langle x \rangle)/Y$ is nilpotent. So, by Corollary \ref{corollary_to_show_that_2-components_are_all_2-components}, $J$ is the only $2$-component of $C_G(x)$. Lemma \ref{key lemma} (iii) shows that $J/O(J)$ is isomorphic to a nontrivial quotient of $SL_{\frac{n}{2}}^{\varepsilon_0}(q_0)$ for some nontrivial odd prime power $q_0$ and some $\varepsilon_0 \in \lbrace +,- \rbrace$ with $q^2 \sim \varepsilon_0 q_0$. 
	\end{proof}
	
	\section{The components of $C_G(t)$} 
	\label{components_of_centralizers}
	The goal of this section is to determine the isomorphism types of $K$ and $L$. In order to do so, we will apply the signalizer functor techniques introduced by Gorenstein and Walter in \cite{GW}. In particular, we will see that $L$ is isomorphic to $SL_2(q^{*})$. This will enable us in Section \ref{G0} to prove that a certain collection of conjugates of $L$ generates a subgroup $G_0$ of $G$ which is isomorphic to a nontrivial quotient of $SL_n^{\varepsilon}(q^{*})$ and normal in $G$. This will complete the proof of Theorem \ref{theorem1}. 
	
	\subsection{$3$-generation of involution centralizers}
	For each $3 \le i \le n$, we define $\mathcal{U}_i$ to be the set of all subgroups $U$ of $PSL_n(q)$ such that $U$ has a subgroup $E$ with $E \in \mathcal{S}_i$ and $m(E) \ge 3$. The following lemma will be important later in this section. 
	
	\begin{lemma}
		\label{3-generation of involution centralizers} 
		Let $1 \le i < n$ such that $t_i$ is defined. Suppose that $i \le \frac{n}{2}$ if $n$ is even. Let $x$ be an involution of $S$ such that $x$ is $G$-conjugate to $t_i$ and such that $\langle x \rangle$ is fully $\mathcal{F}_S(G)$-centralized. Then $C_G(x)$ is $3$-generated in the sense of Definition \ref{def_k-generation}. Moreover, if $i \ge 4$, then we have 
		\begin{equation*}
			C_G(x) = \langle N_{C_G(x)}(U) \ \vert \ U \le C_S(x), U \in \mathcal{U}_i \rangle.
		\end{equation*} 
		If $i = 2$, then we have
		\begin{equation*}
			C_G(x) = \langle N_{C_G(x)}(U) \ \vert \ U \le C_S(x), U \in \mathcal{U}_{n-2} \rangle.
		\end{equation*}
	\end{lemma}
	
	\begin{proof}
		Set $C := C_G(x)$ and $\widebar C := C/O(C)$. Recall that $L_{2'}(C)$ denotes the subgroup of $C$ generated by the $2$-components of $C$ and that $L(\widebar C)$ denotes the product of all components of $\widebar C$. Clearly, $\widebar{L_{2'}(C)} = L(\widebar C)$.   
		
		First we consider the case $(n,i) \ne (6,3)$. Then, by Propositions \ref{2-component_K}, \ref{2-components_ti_1} and \ref{2-components_ti_2}, $C$ has a $2$-component $J$ such that $\widebar J \cong SL_k^{\varepsilon}(q^{*})/O(SL_k^{\varepsilon}(q^{*}))$ for some $k \ge 4$ and such that any elementary abelian subgroup of $Y := C_S(x) \cap J$ of rank at least $2$ lies in $\mathcal{S}_k$. If $i \ge 4$, then we may assume that $k = i$, and if $i = 2$, then $k = n-2$.  
		
		Clearly, $Y \in \mathrm{Syl}_2(J)$. By Lemma \ref{generation2}, we have that $\widebar J$ is $3$-generated. So we have 
		\begin{equation*}
			\widebar J = \langle N_{\widebar J}(\widebar U) \ \vert \ U \le Y, m(U) \ge 3 \rangle.
		\end{equation*} 
		Set $X := C_S(x) \cap L_{2'}(C)$. By the Frattini argument, $L(\widebar C) = \widebar J N_{L(\widebar C)}(\widebar Y)$ and $\widebar C = L(\widebar C) N_{\widebar C}(\widebar X)$. It follows that 
		\begin{equation*}
			\widebar C = \langle N_{\widebar C}(\widebar U) \ \vert \ \textnormal{$U = X$, or $U \le Y$ and $m(U) \ge 3$} \rangle. 
		\end{equation*} 
		Lemma \ref{normalizers_p-subgroups} implies that $C$ is generated by $O(C)$ together with the normalizers $N_C(U)$, where $U = X$, or $U \le Y$ and $m(U) \ge 3$.
		
		Let $E$ denote the subgroup of $S$ generated by $t$, $t_{\lbrace n-2,n-1 \rbrace}$, $t_{\lbrace n-3,n-2 \rbrace}$ and $t_{\lbrace n-4,n-3 \rbrace}$. Clearly, $E \cong E_{16}$. Since $x$ is $G$-conjugate to $t_i$ and $E \le C_G(t_i)$, there is a subgroup $E_x$ of $C_S(x)$ which is $G$-conjugate to $E$. By \cite[Proposition 11.23]{GLS2}, we have 
		\begin{equation*}
			O(C) = \langle C_{O(C)}(D) \ \vert \ D \le E_x, D \cong E_8 \rangle. 
		\end{equation*}
		As remarked above, any elementary abelian subgroup of $Y$ of rank at least $2$ lies in $\mathcal{S}_k$. So, if $U \le Y$ and $m(U) \ge 3$, then $U \in \mathcal{U}_k$. Also $X \in \mathcal{U}_k$. Clearly, any $E_8$-subgroup of $E_x$ lies in $\mathcal{S}_k$ and hence in $\mathcal{U}_k$. We therefore have 
		\begin{equation*}
			C = \langle N_C(U) \ \vert \ U \le C_S(x), U \in \mathcal{U}_k \rangle.
		\end{equation*}
		Consequently, $C$ is $3$-generated, and the last two statements of the lemma are satisfied.
		
		Suppose now that $(n,i) = (6,3)$. By Proposition \ref{2-components_ti_1}, $C$ has precisely two $2$-components $J_1$ and $J_2$, and we have $\widebar{J_1} \cong PSL_3^{\varepsilon}(q^{*}) \cong \widebar{J_2}$. Set $Y_1 := C_S(x) \cap J_1$ and $Y_2 := C_S(x) \cap J_2$. Since $\widebar{J_1}$ is $2$-generated by Lemma \ref{generation1}, we have
		\begin{equation*}
			\widebar{J_1} = \langle N_{\widebar{J_1}}(\widebar U) \ \vert \ U \le Y_1, m(U) \ge 2 \rangle.
		\end{equation*} 
		Let $y$ be an involution of $Y_2$. We have $[\widebar{J_1}, \widebar{J_2}] = 1$ by \cite[6.5.3]{KurzweilStellmacher}, and so $\widebar y$ centralizes $\widebar{J_1}$. As $Z(\widebar{J_1}) = 1$, we have $\widebar y \not\in \widebar{J_1}$. Now let $U \le Y_1$ with $m(U) \ge 2$. Then $\langle \widebar U, \widebar y \rangle$ has rank at least $3$. Moreover, it is clear that $N_{\widebar{J_1}}(\widebar U)$ normalizes $\langle \widebar U, \widebar y \rangle$. Thus 
		\begin{equation*}
			\widebar{J_1} = \langle N_{\widebar{J_1}}(\widebar U) \ \vert \ U \le Y_1Y_2, m(U) \ge 3 \rangle.
		\end{equation*} 
		Interchanging the roles of $J_1$ and $J_2$, we also see that 
		\begin{equation*}
			\widebar{J_2} = \langle N_{\widebar{J_2}}(\widebar U) \ \vert \ U \le Y_1Y_2, m(U) \ge 3 \rangle.
		\end{equation*}
		By the Frattini argument, $\widebar{C} = \widebar{J_1}\widebar{J_2}N_{\widebar C}(\widebar{Y_1} \widebar{Y_2})$. Lemma \ref{normalizers_p-subgroups} implies that $C$ is generated by $O(C)$ together with the normalizers $N_C(U)$, where $U \le Y_1Y_2$ and $m(U) \ge 3$. For any $E_{16}$-subgroup $A$ of $C_S(x)$, we have
		\begin{equation*}
			O(C) = \langle C_{O(C)}(B) \ \vert \ B \le A, B \cong E_8 \rangle.
		\end{equation*}  
		by \cite[Proposition 11.23]{GLS2}. It follows that $C$ is $3$-generated. The proof is now complete. 
	\end{proof}
	
	\begin{lemma}
		\label{3-generation_w} 
		Suppose that $w$ is defined. Let $x$ be an involution of $S$ which is $PSL_n(q)$-conjugate to $w$. Then $C_G(x)$ is $3$-generated.
	\end{lemma}
	
	\begin{proof}
		Set $C := C_G(x)$ and $\widebar C := C/O(C)$. By Proposition \ref{2-components w}, $C$ has a unique $2$-component $J$, and $\widebar J$ is isomorphic to a nontrivial quotient of $SL_{\frac{n}{2}}^{\varepsilon_0}(q_0)$ for some nontrivial odd prime power $q_0$ and some $\varepsilon_0 \in \lbrace +,-\rbrace$ with $q^2 \sim \varepsilon_0 q_0$. Note that $q_0 \equiv \varepsilon_0 \mod 8$.  
		
		First we prove that $\widebar C$ is $3$-generated. Let $R$ be a Sylow $2$-subgroup of $C$ and $Y := R \cap J$. We consider two cases. 
		
		\medskip
		
		\textit{Case 1: $n \ge 8$.}
		
		By Lemma \ref{generation2}, $\widebar J$ is $3$-generated. Hence 
		\begin{equation*}
			\widebar J = \langle N_{\widebar J}(\widebar U) \ \vert \ U \le Y, m(U) \ge 3 \rangle.
		\end{equation*}
		By the Frattini argument, $\widebar C = \widebar J N_{\widebar C}(\widebar Y)$. So $\widebar C$ is $3$-generated.
		
		\medskip
		
		\textit{Case 2: $n = 6$.}
		
		We have $\widebar J \cong PSL_3^{\varepsilon_0}(q_0)$. By Lemma \ref{generation1}, $\widebar J$ is $2$-generated. Applying the Frattini argument, we may conclude that 
		\begin{equation*}
			\widebar C = \langle N_{\widebar C}(\widebar U) \ \vert \ U \le Y, m(U) \ge 2 \rangle. 
		\end{equation*}
		Now let $U \le Y$ with $m(U) \ge 2$. Since $\widebar x$ is a central involution of $\widebar C$ and $Z(\widebar J)$ is trivial, we have $\widebar x \not\in \widebar J$ and hence $\widebar x \not\in \widebar U$. It follows $\langle \widebar U, \widebar x \rangle$ has rank at least $3$. Moreover, as $\widebar x$ is central in $\widebar C$, we have $N_{\widebar C}(\widebar U) \le N_{\widebar C}(\langle \widebar U, \widebar x \rangle)$. Clearly, $\langle \widebar U, \widebar x \rangle \le \widebar R$. It follows that
		\begin{equation*}
			\widebar C = \langle N_{\widebar C}(\widebar U) \ \vert \ U \le R, m(U) \ge 3 \rangle.
		\end{equation*}
		Hence $\widebar C$ is $3$-generated. 
		
		\medskip
		
		Applying Lemma \ref{normalizers_p-subgroups}, we may conclude that $C$ is generated by $O(C)$ together with the normalizers $N_C(U)$, where $U \le R$ and $m(U) \ge 3$. By Lemma \ref{centralizer of w in PSL(n,q)} (iii), $R$ has an elementary abelian $2$-subgroup of rank $4$, say $A$. By \cite[Proposition 11.23]{GLS2}, we have 
		\begin{equation*}
			O(C) = \langle C_{O(C)}(B) \ \vert \ B \le A, B \cong E_8 \rangle. 
		\end{equation*} 
		So $C$ is $3$-generated. 
	\end{proof} 
	
	\begin{corollary}
		\label{3-generation_conclusion}
		Let $x$ be an involution of $S$. Then $C_G(x)$ is $3$-generated.
	\end{corollary}
	
	\begin{proof}
		As a consequence of Proposition \ref{involutions of PSL(n,q)}, $x$ is $G$-conjugate to $t_i$ for some $1 \le i < n$ such that $t_i$ is defined or $PSL_n(q)$-conjugate to $w$ (if defined). So the statement follows from Lemmas \ref{3-generation of involution centralizers} and \ref{3-generation_w}.
	\end{proof}
	
	\subsection{The case $q^{*} = 3$} 
	Recall that our goal is to determine the isomorphism types of $K$ and $L$. First we will deal with the case $q^{*} = 3$. We will prove that, in this case, $O(C_G(t)) = 1$. 
	
	\begin{lemma}
		\label{locally balanced 2-components} 
		Let $x$ be an involution of $S$, and let $J$ be a $2$-component of $C_G(x)$. Let $1 \le i < n$ such that $t_i$ is defined. Suppose that $q^{*} = 3$ and that $x$ is $G$-conjugate to $t_i$. Then $J/O(J)$ is locally balanced. 
	\end{lemma}
	
	\begin{proof}
		By Propositions \ref{construction of L} (iii), \ref{2-components_ti_1} and \ref{2-components_ti_2}, we have $J/O(J) \cong SL_k^{\varepsilon}(3)$ for some $3 \le k < n$. So $J/O(J)$ is locally balanced by Lemma \ref{local balance q=3}. 
	\end{proof}
	
	\begin{lemma}
		\label{existence sequence}
		Let $P$ and $Q$ be subgroups of $S$.
		\begin{enumerate}
			\item[(i)] If $P \in \mathcal{S}$ and $m(P) \le 2$, then there is a subgroup $\widebar P$ of $S$ such that $P < \widebar P$, $\widebar P \in \mathcal{S}$ and $m(\widebar P) = 3$.
			\item[(ii)] If $P$ and $Q$ are elements of $\mathcal{S}$ of rank at least $3$, then there exist some $m \ge 1$ and a sequence
			\begin{equation*}
				P = P_1, \dots, P_m = Q,
			\end{equation*}
			where $P_i$, $1 \le i \le m$, is a subgroup of $S$ of rank at least $2$ lying in $\mathcal{S}$ and where 
			\begin{equation*}
				P_i \subseteq P_{i+1} \ \textnormal{or} \ P_{i+1} \subseteq P_i
			\end{equation*} 
			for all $1 \le i < m$. 
		\end{enumerate}
	\end{lemma}
	
	\begin{proof}
		Suppose that $P \in \mathcal{S}$ and $m(P) \le 2$. Let $\widetilde S$ be a Sylow $2$-subgroup of $SL_n(q)$ such that $S$ is the image of $\widetilde S$ in $PSL_n(q)$. Note that $\widetilde S$ is unique. Since $P$ is an element of $\mathcal{S}$, there exists some elementary abelian $2$-subgroup $\widetilde P$ of $SL_n(q)$ such that $P$ is the image of $\widetilde P$ in $PSL_n(q)$. Clearly, $\widetilde P \le \widetilde S$. We have $m(\widetilde P) \le 3$ as $m(P) \le 2$. By Corollary \ref{conclusion_connectivity_2}, $\widetilde P$ is contained in an $E_{16}$-subgroup of $\widetilde S$. This implies (i). 
		
		We now prove (ii). Suppose that $P$ and $Q$ are elements of $\mathcal{S}$ of rank at least $3$. There are elementary abelian subgroups $\widetilde P$ and $\widetilde Q$ of $SL_n(q)$ such that $P$ is the image of $\widetilde P$ in $PSL_n(q)$ and such that $Q$ is the image of $\widetilde Q$ in $PSL_n(q)$. Clearly, $\widetilde P, \widetilde Q \le \widetilde S$. Also $m(\widetilde P), m(\widetilde Q) \ge 3$. Since $\widetilde S$ is $3$-connected by Corollary \ref{conclusion connectivity}, there exist some $m \ge 1$ and a sequence 
		\begin{equation*}
			\widetilde P = \widetilde{P}_1, \dots, \widetilde{P}_n = \widetilde Q,
		\end{equation*} 
		where $\widetilde{P}_i$ ($1 \le i \le m$) is an elementary abelian subgroup of $\widetilde S$ of rank at least $3$ and where 
		\begin{equation*}
			\widetilde{P}_i \subseteq \widetilde{P}_{i+1} \ \textnormal{or} \ \widetilde{P}_{i+1} \subseteq \widetilde{P}_i
		\end{equation*}
		for all $1 \le i < m$. Let $P_i$, $1 \le i \le m$, denote the image of $\widetilde{P}_i$ in $S$. Then the sequence 
		\begin{equation*}
			P = P_1, \dots, P_m = Q
		\end{equation*}
		has the desired properties. 
	\end{proof}
	
	\begin{lemma}
		\label{W_equal} 
		Suppose that $q^{*} = 3$. For each elementary abelian subgroup $E$ of $S$ of rank at least $2$, let 
		\begin{equation*}
			W_E := \langle O(C_G(x)) \ \vert \ x \in E^{\#} \rangle. 
		\end{equation*}  
		Let $P$ and $Q$ be subgroups of $S$ with $P, Q \in \mathcal{S}$ and $m(P), m(Q) \ge 3$. Then $W_P = W_Q$.
	\end{lemma}
	
	\begin{proof}
		By Lemma \ref{existence sequence} (ii), there exist some $m \ge 1$ and a sequence 
		\begin{equation*}
			P = P_1, \dots, P_m = Q,
		\end{equation*}
		where $P_i$, $1 \le i \le m$, is a subgroup of $S$ of rank at least $2$ lying in $\mathcal{S}$ and where 
		\begin{equation*}
			P_i \subseteq P_{i+1} \ \textnormal{or} \ P_{i+1} \subseteq P_i
		\end{equation*} 
		for all $1 \le i < m$. By Lemma \ref{existence sequence} (i), there is a subgroup $\widebar{P_i}$ of $S$ such that $\widebar{P_i} \in \mathcal{S}$, $m(\widebar{P_i}) \ge 3$ and $P_i \le \widebar{P_i}$ for each $1 \le i \le m$. 
		
		Let $1 \le i \le m$ and let $x$ be an involution of $\widebar{P_i}$. Also let $J$ be a $2$-component of $C_G(x)$. As $\widebar{P_i} \in \mathcal{S}$, we have that $x$ is $G$-conjugate to $t_j$ for some even $2 \le j < n$. Therefore, by Lemma \ref{locally balanced 2-components}, $J/O(J)$ is locally balanced. Applying \cite[Corollary 5.6]{GW}, we may conclude that $G$ is balanced with respect to $\widebar{P_i}$.  
		
		Let $1 \le i < m$. We have $m(P_i \cap P_{i+1}) \ge 2$ since $P_i \subseteq P_{i+1}$ or $P_{i+1} \subseteq P_i$ and $m(P_i), m(P_{i+1}) \ge 2$. Hence $m(\widebar{P_i} \cap \widebar{P_{i+1}}) \ge 2$. Proposition \ref{GW 6.10} (ii) implies 
		\begin{equation*}
			W_{P_i} = W_{\widebar{P_i}} = W_{\widebar{P_i} \cap \widebar{P_{i+1}}} = W_{\widebar{P_{i+1}}} = W_{P_{i+1}}. 
		\end{equation*} 
		Consequently, $W_P = W_Q$, as wanted. 
	\end{proof}
	
	\begin{proposition}
		\label{case_q*_3}
		Suppose that $q^{*} = 3$. Let $x$ be an involution of $S$ which is $G$-conjugate to $t_i$ for some even $2 \le i < n$. Then we have $O(C_G(x)) = 1$. In particular, $O(C_G(t)) = 1$.
	\end{proposition}
	
	\begin{proof}
		We follow the pattern of the proof of \cite[Theorem 9.1]{GW}. Let $E$ be the subgroup of $S$ consisting of all $t_A$, where $A \subseteq \lbrace 1, \dots, n \rbrace$ has even order. For each elementary abelian $2$-subgroup $A$ of $G$ of rank at least $2$, let
		\begin{equation*}
			W_A := \langle O(C_G(y)) \ \vert \ y \in A^{\#} \rangle. 
		\end{equation*}
		Set $W_0 := W_E$ and $M := N_G(W_0)$. We accomplish the proof step by step.  
		
		\medskip 
		
		(1) $N_G(S) \le M$.  
		
		Let $g \in N_G(S)$. Clearly, $E \in \mathcal{S}$, and it is easy to note $E^g$ still lies in $\mathcal{S}$. Lemma \ref{W_equal} implies that $W_0 = W_{E^g}$. On the other hand, we have $(W_0)^g = W_{E^g}$ by Proposition \ref{GW 6.10} (i). So we have $(W_0)^g = W_0$ and hence $g \in M$.  
		
		\medskip
		
		(2) \textit{Let $y$ be an involution of $S$ such that $y$ is $G$-conjugate to $t_j$ for some even $2 \le j < n$. Then $y$ is $M$-conjugate to $t_j$.}
		
		We have $\langle y \rangle \in \mathcal{S}$. By Lemma \ref{existence sequence} (i), there is a subgroup $A$ of $S$ with $\langle y \rangle \le A$, $A \in \mathcal{S}$ and $m(A) = 3$. As a consequence of Lemma \ref{elementary abelian subgroups of SL(n,q)}, there is an element $g$ of $G$ with $A^g \le E$. By Lemma \ref{W_equal} and Proposition \ref{GW 6.10} (i), we have $(W_0)^g = (W_A)^g = W_{A^g} = W_0$. Thus $g \in M$. 
		
		We have $y^g \in E$, and $y^g$ is $G$-conjugate and hence $PSL_n(q)$-conjugate to $t_j$. It is rather easy to show that an element of $E$ is $N_{PSL_n(q)}(E)$-conjugate to $t_j$ if it is $PSL_n(q)$-conjugate to $t_j$. So $y^g$ is $N_{PSL_n(q)}(E)$-conjugate and hence $N_G(E)$-conjugate to $t_j$. As $N_G(E) \le M$, it follows that $y^g$ is $M$-conjugate to $t_j$. Hence $y$ is $M$-conjugate to $t_j$. 
		
		\medskip
		
		(3) \textit{Let $y$ be an involution of $S$ such that $y$ is $G$-conjugate to $t_j$ for some even $2 \le j < n$. Then $C_G(y) \le M$.} 
		
		Because of (2), we may assume that $\langle y \rangle$ is fully $\mathcal{F}_S(G)$-centralized. Then, by Lemma \ref{3-generation of involution centralizers}, $C_G(y)$ is generated by the normalizers $N_{C_G(y)}(U)$, where $U$ is a subgroup of $C_S(y)$ such that there exists $B \le U$ with $B \in \mathcal{S}$ and $m(B) \ge 3$. It suffices to show that each such normalizer lies in $M$. 
		
		Let $U$ and $B$ be as above and let $g \in N_{C_G(y)}(U)$. By Lemma \ref{W_equal} and Proposition \ref{GW 6.10} (i), we have $(W_0)^g = (W_B)^g = W_{B^g} = W_0$. Thus $g \in M$ and hence $N_{C_G(y)}(U) \le M$, as required. 
		
		\medskip
		
		(4) \textit{Let $y$ be an involution of $S$. Then $C_G(y) \le M$.}
		
		We can see from Lemmas \ref{sylow_power_2} and \ref{sylow_general} that $Z(S)$ has an involution $s$ which is $G$-conjugate to $t_j$ for some even $2 \le j < n$. Let $P$ be a Sylow $2$-subgroup of $C_G(y)$ with $s \in P$. By (1), $s \in M$ and hence $s \in P \cap M$. Now let $r \in N_P(P \cap M)$. Then $s^r \in P \cap M$. As a consequence of (1) and (2), $s^r$ and $s$ are $M$-conjugate to $t_j$. Therefore, there is some $m \in M$ with $s^r = s^m$. We have $rm^{-1} \in C_G(s)$, and so $rm^{-1} \in M$ by (3). Hence $r \in M$. Consequently, $N_P(P \cap M) = P \cap M$. It follows that $P = P \cap M$. 
		
		Let $U \le P$ with $m(U) \ge 3$ and let $g \in N_{C_G(y)}(U)$. By Lemma \ref{E8 subgroups of central quotients}, any $E_8$-subgroup of $S$ has an involution which is the image of an involution of $SL_n(q)$. Since $m(U) \ge 3$, it follows that $U$ has an element $u$ which is $G$-conjugate to $t_k$ for some even $2 \le k < n$. By the preceding paragraph, $u, u^g \in U \le P \le M$. As a consequence of (1) and (2), $u$ and $u^g$ are $M$-conjugate to $t_k$. So there is some $m \in M$ with $u^g = u^m$. Hence $gm^{-1} \in C_G(u)$. From (3), we see that $C_G(u) \le M$, and so $gm^{-1} \in M$. Thus $g \in M$ and hence $N_{C_G(y)}(U) \le M$. Since $C_G(y)$ is $3$-generated by Corollary \ref{3-generation_conclusion}, it follows that $C_G(y) \le M$.
		
		\medskip
		
		(5) \textit{$M = G.$}
		
		Assume that $M \ne G$. By \cite[Proposition 17.11]{GLS2}, we may deduce from (1) and (4) that $M$ is strongly embedded in $G$, i.e. $M \cap M^g$ has odd order for any $g \in G \setminus M$. Applying \cite[Chapter 6, 4.4]{Suzuki}, it follows that $G$ has only one conjugacy class of involutions. On the other hand, we see from Proposition \ref{involutions of PSL(n,q)} that $G$ has at least two conjugacy classes of involutions. This contradiction shows that $M = G$.   
		
		\medskip
		
		(6) \textit{Conclusion.}
		
		Let $y \in E^{\#}$ and let $J$ be a $2$-component of $C_G(y)$. By Lemma \ref{locally balanced 2-components}, $J/O(J)$ is locally balanced. So, by \cite[Corollary 5.6]{GW}, $G$ is balanced with respect to $E$. Proposition \ref{GW 6.10} (ii) implies that $W_0$ has odd order. By (5), we have $M = G$ and hence $W_0 \trianglelefteq G$. As $O(G) = 1$ by Hypothesis \ref{hypothesis}, it follows that $W_0 = 1$. So we have $O(C_G(y)) = 1$ for all $y \in E^{\#}$, and the statement of the proposition follows. 
	\end{proof} 
	
	Proposition \ref{case_q*_3} implies that if $q^{*} = 3$, then $K \cong SL_{n-2}^{\varepsilon}(3)$ and $L \cong SL_2(3)$. Our next goal is to find the isomorphism types of $K$ and $L$ for the case $q^{*} \ne 3$. 
	
	In general, $O(C_{PSL_n(q)}(t))$ is not trivial. So, if $q^{*}$ is not assumed to be $3$, we have no chance to prove that $O(C_G(t)) = 1$. However, we will be able to show that 
	\begin{equation*}
		\Delta_G(F) = \bigcap_{a \in F^{\#}} O(C_G(a)) = 1
	\end{equation*}
	for any Klein four subgroup $F$ of $G$ consisting of elements of the form $t_A$, where $A \subseteq \lbrace 1, \dots, n \rbrace$ has even order. This will later enable us to determine the isomorphism types of $K$ and $L$ for the case $q^{*} \ne 3$. 
	
	\subsection{$2$-balance of $G$}
	In this subsection, we prove that $G$ is $2$-balanced when $q^{*} \ne 3$. 
	
	\begin{lemma}
		\label{Delta centralizes K} 
		Set $C := C_G(t)$ and $\widebar C := C/O(C)$. Let $F$ be a Klein four subgroup of $C$. Then $[\Delta_{\overline C}(\overline F), \widebar K] = 1$. 
	\end{lemma} 
	
	\begin{proof}
		We closely follow arguments found in the proof of \cite[Theorem 5.2]{GW}.
		
		First we consider the case that $F$ has a nontrivial element $y$ such that $\widebar y$ centralizes $\widebar K$. Then $\widebar K$ normalizes $O(C_{\widebar C}(\widebar y))$ and, as $\widebar K \trianglelefteq \widebar C$, $O(C_{\widebar C}(\widebar y))$ also normalizes $\widebar K$. It follows that
		\begin{equation*}
			[\widebar K,O(C_{\widebar C}(\widebar y))] \le \widebar K \cap O(C_{\widebar C}(\widebar y)).
		\end{equation*} 
		Hence, $[\widebar K,O(C_{\widebar C}(\widebar y))]$ is a subgroup of $\widebar K$ with odd order. By \cite[1.5.5]{KurzweilStellmacher}, $\widebar{K}$ normalizes $[\widebar K,O(C_{\widebar C}(\widebar y))]$. It follows that  
		\begin{equation*}
			[\widebar K,O(C_{\widebar C}(\widebar y))] \le O(\widebar K).
		\end{equation*} 
		As $O(\widebar K) = 1$, this implies that $O(C_{\widebar C}(\widebar y))$ centralizes $\widebar K$. By definition of $\Delta_{\widebar C}(\widebar F)$, we have $\Delta_{\widebar C}(\widebar F) \le O(C_{\widebar C}(\widebar y))$. Consequently, $\Delta_{\widebar C}(\widebar F)$ centralizes $\overline K$.  
		
		Now we treat the case that $C_{\widebar F}(\widebar K) = 1$. For each subgroup or element $X$ of $C$, let $\widehat X$ denote the image of $\widebar X$ in $\widebar{C}/C_{\widebar C}(\widebar K)$. Since $C_{\widebar F}(\widebar K) = 1$, we have $\widehat{F} \cong \widebar{F}$, and so $\widehat F$ is a Klein four subgroup of $\widehat C$. As $\widebar{K} \cong SL_{n-2}^{\varepsilon}(q^{*})/O(SL_{n-2}^{\varepsilon}(q^{*}))$, we have that $\widebar{K}$ is locally $2$-balanced (see Lemma \ref{local 2-balance of SL and SU}). Using this together with the fact that the group $\widehat C = \widebar{C}/C_{\widebar{C}}(\widebar K)$ is isomorphic to a subgroup of $\mathrm{Aut}(\widebar K)$ containing $\mathrm{Inn}(\widebar K)$, we may conclude that $\Delta_{\widehat C}(\widehat F) = 1$. By \cite[Proposition 3.11]{GW}, if $X$ is a finite group, $B$ a $2$-subgroup of $X$ and $N \trianglelefteq X$, then the image of $O(C_X(B))$ in $X/N$ lies in $O(C_{X/N}(BN/N))$. Thus, if $y$ is an involution of $F$, then the image of $O(C_{\widebar C}(\widebar y))$ in $\widehat C$ lies in $O(C_{\widehat C}(\widehat y))$. It follows that the image of $\Delta_{\widebar C}(\widebar F)$ in $\widehat C$ is contained in $\Delta_{\widehat C}(\widehat F) = 1$. Hence $\Delta_{\widebar C}(\widebar F) \le C_{\widebar{C}}(\widebar K)$. 
	\end{proof} 
	
	\begin{lemma}
		\label{center of C} 
		Let $C := C_G(t)$ and $\widebar C := C/O(C)$. Then $C_{\widebar C}(\widebar K) \cap C_{\widebar C}(\widebar L)$ is a $2$-group.
	\end{lemma} 
	
	\begin{proof}
		For convenience, we denote $C_{\widebar C}(\widebar K) \cap C_{\widebar C}(\widebar L)$ by $C_{\widebar C}(\widebar K, \widebar L)$. Since $\widebar C$ is core-free, we have that $C_{\widebar C}(\widebar K, \widebar L)$ is core-free. So it is enough to prove that $C_{\widebar C}(\widebar K, \widebar L)$ is $2$-nilpotent. By \cite[Theorem 1.4]{Linckelmann}, it suffices to show that $C_{\widebar C}(\widebar K, \widebar L)$ has a nilpotent $2$-fusion system.
		
		Let $X$ denote the subgroup of $T$ consisting of all elements of $T$ of the form
		\begin{equation*}
			\begin{pmatrix}A & \\ & B \end{pmatrix} Z(SL_n(q))
		\end{equation*} 
		with $A \in W \cap Z(GL_{n-2}(q))$, $B \in V \cap Z(GL_2(q))$ and $\mathrm{det}(A) \mathrm{det}(B) = 1$. 
		
		Let $A \in W$ and $B \in V$ with $\mathrm{det}(A)\mathrm{det}(B) = 1$ and 
		\begin{equation*}
			m := \begin{pmatrix}A & \\ & B \end{pmatrix} Z(SL_n(q)) \in T.  
		\end{equation*} 
		Assume that $\widebar m$ centralizes $\widebar K$ and $\widebar L$. Then we have $A \in Z(GL_{n-2}(q))$ by Lemma \ref{centralizer_K}. Since $\widebar m$ centralizes $\widebar L$, $\widebar m$ also centralizes $\widebar{X_2}$. Thus $m$ centralizes $X_2$, and so $B$ centralizes $V \cap SL_2(q)$. Lemma \ref{centralizer_sylow_SL} implies that $B \in Z(GL_2(q))$. So we have $m \in X$. Conversely, if $A \in Z(GL_{n-2}(q))$ and $B \in Z(GL_2(q))$, then $\widebar{m} \in C_{\widebar C}(\widebar K, \widebar L)$ as a consequence of Lemmas \ref{centralizer_K} and \ref{involutions of SL_2(q) fixing S2 subgroup}. It follows that $\widebar T \cap C_{\widebar C}(\widebar K, \widebar L) = \widebar X$.
		
		Let $\mathcal{F} := \mathcal{F}_S(PSL_n(q)) = \mathcal{F}_S(G)$. Since $X$ is central in $C_{PSL_n(q)}(t)$, the only subsystem of $C_{\mathcal{F}}(\langle t \rangle)$ on $X$ is the nilpotent fusion system on $X$. It follows that $\mathcal{F}_{\widebar X}(C_{\widebar C}(\widebar K, \widebar L))$ is nilpotent. So $C_{\widebar C}(\widebar K, \widebar L)$ has a nilpotent $2$-fusion system, as required.  
	\end{proof}
	
	In the following lemma, $A_1$ and $A_2$ have the meanings given to them after Lemma \ref{A_1^+_and_A_2^+_are_normal}. 
	
	\begin{lemma}
		\label{2-components of C_C(u)} 
		Set $C := C_G(t)$. Suppose that $q^{*} \ne 3$. Then $A_1$, $A_2$ and $L$ are the only $2$-components of $C_C(u)$. Moreover, the following hold: 
		\begin{enumerate}
			\item[(i)] $A_1$ is the only $2$-component of $C_C(u)$ containing $u$. 
			\item[(ii)] $A_2$ is the only $2$-component of $C_C(u)$ containing neither $u$ nor $t$. 
			\item[(iii)] $L$ is the only $2$-component of $C_C(u)$ containing $t$. 
		\end{enumerate}
	\end{lemma} 
	
	\begin{proof}
		By definition, $A_1$ and $A_2$ are $2$-components of $C_C(u)$. Also, it is clear from the definition of $L$ (see Proposition \ref{construction of L}) that $L$ is a $2$-component of $C_C(u)$. 
		
		Set $\widebar C := C/O(C)$. As a consequence of Lemma \ref{very long lemma}, $\widebar{A_1}$ and $\widebar{A_2}$ are the only $2$-components of $C_{\widebar K}(\widebar u)$. Moreover, $\widebar L$ is a component of $C_{\widebar C}(\widebar u)$. So Lemma \ref{GW 2.18} shows that $\widebar{A_1}$, $\widebar{A_2}$ and $\widebar L$ are the only $2$-components of $C_{\widebar C}(\widebar u)$. As we have observed after Lemma \ref{A_1^+_and_A_2^+_are_normal}, there is a bijection from the set of $2$-components of $C_C(u)$ to the set of $2$-components of $C_{\widebar C}(\widebar u)$ sending each $2$-component $A$ of $C_C(u)$ to $\widebar A$. Therefore, $A_1$, $A_2$ and $L$ are the only $2$-components of $C_C(u)$.
		
		It remains to prove (i), (ii) and (iii). We have $T_1 \le A_1$ by Lemma \ref{T1 lies in A1} and thus $u \in A_1$. From the definition of $L$, it is clear that $t \in L$. Moreover, $u \not\in L$ since $\widebar t$ is the only involution of $\widebar L$. Similarly, $t \not\in A_1$. Also, it is easy to see from Lemma \ref{very long lemma} that $u$ and $t$ cannot be elements of $A_2$.
	\end{proof} 
	
	\begin{lemma}
		\label{2-balance_G_lemma}
		Suppose that $q^{*} \ne 3$. Let $F$ be a Klein four subgroup of $T$. Then we have $\Delta_G(F) \cap C_G(t) \le O(C_G(t))$.
	\end{lemma} 
	
	\begin{proof}
		Set $C := C_G(t)$, $D := \Delta_G(F) \cap C$ and $\widebar C := C/O(C)$. We are going to show that $\widebar D$ is trivial. 
		
		A direct calculation shows that $D \le \Delta_C(F)$. For each $a \in F^{\#}$, we have $\widebar{O(C_C(a))} \le O(C_{\widebar C}(\widebar a))$ as a consequence of Corollary \ref{centralizers_p-elements}. Therefore, we have $\widebar{\Delta_C(F)} \le \Delta_{\widebar C}(\widebar F)$, and hence $\widebar D \le \Delta_{\widebar C}(\widebar F)$. Lemma \ref{Delta centralizes K} implies that $[\widebar D, \widebar K] = 1$. In particular, $\widebar D \le C_{\widebar C}(\widebar u) = \widebar{C_C(u)}$. Fix a subgroup $D_0$ of $C_C(u)$ with $\widebar{D_0} = \widebar{D}$. Also, let $g \in G$ with $u^g = t$ and $t^g = u$ (such an element exists by Lemma \ref{very easy lemma}). Note that $(D_0)^g \le (C_C(u))^g = C_C(u)$.
		
		We accomplish the proof step by step. 
		
		\medskip
		
		(1) \textit{$A_1$, $A_2$ and $L$ are normal subgroups of $C_C(u)$.}
		
		This is immediate from Lemma \ref{2-components of C_C(u)}.
		
		\medskip
		
		(2) \textit{There is a group isomorphism $\mathrm{Aut}(\widebar{A_1}) \rightarrow \mathrm{Aut}(\widebar L)$ which maps $\mathrm{Inn}(\widebar{A_1})$ to $\mathrm{Inn}(\widebar L)$ and $\mathrm{Aut}_{\widebar{(D_0)^g}}(\widebar{A_1})$ to $\mathrm{Aut}_{\widebar D}(\widebar L)$.}
		
		Let $\mathrm{Aut}_{D_0}(L/O(L))$ denote the image of $\mathrm{Aut}_{D_0}(L)$ under the natural group homomorphism $\mathrm{Aut}(L) \rightarrow \mathrm{Aut}(L/O(L))$. Also, let $\mathrm{Aut}_{(D_0)^g}(A_1/O(A_1))$ denote the image of $\mathrm{Aut}_{(D_0)^g}(A_1)$ under the natural group homomorphism $\mathrm{Aut}(A_1) \rightarrow \mathrm{Aut}(A_1/O(A_1))$. 
		
		From Lemma \ref{2-components of C_C(u)}, it is clear that $(A_1)^{g^{-1}} = L$. The group isomorphism $c_{g^{-1}}|_{A_1,L}$ induces a group isomorphism $A_1/O(A_1) \rightarrow L/O(L)$, and this group isomorphism induces a group isomorphism $\mathrm{Aut}(A_1/O(A_1)) \rightarrow \mathrm{Aut}(L/O(L))$. By a direct calculation, the group isomorphism just mentioned maps $\mathrm{Aut}_{(D_0)^g}(A_1/O(A_1))$ to $\mathrm{Aut}_{D_0}(L/O(L))$ and $\mathrm{Inn}(A_1/O(A_1))$ to $\mathrm{Inn}(L/O(L))$.
		
		We have $A_1/(A_1 \cap O(C)) \cong \widebar{A_1} \cong SL_2(q^{*})$. As $SL_2(q^{*})$ is core-free, it follows that $A_1 \cap O(C) = O(A_1)$. So the natural group homomorphism $A_1 \rightarrow \widebar{A_1}$ induces a group isomorphism $A_1/O(A_1) \rightarrow \widebar{A_1}$. This group isomorphism induces a group isomorphism $\mathrm{Aut}(A_1/O(A_1)) \rightarrow \mathrm{Aut}(\widebar{A_1})$. By a direct calculation, the group isomorphism just mentioned maps $\mathrm{Aut}_{(D_0)^g}(A_1/O(A_1))$ to $\mathrm{Aut}_{\widebar{(D_0)^g}}(\widebar{A_1})$ and $\mathrm{Inn}(A_1/O(A_1))$ to $\mathrm{Inn}(\widebar{A_1})$. In a very similar way, we obtain an isomorphism $\mathrm{Aut}(L/O(L)) \rightarrow \mathrm{Aut}(\widebar{L})$ which maps $\mathrm{Aut}_{D_0}(L/O(L))$ to $\mathrm{Aut}_{\widebar{D_0}}(\widebar L) = \mathrm{Aut}_{\widebar D}(\widebar{L})$ and $\mathrm{Inn}(L/O(L))$ to $\mathrm{Inn}(\widebar{L})$. 
		
		As a consequence of the preceding observations, there is a group isomorphism $\mathrm{Aut}(\widebar{A_1}) \rightarrow \mathrm{Aut}(\widebar L)$ which maps $\mathrm{Inn}(\widebar{A_1})$ to $\mathrm{Inn}(\widebar L)$ and $\mathrm{Aut}_{\widebar{(D_0)^g}}(\widebar{A_1})$ to $\mathrm{Aut}_{\widebar D}(\widebar L)$, as asserted. 
		
		\medskip 
		
		(3) \textit{$\mathrm{Aut}_{\widebar{(D_0)^g}}(\widebar{A_1}) \le \mathrm{Inn}(\widebar{A_1})$.}
		
		As observed above, $\widebar{D_0} = \widebar D$ centralizes $\widebar K$. In particular, $\widebar D$ centralizes $\widebar{A_2}$. This implies that $[D_0,A_2] \le O(C)$. As $D_0$ normalizes $A_2$ by (1), we also have that $[D_0,A_2] \le A_2$. Consequently, $[D_0,A_2] \le O(A_2)$. Because of Lemma \ref{2-components of C_C(u)}, we have $(A_2)^g = A_2$. It follows that $[(D_0)^g,A_2] \le O(A_2)$. This easily implies $[\widebar{(D_0)^g},\widebar{A_2}] \le O(\widebar{A_2})$. As $\widebar{A_2} \cong SL_{n-4}^{\varepsilon}(q^{*})$ by Lemma \ref{very long lemma}, we have $O(\widebar{A_2}) \le Z(\widebar{A_2})$. It follows that $[\widebar{A_2}, \widebar{(D_0)^g},\widebar{A_2}] = [\widebar{(D_0)^g},\widebar{A_2},\widebar{A_2}] \le [Z(\widebar{A_2}),\widebar{A_2}] = 1$. The Three Subgroups Lemma \cite[1.5.6]{KurzweilStellmacher} implies $[\widebar{A_2},\widebar{(D_0)^g}] = [\widebar{A_2},\widebar{A_2},\widebar{(D_0)^g}] = 1$. Hence, $\widebar{(D_0)^g}$ centralizes $\widebar{A_2}$. By (1), $\widebar{(D_0)^g}$ normalizes $\widebar{A_1}$. Clearly, $\mathrm{Aut}_{\widebar{(D_0)^g}}(\widebar K)$ has odd order. The assertion now follows from Lemmas \ref{very long lemma} (iii), \ref{auto_lemma_SL} and \ref{auto_lemma_SU}.
		
		\medskip
		
		(4) \textit{$\widebar D \le \bigcap_{y \in F^{\#}} O(C_{\widebar L}(\widebar y))$.}
		
		As a consequence of (2) and (3), we have $\mathrm{Aut}_{\widebar D}(\widebar L) \le \mathrm{Inn}(\widebar L)$. This implies $\widebar D \le \widebar L C_{\widebar C}(\widebar L)$. By \cite[6.5.3]{KurzweilStellmacher}, $\widebar L \le C_{\widebar C}(\widebar K)$. As observed above, $[\widebar D, \widebar K] = 1$ and hence $\widebar D \le C_{\widebar C}(\widebar K)$. It follows that $\widebar D$ is a subgroup of $\widebar L (C_{\widebar C}(\widebar L) \cap C_{\widebar C}(\widebar K))$. By Lemma \ref{center of C}, $C_{\widebar C}(\widebar L) \cap C_{\widebar C}(\widebar K)$ is a $2$-group. As $\widebar D$ has odd order and $\widebar L \trianglelefteq \widebar C$, this implies that $\widebar D \le \widebar L$. Now we see that
		\begin{align*}
			\widebar D &\le \widebar L \cap \Delta_{\widebar C}(\widebar F) \\
			&= \bigcap_{y \in F^{\#}} (\widebar L \cap O(C_{\widebar C}(\widebar y))) \\
			&= \bigcap_{y \in F^{\#}} (C_{\widebar L}(\widebar y) \cap O(C_{\widebar C}(\widebar y))) \\
			&= \bigcap_{y \in F^{\#}} O(C_{\widebar L}(\widebar y)).
		\end{align*} 
		
		\medskip
		
		(5) \textit{Conclusion.} 
		
		As $F$ is a Klein four subgroup of $T$, we have $F = \langle y_1, y_2 \rangle$ for two commuting involutions $y_1$ and $y_2$ of $T$. For $i \in \lbrace 1,2\rbrace$, we have
		\begin{equation*}
			y_i = \begin{pmatrix} A_i & \\ & B_i \end{pmatrix} Z(SL_n(q))
		\end{equation*} 
		for some $A_i \in W$ and $B_i \in V$ with $\mathrm{det}(A_i) \mathrm{det}(B_i) = 1$. Let $y_3 := y_1y_2, A_3 := A_1A_2$ and $B_3 := B_1B_2$. As $y_1, y_2, y_3$ are involutions, we have $(B_i)^2 \in Z(GL_2(q))$ for each $i \in \lbrace 1,2,3 \rbrace$. 
		
		It is easy to note that $\widebar{X_2} \in \mathrm{Syl}_2(\widebar L)$. If $B \in V \cap SL_2(q)$ and 
		\begin{equation*}
			y := \begin{pmatrix} I_{n-2} & \\ & B \end{pmatrix} Z(SL_n(q)) \in X_2,
		\end{equation*} 
		then
		\begin{equation*}
			y^{y_i} = \begin{pmatrix} I_{n-2} & \\ & B^{B_i} \end{pmatrix} Z(SL_n(q))  
		\end{equation*} 
		for each $i \in \lbrace 1,2,3 \rbrace$. Applying Lemma \ref{final lemma for 2-balance}, we deduce that 
		\begin{equation*}
			\bigcap_{y \in F^{\#}} O(C_{\widebar L}(\widebar y)) = 1. 
		\end{equation*} 
		So we have $\widebar D = 1$ by (4). This completes the proof.
	\end{proof}

\begin{lemma}
	\label{2-balance} 
	Suppose that $q^{*} \ne 3$. Then $G$ is $2$-balanced.
\end{lemma} 

\begin{proof}
	Let $F$ be a Klein four subgroup of $G$ and let $a$ be an involution of $G$ centralizing $F$. We have to show that $\Delta_G(F) \cap C_G(a) \le O(C_G(a))$. 
	
	Assume that $a$ is $G$-conjugate to $t$. Then there is some $g \in G$ with $a^g = t$ and $F^g \le T$. By Lemma \ref{2-balance_G_lemma}, we have $\Delta_G(F^g) \cap C_G(t) \le O(C_G(t))$. Clearly $\Delta_G(F)^g = \Delta_G(F^g)$. It follows that $\Delta_G(F) \cap C_G(a) \le O(C_G(a))$.
	
	Assume now that $a$ is not $G$-conjugate to $t$. Let $J$ be a $2$-component of $C_G(a)$. By Propositions \ref{2-components_ti_1}, \ref{2-components_ti_2} and \ref{2-components w}, either $J/O(J) \cong SL_k^{\varepsilon}(q^{*})/O(SL_k^{\varepsilon}(q^{*}))$ for some $k \ge 3$, or $J/O(J)$ is isomorphic to a nontrivial quotient of $SL_{\frac{n}{2}}^{\varepsilon_0}(q_0)$ for some nontrivial odd prime power $q_0$ and some $\varepsilon_0 \in \lbrace +,- \rbrace$. So $J/O(J)$ is locally $2$-balanced by Lemma \ref{local 2-balance of SL and SU}. Applying \cite[Theorem 5.2]{GW}, we may conclude that $\Delta_{C_G(a)}(F) \le O(C_G(a))$. A direct calculation shows that $\Delta_G(F) \cap C_G(a) \le \Delta_{C_G(a)}(F)$. Hence $\Delta_G(F) \cap C_G(a) \le O(C_G(a))$.  
\end{proof} 
	
	\subsection{The case $q^{*} \ne 3$: Triviality of $\Delta_G(F)$}
	\begin{lemma}
		\label{triviality remaining cases} 
		Suppose that $q^{*} \ne 3$. Assume moreover that $q \equiv 1 \ \mathrm{mod} \ 4$ or $n \ge 7$. Then we have $\Delta_G(F) = 1$ for each Klein four subgroup $F$ of $S$. 
	\end{lemma}
	
	\begin{proof}
		We follow the pattern of the proof of \cite[Theorem 9.1]{GW}.
		
		For each elementary abelian $2$-subgroup $A$ of $G$ of rank at least $3$, we define 
		\begin{equation*}
			W_A := \langle \Delta_G(F) \ \vert \ F \le A, m(F) = 2 \rangle.
		\end{equation*}
		Let $P$ and $Q$ be elementary abelian subgroups of $S$ of rank at least $3$. We claim that $W_P = W_Q$. By Corollary \ref{conclusion connectivity} (iii), $S$ is $3$-connected. So there exist a natural number $m \ge 1$ and a sequence 
		\begin{equation*}
			P = P_1, \dots, P_m = Q
		\end{equation*}
		such that $P_i$, $1 \le i \le m$, is an elementary abelian subgroup of $S$ of rank at least $3$ and such that 
		\begin{equation*}
			P_i \subseteq P_{i+1} \ \textnormal{or} \ P_{i+1} \subseteq P_i
		\end{equation*}  
		for all $1 \le i < m$. By Lemma \ref{2-balance}, $G$ is $2$-balanced. Proposition \ref{GW 6.10} (ii) implies that $W_{P_i} = W_{P_{i+1}}$ for all $1 \le i < m$. Therefore, $W_P = W_Q$, as asserted.  
		
		We use $W_0$ to denote $W_P$, where $P$ is an elementary abelian subgroup of $S$ of rank at least $3$. Let $M := N_G(W_0)$. We accomplish the proof step by step. 
		
		\medskip
		
		(1) $N_G(S) \le M$. 
		
		Let $g \in N_G(S)$. Take an elementary abelian subgroup $P$ of $S$ with $m(P) \ge 3$. By Proposition \ref{GW 6.10} (i), we have $(W_0)^g = (W_P)^g = W_{P^g} = W_0$. Thus $g \in M$.
		
		\medskip
		
		(2) \textit{Let $x$ be an involution of $S$. Then $C_G(x) \le M$.} 
		
		By Corollary \ref{conclusion_connectivity_2}, there is an elementary abelian subgroup $P$ of $S$ with $x \in P$ and $m(P) = 4$. Clearly, $P \le C_G(x)$. Let $R$ be a Sylow $2$-subgroup of $C_G(x)$ containing $P$. By Corollary \ref{3-generation_conclusion}, $C_G(x)$ is $3$-generated. Hence, $C_G(x)$ is generated by the normalizers $N_{C_G(x)}(U)$, where $U \le R$ and $m(U) \ge 3$. It suffices to show that each such normalizer lies in $M$.
		
		So let $U$ be a subgroup of $R$ with $m(U) \ge 3$, and let $g \in N_{C_G(x)}(U)$. Let $Q$ be an elementary abelian subgroup of $U$ with $m(Q) = 3$, and let $h \in G$ with $R^h \le S$. Then $W_{Q^h} = W_{Q^{gh}} = W_{P^h} = W_0$. Proposition \ref{GW 6.10} (i) implies that $W_Q = W_{Q^g} = W_P = W_0$. Applying Proposition \ref{GW 6.10} (i) again, it follows that $(W_0)^g = (W_Q)^g = W_{Q^g} = W_0$. Hence $g \in M$ and thus $N_{C_G(x)}(U) \le M$.
		
		\medskip
		
		(3) $M = G$. 
		
		Assume that $M \ne G$. By \cite[Proposition 17.11]{GLS2}, we may deduce from (1) and (2) that $M$ is strongly embedded in $G$, i.e. $M \cap M^g$ has odd order for any $g \in G \setminus M$. Applying \cite[Chapter 6, 4.4]{Suzuki}, it follows that $G$ has only one conjugacy class of involutions. On the other hand, we see from Proposition \ref{involutions of PSL(n,q)} that $G$ has at least two conjugacy classes of involutions. This contradiction shows that $M = G$.  
		
		\medskip
		
		(4) \textit{Conclusion.} 
		
		Let $F$ be a Klein four subgroup of $S$. By Corollary \ref{conclusion_connectivity_2}, there is an elementary abelian subgroup $P$ of $S$ with $F \le P$ and $m(P) = 4$. Clearly, $\Delta_G(F) \le W_P$. Since $G$ is $2$-balanced, $W_P$ has odd order by Proposition \ref{GW 6.10} (ii). Since $W_P = W_0$, we have $W_P \trianglelefteq G$ by (3). As $O(G) = 1$ by Hypothesis \ref{hypothesis}, it follows that $W_P = 1$. Hence $\Delta_G(F) = 1$.  
	\end{proof}
	
	Next, we deal with the case that $n = 6$, $q \equiv 3 \mod 4$ and $q^{*} \ne 3$. We show that, in this case, $\Delta_G(F) = 1$ for each Klein four subgroup $F$ of $S$ consisting of elements of the form $t_A$, where $A \subseteq \lbrace 1, \dots, n \rbrace$ has even order. We need the following lemma. 
	
	\begin{lemma}
		\label{lemma_on_K} 
		Suppose that $q^{*} \ne 3$. Set $\ell := n-4$. Let $E$ be the subgroup of $T$ consisting of all $t_A$, where $A \subseteq \lbrace 1, \dots, n \rbrace$ has even order. Let $E_1$ denote the subgroup of $X_1$ consisting of all $t_A$, where $A$ is a subset of $\lbrace 1,\dots,n-2 \rbrace$ of even order. Then we may choose elements $m_1, \dots, m_{\ell} \in N_K(E_1)$ and an $E_8$-subgroup $E_0$ of $E$ with
		\begin{equation*}
			K = \langle O(K), L_{2'}(C_K(E_0)), L_{2'}(C_K(E_0))^{m_1}, \dots, L_{2'}(C_K(E_0))^{m_{\ell}} \rangle.
		\end{equation*}    
	\end{lemma}
	
	\begin{proof}
		Set $C := C_G(t)$ and $\widebar C := C/O(C)$. Let $H := SL_{n-2}^{\varepsilon}(q^{*})/O(SL_{n-2}^{\varepsilon}(q^{*}))$. Let $\widetilde{D}$ be the subgroup of $SL_{n-2}^{\varepsilon}(q^{*})$ consisting of all diagonal matrices in $SL_{n-2}^{\varepsilon}(q^{*})$ with diagonal entries in $\lbrace 1, -1 \rbrace$, and let $D$ denote the image of $\widetilde{D}$ in $H$. Denote by $H_1$ the image of
		\begin{equation*}
			\left\lbrace \begin{pmatrix} A & \\ & I_{n-4}\end{pmatrix} \ : \ A \in SL_2^{\varepsilon}(q^{*})  \right\rbrace
		\end{equation*}
		in $H$.
		
		We claim that there is a group isomorphism $\psi: \widebar K \rightarrow H$ which maps $\widebar{E_1}$ to $D$ and $\widebar{A_1}$ to $H_1$. By Lemma \ref{very long lemma} (iii), there is a group isomorphism $\varphi: \widebar{K} \rightarrow H$ under which $\widebar{A_1}$ corresponds to $H_1$. Since $\widebar u$ is the only involution of $\widebar{A_1}$, we have that $\widebar u^{\varphi}$ is the image of $\mathrm{diag}(-1,-1,1,\dots,1) \in SL_{n-2}^{\varepsilon}(q^{*})$ in $H$. Clearly, $\widebar {E_1}$ is elementary abelian of order $2^{n-3}$. Using Lemma \ref{elementary abelian subgroups of SL(n,q)}, we conclude that $\widebar{E_1}^{\varphi}$ is $H$-conjugate to $D$. So there is some $\alpha \in \mathrm{Inn}(H)$ mapping $\widebar{E_1}^{\varphi}$ to $D$. We may assume that $\alpha$ centralizes $\widebar{u}^{\varphi}$. Then ${H_1}^{\alpha} = H_1$, and the isomorphism $\psi := \varphi\alpha$ maps $\widebar{E_1}$ to $D$ and $\widebar{A_1}$ to $H_1$, as desired.
		
		Using Lemma \ref{generation0}, we can find elements $x_1, \dots, x_{\ell} \in N_H(D)$ such that $H = \langle H_1$, ${H_1}^{x_1}$, \dots, ${H_1}^{x_{\ell}} \rangle$. Therefore, $K$ has elements $m_1$, $\dots$, $m_{\ell}$ such that
		\begin{equation*}
			\widebar K = \langle \widebar{A_1}, \widebar{A_1}^{\widebar{m_1}}, \dots, \widebar{A_1}^{\widebar{m_{\ell}}} \rangle
		\end{equation*}
		and $\widebar{m_1}, \dots, \widebar{m_{\ell}} \in N_{\widebar K}(\widebar{E_1})$. From Lemma \ref{normalizers_p-subgroups}, we see that $N_{\widebar{K}}(\widebar E_1) = \widebar{N_K(E_1)}$. So we may assume $m_i \in N_K(E_1)$ for $i \in \lbrace 1, \dots, \ell \rbrace$. Let $E_0 := \langle u, t_{\lbrace 3,4 \rbrace}, t_{\lbrace 4,5 \rbrace} \rangle$. By Lemma \ref{A_1^+_and_A_2^+_are_normal}, we have $\widebar{A_1} \trianglelefteq C_{\widebar C}(\widebar u)$. In particular, $\widebar{E_0}$ normalizes $\widebar{A_1}$. Moreover, $\widebar{E_0}$ centralizes $\widebar{T_1}$. We have $\widebar{A_1} \cong SL_2(q^{*})$ and $\widebar{T_1} \in \mathrm{Syl}_2(\widebar{A_1})$ (see Lemma \ref{very long lemma}). Applying Lemma \ref{involutions of SL_2(q) fixing S2 subgroup}, we conclude that $\widebar{A_1} \le C_{\widebar K}(\widebar {E_0})$. As $\widebar{A_1} \trianglelefteq C_{\widebar K}(\widebar u)$ and $\widebar{A_1} \le C_{\widebar K}(\widebar{E_0}) \le C_{\widebar K}(\widebar u)$, we even have that $\widebar{A_1}$ is a component of $C_{\widebar K}(\widebar {E_0})$. It follows that 
		\begin{equation*}
			\widebar K = \langle L_{2'}(C_{\widebar K}(\widebar{E_0})), L_{2'}(C_{\widebar K}(\widebar{E_0}))^{\widebar{m_1}}, \dots, L_{2'}(C_{\widebar K}(\widebar{E_0}))^{\widebar{m_{\ell}}} \rangle.
		\end{equation*} 
		Let $k \in K$ such that $\widebar{k} \in C_{\widebar K}(\widebar{E_0})$. As $K \trianglelefteq C$, we have $[k,E_0] \le O(C) \cap K = O(K)$. Thus $k O(K) \in C_{C/O(K)}(E_0O(K)/O(K))$. By Lemma \ref{normalizers_p-subgroups}, there is an element $z \in C_C(E_0)$ such that $kO(K) = zO(K)$. Observing that $z \in C_K(E_0)$ and that $\widebar{k} = \widebar{z}$, we may conclude that $C_{\widebar K}(\widebar{E_0}) = \widebar{C_K(E_0)}$. If $1 \le i \le \ell$, then $L_{2'}(C_{\widebar K}(\widebar{E_0}))^{\widebar{m_i}} = L_{2'}(\widebar{C_K(E_0)})^{\widebar{m_i}} = \widebar{L_{2'}(C_K(E_0))}^{\widebar{m_i}} = \widebar{L_{2'}(C_K(E_0))^{m_i}}$, where the second equality follows from Proposition \ref{2-components modulo odd order subgroup}. It follows that 
		\begin{equation*}
			K = \langle O(K), L_{2'}(C_K(E_0)), L_{2'}(C_K(E_0))^{m_1}, \dots, L_{2'}(C_K(E_0))^{m_{\ell}} \rangle.
		\end{equation*} 
		This completes the proof. 
	\end{proof}
	
	\begin{lemma}
		\label{delta_trivial_1} 
		Suppose that $n = 6$, $q \equiv 3 \ \mathrm{mod} \ 4$ and $q^{*} \ne 3$. Let $E$ denote the subgroup of $S$ consisting of all $t_A$, where $A$ is a subset of $\lbrace 1, \dots, n \rbrace$ of even order. Then $\Delta_G(F) = 1$ for any Klein four subgroup $F$ of $E$.  
	\end{lemma} 
	
	\begin{proof}
		We follow the pattern of the proof of \cite[Theorem 9.1]{GW}. 
		
		Set $W_0 := \langle \Delta_G(F) \ \vert \ F \le E, m(F) = 2 \rangle$ and $M := N_G(W_0)$. Since $T$ is the image of 
		\begin{equation*}
			\left \lbrace \begin{pmatrix} A & \\ & B \end{pmatrix} \ : \ A \in W, B \in V, \mathrm{det}(A)\mathrm{det}(B) = 1 \right \rbrace
		\end{equation*} 
		in $PSL_n(q)$, we have $T \in \mathrm{Syl}_2(PSL_n(q))$ by Lemma \ref{sylow_general}. Hence $S = T$ and thus $t \in Z(S)$. By choice of $W$ (see Section \ref{Preliminary discussion and notation}), we have
		\begin{equation*}
			W = \left\lbrace\begin{pmatrix}
				A & \\
				& B \\ 
			\end{pmatrix} 
			\ : \ A, B \in V \right \rbrace 
			\cdot
			\left \langle 
			\begin{pmatrix}
				& I_{2} \\
				I_{2} &  \\ 
			\end{pmatrix} \right \rangle
		\end{equation*}   
		We accomplish the proof step by step.
		
		\medskip
		
		(1) \textit{For each subgroup $E_0$ of $E$ with order at least $8$, we have $N_G(E_0) \le M$.}
		
		Clearly, $E \cong E_{16}$. Therefore, the statement follows from the $2$-balance of $G$ (see Lemma \ref{2-balance}) and Proposition \ref{GW 6.10} (ii). 
		
		\medskip
		
		(2) $N_G(S) \le M$.
		
		First we prove $S \le M$. By (1), we have $E \le M$. As $q \equiv 3 \ \mathrm{mod} \ 4$ and $S = T$, any element of $S$ can be written as a product of an element of $E$ and an element of $S$ induced by a matrix of the form 
		\begin{equation*}
			\begin{pmatrix}
				A & \\ & B
			\end{pmatrix} 
		\end{equation*}   
		with $A \in W \cap SL_4(q)$ and $B \in V \cap SL_2(q)$. So, in order to prove that $S \le M$, it suffices to show that each element of $S$ induced by a matrix of this form lies in $M$. If $B \in V \cap SL_2(q)$, then the image of
		\begin{equation*}
			\begin{pmatrix}
				I_4 & \\ & B 
			\end{pmatrix}
		\end{equation*} 
		in $S$ centralizes the group $\langle t_{\lbrace 1,2 \rbrace}, t_{\lbrace 2,3 \rbrace}, t_{\lbrace 3,4 \rbrace} \rangle \cong E_8$. So it is contained in $M$ by (1). Hence, in order to prove that $S \le M$, it suffices to show that if $A \in W \cap SL_4(q)$, then the image of
		\begin{equation*}
			\begin{pmatrix}
				A & \\ & I_2
			\end{pmatrix}
		\end{equation*} 
		in $S$ lies in $M$. So assume that $A \in W \cap SL_4(q)$. By the structure of $W$, there are elements $M_1$, $M_2$ of $V$ such that $\mathrm{det}(M_1) = \mathrm{det}(M_2)$ and 
		\begin{equation*}
			A = 
			\begin{pmatrix}
				M_1 & \\ & M_2
			\end{pmatrix} \
			\textnormal{or} \
			A = \begin{pmatrix}
				M_1 & \\ & M_2
			\end{pmatrix} 
			\begin{pmatrix}
				& I_2 \\ I_2 &
			\end{pmatrix}.
		\end{equation*}
		The image of  
		\begin{equation*}
			\begin{pmatrix}
				M_1 & & \\ & M_2 & \\ & & I_2
			\end{pmatrix}
		\end{equation*}
		in $S$ can be written as a product of an element of $E$ and an element of $S$ induced by a matrix of the form 
		\begin{equation*}
			\begin{pmatrix}
				\widetilde{M_1} & & \\ & \widetilde{M_2} & \\ & & I_2
			\end{pmatrix}
		\end{equation*}
		with $\widetilde{M_1},\widetilde{M_2} \in V \cap SL_2(q)$. The images of 
		\begin{equation*}
			\begin{pmatrix}
				\widetilde{M_1} & \\ & I_4 
			\end{pmatrix} \ \textnormal{and} \ \begin{pmatrix}
				I_2 & & \\ & \widetilde{M_2} & \\ & & I_2 
			\end{pmatrix} 
		\end{equation*}
		in $S$ centralize the groups $\langle t_{\lbrace 3,4 \rbrace}, t_{\lbrace 4,5 \rbrace}, t_{\lbrace 5,6 \rbrace} \rangle$ and $\langle t_{\lbrace 1,2 \rbrace}, t_{\lbrace 2,5 \rbrace}, t_{\lbrace 5,6 \rbrace} \rangle$, respectively. So they are elements of $M$. It follows that the image of
		\begin{equation*}
			\begin{pmatrix}
				M_1 & & \\ & M_2 & \\ & & I_2
			\end{pmatrix}
		\end{equation*}
		in $S$ lies in $M$. The image of the block matrix 
		\begin{equation*}
			\begin{pmatrix}
				& I_2 & \\
				I_2 & & \\
				& & I_2
			\end{pmatrix}
		\end{equation*} 
		in $S$ normalizes $E$ and is thus contained in $M$. It follows that the image of
		\begin{equation*}
			\begin{pmatrix}
				A & \\ & I_2
			\end{pmatrix}
		\end{equation*}
		in $S$ lies in $M$. Consequently, $S \le M$. 
		
		By Lemma \ref{PSL(n,q)-automorphisms of S}, $\mathrm{Aut}_{PSL_n(q)}(S) = \mathrm{Inn}(S)$. As $\mathcal{F}_S(G) = \mathcal{F}_S(PSL_n(q))$, it follows that $\mathrm{Aut}_G(S) = \mathrm{Inn}(S)$, and so $N_G(S) = SC_G(S)$. We have seen above that $S \le M$, and we have $C_G(S) \le M$ by (1). Hence $N_G(S) \le M$.
		
		\medskip
		
		(3) $C_G(t) \le M$. 
		
		Let $E_1$ be the subgroup of $X_1$ consisting of all $t_A$, where $A$ is a subset of $\lbrace 1, \dots, n-2 \rbrace$ of even order. As a consequence of Lemma \ref{lemma_on_K}, there is an $E_8$-subgroup $E_0$ of $E$ such that $K = \langle O(K), C_K(E_0), N_K(E_1) \rangle$. By (1), $C_K(E_0)$ and $N_K(E_1)$ are subgroups of $M$. By \cite[Proposition 11.23]{GLS2}, we have 
		\begin{equation*}
			O(K) = \langle C_{O(K)}(B) \ \vert \ B \le E, m(B) = 3 \rangle.
		\end{equation*} 
		Therefore, $O(K) \le M$ by (1). Consequently, $K \le M$. By the Frattini argument, 
		\begin{equation*}
			C_G(t) = K N_{C_G(t)}(X_1).
		\end{equation*} 
		So it suffices to show that $N_{C_G(t)}(X_1) \le M$. Since $\mathcal{F}_S(G) = \mathcal{F}_S(PSL_n(q))$, we may conclude from Lemma \ref{automorphisms of X1} that $\mathrm{Aut}_{C_G(t)}(X_1)$ is a $2$-group. Hence, $N_{C_G(t)}(X_1)/C_{C_G(t)}(X_1)$ is a $2$-group. As $X_1 \trianglelefteq T = S \in \mathrm{Syl}_2(C_G(t))$, it follows that $N_{C_G(t)}(X_1) = S C_{C_G(t)}(X_1)$. We have $S \le M$ by (2), and $C_{C_G(t)}(X_1) \le C_G(E_1) \le M$ by (1). Consequently, $N_{C_G(t)}(X_1) \le M$, as required.
		
		\medskip
		
		(4) \textit{Let $x$ be an involution of $S$ which is $G$-conjugate to $t$. Then $x$ is $M$-conjugate to $t$.}
		
		It is easy to see that if an element of $T$ is $PSL_n(q)$-conjugate to $t$, then it is $C_{PSL_n(q)}(t)$-conjugate to an element of $E$. As $\mathcal{F}_S(G) = \mathcal{F}_S(PSL_n(q))$ and $S = T$, it follows that $x$ is $C_G(t)$-conjugate and hence $M$-conjugate to an element $y$ of $E$. It is rather easy to show that if an element of $E$ is $PSL_n(q)$-conjugate to $t$, then it is $N_{PSL_n(q)}(E)$-conjugate to $t$. So, as $\mathcal{F}_S(G) = \mathcal{F}_S(PSL_n(q))$, we have that $y$ is $N_G(E)$-conjugate to $t$. By (1), $N_G(E) \le M$, and so $x$ is $M$-conjugate to $t$.  
		
		\medskip
		
		(5) \textit{Let $x$ be an involution of $S$. Then $C_G(x) \le M$.}
		
		Let $R$ be a Sylow $2$-subgroup of $C_G(x)$ with $C_S(x) \le R$. We have $t \in Z(S) \le C_S(x)$ and $t \in M$. Thus $t \in R \cap M$. Let $r \in N_R(R \cap M)$. Then $y := t^r \in R \cap M$. As a consequence of (4), $y$ is $M$-conjugate to $t$. So there is an element $m$ of $M$ such that $t^r = y = t^m$. We have $rm^{-1} \in C_G(t) \le M$ by (3), and so $r \in R \cap M$. Hence, $N_R(R \cap M) = R \cap M$, and thus $R = R \cap M$.
		
		By Corollary \ref{3-generation_conclusion}, $C_G(x)$ is $3$-generated. Therefore, $C_G(x)$ is generated by the normalizers $N_{C_G(x)}(U)$, where $U \le R$ and $m(U) \ge 3$. It suffices to show that each such normalizer lies in $M$. 
		
		So let $U \le R$ with $m(U) \ge 3$, and let $g \in N_{C_G(x)}(U)$. Take an elementary abelian subgroup $Q$ of $U$ of rank $3$. Lemma \ref{E8 subgroups of central quotients} shows that any $E_8$-subgroup of $S$ has an involution which is the image of an involution of $SL_n(q)$. This implies that $Q$ has an element $s$ which is $G$-conjugate to $t$. Since $s, s^g \in U \le R \le M$, we see from (4) that $s$ and $s^g$ are $M$-conjugate to $t$. So there are elements $m,m' \in M$ such that $s = t^m$ and $s^g = t^{m'}$. We have $t^{m'} = s^g = (t^m)^g = t^{mg}$. Thus $mgm'^{-1} \in C_G(t) \le M$, and hence $g \in M$. It follows that $N_{C_G(x)}(U) \le M$.
		
		\medskip
		
		(6) $M = G$.  
		
		Assume that $M \ne G$. By \cite[Proposition 17.11]{GLS2}, we may deduce from (2) and (5) that $M$ is strongly embedded in $G$, i.e. $M \cap M^g$ has odd order for any $g \in G \setminus M$. Applying \cite[Chapter 6, 4.4]{Suzuki}, it follows that $G$ has only one conjugacy class of involutions. On the other hand, we see from Proposition \ref{involutions of PSL(n,q)} that $G$ has precisely two conjugacy classes of involutions. This contradiction shows that $M = G$. 
		
		\medskip
		
		(7) \textit{Conclusion.} 
		
		Let $F$ be a Klein four subgroup of $E$. Clearly, $\Delta_G(F) \le W_0$. By (6), we have $W_0 \trianglelefteq G$. Since $G$ is $2$-balanced, $W_0$ has odd order by Proposition \ref{GW 6.10} (ii). As $O(G) = 1$ by Hypothesis \ref{hypothesis}, it follows that $W_0 = 1$. Hence $\Delta_G(F) = 1$.  
	\end{proof}
	
	\subsection{Quasisimplicity of the $2$-components of $C_G(t)$} 
	In this subsection, we determine the isomorphism types of $K$ and $L$. 
	
	\begin{lemma}
		\label{involution centralizers modulo cores have B-property}
		Let $x$ and $y$ be two commuting involutions of $G$. Set $C := C_G(x)$ and $\widebar C := C/O(C)$. Then any $2$-component of $C_{\widebar C}(\widebar y)$ is a component of $C_{\widebar C}(\widebar y)$. 
	\end{lemma} 
	
	\begin{proof}
		By \cite[Corollary 3.2]{GW}, $L_{2'}(C_{\widebar C}(\widebar y)) = L_{2'}(C_{L(\widebar C)}(\widebar y))$. We know from Section \ref{2-components of involution centralizers} that $L(\widebar C)$ is a $K$-group, i.e. the composition factors of $L(\widebar C)$ are known finite simple groups. Applying \cite[Theorem 3.5]{Gorenstein83}, we conclude that $L_{2'}(C_{L(\widebar C)}(\widebar y)) = L(C_{L(\widebar C)}(\widebar y))$. Therefore, any $2$-component of $C_{L(\widebar C)}(\widebar y)$ is a component of $C_{L(\widebar C)}(\widebar y)$. So any $2$-component of $C_{\widebar C}(\widebar y)$ is a component of $C_{\widebar C}(\widebar y)$.
		
		Instead of using \cite[Theorem 3.5]{Gorenstein83}, the lemma could be proved directly by using Corollary \ref{q q star corollary} (i) and the results of Section \ref{2-components of involution centralizers}.
	\end{proof} 
	
	\begin{proposition}
		\label{isomorphism_type_K} 
		$K$ is isomorphic to a quotient of $SL_{n-2}^{\varepsilon}(q^{*})$ by a central subgroup of odd order. 
	\end{proposition} 
	
	\begin{proof}
		The proof is inspired from the proof of \cite[Theorem 10.1]{GW}.
		
		For $q^{*} = 3$, the proposition follows from Proposition \ref{case_q*_3}. From now on, we assume that $q^{*} \ne 3$. 
		
		Set $C := C_G(t)$. Let $E$ denote the subgroup of $T$ consisting of all $t_A$, where $A \subseteq \lbrace 1, \dots, n \rbrace$ has even order. We assume $m_1, \dots, m_{\ell}$, where $\ell := n-4$, to be elements of $K$ and $E_0$ to be an $E_8$-subgroup of $E$ with 
		\begin{equation*}
			K = \langle O(K), L_{2'}(C_K(E_0)), L_{2'}(C_K(E_0))^{m_1}, \dots, L_{2'}(C_K(E_0))^{m_{\ell}} \rangle.
		\end{equation*} 
		Such elements $m_1, \dots, m_{\ell}$ and such a subgroup $E_0$ exist by Lemma \ref{lemma_on_K}. 
		
		The proof will be accomplished step by step. 
		
		\medskip
		
		(1) \textit{Let $f$ be an involution of $E_0$. Then $L_{2'}(C_K(E_0)) \le L_{2'}(C_C(f))$.} 
		
		As $K \trianglelefteq C$, we have $C_K(E_0) \trianglelefteq C_C(E_0)$. This implies $L_{2'}(C_K(E_0)) \le L_{2'}(C_C(E_0))$. By \cite[Theorem 3.1]{GW}, we have $L_{2'}(C_{C_C(f)}(E_0)) \le L_{2'}(C_C(f))$. Clearly, $C_{C_C(f)}(E_0) = C_C(E_0)$. It follows that $L_{2'}(C_K(E_0)) \le L_{2'}(C_C(E_0)) \le L_{2'}(C_C(f))$.
		
		\medskip
		
		(2) \textit{Let $F$ be a Klein four subgroup of $E_0$. Set $D := [C_{O(K)}(F),L_{2'}(C_K(E_0))]$. Then $D = 1$.}
		
		Clearly, $L_{2'}(C_K(E_0))$ normalizes $C_{O(K)}(F)$. Also, $O^{2'}(L_{2'}(C_K(E_0))) = L_{2'}(C_K(E_0))$, and $C_{O(K)}(F)$ is a $2'$-group. Applying \cite[Proposition 4.3 (i)]{GLS2}, we conclude that $D = [D,L_{2'}(C_K(E_0))]$.
		
		Now let $f$ be an involution of $F$. We are going to show that $D \le O(C_G(f))$. Set $M := L_{2'}(C_C(f))$. By (1), $L_{2'}(C_K(E_0)) \le M$. Also, $D \le C_C(F) \le C_C(f)$ and $M \trianglelefteq C_C(f)$. It follows that $D = [D,L_{2'}(C_K(E_0))] \le [C_C(f), M] \le M$.
		
		Let $\widebar{C_G(f)} := C_G(f)/O(C_G(f))$. By Corollary \ref{centralizers_p-elements}, $C_{\widebar{C_G(f)}}(\widebar t) = \widebar{C_C(f)}$. As a consequence of Proposition \ref{2-components modulo odd order subgroup}, $L_{2'}(C_{\widebar{C_G(f)}}(\widebar t)) = \widebar{M}$. Lemma \ref{involution centralizers modulo cores have B-property} implies that $\widebar M = L(C_{\widebar{C_G(f)}}(\widebar t))$. It easily follows that $O(\widebar M)$ is central in $\widebar{M}$. 
		
		From the definition of $D$, it is clear that $D \le O(K)$. So we have $D \le M \cap O(K) \le O(M)$. It follows that $\widebar D \le \widebar{O(M)} \le O(\widebar M) \le Z(\widebar M)$. In particular, $\widebar{L_{2'}(C_K(E_0))}$ centralizes $\widebar D$. Thus $D = [D,L_{2'}(C_K(E_0))] \le O(C_G(f))$.
		
		Since $f$ was arbitrarily chosen, it follows that $D \le \Delta_G(F)$. By Lemmas \ref{triviality remaining cases} and \ref{delta_trivial_1}, we have $\Delta_G(F) = 1$. Consequently, $D = 1$, as wanted. 
		
		\medskip 
		
		(3) $O(K) \le Z(K)$. 
		
		By \cite[Proposition 11.23]{GLS2}, we have 
		\begin{equation*}
			O(K) = \langle C_{O(K)}(F) : F \le E_0, m(F) = 2 \rangle. 
		\end{equation*} 
		
		Because of (2), it follows that $O(K)$ centralizes $L_{2'}(C_K(E_0))$. By choice of $E_0$, we have 
		\begin{equation*}
			K = \langle O(K), L_{2'}(C_K(E_0)), L_{2'}(C_K(E_0))^{m_1}, \dots, L_{2'}(C_K(E_0))^{m_{\ell}} \rangle
		\end{equation*} 
		for some $m_1, \dots, m_{\ell} \in K$. It follows that $K = O(K)C_K(O(K))$. Therefore, $C_K(O(K))$ has odd index in $K$. We have $O^{2'}(K) = K$ since $K$ is a $2$-component of $C$. It follows that $K = C_K(O(K))$. Consequently, $O(K) \le Z(K)$. 
		
		\medskip
		
		(4) \textit{Conclusion.}
		
		Applying \cite[Lemma 4.11]{GLS2}, we deduce from (3) that $K$ is a component of $C$. Therefore, $K$ is quasisimple. We have
		\begin{equation*}
			K/Z(K) \cong (K/O(K))/Z(K/O(K)) \cong PSL_{n-2}^{\varepsilon}(q^{*}). 
		\end{equation*} 
		Applying Lemmas \ref{Schur_PSL} and \ref{Schur_PSU}, we conclude that $K \cong SL_{n-2}^{\varepsilon}(q^{*})/Z$ for some central subgroup $Z$ of $SL_{n-2}^{\varepsilon}(q^{*})$. Using Proposition \ref{SL_SU_fusion}, or using the order formulas for $\vert SL_{n-2}^{\varepsilon}(q^{*}) \vert$ and $\vert SL_{n-2}(q) \vert$ given by \cite[Proposition 1.1 and Corollary 11.29]{Grove}, we see that
		\begin{equation*}
			\vert SL_{n-2}^{\varepsilon}(q^{*}) \vert_2 = \vert SL_{n-2}(q) \vert_2 = \vert X_1 \vert = \vert K \vert_2 = \vert SL_{n-2}^{\varepsilon}(q^{*})/Z \vert_2.
		\end{equation*} 
		Thus $Z$ has odd order.
	\end{proof}
	
	\begin{proposition}
		\label{isomorphism_type_L}
		We have $L \cong SL_2(q^{*})$ and $L \trianglelefteq C_G(t)$. Moreover, $L$ is the only normal subgroup of $C_G(t)$ which is isomorphic to $SL_2(q^{*})$. 
	\end{proposition} 
	
	\begin{proof}
		For $q^{*} = 3$, this follows from Propositions \ref{case_q*_3} and \ref{construction of L}.
		
		Assume now that $q^{*} \ne 3$. Let $\widetilde K := KO(C_G(t))$. By the last statement in Proposition \ref{2-components modulo odd order subgroup}, $K = O^{2'}(\widetilde K)$. Let $i \in \lbrace 1,2 \rbrace$. Since $A_i$ is a $2$-component of $C_{C_G(t)}(u)$, we have $A_i = O^{2'}(A_i)$. Also, $A_i \le \widetilde K$, and so $A_i \le O^{2'}(\widetilde K) = K$. It follows that $A_i$ is a $2$-component of $C_K(u)$. 
		
		By Proposition \ref{isomorphism_type_K}, we have $K \cong SL_{n-2}^{\varepsilon}(q^{*})/Z$ for some central subgroup $Z$ of $SL_{n-2}^{\varepsilon}(q^{*})$ with odd order. It is easy to see that if $m$ is a non-central involution of $SL_{n-2}^{\varepsilon}(q^{*})/Z$ and $J$ is a $2$-component of its centralizer in $SL_{n-2}^{\varepsilon}(q^{*})/Z$, then $J \cong SL_k^{\varepsilon}(q^{*})$ for some $k \ge 2$. Since $u$ is a non-central involution of $K$ and $A_1/O(A_1) \cong SL_2(q^{*})$, it follows that $A_1 \cong SL_2(q^{*})$. By definition of $L$ (see Proposition \ref{construction of L}), $L$ is isomorphic to $A_1$. So we have $L \cong SL_2(q^{*})$. 
		
		Let $L_0$ be the $2$-component of $C_G(t)$ associated to $LO(C_G(t))/O(C_G(t))$. By \cite[6.5.2]{KurzweilStellmacher}, we have $[L_0,K] = 1$. Hence $L_0 \le C_{C_G(t)}(u)$. So $L_0$ is a $2$-component of $C_{C_G(t)}(u)$. Clearly $A_1 \ne L_0 \ne A_2$. Lemma \ref{2-components of C_C(u)} implies that $L_0 = L$. From Proposition \ref{construction of L} (iii), we see that $L = L_0 \trianglelefteq C_G(t)$.  
		
		Proposition \ref{construction of L} (iii) also shows that $K$ and $L$ are the only $2$-components of $C_G(t)$. So $L$ is the only normal subgroup of $C_G(t)$ isomorphic to $SL_2(q^{*})$. 
	\end{proof}
	
	\section{The subgroup $G_0$}
	\label{G0} 
	Let $A$ be a subset of $\lbrace 1, \dots, n \rbrace$ with order $2$. Then $t_A$ is $G$-conjugate to $t$. Proposition \ref{isomorphism_type_L} implies that $C_G(t_A)$ has a unique normal subgroup isomorphic to $SL_2(q^{*})$. We denote this subgroup by $L_{A}$, and we define $G_0$ to be the subgroup of $G$ generated by the groups $L_A$, where $A = \lbrace i,i+1 \rbrace$ for some $1 \le i < n$. We are going to prove that $G_0 \trianglelefteq G$ and that $G_0$ is isomorphic to a nontrivial quotient of $SL_n^{\varepsilon}(q^{*})$. This will complete the proof of Theorem \ref{theorem1}.
	
	By Proposition \ref{isomorphism_type_K}, $K$ is isomorphic to a quotient of $SL_{n-2}^{\varepsilon}(q^{*})$ by a central subgroup of odd order. By the proof of Proposition \ref{isomorphism_type_L}, $A_1$ and $A_2$ are $2$-components of $C_K(u)$ if $q^{*} \ne 3$.  
	
	\begin{lemma}
		\label{Curtis-Tits_lemma}
		Let $Z \le Z(SL_{n-2}^{\varepsilon}(q^{*}))$ with $K \cong H := SL_{n-2}^{\varepsilon}(q^{*})/Z$. Let $H_1$ be the image of 
		\begin{equation*}
			\left \lbrace \begin{pmatrix} A & \\ & I_{n-4} \end{pmatrix} \ : \ A \in SL_2^{\varepsilon}(q^{*}) \right \rbrace
		\end{equation*} 
		in $H$ and $H_2$ the image of 
		\begin{equation*}
			\left \lbrace \begin{pmatrix} I_2 & \\ & A \end{pmatrix} \ : \ A \in SL_{n-4}^{\varepsilon}(q^{*}) \right \rbrace
		\end{equation*}
		in $H$. Then there is a group isomorphism $\varphi: K \rightarrow H$ which maps $A_1$ to $H_1$ and $A_2$ to $H_2$. 
	\end{lemma} 
	
	\begin{proof}
		For $q^{*} = 3$, this follows from Proposition \ref{case_q*_3} and Lemma \ref{very long lemma} (iii). 
		
		Assume now that $q^{*} \ne 3$. Let $\varphi : K \rightarrow H$ be a group isomorphism. For each even natural number $k$ with $2 \le k < n-2$, let $h_k$ be the image of
		\begin{equation*}
			\begin{pmatrix}
				-I_k & \\ & I_{n-2-k}
			\end{pmatrix} 
		\end{equation*} 
		in $H$. It is easy to note that each non-central involution of $H$ is conjugate to $h_k$ for some even $2 \le k < n-2$. As $u$ is a non-central involution of $K$, we may assume that $u^{\varphi} = t_k$ for some even $2 \le k < n-2$.  
		
		Let $\widetilde{H_1}$ be the image of 
		\begin{equation*}
			\left \lbrace \begin{pmatrix} A & \\ & I_{n-2-k} \end{pmatrix} \ : \ A \in SL_k^{\varepsilon}(q^{*}) \right \rbrace
		\end{equation*}
		in $H$ and $\widetilde{H_2}$ be the image of 
		\begin{equation*}
			\left \lbrace \begin{pmatrix} I_k & \\ & A \end{pmatrix} \ : \ A \in SL_{n-2-k}^{\varepsilon}(q^{*}) \right \rbrace
		\end{equation*}
		in $H$. It is easy to note that the $2$-components of $C_H(t_k)$ are precisely the quasisimple elements of $\lbrace \widetilde{H_1}, \widetilde{H_2} \rbrace$. Also, $t_k \in \widetilde{H_1}$, but $t_k \not\in \widetilde{H_2}$. On the other hand, $A_1$ and $A_2$ are the $2$-components of $C_K(u)$, and we have $u \in A_1$. This implies $(A_1)^{\varphi} = \widetilde{H_1}$ and $(A_2)^{\varphi} = \widetilde{H_2}$. Since $A_1 \cong L \cong SL_2(q^{*})$, we have $k = 2$, and hence $\widetilde{H_1} = H_1$ and $\widetilde{H_2} = H_2$.  
	\end{proof}
	
	\begin{lemma}
		\label{Curtis-Tits} 
		Let $1 \le i < j < n$. Set $A := \lbrace i,i+1 \rbrace$ and $B:= \lbrace j,j+1 \rbrace$. Then:  
		\begin{enumerate}
			\item[(i)] If $i+1 < j$, then $[L_A,L_B] = 1$. 
			\item[(ii)] Suppose that $j = i+1$. Then there is a group isomorphism from $\langle L_A, L_B \rangle$ to $SL_3^{\varepsilon}(q^{*})$ under which $L_A$ corresponds to the subgroup 
			\begin{equation*}
				\left \lbrace \left( \begin{array}{c|c} M & \begin{matrix} 0 \\ 0 \end{matrix} \\ \hline \begin{matrix} 0 & 0 \end{matrix} & 1 \end{array} \right) \ : \ M \in SL_2^{\varepsilon}(q^{*}) \right\rbrace
			\end{equation*}
			of $SL_3^{\varepsilon}(q^{*})$ and under which $L_B$ corresponds to the subgroup 
			\begin{equation*}
				\left \lbrace \left( \begin{array}{c|cc} 1 & \begin{matrix} 0 & 0 \end{matrix} \\ \hline \begin{matrix} 0 \\ 0 \end{matrix} & M \end{array} \right) \ : \ M \in SL_2^{\varepsilon}(q^{*}) \right\rbrace
			\end{equation*}
			of $SL_3^{\varepsilon}(q^{*})$. 
			\item[(iii)] Suppose that $1 \le i \le n-3$ and that $j = i+1$. Set $k := i+2$ and $C := \lbrace k,k+1 \rbrace$. Then $\langle L_A, L_B, L_C \rangle$ is isomorphic to $SL_4^{\varepsilon}(q^{*})$. 
		\end{enumerate} 
	\end{lemma} 
	
	\begin{proof}
		Let $H$, $H_1$, $H_2$ and $\varphi$ be as in Lemma \ref{Curtis-Tits_lemma}. For each $D \subseteq \lbrace 1, \dots, n-2\rbrace$ of even order, let $h_D$ be the image of the matrix $\mathrm{diag}(d_1,\dots, d_{n-2}) \in SL_{n-2}^{\varepsilon}(q^{*})$ in $H$, where $d_{\ell} = -1$ if $\ell \in D$ and $d_{\ell} = 1$ if $\ell \in \lbrace 1, \dots, n-2 \rbrace \setminus D$. Note that $u^{\varphi} = h_{\lbrace 1,2 \rbrace}$. Let $J$ be the subgroup of $H$ consisting of all $h_D$, where $D \subseteq \lbrace 1, \dots, n-2 \rbrace$ has even order, and let $E_1$ denote the subgroup of $X_1$ consisting of all $t_D$, where $D \subseteq \lbrace 1, \dots, n-2 \rbrace$ has even order. From Lemma \ref{elementary abelian subgroups of SL(n,q)}, we see that $(E_1)^{\varphi}$ is $C_H(u^{\varphi})$-conjugate to $J$. Upon replacing $\varphi$ by a composite of $\varphi$ and an inner automorphism of $H$, we may (and will) assume that $(E_1)^{\varphi} = J$.  
		
		From the definition of $L$ (Proposition \ref{construction of L}), it is easy to see that $L_{\lbrace 1,2 \rbrace} = A_1$. 
		
		We now prove (i). Assume that $i+1 < j$. Since $\mathcal{F}_S(G) = \mathcal{F}_S(PSL_n(q))$, there is some $g \in G$ with $(t_A)^g = t_{\lbrace 1,2 \rbrace} = u$ and $(t_B)^g = t_{\lbrace 3,4 \rbrace}$. So it suffices to show that $[L_{\lbrace 1,2 \rbrace},L_{\lbrace 3,4 \rbrace}] = 1$. Let $h$ denote the image of $t_{\lbrace 3,4 \rbrace}$ under $\varphi$. Then $h \in H_2$ since $t_{\lbrace 3,4 \rbrace} \in T_2 \le A_2$. Therefore, and since $h$ is conjugate to $u^{\varphi} = h_{\lbrace 1,2 \rbrace}$, we may choose $\varphi$ such that $h = h_{\lbrace 3,4 \rbrace}$ (and for the rest of the proof of (i), we will assume that $\varphi$ has been chosen in this way). We see from Lemma \ref{generation0} (ii) that there is an $a \in H$ with $(h_{\lbrace 1,2 \rbrace})^a = h_{\lbrace 3,4 \rbrace}$ and $(H_1)^a \le H_2$. In particular, $[H_1,(H_1)^a] = 1$. If $k$ is the preimage of $a$ under $\varphi$, then $u^k = t_{\lbrace 3,4 \rbrace}$ and $[A_1,(A_1)^k] = 1$. We also have $(A_1)^k = L_{\lbrace 3,4 \rbrace}$ and thus $[L_{\lbrace 1,2 \rbrace},L_{\lbrace 3,4 \rbrace}] = 1$.
		
		We now prove (ii). Assume that $j = i+1$. Since $\mathcal{F}_S(G) = \mathcal{F}_S(PSL_n(q))$, there is some $g \in G$ with $(t_A)^g = t_{\lbrace 1,2 \rbrace}$ and $(t_B)^g = t_{\lbrace 2,3 \rbrace}$. Therefore, it is enough to prove (ii) under the assumption that $i = 1$, and we will assume that this is the case. We see from Lemmas \ref{very long lemma} (ii) and \ref{T1 lies in A1} that $X_1 \cap A_2 = T_2$. Thus $t_{\lbrace 2,3 \rbrace} \not\in A_2$. Let $h$ denote the image of $t_{\lbrace 2,3 \rbrace}$ under $\varphi$. Then $h \not\in H_2$. Therefore, and since $h$ is conjugate to $u^{\varphi} = h_{\lbrace 1,2 \rbrace}$, we may choose $\varphi$ such that $h = h_{\lbrace 2,3 \rbrace}$ (and for the rest of the proof of (ii), we will assume that $\varphi$ has been chosen in this way). Let $\widetilde{H_1}$ be the image of 
		\begin{equation*}
			\left \lbrace \begin{pmatrix} 1 & & \\ & M & \\ & & I_{n-5} \end{pmatrix} \ : \ M \in SL_2^{\varepsilon}(q^{*}) \right \rbrace
		\end{equation*} 
		in $H$. By Lemma \ref{generation0} (ii), there is some $a \in H$ with $(h_{\lbrace 1,2 \rbrace})^a = h_{\lbrace 2,3 \rbrace}$ and $(H_1)^a = \widetilde{H_1}$. Let $k$ be the preimage of $a$ under $\varphi$. Then $u^k = t_{\lbrace 2,3 \rbrace}$ and hence $ L_{\lbrace 2,3 \rbrace} = (L_{\lbrace 1,2 \rbrace})^k = (A_1)^k$. We see now that $\varphi$ induces an isomorphism from $\langle L_{\lbrace 1,2 \rbrace}, L_{\lbrace 2,3 \rbrace} \rangle$ to $\langle H_1, \widetilde{H_1} \rangle$ mapping $L_{\lbrace 1,2 \rbrace}$ to $H_1$ and $L_{\lbrace 2,3 \rbrace}$ to $\widetilde{H_1}$. With this observation, it is easy to complete the proof of (ii).
		
		We now prove (iii). Assume that $1 \le i \le n-3$ and that $j = i+1$. Let $k$ and $C$ be as in the statement of (iii). Since $\mathcal{F}_S(G) = \mathcal{F}_S(PSL_n(q))$, there is some $g \in G$ with $(t_A)^g = t_{\lbrace 1,2 \rbrace} = u$, $(t_B)^g = t_{\lbrace 2,3 \rbrace}$ and $(t_C)^g = t_{\lbrace 3,4 \rbrace}$. Therefore, it is enough to show that $\langle L_{\lbrace 1,2 \rbrace}, L_{\lbrace 2,3 \rbrace}, L_{\lbrace 3,4 \rbrace} \rangle$ is isomorphic to $SL_4^{\varepsilon}(q^{*})$. Let $h := (t_{\lbrace 2,3 \rbrace})^{\varphi}$ and $\widetilde h := (t_{\lbrace 3,4 \rbrace})^{\varphi}$. As in the proof of (ii), we can choose $\varphi$ such that $h = h_{\lbrace 2,3 \rbrace}$. Also, $\widetilde h = h_D$ for some $D \subseteq \lbrace 1,\dots,n-2 \rbrace$ of order $2$. We have $t_{\lbrace 3,4 \rbrace} \in T_2 \le A_2$ and hence $h_D = \widetilde h \in H_2$. Therefore, $D \cap \lbrace 1,2 \rbrace = \emptyset$. We claim that $D \cap \lbrace 2,3 \rbrace = \lbrace 3 \rbrace$. Assume not. Then $D \cap \lbrace 2,3 \rbrace = \emptyset$, and it is easy to find an element $a \in N_H(J)$ with $h^a = h_{\lbrace 1,2 \rbrace} = u^{\varphi}$ and $(\widetilde h)^a = h_{\lbrace 3,4 \rbrace} \in H_2$. So there is some $k \in N_K(E_1)$ with $(t_{\lbrace 2,3 \rbrace})^k = u$ and $(t_{\lbrace 3,4 \rbrace})^k \in T_2$. On the other hand, it is easy to see from $\mathcal{F}_S(G) = \mathcal{F}_S(PSL_n(q))$ that there is no $g \in K$ with $(t_{\lbrace 2,3 \rbrace})^g = u$ and $(t_{\lbrace 3,4 \rbrace})^g \in T_2$. This contradiction shows that $D \cap \lbrace 2,3 \rbrace = \lbrace 3 \rbrace$. So we can choose $\varphi$ such that $h = h_{\lbrace 2,3 \rbrace}$ and $\widetilde h = h_{\lbrace 3,4 \rbrace}$. Now the proof of (iii) can be completed by using similar arguments as in the proof of (ii). 
	\end{proof}
	
	\begin{proposition}
		\label{isomorphism_type_G0}
		$G_0$ is isomorphic to a nontrivial quotient of $SL_n^{\varepsilon}(q^{*})$. 
	\end{proposition} 
	
	\begin{proof}
		Assume that $\varepsilon = +$. By Lemma \ref{Curtis-Tits}, the groups $L_{\lbrace 1,2 \rbrace}, \dots, L_{\lbrace n-1,n \rbrace}$ form a weak Curtis-Tits system in $G$ of type $SL_n(q^{*})$ (in the sense of \cite[p. 9]{GLS8}). Applying a version of the Curtis-Tits theorem, namely \cite[Chapter 13, Theorem 1.4]{GLS8}, we conclude that $G_0$ is isomorphic to a quotient of $SL_n(q^{*})$. 
		
		Assume now that $\varepsilon = -$. Then Lemma \ref{Curtis-Tits} shows that $G_0$ has a weak Phan system of rank $n-1$ over $\mathbb{F}_{{q^{*}}^2}$ (in the sense of \cite[p. 288]{BennettShpectorov}). If $q^{*} \ne 3$, then \cite[Theorem 1.2]{BennettShpectorov} implies that $G_0$ is isomorphic to a quotient of $SU_n(q^{*})$. If $q^{*} = 3$, the same follows from \cite[Theorem 1.3]{BennettShpectorov} and Lemma \ref{Curtis-Tits} (iii).
	\end{proof}
	
	\begin{lemma}
		\label{sylow_contained} 
		Let $R$ be a Sylow $2$-subgroup of $G_0$. Then $R \in \mathrm{Syl}_2(G)$ and $\mathcal{F}_R(G_0) = \mathcal{F}_R(G)$.
	\end{lemma} 
	
	\begin{proof}
		Since $q \sim \varepsilon q^{*}$, we have that the $2$-fusion system of $PSL_n^{\varepsilon}(q^{*})$ is isomorphic to the $2$-fusion system of $PSL_n(q)$ (see Proposition \ref{PSL_PSU_fusion}). Clearly, $G_0/Z(G_0) \cong PSL_n^{\varepsilon}(q^{*})$. So the $2$-fusion system of $G_0/Z(G_0)$ is isomorphic to the $2$-fusion system of $G$. It easily follows that $\vert G_0 \vert_2 = \vert G_0/Z(G_0) \vert_2 = \vert G \vert_2$, and Lemma \ref{factor_systems_fusion_categories} shows that the $2$-fusion system of $G_0$ is isomorphic to that of $G_0/Z(G_0)$ and hence to that of $G$. This completes the proof. 
	\end{proof} 
	
	\begin{lemma}
		\label{KL_in_G0_lemma} 
		The following hold.
		\begin{enumerate}
			\item[(i)] If $q^{*} \ne 3$, then $O^{2'}(O^{2}(C_G(t))) = KL$. 
			\item[(ii)] If $q^{*} = 3$, then $O^2(C_G(t)) = KL$.
		\end{enumerate}
	\end{lemma}
	
	\begin{proof}
		Set $C := C_G(t)$. 
		
		Assume that $q^{*} \ne 3$. Then $KL$ is perfect. This implies that $KL = O^{2'}(O^2(KL)) \le O^{2'}(O^{2}(C))$. Since $T \cap KL = (T \cap K)(T \cap L) = X_1 X_2$, Lemmas \ref{factor_system_C_G(t)_nilpotent} and \ref{factor_systems_fusion_categories} show that $C/KL$ has a nilpotent $2$-fusion system. So $C/KL$ is $2$-nilpotent by \cite[Theorem 1.4]{Linckelmann}. This implies $O^{2'}(O^{2}(C)) \le KL$.
		
		We assume now that $q^{*} = 3$. Then $KL = O^2(KL)$ since $K$ is perfect and $L \cong SL_2(3)$. Thus $KL \le O^2(C)$. In order to prove equality, it suffices to show that $C/KL$ is a $2$-group. By Proposition \ref{case_q*_3} and Lemma \ref{centralizer_q_is_3} (i), $C/KC_C(K)$ is a $2$-group. By \cite[6.5.2]{KurzweilStellmacher}, we have $L \le C_C(K)$. It is enough to show that $C_C(K)/L$ is a $2$-group.
		
		We have $O^2(C_{C}(K)) \cap T \le O^2(C_{C}(X_1)) \cap T = X_2$ by Lemma \ref{lemma_hyperfocal_subgroup} and the hyperfocal subgroup theorem \cite[Theorem 1.33]{Craven}. On the other hand, $X_2 \le L = O^2(L) \le O^2(C_{C}(K))$. Consequently, $X_2 = O^2(C_{C}(K)) \cap T \in \mathrm{Syl}_2(O^2(C_C(K)))$. Set $U := C_{O^2(C_C(K))}(X_2)$. We have $X_2 \trianglelefteq C$ since $X_2$ is the unique Sylow $2$-subgroup of $L \cong SL_2(3)$. So we have $U \trianglelefteq C$. Hence $Z(X_2) = X_2 \cap U \in \mathrm{Syl}_2(U)$. Applying \cite[7.2.2]{KurzweilStellmacher}, we conclude that $U$ is $2$-nilpotent. We have $O(U) = 1$ since $U \trianglelefteq C$ and $O(C) = 1$ by Proposition \ref{case_q*_3}. It follows that $U = Z(X_2)$. 
		
		Clearly, $O^{2}(C_C(K))/U$ is isomorphic to a subgroup of $\mathrm{Aut}(X_2)$. We have $|O^{2}(C_C(K))/U|_2 = 4$ since $Q_8 \cong X_2 \in \mathrm{Syl}_2(O^2(C_C(K)))$ and $U = Z(X_2)$. Also, $|O^{2}(C_C(K))/U| \ge 12$ since $L \le O^2(C_C(K))$. As $\mathrm{Aut}(X_2) \cong \mathrm{Aut}(Q_8) \cong S_4$ by \cite[5.3.3]{KurzweilStellmacher}, it follows that $|O^{2}(C_C(K))/U| = 12$. This implies $O^2(C_C(K)) = L$. So $C_C(K)/L$ is a $2$-group, as required.
	\end{proof}
	
	\begin{lemma}
		\label{KL_in_G0} 
		We have $KL \le G_0$. 
	\end{lemma}
	
	\begin{proof}
		We have $t \in X_2 \le L = L_{\lbrace n-1, n \rbrace} \le G_0$. Let $R \in \mathrm{Syl}_2(G_0)$ with $t \in R$ such that $\langle t \rangle$ is fully centralized in $\mathcal{G} := \mathcal{F}_R(G_0)$. By Lemma \ref{sylow_contained}, $R \in \mathrm{Syl}_2(G)$ and $\mathcal{G} = \mathcal{F}_R(G)$. Therefore, $C_R(t) \in \mathrm{Syl}_2(C_G(t))$ and $C_{\mathcal{G}}(\langle t \rangle) = \mathcal{F}_{C_R(t)}(C_G(t))$. Also, $T = C_S(t) \in \mathrm{Syl}_2(C_G(t))$ and $C_{\mathcal{F}_S(G)}(\langle t \rangle) = \mathcal{F}_T(C_G(t))$. So, by Lemma \ref{components_C1_C2}, $C_{\mathcal{G}}(\langle t \rangle)$ has a component isomorphic to the $2$-fusion system of $SL_{n-2}(q)$. 
		
		Let $Z \le Z(SL_n^{\varepsilon}(q^{*}))$ with $G_0 \cong SL_n^{\varepsilon}(q^{*})/Z$. By the proof of Lemma \ref{sylow_contained}, $Z(G_0)$ has odd order. 
		
		Let $\widetilde x$ be an element of $SL_n^{\varepsilon}(q^{*})$ such that $x := \widetilde x Z$ is an involution of $SL_n^{\varepsilon}(q^{*})/Z$. Set $C := C_{SL_n^{\varepsilon}(q^{*})/Z}(x)$. It is easy to note that the $2$-components of $C$ are precisely the images of the $2$-components of $C_{SL_n^{\varepsilon}(q^{*})}(\widetilde x)$ in $SL_n^{\varepsilon}(q^{*})/Z$. Using this, it is not hard to see from Lemmas \ref{involutions_GL(n,q)} and \ref{diagonalizable_involutions_GU(n,q)} that one of the following holds: 
		\begin{enumerate} 
			\item[(1)] $q^{*} \ne 3$, $O^{2'}(O^2(C)) = K_0L_0$, where $K_0$ and $L_0$ are subnormal subgroups of $C$ such that $K_0 \cong SL_{n-i}^{\varepsilon}(q^{*})$ and $L_0 \cong SL_i^{\varepsilon}(q^{*})$ for some $1 \le i < n$. Moreover, the $2$-components of $C$ are precisely the quasisimple elements of $\lbrace K_0, L_0 \rbrace$. 
			\item[(2)] $q^{*} = 3$, $O^2(C) = K_0L_0$, where $K_0$ and $L_0$ are subnormal subgroups of $C$ such that $K_0 \cong SL_{n-i}^{\varepsilon}(q^{*})$ and $L_0 \cong SL_i^{\varepsilon}(q^{*})$ for some $1 \le i < n$. Moreover, the $2$-components of $C$ are precisely the quasisimple elements of $\lbrace K_0, L_0 \rbrace$. 
			\item[(3)] $C$ has precisely one $2$-component, and this $2$-component is isomorphic to a nontrivial quotient of $SL_{n/2}((q^{*})^2)$. 
		\end{enumerate}   
		
		As seen above, $C_{\mathcal{G}}(\langle t \rangle) = \mathcal{F}_{C_R(t)}(C_{G_0}(t))$ has a component isomorphic to the $2$-fusion system of $SL_{n-2}(q)$. By Proposition \ref{subsystems induced by 2-components}, this component is induced by a $2$-component of $C_{G_0}(t)$. In view of the preceding observations, we can conclude that $C_{G_0}(t)$ has subgroups $K_0$ and $L_0$ with $K_0 \cong SL_{n-2}^{\varepsilon}(q^{*})$ and $L_0 \cong SL_2(q^{*})$ such that $O^{2'}(O^2(C_{G_0}(t))) = K_0L_0$ if $q^{*} \ne 3$ and $O^2(C_{G_0}(t)) = K_0 L_0$ if $q^{*} = 3$. 
		
		Clearly, $O^{2'}(O^2(C_{G_0}(t))) \le O^{2'}(O^2(C_G(t)))$ and $O^2(C_{G_0}(t)) \le O^2(C_G(t))$. Lemma \ref{KL_in_G0_lemma} implies that $K_0 L_0 \le KL$. If $n$ is odd, then it is easy to see that $\vert K_0 L_0 \vert = \vert K_0 \vert \vert L_0 \vert \ge \vert K \vert \vert L \vert = \vert KL \vert$. If $n$ is even, then one can easily see that $\vert K_0 L_0 \vert = \frac{1}{2} |K_0||L_0| \ge \frac{1}{2} |K||L| = |KL|$. Consequently, $K_0L_0 \le KL$ and $|K_0L_0| \ge |KL|$. It follows that $KL = K_0L_0 \le G_0$.
	\end{proof}
	
	\begin{corollary}
		\label{all_included}
		Let $x$ be an involution of $G_0$ which is $G$-conjugate to $t$. Let $L_0$ be the unique normal $SL_2(q^{*})$-subgroup of $C_G(x)$, and let $K_0$ be the component of $C_G(x)$ different from $L_0$. Then we have $K_0 L_0 \le G_0$. 
	\end{corollary} 
	
	\begin{proof}
		Since $t \in G_0$, we see from Lemma \ref{sylow_contained} that there is some $g \in G_0$ with $x = t^g$. Clearly, $(K_0L_0) = (KL)^g$, and so $K_0L_0 \le G_0$ by Lemma \ref{KL_in_G0}.
	\end{proof}
	
	\begin{lemma}
		\label{N_G_S_in_normalizer}
		We have $N_G(S) \le N_G(G_0)$. 
	\end{lemma} 
	
	\begin{proof}
		Set $M := N_G(G_0)$. Let $s \in N_S(S \cap M)$, and let $1 \le i \le n-1$. We have $t_{\lbrace i,i+1 \rbrace} \in S \cap L_{\lbrace i,i+1 \rbrace} \le S \cap G_0 \le S \cap M$, and hence $(t_{\lbrace i,i+1 \rbrace})^s \in S \cap M \le M$. Since $G_0$ has odd index in $M$ by Lemma \ref{sylow_contained}, we even have $(t_{\lbrace i,i+1 \rbrace})^s \in G_0$. Corollary \ref{all_included} implies that $(L_{\lbrace i,i+1 \rbrace})^s \le G_0$. So we have $s \in M$ by the definition of $G_0$. Thus $N_S(S \cap M) = S \cap M$ and hence $S \le M$. It is clear that $C_G(S) \le M$. Using Lemma \ref{PSL(n,q)-automorphisms of S}, we conclude that $N_G(S) = SC_G(S) \le M$.
	\end{proof} 
	
	\begin{lemma}
		\label{centralizers_included}
		If $x$ is an involution of $S$, then $C_G(x) \le N_G(G_0)$. 
	\end{lemma}
	
	\begin{proof}
		Set $M := N_G(G_0)$. 
		
		We begin by proving that $C_G(t) \le M$. We have $K \le G_0 \le M$ by Lemma \ref{KL_in_G0} and $C_G(t) = K N_{C_G(t)}(X_1)$ by the Frattini argument. Also, $N_{C_G(t)}(X_1) = T C_{C_G(t)}(X_1)$ as a consequence of Lemma \ref{automorphisms of X1}, and $T \le M$ by Lemma \ref{N_G_S_in_normalizer}. So it suffices to show that $C_{C_G(t)}(X_1) \le M$. 
		
		Let $z \in C_{C_G(t)}(X_1)$. In order to prove $z \in M$, it is enough to show that $(L_{\lbrace i, i+1 \rbrace})^z \le G_0$ for all $1 \le i < n$. If $1 \le i < n$ and $i \ne n-2$, we have $z \in C_G(t_{\lbrace i,i+1 \rbrace})$ and hence $(L_{\lbrace i,i+1 \rbrace})^z = L_{\lbrace i,i+1 \rbrace} \le G_0$. It remains to show that $(L_{\lbrace n-2,n-1 \rbrace})^z \le G_0$. Since $\mathcal{F}_S(G) = \mathcal{F}_S(PSL_n(q))$, there is some $g \in G$ with $t^g = u$, $u^g = t$ and $(t_{\lbrace 2,3 \rbrace})^g = t_{\lbrace n-2,n-1 \rbrace}$. From the definition of $L$ (Proposition \ref{construction of L}), it is easy to see that $L_{\lbrace 1,2 \rbrace} = A_1 \le K$. Since $u = t_{\lbrace 1,2 \rbrace}$ and $t_{\lbrace 2,3 \rbrace}$ are $K$-conjugate, we thus have $L_{\lbrace 2,3 \rbrace} \le K \le L_{2'}(C_G(t))$. Hence $L_{\lbrace n-2,n-1 \rbrace} = (L_{\lbrace 2,3 \rbrace})^g \le L_{2'}(C_G(t))^g = L_{2'}(C_G(u))$. Since $z$ centralizes $u$, it follows that $(L_{\lbrace n-2,n-1 \rbrace})^z \le L_{2'}(C_G(u))$. From Corollary \ref{all_included}, we see that $L_{2'}(C_G(u)) \le G_0$. So we have $(L_{\lbrace n-2,n-1 \rbrace})^z \le G_0$, and it follows that $C_{C_G(t)}(X_1) \le M$. Consequently, $C_G(t) \le M$.  
		
		Since $G_0$ has odd index in $M$ by Lemma \ref{sylow_contained}, we see from Lemma \ref{N_G_S_in_normalizer} that $S \le G_0$. Also, $\mathcal{F}_S(G_0) = \mathcal{F}_S(G)$ by Lemma \ref{sylow_contained}. As $C_G(t) \le M$, it follows that $C_G(x) \le M$ for any involution $x$ of $S$ which is $G$-conjugate to $t$. 
		
		Assume now that $x$ is an involution of $S$ which is $G$-conjugate to $t_i$ for some even natural number $i$ with $4 \le i < n$ such that $i \le \frac{n}{2}$ if $n$ is even. We are going to show that $C_G(x) \le M$. Arguing by induction over $i$ and using the preceding observations, we may assume that for each even $2 \le j < i$ and each involution $y$ of $S$ which is $G$-conjugate to $t_j$, we have $C_G(y) \le M$. Furthermore, we may assume that $\langle x \rangle$ is fully $\mathcal{F}_S(G)$-centralized since $\mathcal{F}_S(G) = \mathcal{F}_S(G_0)$. 
		
		As a consequence of Lemma \ref{3-generation of involution centralizers}, $C_G(x)$ is generated by the normalizers $N_{C_G(x)}(U)$, where $U$ is a subgroup of $C_S(x)$ containing a $G$-conjugate of $t_j$ for some even $2 \le j < i$. We show that each such normalizer is contained in $M$. Thus let $U$ be a subgroup of $C_S(x)$ and let $y$ be an element of $U$ which is $G$-conjugate to $t_j$ for some even $2 \le j < i$. Also, let $g \in N_{C_G(x)}(U)$. Then $y^g \in U \le C_S(x) \le S$. Since $\mathcal{F}_S(G_0) = \mathcal{F}_S(G)$, we have that $y$ and $y^g$ are $G_0$-conjugate. Hence, there is some $m \in G_0$ with $y^g = y^m$. We have $mg^{-1} \in C_G(y) \le M$. This implies $g \in M$ since $m \in G_0 \le M$. So we have $N_{C_G(x)}(U) \le M$ and hence $C_G(x) \le M$. 
		
		Assume now that $x$ is an arbitrary involution of $S$. We are going to prove that $C_G(x) \le M$. Since $\mathcal{F}_S(G) = \mathcal{F}_S(G_0)$, we may assume that $\langle x \rangle$ is fully $\mathcal{F}_S(G)$-centralized. By Corollary \ref{3-generation_conclusion}, $C_G(x)$ is $3$-generated. Therefore, $C_G(x)$ is generated by the normalizers $N_{C_G(x)}(U)$, where $U \le C_S(x)$ and $m(U) \ge 3$. Take some $U \le C_S(x)$ with $m(U) \ge 3$. By Lemma \ref{E8 subgroups of central quotients}, any $E_8$-subgroup of $S$ has an involution which is the image of an involution of $SL_n(q)$. It follows that $U$ has an element $y$ which is $G$-conjugate to $t_k$ for some even $2 \le k < n$. By the preceding observations, $C_G(y) \le M$. Arguing as above, we can conclude that $N_{C_G(x)}(U) \le M$. It follows that $C_G(x) \le M$. 
	\end{proof}
	
	\begin{proposition} 
		\label{final_proposition}
		We have $G_0 \trianglelefteq G$.  
	\end{proposition}
	
	\begin{proof}
		Suppose that $M := N_G(G_0)$ is a proper subgroup of $G$. By \cite[Proposition 17.11]{GLS2}, we may deduce from Lemmas \ref{N_G_S_in_normalizer} and \ref{centralizers_included} that $M$ is strongly embedded in $G$. Therefore, by \cite[Chapter 6, 4.4]{Suzuki}, $G$ has only one conjugacy class of involutions. On the other hand, we see from Proposition \ref{involutions of PSL(n,q)} that $G$ has at least two conjugacy classes of involutions. This contradiction shows that $M = G$. Hence $G_0 \trianglelefteq G$. 
	\end{proof} 
	
	With Propositions \ref{isomorphism_type_G0} and \ref{final_proposition}, we have completed the proof of Theorem \ref{theorem1}. 
	
	\section{Proofs of the main results}
	\label{proofs_main_results}
	 \begin{proof}[Proof of Theorem \ref{A}]
	 	By Section \ref{small_cases}, Theorem \ref{A} is true for $n \le 5$. 
	 	
	 	Suppose now that $n \ge 6$. Let $q$ be a nontrivial odd prime power, and let $G$ be a finite simple group satisfying (\ref{CK}). 
	 	
	 	Recall that a natural number $k \ge 6$ is said to satisfy $P(k)$ if whenever $q_0$ is a nontrivial odd prime power and $H$ is a finite simple group satisfying (\ref{CK}) and realizing the $2$-fusion system of $PSL_k(q_0)$, we have $H \cong PSL_k^{\varepsilon}(q^{*})$ for some nontrivial odd prime power $q^{*}$ and some $\varepsilon \in \lbrace +,- \rbrace$ with $\varepsilon q^{*} \sim q_0$. Theorem \ref{theorem1} shows that $P(k)$ is satisfied for all natural numbers $k \ge 6$. 
	 	
	 	Therefore, if the $2$-fusion system of $G$ is isomorphic to the $2$-fusion system of $PSL_n(q)$, then condition (i) of Theorem \ref{A} is satisfied. 
	 	
	 	Conversely, if one of the conditions (i), (ii), (iii) of Theorem \ref{A} is satisfied, then this can only be condition (i), and Proposition \ref{PSL_PSU_fusion} implies that the $2$-fusion system of $G$ is isomorphic to the $2$-fusion system of $PSL_n(q)$.
	 \end{proof} 
	 
	 \begin{proof}[Proof of Theorem \ref{B}]
	 	Let $q$ be a nontrivial odd prime power and let $n \ge 2$ be a natural number, where $q \equiv 1$ or $7 \mod 8$ if $n = 2$. Let $G$ be a finite simple group and $S \in \mathrm{Syl}_2(G)$. Suppose that $\mathcal{F}_S(G)$ has a normal subsystem $\mathcal{E}$ on a subgroup $T$ of $S$ such that $\mathcal{E}$ is isomorphic to the $2$-fusion system of $PSL_n(q)$ and such that $C_S(\mathcal{E}) = 1$. We have to show that $\mathcal{F}_S(G)$ is isomorphic to the $2$-fusion system of $PSL_n(q)$. 
	 	
	 	By Lemma \ref{PSL as Goldschmidt group}, $PSL_n(q)$ is not a Goldschmidt group. Applying \cite[Theorem 5.6.18]{Aschbacher2021}, we conclude that $\mathcal{E}$ is simple. We see from \cite[Theorem B]{BMO2019} that $\mathcal{E}$ is tamely realized by some finite simple group of Lie type $K$. 
	 	
	 	By Theorem \ref{A}, we have $K \cong PSL_n^{\varepsilon}(q^{*})$ for some nontrivial odd prime power $q^{*}$ and some $\varepsilon \in \lbrace +,- \rbrace$ with $\varepsilon q^{*} \sim q$. 
	 	
	 	By Propositions \ref{prop_2-nilpotence_out1} and \ref{prop_2-nilpotence_out2}, we have that $\mathrm{Out}(K)$ is $2$-nilpotent. Now Proposition \ref{corollary_oliver} implies that $\mathcal{F}_S(G)$ is tamely realized by a subgroup $L$ of $\mathrm{Aut}(K)$ containing $\mathrm{Inn}(K)$ such that the index of $\mathrm{Inn}(K)$ in $L$ is odd. By Lemma \ref{lemma_fusion_out_PSL_PSU}, the $2$-fusion system of $L$ is isomorphic to the $2$-fusion system of $\mathrm{Inn}(K) \cong K$ and hence isomorphic to the $2$-fusion system of $PSL_n(q)$. So $\mathcal{F}_S(G)$ is isomorphic to the $2$-fusion system of $PSL_n(q)$. 
	 \end{proof}
	 
	 \begin{proof}[Proof of Corollary \ref{C}]
	 	Let $q$ be a nontrivial odd prime power and let $n \ge 2$ be a natural number, where $q \equiv 1$ or $7 \mod 8$ if $n = 2$. Let $G$ be a finite simple group and let $S$ be a Sylow $2$-subgroup of $G$. Suppose that $F^{*}(\mathcal{F}_S(G))$ is isomorphic to the $2$-fusion system of $PSL_n(q)$. 
	 	
	 	We have $F^{*}(\mathcal{F}_S(G)) \trianglelefteq \mathcal{F}_S(G)$ and $C_S(F^{*}(\mathcal{F}_S(G))) = Z(F^{*}(\mathcal{F}_S(G))) = 1$. So Theorem \ref{B} implies that $\mathcal{F}_S(G)$ is isomorphic to the $2$-fusion system of $PSL_n(q)$.
	 \end{proof}

	\section*{Acknowledgements}
	This paper is based on the author’s PhD thesis, written at the University of Aberdeen under the supervision of Professor Ellen Henke and Professor Benjamin Martin. The author is deeply grateful to them for their guidance and support. Moreover, the author would like to thank Professor Ron Solomon for helpful discussions.

	\bibliographystyle{amsplain}
	\bibliography{mybibfile}

\providecommand{\bysame}{\leavevmode\hbox to3em{\hrulefill}\thinspace}
\providecommand{\MR}{\relax\ifhmode\unskip\space\fi MR }
\providecommand{\MRhref}[2]{%
  \href{http://www.ams.org/mathscinet-getitem?mr=#1}{#2}
}
\providecommand{\href}[2]{#2}
\begin{thebibliography}{10}

\bibitem{Atlas}
Rachel Abbott, John Bray, Steve Linton, Simon Nickerson, Simon Norton, Richard
  Parker, Ibrahim Suleiman, Jonathan Tripp, Peter Walsh, and Robert Wilson,
  \emph{{ATLAS of Finite Group Representations - Version 3}},
  \url{http://brauer.maths.qmul.ac.uk/Atlas/v3/}, 2021, [Online; accessed
  11-May-2021].

\bibitem{AlperinBrauerGorenstein1}
J.~L. Alperin, Richard Brauer, and Daniel Gorenstein, \emph{Finite groups with
  quasi-dihedral and wreathed {S}ylow {$2$}-subgroups}, Trans. Amer. Math. Soc.
  \textbf{151} (1970), 1--261. \MR{284499}

\bibitem{AlperinBrauerGorenstein2}
\bysame, \emph{Finite simple groups of {$2$}-rank two}, Scripta Math.
  \textbf{29} (1973), no.~3-4, 191--214. \MR{401902}

\bibitem{Andersen}
Kasper K.~S. Andersen, Bob Oliver, and Joana Ventura, \emph{Reduced, tame and
  exotic fusion systems}, Proc. Lond. Math. Soc. (3) \textbf{105} (2012),
  no.~1, 87--152. \MR{2948790}

\bibitem{FiniteGroupTheory}
M.~Aschbacher, \emph{Finite group theory}, second ed., Cambridge Studies in
  Advanced Mathematics, vol.~10, Cambridge University Press, Cambridge, 2000.
  \MR{1777008}

\bibitem{Aschbacher1974}
Michael Aschbacher, \emph{Finite groups with a proper {$2$}-generated core},
  Trans. Amer. Math. Soc. \textbf{197} (1974), 87--112. \MR{364427}

\bibitem{generalizedfittingsubsystem}
\bysame, \emph{The generalized {F}itting subsystem of a fusion system}, Mem.
  Amer. Math. Soc. \textbf{209} (2011), no.~986, vi+110. \MR{2752788}

\bibitem{Aschbacher2019}
\bysame, \emph{On fusion systems of component type}, Mem. Amer. Math. Soc.
  \textbf{257} (2019), no.~1236, v+182. \MR{3898993}

\bibitem{Aschbacher2020}
\bysame, \emph{The 2-fusion system of an almost simple group}, J. Algebra
  \textbf{561} (2020), 5--16. \MR{4135537}

\bibitem{Aschbacher2021}
\bysame, \emph{Quaternion fusion packets}, Contemporary Mathematics, vol. 765,
  American Mathematical Society, [Providence], RI, [2021] \copyright 2021.
  \MR{4240597}

\bibitem{AKO}
Michael Aschbacher, Radha Kessar, and Bob Oliver, \emph{Fusion systems in
  algebra and topology}, London Mathematical Society Lecture Note Series, vol.
  391, Cambridge University Press, Cambridge, 2011. \MR{2848834}

\bibitem{AO}
Michael Aschbacher and Bob Oliver, \emph{Fusion systems}, Bull. Amer. Math.
  Soc. (N.S.) \textbf{53} (2016), no.~4, 555--615. \MR{3544261}

\bibitem{Ballester}
Adolfo Ballester-Bolinches, Ram\'{o}n Esteban-Romero, and Mohamed Asaad,
  \emph{Products of finite groups}, De Gruyter Expositions in Mathematics,
  vol.~53, Walter de Gruyter GmbH \& Co. KG, Berlin, 2010. \MR{2762634}

\bibitem{BennettShpectorov}
Curtis~D. Bennett and Sergey Shpectorov, \emph{A new proof of a theorem of
  {P}han}, J. Group Theory \textbf{7} (2004), no.~3, 287--310. \MR{2062999}

\bibitem{BMO2012}
Carles Broto, Jesper~M. M\o{}ller, and Bob Oliver, \emph{Equivalences between
  fusion systems of finite groups of {L}ie type}, J. Amer. Math. Soc.
  \textbf{25} (2012), no.~1, 1--20. \MR{2833477}

\bibitem{BMO2019}
\bysame, \emph{Automorphisms of fusion systems of finite simple groups of {L}ie
  type}, Mem. Amer. Math. Soc. \textbf{262} (2019), no.~1267, 1--120.
  \MR{4071770}

\bibitem{BurnessGiudici}
Timothy~C. Burness and Michael Giudici, \emph{Classical groups, derangements
  and primes}, Australian Mathematical Society Lecture Series, vol.~25,
  Cambridge University Press, Cambridge, 2016. \MR{3443032}

\bibitem{CarterFong}
Roger Carter and Paul Fong, \emph{The {S}ylow {$2$}-subgroups of the finite
  classical groups}, J. Algebra \textbf{1} (1964), 139--151. \MR{166271}

\bibitem{Craven}
David~A. Craven, \emph{The theory of fusion systems}, Cambridge Studies in
  Advanced Mathematics, vol. 131, Cambridge University Press, Cambridge, 2011,
  An algebraic approach. \MR{2808319}

\bibitem{Dieudonne}
Jean Dieudonn\'{e}, \emph{On the automorphisms of the classical groups. {W}ith
  a supplement by {L}oo-{K}eng {H}ua}, Mem. Amer. Math. Soc. \textbf{2} (1951),
  vi+122. \MR{45125}

\bibitem{FineRosenberger}
Benjamin Fine and Gerhard Rosenberger, \emph{Number theory}, Birkh\"{a}user
  Boston, Inc., Boston, MA, 2007, An introduction via the distribution of
  primes. \MR{2261276}

\bibitem{Fumagalli}
Francesco Fumagalli, \emph{On the group of automorphisms of finite wreath
  products}, Rend. Sem. Mat. Univ. Padova \textbf{115} (2006), 15--28.
  \MR{2245584}

\bibitem{Glauberman}
George Glauberman, \emph{Central elements in core-free groups}, J. Algebra
  \textbf{4} (1966), 403--420. \MR{202822}

\bibitem{Gorenstein}
Daniel Gorenstein, \emph{Finite groups}, second ed., Chelsea Publishing Co.,
  New York, 1980. \MR{569209}

\bibitem{Gorenstein1983}
\bysame, \emph{Finite simple groups}, University Series in Mathematics, Plenum
  Publishing Corp., New York, 1982, An introduction to their classification.
  \MR{698782}

\bibitem{Gorenstein83}
\bysame, \emph{The classification of finite simple groups. {V}ol. 1}, The
  University Series in Mathematics, Plenum Press, New York, 1983, Groups of
  noncharacteristic $2$ type. \MR{746470}

\bibitem{GLS0}
Daniel Gorenstein, Richard Lyons, and Ronald Solomon, \emph{The classification
  of the finite simple groups}, Mathematical Surveys and Monographs, vol.~40,
  American Mathematical Society, Providence, RI, 1994. \MR{1303592}

\bibitem{GLS2}
\bysame, \emph{The classification of the finite simple groups. {N}umber 2.
  {P}art {I}. {C}hapter {G}}, Mathematical Surveys and Monographs, vol.~40,
  American Mathematical Society, Providence, RI, 1996, General group theory.
  \MR{1358135}

\bibitem{GLS3}
\bysame, \emph{The classification of the finite simple groups. {N}umber 3.
  {P}art {I}. {C}hapter {A}}, Mathematical Surveys and Monographs, vol.~40,
  American Mathematical Society, Providence, RI, 1998, Almost simple
  $K$-groups. \MR{1490581}

\bibitem{GLS8}
\bysame, \emph{The classification of the finite simple groups. {N}umber 8.
  {P}art {III}. {C}hapters 12--17. {T}he generic case, completed}, Mathematical
  Surveys and Monographs, vol.~40, American Mathematical Society, Providence,
  RI, 2018. \MR{3887657}

\bibitem{GorensteinWalter}
Daniel Gorenstein and John~H. Walter, \emph{The characterization of finite
  groups with dihedral {S}ylow {$2$}-subgroups. {I}}, J. Algebra \textbf{2}
  (1965), 85--151. \MR{177032}

\bibitem{GW}
\bysame, \emph{Balance and generation in finite groups}, J. Algebra \textbf{33}
  (1975), 224--287. \MR{357583}

\bibitem{Grove}
Larry~C. Grove, \emph{Classical groups and geometric algebra}, Graduate Studies
  in Mathematics, vol.~39, American Mathematical Society, Providence, RI, 2002.
  \MR{1859189}

\bibitem{Henke}
Ellen Henke, \emph{Products in fusion systems}, J. Algebra \textbf{376} (2013),
  300--319. \MR{3003728}

\bibitem{Huppert}
B.~Huppert, \emph{Endliche {G}ruppen. {I}}, Die Grundlehren der mathematischen
  Wissenschaften, Band 134, Springer-Verlag, Berlin-New York, 1967.
  \MR{0224703}

\bibitem{Kondratev}
A.~S. Kondrat'ev, \emph{Normalizers of {S}ylow 2-subgroups in finite simple
  groups}, Mat. Zametki \textbf{78} (2005), no.~3, 368--376. \MR{2227510}

\bibitem{KurzweilStellmacher}
Hans Kurzweil and Bernd Stellmacher, \emph{The theory of finite groups},
  Universitext, Springer-Verlag, New York, 2004, An introduction, Translated
  from the 1998 German original. \MR{2014408}

\bibitem{LiZhangYi}
Changwen Li, Xuemei Zhang, and Xiaolan Yi, \emph{On partially
  {$\tau$}-quasinormal subgroups of finite groups}, Hacet. J. Math. Stat.
  \textbf{43} (2014), no.~6, 953--961. \MR{3331152}

\bibitem{Linckelmann}
Markus Linckelmann, \emph{Introduction to fusion systems}, Group representation
  theory, EPFL Press, Lausanne, 2007, pp.~79--113. \MR{2336638}

\bibitem{Mason2}
David~R. Mason, \emph{Finite simple groups with {S}ylow {$2$}-subgroup dihedral
  wreath {$Z_{2}$}}, J. Algebra \textbf{26} (1973), 10--68. \MR{318294}

\bibitem{Mason1}
\bysame, \emph{Finite simple groups with {S}ylow {$2$}-subgroups of type {${\rm
  PSL}(4,\,q),\,q$} odd}, J. Algebra \textbf{26} (1973), 75--97. \MR{318295}

\bibitem{Mason3}
\bysame, \emph{Finite simple groups with {S}ylow {$2$}-subgroups of type
  {$PSL(5, q), q$} odd}, Math. Proc. Cambridge Philos. Soc. \textbf{79} (1976),
  no.~2, 251--269. \MR{396737}

\bibitem{MSS}
Ulrich Meierfrankenfeld, Bernd Stellmacher, and Gernot Stroth, \emph{Finite
  groups of local characteristic {$p$}: an overview}, Groups, combinatorics \&
  geometry ({D}urham, 2001), World Sci. Publ., River Edge, NJ, 2003,
  pp.~155--192. \MR{1994966}

\bibitem{Oliver2016}
Bob Oliver, \emph{Reductions to simple fusion systems}, Bull. Lond. Math. Soc.
  \textbf{48} (2016), no.~6, 923--934. \MR{3608937}

\bibitem{Phan1970}
Kok-wee Phan, \emph{A theorem on special linear groups}, J. Algebra \textbf{16}
  (1970), 509--518. \MR{269732}

\bibitem{Phan1972}
\bysame, \emph{A characterization of the finite groups {${\rm PSL}(n,\,q)$}},
  Math. Z. \textbf{124} (1972), 169--185. \MR{296173}

\bibitem{Phan1975}
Kok~Wee Phan, \emph{A characterization of the finite groups {${\rm
  PSU}(n,\,q)$}}, J. Algebra \textbf{37} (1975), no.~2, 313--339. \MR{390044}

\bibitem{Steinberg}
Robert Steinberg, \emph{Lectures on {C}hevalley groups}, University Lecture
  Series, vol.~66, American Mathematical Society, Providence, RI, 2016, Notes
  prepared by John Faulkner and Robert Wilson, Revised and corrected edition of
  the 1968 original [ MR0466335], With a foreword by Robert R. Snapp.
  \MR{3616493}

\bibitem{Suzuki}
Michio Suzuki, \emph{Group theory. {II}}, Grundlehren der Mathematischen
  Wissenschaften [Fundamental Principles of Mathematical Sciences], vol. 248,
  Springer-Verlag, New York, 1986, Translated from the Japanese. \MR{815926}

\end{thebibliography}
	
\end{document}